\documentclass[reqno,a4paper]{amsart}
\usepackage[english]{babel}

\usepackage{graphicx, amsmath, amsfonts,amssymb,amsthm, amscd,amsbsy,amstext, tikz,tikz-cd, mathrsfs, mathtools,mathpazo, enumerate,enumitem, tensor, a4wide,xfrac} 
\usepackage{latexsym,ifthen,stmaryrd,fancyhdr,empheq,verbatim, epigraph, youngtab} 
\usepackage{dynkin-diagrams}
\mathtoolsset{showonlyrefs,showmanualtags}

\usepackage{microtype}

\usepackage{hyperref}
\hypersetup{anchorcolor=black, linktocpage=true,
colorlinks=true,
citecolor=black,
linkcolor=black,
urlcolor=black,
breaklinks=true,
plainpages=true
}

\usetikzlibrary{positioning,shapes,shadows,arrows}

\usepackage{color}
\definecolor{MyDarkBlue}{rgb}{0.15,0.25,0.45}
\definecolor{blue2}{rgb}{0.0,0.43,0.84}
\definecolor{red2}{rgb}{0.74,0.0,0.53}

\usepackage[hyperpageref]{backref}

\usepackage[utf8x]{inputenc}

\usepackage[toc,page]{appendix}
\usepackage[mathscr]{euscript} 
\allowdisplaybreaks

\newif\ifpersonal
\personaltrue 
\ifpersonal
\newcommand{\personal}[1]{\textcolor[rgb]{0,0,1}{(Personal: #1)}}
\newcommand{\todo}[1]{\textcolor{red}{(Todo: #1)}}
\else
\newcommand*{\personal}[1]{\ignorespaces}
\newcommand*{\todo}[1]{\ignorespaces}
\fi

\newcommand{\calA}{\mathcal{A}}

\newcommand{\calC}{\mathcal{C}}

\newcommand{\calE}{\mathcal{E}}
\newcommand{\calF}{\mathcal{F}}
\newcommand{\calG}{\mathcal{G}}
\newcommand{\calH}{\mathcal{H}}
\newcommand{\calI}{\mathcal{I}}
\newcommand{\calJ}{\mathcal{J}}
\newcommand{\calK}{\mathcal{K}}
\newcommand{\calL}{\mathcal{L}}
\newcommand{\calM}{\mathcal{M}}
\newcommand{\calN}{\mathcal{N}}

\newcommand{\calP}{\mathcal{P}}
\newcommand{\calQ}{\mathcal{Q}}

\newcommand{\calS}{\mathcal{S}}
\newcommand{\calT}{\mathcal{T}}

\newcommand{\calV}{\mathcal{V}}
\newcommand{\calX}{\mathcal{X}}
\newcommand{\calY}{\mathcal{Y}}

\newcommand{\calZ}{\mathcal{Z}}

\newcommand{\R}{\mathbb{R}}
\newcommand{\N}{\mathbb{N}}
\newcommand{\C}{\mathbb{C}}
\newcommand{\Z}{\mathbb{Z}}

\newcommand{\PP}{\mathbb{P}}
\newcommand{\A}{\mathbb{A}}

\newcommand{\LL}{\mathbb{L}}

\newcommand{\scrA}{\mathscr{A}}
\newcommand{\scrB}{\mathscr{B}}
\newcommand{\scrC}{\mathscr{C}}

\newcommand{\scrE}{\mathscr{E}}
\newcommand{\scrF}{\mathscr{F}}
\newcommand{\scrG}{\mathscr{G}}

\newcommand{\scrJ}{\mathscr{J}}
\newcommand{\scrO}{\mathscr{O}}

\newcommand{\scrT}{\mathscr{T}}

\newcommand{\scrQ}{\mathscr{Q}}

\newcommand{\sfP}{\mathsf{P}}

\newcommand{\bfE}{\mathbf{E}}

\newcommand{\bfQ}{\mathbf{Q}}

\newcommand{\ev}{\mathsf{ev}}

\newcommand{\ch}{\mathsf{ch}}
\newcommand{\id}{\mathsf{id}}
\newcommand{\Spec}{\mathsf{Spec}}

\newcommand{\Hom}{\mathsf{Hom}}
\newcommand{\calHom}{\mathcal{H}\mathsf{om}}
\newcommand{\Ext}{\mathsf{Ext}}
\newcommand{\calExt}{\mathcal{E}\mathsf{xt}}
\newcommand{\Aut}{\mathsf{Aut}}
\newcommand{\End}{\mathsf{End}}
\newcommand{\Tor}{\mathsf{Tor}}
\newcommand{\Pic}{\mathsf{Pic}}

\newcommand{\catCoh}{\mathsf{Coh}}
\newcommand{\catPerf}{\mathsf{Perf}}
\newcommand{\catQCoh}{\mathsf{QCoh}}
\newcommand{\catDb}{\mathsf{D}^\mathsf{b}}

\newcommand{\dSt}{\mathsf{dSt}}
\newcommand{\Sh}{\mathsf{Sh}}
\newcommand{\dAff}{\mathsf{dAff}}

\newcommand{\derivedPerf}{\mathbf{Perf}}
\newcommand{\derivedCoh}{\mathbf{Coh}}
\newcommand{\derivedextCoh}{\mathbf{Coh}^{\mathsf{ext}}}

\newcommand{\derivedfSP}{f\textrm{-}\mathbf{SP}}

\newcommand{\Hilb}{\mathsf{Hilb}}
\newcommand{\Quot}{\mathsf{Quot}}
\newcommand{\FHilb}{\mathsf{FHilb}}

\newcommand{\Coh}{\calC oh}

\newcommand{\SP}{\mathcal{SP}}
\newcommand{\fSP}{f\textrm{-}\mathcal{SP}}

\newcommand{\trunc}[1]{\tensor*[^{\mathsf{cl}}]{#1}{}}

\makeatletter
\newcommand{\ostar}{\mathbin{\mathpalette\make@circled\star}}
\newcommand{\make@circled}[2]{%
\ooalign{$\m@th#1\smallbigcirc{#1}$\cr\hidewidth$\m@th#1#2$\hidewidth\cr}%
}
\newcommand{\smallbigcirc}[1]{%
\vcenter{\hbox{\scalebox{0.77778}{$\m@th#1\bigcirc$}}}%
}
\makeatother

\newcommand{\triend}{\parbox{2mm}{\hfill} \hfill\mbox{\hspace{0.2mm}}\hfill$\triangle$}
\newcommand{\ocend}{\parbox{2mm}{\hfill} \hfill\mbox{\hspace{0.2mm}}\hfill$\oslash$}

\newtheorem{theorem}{Theorem}

\newtheorem{theoremintroduction}{Theorem}

\newtheorem{proposition}[theorem]{Proposition}
\newtheorem{lemma}[theorem]{Lemma}
\newtheorem{corollary}[theorem]{Corollary}

\newtheorem{corollary*}{Corollary}
\newtheorem*{theorem*}{Theorem}
\newtheorem*{proposition*}{Proposition}
\newtheorem*{conjecture*}{Conjecture}

\numberwithin{equation}{section}
\numberwithin{theorem}{section}

\theoremstyle{remark}
\newtheorem{ex}[theorem]{Example}
\newenvironment{example}{\begin{ex}}{\triend\end{ex}}

\theoremstyle{remark}
\newtheorem{rem}[theorem]{Remark}
\newtheorem{notati}[theorem]{Notation}
\newtheorem{war}[theorem]{Warning}
\newenvironment{warning}{\begin{war}}{\triend\end{war}}
\newenvironment{remark}{\begin{rem}}{\triend\end{rem}}

\newtheorem*{rem*}{Remark}
\newenvironment{remark*}{\begin{rem*}}{\triend\end{rem*}}

\theoremstyle{definition}
\newtheorem{defin}[theorem]{Definition}
\newtheorem*{defin*}{Definition}
\newenvironment{definition}{\begin{defin}}{\ocend\end{defin}}
\newenvironment{definition*}{\begin{defin*}}{\ocend\end{defin*}}
\newtheorem{constru}[theorem]{Construction}

\newcounter{assumption}
\newtheorem{assum}[assumption]{Assumption}

\newenvironment{assumption}{\begin{assum}}{\ocend\end{assum}}

\title[Flops and Hilbert schemes of space curve singularities]{Flops and Hilbert schemes of space curve singularities}

\author[D.-E.~Diaconescu]{Duiliu-Emanuel Diaconescu}
\address[Duiliu-Emanuel Diaconescu]{New High Energy Theory Center - Serrin Building, Rutgers, The State University Of New Jersey, 126 Frelinghuysen Rd., Piscataway, NJ 08854-8019, USA}
\curraddr{}
\email{\href{mailto:duiliu@physics.rutgers.edu}{duiliu@physics.rutgers.edu}}

\author[M.~Porta]{Mauro Porta}
\address[Mauro Porta]{Institut de recherche mathématique avancée (IRMA), Université de Strasbourg, France}
\curraddr{}
\email{\href{mailto:porta@math.unistra.fr}{porta@math.unistra.fr}}

\author[F.~Sala]{Francesco Sala}
\address[Francesco Sala]{Università di Pisa, Dipartimento di Matematica, Largo Bruno Pontecorvo 5, 56127 Pisa (PI), Italy}
\address{Kavli IPMU (WPI), UTIAS, The University of Tokyo, Kashiwa, Chiba 277-8583, Japan}
\curraddr{}
\email{\href{mailto:francesco.sala@unipi.it}{francesco.sala@unipi.it}}

\author[A.~Vosoughinia]{Arian Vosoughinia}
\address[Arian Vosoughinia]{New High Energy Theory Center - Serrin Building, Rutgers, The State University Of New Jersey, 126 Frelinghuysen Rd., Piscataway, NJ 08854-8019, USA}
\curraddr{}
\email{\href{mailto:arianvwr@gmail.com}{arianvwr@gmail.com}}

\thanks{The work of F.~S. was partially supported by JSPS KAKENHI Grant Number JP21K03197, while the work of D.-E.~D. and A.~V. was partially supported by NSF grant DMS-1802410.}



\begin{document}

\begin{abstract}
	Using pagoda flop transitions between smooth projective threefolds, a relation is derived between the Euler numbers of moduli spaces of stable pairs which are scheme-theoretically supported on a fixed singular space curve and Euler numbers of Flag Hilbert schemes associated to a plane curve singularity. When the space curve singularity is locally complete intersection, one obtains a relation between the latter and Euler numbers of Hilbert schemes of the space curve singularity. It is also shown that this relation yields explicit results for a class of torus-invariant locally complete intersection singularities. These results spark a series of open questions in low-dimensional topology, combinatorics, and representation theory.
\end{abstract}

\maketitle\thispagestyle{empty}

\tableofcontents


\section{Introduction}

The topology of punctual Hilbert schemes of plane curve singularities has been studied very intensively in the mathematical literature due to its connections with knot polynomials \cite{CDGZ-Alexander-polynomial,OS-HOMFLY,ORS-HOMFLY,Maulik-Stable-pairs} and representation theory \cite{GORS-Knots, Rennemo-Hilbert,GK-Hilbert}. In particular, the Oblomkov-Shende conjecture \cite{OS-HOMFLY}, proven by Maulik \cite{Maulik-Stable-pairs}, provides a remarkable formula relating the generating function of Euler numbers of punctual Hilbert schemes of a 
plane curve singularity to the HOMFLY polynomial of its link. In contrast, punctual Hilbert schemes of space curve singularities have been studied to a much lesser degree in the literature. The main result known to date is due to \cite{BRV-Motivic}, which proves that the motivic zeta function of the Hilbert of a reduced curve is rational. 

The present study is motivated by the absence of concrete results concerning the topology of punctual Hilbert schemes of space curve singularities, as well as possible connections to knot polynomials and representation theory. One problem that immediately manifests itself is the lack of a systematic classification of such singularities. Namely, it is a priori far from clear if a particular class of space curve singularities can be singled out as a good starting point for this purpose. The main idea of the present paper is that such a class consists of space curve singularities which are related to plane curve singularities by flop transitions\footnote{A similar approach was employed in \cite{Maulik-Stable-pairs} in the proof of the Oblomkov-Shende conjecture, as well as in \cite{Toda-S-duality} for punctual Hilbert schemes of  Kleinian surface singularities of type A.}. As a first step in this direction, it will be shown in this paper that flop transitions lead to explicit results for Euler numbers of punctual Hilbert schemes of certain space curve singularities. Moreover, several possible open questions in low-dimensional topology, combinatorics, and representation theory are briefly discussed in \S\ref{sec:openquest-top} and \S\ref{sec:openquest-rep}. 

\subsection{Main results}

Let 
\begin{align}\label{eq:flop-diagram-O}
	\begin{tikzcd}[ampersand replacement=\&]
		Y^+ \ar{dr}[swap]{f^+} \ar[dashed]{rr}{\phi}\& \& Y^- \ar{dl}{f^-}\\
		\& X \&
	\end{tikzcd}
\end{align}
be a \textit{flop transition} between smooth connected projective complex threefolds, where $f^\pm$ are \textit{contractions}\footnote{In the sense of Definition~\ref{def:contraction}.} and $X$ has a unique singular point $\nu\in X$ so that the scheme-theoretic exceptional loci $\Sigma^\pm \coloneqq (f^\pm)^{-1}(\nu)$ are smooth connected $(0,-2)$-curves on $Y^\pm$. Such transitions were first studied in \cite{Reid-Minimal-Models} for complex analytic spaces, and more recently, in \cite{Spherical_twists}, and, as part of a more general framework, in \cite{DW16}, for quasi-projective threefolds. In particular, it was shown in \textit{loc. cit.} that any 
$(0,-2)$ curve $\Sigma$ in a smooth threefold $Y$ determines uniquely a numerical invariant $n\in \Z$, $n \geq 2$, which admits several equivalent definitions. Moreover, given a diagram of the form \eqref{eq:flop-diagram-A}, the numerical invariants $n^\pm$ associated to the curves  $\Sigma^\pm$ coincide, and their common value will be denoted by $n$. 

For a \textit{contraction} $f\colon Y\to X$ of a $(0,-2)$-curve $\Sigma\subset Y$, in \cite{Reid-Minimal-Models}, the invariant $n$ was defined as the width of $\Sigma$, i.e., the maximal length of a trivial thickening $\Sigma \subset \Sigma \times\Spec\, \C[t]/(t^k)\subset Y$. More recently, the same invariant $n$ was identified in \cite{DW16} with the length of the \textit{contraction algebra} of $\Sigma$ defined in \cite{Spherical_twists, DW16}. For the purposes of the present paper, it suffices to note that the contraction algebra controls the non-commutative deformation theory 
of the structure sheaf $\scrO_\Sigma$ as an $\scrO_Y$-module. In particular, it was shown in \cite{DW16} that the component of the moduli space of simple coherent sheaves on $Y$ containing $[\scrO_\Sigma]$ is isomorphic to $\Spec\, \C[t]/(t^n)$ (see Theorem~\ref{thm:moduli-sigma} in Appendix~\ref{sec:theorem-filtration}.) We prove in Proposition \ref{prop:hilb-sigma}, that a similar result holds for the connected component of Hilbert schemes of curves on $Y$ containing $[\Sigma \subset Y]$. 

Furthermore, Theorem~\ref{thm:filtration} shows that on each side of the transition \eqref{eq:flop-diagram-O} there exists a unique closed stratification
\begin{align}
	\Sigma^\pm=\Sigma_1^\pm \subset \Sigma_2^\pm \subset \cdots \subset \Sigma_n^\pm
\end{align}
whose structure sheaves fit into exact sequences of the form 
\begin{align}
	0\longrightarrow \scrO_{\Sigma_i^\pm} \longrightarrow \scrO_{\Sigma_{i+1}^\pm} \longrightarrow \scrO_{\Sigma^\pm} \longrightarrow 0 \ ,
\end{align}
for $1\leq i \leq n-1$.

Now consider a smooth connected Weil divisor $W\subset X$ containing the singular point, and let 
\begin{align}\label{eq:surface-diagram-O} 
	\begin{tikzcd}[ampersand replacement=\&]
		S^+ \ar{dr}[swap]{f_{S^+}} \ar[dashed]{rr}{\phi_S}\& \& S^- \ar{dl}{f_{S^-}}\\
		\& W \&
	\end{tikzcd}
\end{align}
be its strict transforms under the two contractions. We shall assume that the canonical projection $f_{S^+}\colon S^+\to W$ is an isomorphism and $S^+$ intersects $\Sigma^+$ transversely at a single point $p$. Furthermore, we shall assume that $S^-$ is reduced and irreducible, with a zero-dimensional singular locus. 

The above geometric framework is  specified in the main text through Assumptions~\ref{assumption:f}, \ref{assumption:divisor}, and \ref{assumption:surface}. The first assumption encodes the basic properties of diagram \eqref{eq:flop-diagram-O}, while the second states existence of the pair $(S^+, W)$ in diagram \eqref{eq:surface-diagram-O}. The third encodes the basic properties of surface $S^-$ in the same diagram. The first two assumptions are systematically used in \S\ref{sec:geometric-setup} -- \S\ref{sec:flop-identity} and Appendix~\ref{sec:proof-C-framed-identity}, while the third is used only in \S\ref{subsec:transformation-classes} and \S\ref{subsec:flop-identity}. 
	
Consider a reduced irreducible plane curve $C\subset W$ so that $\nu \in C^{\mathsf{sing}}$ and let $C^\pm \subset S^\pm$ be its strict transforms. Under the current assumptions, the natural projection $C^+ \to C$ is an isomorphism and $p\in (C^+)^{\mathsf{sing}}$ is the only point of intersection of $C^+$ and $\Sigma^+$. On the other hand, $C^-\subset S^-$ is a reduced irreducible closed subscheme which intersects $\Sigma^-$ at finitely many points $y_1, \ldots, y_d$. Note that $(f_{S^-})^{-1}(C)$ is the scheme-theoretic union of two irreducible components:
\begin{align}
	(f_{S^-})^{-1}(C) = C^-\cup Z_{\Sigma^-}\ ,
\end{align}
where $(Z_{\Sigma^-})_{\mathsf{red}}=\Sigma^-$. Let $m_{\Sigma^-}\in \N$ be the \textit{multiplicity} of $Z_{\Sigma^-}$ along $\Sigma^-$, defined by $\ch_2(\scrO_{Z_{\Sigma^-}}) = m_{\Sigma^-} [\Sigma^-]$ in $N_1(Y^-)$.\footnote{Here and in the following, $\N\coloneqq\Z_{\geq 0}$.}

For $k\in \N$, let $\Hilb^k(C^+)$ denote the \textit{Hilbert scheme of $k$ points on $C^+$}: it parametrizes length $k$ zero-dimensional subschemes of $C^+$. Let $\FHilb_{T_n}^k(C^+; m)\subset \Hilb^k(C^+)\times \Hilb^{k+m}(C^+)$ be the \textit{Flag Hilbert scheme}\footnote{Following, the usual notation used in the literature for the Flag Hilbert scheme, we should have denoted it as $\FHilb_{T_n}^{[k, m]}(C^+)$: since in the following $k$ and $m$ will play different roles, we preferred to differentiate them in our notation.} parametrizing flags of ideal sheaves $\calI_1\subset \calI_2\subset \scrO_{C^+}$ so that the ideal $\calI_2$ has colength $k$ and $\calI_2/\calI_1$ is the pushforward of a length $m$ zero-dimensional sheaf on $T_n$, where $T_n$ is the scheme-theoretic intersection of $\Sigma_n$ and $S^+$ in $Y^+$. Let $\FHilb_{T_n, p}^k(C^+; m)$ be the \textit{punctual} Flag Hilbert scheme, i.e., the inverse image of the punctual Hilbert scheme $\Hilb_p^{k+m}(C^+)\subset \Hilb^{k+m}(C^+)$ via the forgetful morphism $\FHilb_{T_n}^{k}(C^+; m)\to \Hilb^{k+m}(C^+)$ sending $(\calI_1\subset \calI_2)$ to $\calI_1$. 

On the other hand, let $\SP_{C^-}(Y^-)$ be the moduli space of Pandharipande-Thomas stable pairs on $Y^-$ scheme-theoretic supported on $C^-$, as introduced in \cite{PT-Curve-counting}, i.e, the moduli space of pairs $(\calF,s)$ with $\calF$ a pure one-dimensional sheaf on $Y^-$, which is scheme-theoretically supported on $C^-$, and $s\colon \scrO_{Y^-}\to \calF$ a generically surjective section. Let $\SP_{C^-}^k(Y^-)$ denote its open and closed component defined by $\chi(\mathsf{Coker}(s))=k$. For each $1\leq i \leq d$, let $\SP_{C^-,\, y_i}^k(Y^-)\subset \SP_{C^-}^k(Y^-)$ be the reduced closed subscheme defined by the condition that $\mathsf{Coker}(s)$ is set-theoretically supported at $y_i$. Then, the first main result of this paper is:
\begin{theoremintroduction}\label{thm:main-A} 
	Let us consider a flop transition of the form \eqref{eq:flop-diagram-A} satisfying Assumptions~\ref{assumption:f}, \ref{assumption:divisor}, and \ref{assumption:surface}. Let $C\subset W$ be a reduced irreducible plane curve so that $\nu \in C^{\mathsf{sing}}$ and let $C^\pm \subset S^\pm$ be its strict transforms. Then, the following identity for generating functions of Euler numbers holds:
	\begin{align}\label{eq:main-identity-A} 
		q^{\chi(\scrO_{C^+})} \sum_{k\geq 0} \chi(\FHilb_{T_n, p}^k(C^+; m_{\Sigma^-}))q^k = q^{\chi(\scrO_{C^-})}\prod_{i=1}^d \sum_{k\geq 0} \chi(\SP_{C^-,\, y_i}^k(Y^-))q^k\ .
	\end{align}
\end{theoremintroduction}

\begin{remark*}
	By analogy to \cite[Proposition~12 and Theorem~13]{OS-HOMFLY}, note that 
	\cite[Theorem~3]{PT-BPS} implies that the right hand side of the main identity in Theorem~\ref{thm:main-A} is the expansion of a rational function of the form $P_-(q)/(1-q)$, where $P_-(q)$ is a polynomial with integer coefficients.
\end{remark*}

Furthermore, for $k\in \N$ let $\Hilb_{y_i}^k(C^-)$ be the punctual Hilbert scheme parametrizing length $k$ zero-dimensional subschemes of $C^-$, set-theoretically supported at $y_i$. Then, Theorem~\ref{thm:main-A} yields the following result via \cite[Proposition~B.5]{PT-BPS}. 
\begin{theoremintroduction}\label{thm:main-B}
	Under the same hypotheses of Theorem~\ref{thm:main-A}, assume in addition that $C^-$ has locally complete intersection singularities. Then, the following identity for generating functions of Euler numbers holds: 
	\begin{align}\label{eq:main-identity-B} 
		q^{\chi(\scrO_{C^+})} \sum_{k\geq 0} \chi(\FHilb_{T_n, p}^k(C^+; m_{\Sigma^-}))q^k = q^{\chi(\scrO_{C^-})}\prod_{i=1}^d \sum_{k\geq 0} \chi(\Hilb_{y_i}^k(C^-))q^k\ .
	\end{align}
\end{theoremintroduction}

\begin{remark*}
	Although formally setting $n=1$ does not make sense in our setting, since it would contradict the hypothesis that $\Sigma$ is a $(0,-2)$-curve, it is worth noticing that the $n=1$ analog of $\FHilb_{T_n, p}^k(C^+; m_{\Sigma^-})$ is the punctual Flag Hilbert scheme considered in \cite{OS-HOMFLY}. Thus, the LHS of \eqref{eq:main-identity-A} and \eqref{eq:main-identity-B} for $n=1$ would correspond to the coefficient of $q$-degree $2m_{\Sigma^-}-\mu+1$ of the HOMFLY polynomial of the corresponding link (up to some constant factor), where $\mu$ is the Milnor number of the singularity. Following this analogy, the computation of the LHS of \eqref{eq:main-identity-A} and \eqref{eq:main-identity-B}  for arbitrary $n\geq 2$ can be seen as a first step into a generalization of \cite{OS-HOMFLY}.
\end{remark*}

\subsection{Strategy of the proof}

The proof of Theorem~\ref{thm:main-A} uses in an essential way the notion of $f$-stable pairs introduced by Bryan and Steinberg \cite{BS-Curve-counting-crepant}  for a smooth crepant resolution of threefold singularities $f\colon Y \to X$. This datum determines a torsion pair $(\scrT_f, \scrF_f)$ on the abelian category of coherent sheaves on $Y$, where the torsion part $\scrT_f$ is the full subcategory of sheaves $\calF$ so that $\R f_\ast\, \calF$ is isomorphic to a zero-dimensional sheaf in $\catDb(X)$. An $f$-stable pair is then a pair $(\calF,s)$, with $s\colon \scrO_Y\to \calF$ a section, so that $\calF$ is a one dimensional sheaf in $\scrF_f$ and $\mathsf{Coker}(s) \in \scrT_f$. The main result in \cite{BS-Curve-counting-crepant} provides an explicit relation between Donaldson-Thomas invariants and $f$-stable pair invariants, which is en essential step in the proof of the crepant resolution conjecture. The study of enumerative invariants in this setting has also been investigated  in \cite{Padurariu-stable-pairs}.\footnote{Note that in \cite{Padurariu-stable-pairs}, the author constructs a virtual fundamental class associated to the moduli stack of $f$-stable pairs. In \S\ref{subsec:f-stable-pairs}, we construct a derived enhancement of the moduli stack in a canonical way.}

In this paper, we will apply the construction of \cite{BS-Curve-counting-crepant} to the contraction $f^+\colon Y^+\to X$ in diagram \eqref{eq:flop-diagram-O} (satisfying Assumptions~\ref{assumption:f} and \ref{assumption:divisor}) in order to obtain a relation between the stable pair invariants of $Y^+$ and the Euler numbers of flag Hilbert schemes of $C^+$. Generalizing the analogous notion used in \cite{DHS-HOMFLY,Maulik-Stable-pairs}, we will employ a framing condition along $C^+$. A pair $(\calF,s\colon \scrO_Y\to \calF)$ will be said to be \textit{$C^+$-framed} if $\calF$ fits into a short exact sequence 
\begin{align}
	0\longrightarrow \calF_\Sigma \longrightarrow \calF \longrightarrow \calF_{C^+} \longrightarrow 0\ , 
\end{align}
where $\calF_\Sigma$ is set theoretically supported on $\Sigma^+$ and $\calF_{C^+}$ is the pushforward of a rank one torsion-free sheaf on $C^+$. This notion applies both to Pandharipande-Thomas stable pairs \cite{PT-Curve-counting}, as well as  $f^+$-stable pairs. As shown in  \S\ref{subsec:C-framed-moduli-stack}, for fixed topological invariants \eqref{eq:fixed-top}, 
\begin{align}\label{eq:fixed-top}
	\ch_2(\calF) = [C^+]+m [\Sigma^+]\quad\text{and} \quad \chi(\calF) = \ell\ ,
\end{align}
the corresponding moduli stacks $\SP_{C^+}(Y^+;m,\ell)$ and $f^+\textrm{-}\SP_{C^+}(Y^+;m,\ell)$, are algebraic spaces of finite type. Furthermore, an analogous construction yields a moduli space $\SP_{C^-}(Y^-; m,\ell)$ of $C^-$-framed Pandharipande-Thomas stable pairs on $Y^-$. 

The first step in the proof of Theorem~\ref{thm:main-A} is the construction of an equivalence
\begin{align}\label{eq:SP-isomorphism}
	\begin{tikzcd}[ampersand replacement=\&]
			f^+\textrm{-}\SP_{C^+}(Y^+; m, k+\chi(\scrO_{C^+}))\arrow{r}{\sim} \& \FHilb_{T_n}^k(C^+; m) 
	\end{tikzcd}
\end{align}
for any $m,k\in \N$. This is carried out in Theorem~\ref{thm:C-framed-Hilbert-A} and Proposition~\ref{prop:C-framed-Hilbert-C} by showing that both sides are equivalent to a relative Quot scheme. This result is very specific to contractions of rigid $(0,-2)$ curves, and relies on the detailed analysis of semistable sheaves with set theoretic support on the exceptional locus carried out in \S\ref{sec:semistable-exceptional-locus}. 

The next steps will use the 
generating functions 
\begin{align}
	f^+\textrm{-}\mathsf{PT}_{C^+}(Y^+; q, Q)&\coloneqq  \sum_{\ell\in \Z} \sum_{m\in \N} \chi(f^+\textrm{-}\mathcal{SP}_{C^+}(Y^+; m, \ell)) q^\ell Q^m \ , \\[4pt]
	\mathsf{PT}_{C^\pm}(Y^\pm; q, Q)&\coloneqq  \sum_{\ell\in \Z} \sum_{m\in \N} \chi(\SP_{C^\pm}(Y^\pm; m, \ell)) q^\ell Q^m \ , \\[4pt]
	\mathsf{PT}_{\mathsf{ex}}(Y^\pm; q, Q)&\coloneqq \sum_{\ell\in \Z} \sum_{m\in \N} \chi(\SP(Y^\pm; m[\Sigma^\pm], \ell)) q^\ell Q^m\ .
\end{align}
defined in \S\ref{sec:framed-wallcrossing-identity}. As explained there, these are in elements of the ring $\Lambda_q[[Q]]$, where $\Lambda_q$ is the ring of Laurent formal power series in $q$ with complex coefficients. Moreover, $\mathsf{PT}_{\mathsf{ex}}(Y^\pm)$ is an invertible element of $\C[[q,Q]]$. 

Then the second step is the following:
\begin{theoremintroduction}[{Theorem~\ref{thm:C-framed-identity}}]\label{thm:main-C}
	Let $f^+\colon Y^+\to X$ be a contraction satisfying Assumptions~\ref{assumption:f} and \ref{assumption:divisor}. Let $C\subset W$ be a reduced irreducible plane curve so that $\nu \in C^{\mathsf{sing}}$ and let $C^+ \subset S^+$ be its strict transforms. Then, the following identity holds in $\Lambda_q[[Q]]$:
	\begin{align}
		f^+\textrm{-}\mathsf{PT}_{C^+}(Y^+; q,Q) = \frac{\mathsf{PT}_{C^+}(Y^+; q,Q)}{\mathsf{PT}_{\mathsf{ex}}(Y^+; q,Q)}\ .
	\end{align}
\end{theoremintroduction}
The proof of above result follows the same lines as the proof of \cite[Theorem~6]{BS-Curve-counting-crepant}, although some intermediate arguments have to be modified in the $C^+$-framed setup. For completeness, the details are included in Appendix~\ref{sec:proof-C-framed-identity}.

The final step in the proof of Theorem~\ref{thm:main-A} is a flop identity relating $C^\pm$-framed topological stable pair invariants on both sides the transition. 
As shown in Proposition~\ref{prop:C-framed-flop-A}, one has the identity 
\begin{align}\label{eq:flop-identity-O} 
		Q^{m_{\Sigma^-}}\, \frac{\mathsf{PT}_{C^+}(Y^+; q, Q^{-1})}{\mathsf{PT}_{\mathsf{ex}}(Y^+; q, Q^{-1})}= \frac{\mathsf{PT}_{C^-}(Y^-; q, Q)}{\mathsf{PT}_{\mathsf{ex}}(Y^-; q, Q)}\ ,
\end{align}
in the complex vector space $\C\{q,Q\}$ of bi-infinite formal power series $\sum_{\ell,m\in \Z} c_{\ell, m} q^\ell Q^m$. In particular, this implies that both sides of the above equation are polynomial functions of $Q$ with coefficients in $\Lambda_q$. The proof of this identity is completely analogous to the proof of \cite[Proposition~2.4]{Maulik-Stable-pairs}. 

\subsection{Explicit examples}\label{subsec:examples}

Fix $r,s,n\in \N$, with $r>s>n\geq 2$ and $r, s$ coprime. Consider the planar curve singularity $C_{r,s}$ given by the equation $x^r=w^s$ and corresponding to the $(s, r)$ torus knot. Denote by $p$ the origin of $\C^2$. Let $T_n$ be the closed subscheme given by the equations $x=0$ and $w^n=0$. For any $t,m\geq 1$, let 
\begin{align}
	Z_{t,m}(q,x)  \coloneqq  \prod_{i=1}^t (1+xq^i + \cdots + (xq^i)^{m})
\end{align}
and let $Z_{t,m}^{(k)}(q)$ denote the coefficient of $x^k$ in $Z_{t,m}(q,x)$. 
\begin{remark*}
	Note that
	\begin{align} 
		Z^{(k)}_{t,m}(q) = \sum_{\genfrac{}{}{0pt}{}{i,j\geq 0}{i + (m+1)j =k}} (-1)^j\, h_i\big(q,\ldots, q^t\big) e_j\big(q^{m+1}, \ldots, q^{t(m+1)}\big)\ ,
	\end{align}
	where $h_i(z_1, \ldots, z_t)$ is the $i$-th complete homogeneous symmetric function and $e_j(z_1, \ldots, z_t)$ is the $j$-th elementary symmetric function, both in the variables $z_1, \ldots, z_t$. 
\end{remark*}
Set
\begin{align}
	Z_{r,s,n}(q)  \coloneqq  \frac{1}{1-q^s}\sum_{m=0}^{n-1} \sum_{k=s-n}^{s-m-1} q^{mr} Z^{(k)}_{r-1,n}(q)\ .
\end{align}

In \S\ref{subsec:torus-invariant-+} we prove the following:
\begin{theoremintroduction}\label{thm:main-D}
	Let $\FHilb_{T_n, p}^k(C_{r,s}; s)$ be the punctual Flag Hilbert scheme parametrizing flags of ideal sheaves $\calI_1\subset \calI_2\subset \scrO_{C_{r,s}}$ so that $\scrO_{C_{r,s}}/\calI_1$ is set-theoretic supported at $p$, the ideal $\calI_2$ has colength $k$, and $\calI_2/\calI_1$ is the pushforward of a length $s$ zero-dimensional sheaf on $T_n$. Then 
	\begin{align}\label{eq:main-D}
		\sum_{k\geq 0} \chi(\FHilb_{T_n, p}^k(C_{r,s}; s)) q^k = Z_{r,s,n}(q) \ .
	\end{align}
\end{theoremintroduction} 

Now, consider the complete intersection curves of the form 
\begin{align}
	C_{r,t,n}\colon\begin{cases}
		xv - w^n =0 \\[2pt]
		x^{r-t}- v^t =0 
	\end{cases}
\end{align}
in $\C^3$, with $r,t,n\geq 2$ such that $r>t$ and the pair $(r,nt)$ is coprime. Thanks to Theorems~\ref{thm:main-B} and \ref{thm:main-D}, we can derive:
\begin{theoremintroduction}[{Theorem~\ref{thm:lci-curves}}]\label{thm:main-E}
	For any $k\in \N$, let $\Hilb_o^k(C_{r,t,n})$ be punctual Hilbert scheme parametrizing length $k$ zero-dimensional subschemes of $C_{r,t,n}$ with support at the origin $o$. Then 
	\begin{align}\label{eq:main-E}
		\sum_{k\in \N} \chi(\Hilb_o^k(C_{r,t,n})) q^k = q^{-nt(t-1)/2} Z_{r,tn,n}(q)\ .
	\end{align}
\end{theoremintroduction} 

\begin{remark*}
	In principle one could try to compute the left-hand-side of identity \eqref{eq:main-E} in Theorem~\ref{thm:main-E} directly by localization with respect to the given torus action. This approach leads to a very difficult counting problem for constrained three dimensional partitions. The above result shows that this counting problem is equivalent in a nontrivial way to a tractable counting problem for constrained two dimensional partitions.
\end{remark*}

We conclude the introduction with a brief discussion of possible directions for further research. 

\subsection{Possible generalizations}\label{subsec:gen} 

We will first discuss potential generalizations of Theorems~\ref{thm:main-A} and \ref{thm:main-B}. For convenience, the main steps in the proofs of these theorems are summarized in the following diagram:
\begin{align}\label{eq:proofchartA} 
	\begin{tikzcd}[ampersand replacement=\&, column sep=7em, row sep=large]
		\SP_{C^+}(Y^+; m_{\Sigma^-}) \ar[r, leftrightarrow, "\text{Flop\ identity\ \eqref{eq:flop-identity-O}}"] \ar[d, leftrightarrow, swap, "\text{Theorem~\ref{thm:main-C}}"] 
		 \&  \SP_{C^-}(Y^-; 0) \ar[dd,leftrightarrow]\\
		f\textrm{-}\SP_{C^+}(Y^+; m_{\Sigma^-}) \ar[d, leftrightarrow, swap, "\text{Equivalence~\eqref{eq:SP-isomorphism}}"] \& 	\\
		\FHilb_{T_n}(C^+) \& \Hilb(C^-)
	\end{tikzcd}
\end{align}
where $\SP_{C^\pm}(Y^\pm; m)$ denotes the moduli stack of $C^\pm$-framed stable pairs $(\scrO_{Y^\pm}\to \calF)$ with $\ch_2(\calF) =[C^\pm]+ m[\Sigma^\pm]$.  For ease of exposition, we will refer to such objects as \textit{$C^\pm$-framed stable with multiplicity $m$ along $\Sigma^\pm$}.  Furthermore,
\begin{align}
	\FHilb_{T_n}(C^+)\coloneqq \bigsqcup_{k\in \N}\FHilb_{T_n}^k(C^+)\quad \text{and}\quad \Hilb(C^-)\coloneqq \bigsqcup_{k\in \N} \Hilb^k(C^-)\ .
\end{align}
Moreover, note that the right vertical arrow in diagram \eqref{eq:proofchartA} represents the correspondence proven in \cite[Appendix~B]{PT-BPS}, between stable pairs $\scrO_{Y^-} \to \calF$ on $Y^-$, with $\calF$ scheme-theoretically supported on $C^-$, and ideal sheaves of zero dimensional subschemes of $C^-$. This requires $C^-$ to have locally complete intersection singularities. In the above terminology, such objects are called \textit{$C^-$-framed stable pairs with multiplicity zero along $\Sigma^-$}. 

We also recall that the geometric framework of this paper is encoded in a diagram of 
the form 
\begin{align}
	\begin{tikzcd}[ampersand replacement=\&]
	C^+\ar[r]\ar[ddrr] \& 	S^+ \ar{dr}[swap]{f_{S^+}} \ar[dashed]{rr}{\phi_S}\& \& S^- \ar{dl}{f_{S^-}} \& C^- \ar[l] \ar[ddll]\\
		\& \& W \&\& \\
	\& \& C \ar[u] \& \&
	\end{tikzcd}
\end{align}
where $W\subset X$ is a smooth Weil divisor on $X$ passing through the singular point, and $C$ is a reduced irreducible singular curve on $W$ whose singular locus contains the singular locus of $X$. The surfaces $S^\pm$ and the curves $C^\pm$ are the strict transforms of $W$, and $C$ respectively under the contraction maps $f^\pm\colon Y^\pm \to X$. Moreover, its is further assumed that $S^+$ is smooth and intersects $\Sigma^+$ transversely at a single point $p$. Then a local computation shows that $S^-$ has an canonical $A_{n-1}$ singularity and it also contains $\Sigma^-$ as a closed subscheme.

The above framework is amenable to several possible natural generalizations, as follows. 
\begin{enumerate}\itemsep0.2cm 
	\item Can one prove similar results in a more general setting, allowing $S^-$ to develop other types of singularities? 
	
	\item Can one obtain similar results for non-reduced curves $C^+$ and $C^-$?
	
	\item Diagram~\eqref{eq:proofchartA} is left/right asymmetric since the upper-right-corner consists of $C^-$-framed stable pairs with multiplicity \textit{zero} along $\Sigma^-$, while the upper-left-hand corner consists of 
	$C^+$-framed stable pairs with \textit{nonzero} multiplicity along $\Sigma^+$. Is there a more general framework in which the left/right symmetry is restored, allowing arbitrary multiplicities along $\Sigma^+$ and $\Sigma^-$?
\end{enumerate}
 
Let us briefly address the above questions. First note that one can easily relax the smoothness assumption on the Weil divisor $W$, allowing it to develop any type of isolated surface singularity at the singular point of $X$. This will result in a large class of singularities of $S^\pm$, as well as $C^\pm$, which  can be determined in principle by local computations. At the same time, one can allow the curve $C\subset W$ to be non-reduced, in which case  $C^\pm$ will be non-reduced as well. We expect that such a general geometric framework will give rise to a symmetric diagram of the form 
\begin{align}\label{eq:proofchartB} 
	\begin{tikzcd}[ampersand replacement=\&, column sep=7em, row sep=large]
		\SP_{C^+}(Y^+) \ar[r, leftrightarrow, "\text{Wallcrossing}"] \ar[d, leftrightarrow, swap, "\text{Wallcrossing}"] 
		\&  \SP_{C^-}(Y^-) \ar[d,  "\text{Wallcrossing}"]\\
		\fSP_{C^+}(Y^+) \& \fSP_{C^-}(Y^-)
	\end{tikzcd}\ .
\end{align}
Here $\SP_{C^\pm}(Y^\pm)$ are moduli spaces of $C^\pm$-framed stable pairs with numerical invariants in $[C^\pm]+\Z[\Sigma^\pm]$, while $\fSP_{C^\pm}(Y^\pm)$ are their relative counterparts with respect to the contractions $f^\pm\colon Y^\pm \to X$. In more detail, we expect suitable generalizations of the wallcrossing identities in Theorem~\ref{thm:main-C} and Equation~\eqref{eq:flop-identity-O} to hold by analogy to \cite{Calabrese-DT-Flops}, \cite{Maulik-Stable-pairs}, and \cite{BS-Curve-counting-crepant}. This is based on the fact that the wallcrossing formulas proven in \textit{loc. cit.} are based on general identities in motivic Hall algebras rather than detailed scheme-theoretic properties.
 
  At the same time we expect the remaining steps, i.e., the completion to a diagram of the form 
\begin{align}\label{eq:proofchartC} 
	\begin{tikzcd}[ampersand replacement=\&, column sep=7em, row sep=large]
		\SP_{C^+}(Y^+) \ar[r, leftrightarrow, "\text{Wallcrossing}"] \ar[d, leftrightarrow, swap, "\text{Wallcrossing}"] 
		\&  \SP_{C^-}(Y^-) \ar[d,  "\text{Wallcrossing}"]\\
		\fSP_{C^+}(Y^+)\ar[d,leftrightarrow,dashed] \& \fSP_{C^-}(Y^-) \ar[d,leftrightarrow,dashed]\\
		\text{Moduli\ space\ on}\ C^+ \& \text{Moduli\ space\ on}\ C^-
	\end{tikzcd}
\end{align}
to be very difficult. In fact, in arbitrary situations, untractable. 

For stable pairs with scheme-theoretic support on $C^\pm$, the associated moduli space is isomorphic to the Hilbert scheme of points on $C^\pm$ by the correspondence proven in \cite{PT-BPS}, as long as the curve 
$C^\pm$ is Gorenstein. However, the top horizontal wallcrossing arrow in \eqref{eq:proofchartC} involves a shift of numerical invariants, i.e., it relates the stable pair invariants $\mathsf{PT}_{C^-}(Y^-; m^-)$ to stable pair invariants $\mathsf{PT}_{C^+}(Y^+; m^+)$, where $m^+$ and $m^-$ are related by a shift determined by the local geometry of our configurations. Diagram~\eqref{eq:proofchartA} is a particular case of this phenomenon, where $m^-=0$ and $m^+=m_{\Sigma^-}$. As a result, one has to generalize the results of \cite{PT-BPS} to stable pairs of the form 
\begin{align} 
	\begin{tikzcd}[ampersand replacement=\&]
		\& \& \scrO_{Y^\pm}\ar[d] \& \& \\
		0\ar[r] \& \calF_{\Sigma^\pm} \ar[r] \& \calF \ar[r] \& \calF_{0} \ar[r] \& 0
	\end{tikzcd}\ ,
\end{align} 
with $\calF_{0}$ scheme-theoretically supported on $C^\pm$ and $\calF_{\Sigma^\pm}$ set-theoretically supported on $\Sigma^\pm$. This requires in turn a suitable generalization of Theorem~\ref{thm:C-framed-Hilbert-A}, which poses significant challenges. The difficulty resides in the fact that certain essential steps in the proof rely heavily on the current assumptions. For example, an essential step in the proof of Theorem~\ref{thm:Pi} is Lemma~\ref{lem:Pi-C}--\eqref{item:Pi-C-1}, which shows that $\calF_0$ is semistable, fixing at the same time its $S$-equivalence class. In conjunction with Corollary~\ref{cor:stability-G}, this leads to a reformulation of the moduli problem for $f$-stable pairs in terms of nested ideal sheaves on $C^+$. For more general configurations $C^+\subset S^+\subset Y^+$, a similar analysis will produce a long list of complicated constraints on the Harder-Narasimhan filtration of $\calF_0$, typically leading to an untractable moduli problem. Therefore, the main challenge is to relax the current assumptions, while maintaining sufficient control of the local geometry to obtain tractable moduli problems on both sides of the transition. We leave the investigation of this problem  for future work.

\subsection{Open questions in low-dimensional topology}\label{sec:openquest-top} 

Given the existing results on plane curve singularities and knot polynomials \cite{OS-HOMFLY, ORS-HOMFLY}, a natural question is whether such a relation might extend at least to the class of space singularities considered in \S\ref{subsec:examples}. While there is no natural construction assigning a link in a three-manifold to a space curve singularity, some partial insight can be gained from the \textit{large $N$ duality} approach employed in \cite{Alg_knots_lagrangians}. In the present framework, this leads to large $N$ duality for extremal transitions associated to threefold singularities of analytic type $xy + z^2 - w^{2n} =0$, with $n\geq 2$. By analogy to \textit{loc. cit.}, the curve singularities in \S\ref{subsec:examples} are related through such transitions to certain configurations of links in singular reducible lagrangian cycles. At the moment there is no known construction of polynomial invariants associated to such configurations, either in algebraic topology or topological string theory. In fact the explicit results obtained in \S\ref{subsec:examples} serve as a motivating factor for this new direction of research.

\subsection{Open questions in combinatorics and representation theory}\label{sec:openquest-rep} 

Motivated by the results of \cite{GORS-Knots,GK-Hilbert}, a possible open question is whether there is any connection with the representation theory of Cherednik algebras, double affine Hecke algebras or Coulomb branch algebras. As a preliminary remark, note that the localized equivariant homology of the flag Hilbert scheme in Theorem~\ref{thm:main-D} is identified via pushforward to a subspace of the equivariant homology of the punctual Hilbert scheme of the plane curve singularity $C_{r,s}$. Furthermore, as shown in \cite{GK-Hilbert}, the latter carries a natural action of the rational Cherednik algebra. As the resulting representation is known to be irreducible, obviously, this action cannot preserve any subspace. Nevertheless, this observation constitutes a starting point for a more systematic investigation of a possibly more subtle representation theoretic connection.  

In a series of papers \cite{GM-I, GM-II}, Gorsky and Mazin investigated the topology of the compactified Jacobian of a plane curve singularity using \textit{$(q,t)$-Catalan numbers} and \textit{Dyck path} combinatorics. Given the well-known connection between the punctual Hilbert scheme of a curve singularity and its compactified Jacobian, it is natural to ask whether there exists an alternative formulation of Theorem~\ref{thm:lci-curves} in terms of Dyck path combinatorics. We note that our functions $Z_{r,s, n}(q)$ and $Z_{t,m}(q)$ can be expressed in terms of constrained two-dimensional partitions (see \S\ref{subsec:torus-invariant-+}). However, at present, we are unable to establish an explicit relationship--if any--between these functions and Dyck paths.

\subsection{Outline}

In \S\ref{sec:geometric-setup}, we study the general framework in which we develop the theory of ($C$-framed) $f$-stable pairs. In particular, we introduce the assumptions on a map $f\colon Y\to X$ between projective threefolds and characterize the geometric implications of them. In \S\ref{sec:semistable-exceptional-locus}, we study semistable sheaves supported on the exceptional locus and provide an explicit equivalence between their moduli stacks and the moduli stacks of certain zero-dimensional sheaves. In \S\ref{sec:stable-pairs} we develope the theory of $f$-stable pairs in our framework and the construction of their moduli stacks: since we follow an approach developed in \cite{DPS-torsion-pairs}, we introduce a suitable torsion pair on the abelian category of coherent sheaves on the smooth threefold and we prove its openness (in the sense of Lieblich -- see \cite[Appendix~A]{AB-Bridgeland-stable}). \S\ref{sec:framed-stable-pairs} is devoted to the introduction of $C$-framed $f$-stable pairs and their moduli stacks. In addition, we provide an relation to Hilbert schemes on the curve $C$. In \S\ref{sec:framed-wallcrossing-identity} and \S\ref{sec:flop-identity} we prove the main identities between generating functions of Euler numbers of moduli stacks of ($C$-framed) $f$-stable pairs, which will allow us to prove Theorems \ref{thm:main-A}, \ref{thm:main-B}, and \ref{thm:main-C}. \S\ref{sec:examples} is devoted to explicit examples in the toric framework and the proofs of Theorems~\ref{thm:main-D} and \ref{thm:main-E}. Finally the paper has three appendices. Appendix~\ref{sec:theorem-filtration} contains the proof of Theorem~\ref{thm:filtration}, while in Appendix~\ref{sec:proof-C-framed-identity} we provide the detailed proof of Theorem~\ref{thm:C-framed-identity}, which is a $C$-framed version of the main result of \cite{BS-Curve-counting-crepant}. Appendix~\ref{sec:openness-orthogonality} provides a criterion to prove the openness of the torsion-free part of a torsion pair. 

\subsection*{Acknowledgments}

The paper was finalized during a research visit of the third-named author at Kavli IPMU, the University of Tokyo. He thanks the institution for providing an exceptional research environment. We would like to thank Yan Soibelman for very helpful discussions and Michael Wemyss for very stimulating discussions about contraction algebras and Reid's width, and for bringing to our attention the paper \cite{Spherical_twists}: it was our starting point to prove Theorem~\ref{thm:filtration}. We also thank the anonymous referee for their insightful comments and feedback, which significantly improved the presentation of this paper.

The paper was finalized during a research visit of the third-named author at Kavli IPMU, the University of Tokyo. He thanks the institution for providing an exceptional research environment. Moreover, he would like to thank Andrei Neguţ for very helpful discussions, which took place under the MIT-UNIPI Project (XI call). Finally, he acknowledges the MIUR Excellence Department Project awarded to the Department of Mathematics, University of Pisa, CUP I57G22000700001. He is a member of GNSAGA of INDAM.

\section{Geometric setup}\label{sec:geometric-setup}

In this section, we shall introduce the geometric framework over which we shall study ($C$-framed) $f$-stable pairs. In particular, we state the Assumptions we shall impose and describe the geometric consequences of them. 

\subsection{$(0,-2)$ curves on threefolds}\label{subsec:rigid-curves} 

Let $X$ and $Y$ be irreducible projective complex threefolds and let $f\colon Y\to X$ be a morphism. 
\begin{definition}\label{def:contraction}
	We say that $f$ is a \textit{contraction} if
	\begin{itemize}\itemsep0.2cm
		\item $Y$ is smooth, $X$ is reduced, normal\footnote{In the usual definition of \textit{contractions}, it is customary to require that $X$ is Gorenstein. However, in our definition of a contraction, we do not impose this condition explicitly. Instead, we note that for a contraction satisfying Assumption~\ref{assumption:f} in our framework, $X$ is guaranteed to be Gorenstein by Lemma~\ref{lem:contraction-A}.}, and with a unique singular point $\nu\in X$, and
		\item the scheme-theoretic inverse image $\Sigma\coloneqq f^{-1}(\nu)$ is isomorphic to a rational curve on $Y$ and $f$ maps the complement $Y\smallsetminus \Sigma$ isomorphically onto $X\smallsetminus \{\nu\}$. 
	\end{itemize}
\end{definition}

We shall impose the following assumption.
\begin{assumption}\label{assumption:f} 
	The morphism $f$ is a contraction such that $\Sigma$ is a smooth connected $(0,-2)$-curve\footnote{i.e., $\Sigma$ is reduced, $\Sigma\simeq \PP^1$, and $\calN_{Y/\Sigma}\simeq \scrO_{\PP^1}\oplus \scrO_{\PP^1}(-2)$, where $\calN_{Y/\Sigma}$ is the normal bundle of the inclusion $\Sigma\subset Y$.} on $Y$.
\end{assumption}
Let us denote by $i_\Sigma\colon \Sigma \to Y$ the corresponding closed embedding and by $\calI_\Sigma\subset \scrO_Y$ the corresponding defining ideal. 

\begin{lemma}\label{lem:contraction-A} 
	Under the conditions stated in Assumption~\ref{assumption:f}, the threefold $X$ is Gorenstein with rational singularities. Therefore the dualizing sheaf 
	$\omega_X$ is locally free and 
	\begin{align}\label{eq:dirimg-A}
		\R^0f_\ast \scrO_Y = \scrO_X\quad\text{and}\quad \R^k f_\ast \scrO_{Y}=0 \ \text{ for } \ k\geq 1\ . 
	\end{align}
	Moreover, $f$ is crepant, i.e., $f^\ast\omega_X \simeq \omega_Y$.
\end{lemma} 

\begin{proof}
	As shown in \cite[Remark~5.13]{Reid-Minimal-Models} (see also \cite{KM-Gorenstein}), the underlying complex analytic space $X^{\mathsf{an}}$ has a cDV singular point at $\nu$. Then \cite[Theorem~5.42]{KM-Birational} shows that the singular point $\nu\in X$ is rational. This further implies via \cite[Theorem~5.11 and Corollary~5.11]{KM-Birational} that $f\colon Y\to X$ is a rational resolution of singularities. Moreover, by \cite[Theorem~2.6.II]{Reid-Canonical}, $X$ is also Gorenstein and one has $\omega_X \simeq f_\ast\omega_Y$. Set $\calL \coloneqq \omega_Y \otimes f^\ast\omega_X$. Since $f_\ast\scrO_Y=\scrO_X$ this implies that $f_\ast \calL \simeq \scrO_X$. Therefore $\calL$ and $\scrO_Y$ are isomorphic on the complement of $\Sigma$. Since $\Sigma$ has codimension two, this implies that $\calL\simeq \scrO_Y$. 
\end{proof}

\subsection{Numerical classes}\label{subsec:numerical-classes} 

 Let $N_1(Y)$ be the group of \textit{numerical equivalence classes} of one-cycles of $Y$ (cf.\ \cite[Chapter~19]{Fulton-Intersection-Theory}). 

\begin{lemma}\label{lem:exclass} 
	Let $Z\subset Y$ be a purely\footnote{i.e., all irreducible components of $Z$ are one-dimensional.} one-dimensional closed subscheme so that $[Z]=n[\Sigma]$ in $N_1(Y)$ for some $n \in \Z_{>0}$. Then $Z_{\mathsf{red}}=\Sigma$.
\end{lemma} 

\begin{proof}
	Let $Z'\subseteq Z$ be the maximal closed subscheme of $Z$ such that $Z'_{\mathsf{red}}=\Sigma$. Let $Z''$ be the closure of the complement $Z\smallsetminus Z'$. Then, by assumption, one has $[Z''] = n''[\Sigma]$ for some $n''\in\Z_{\geq 0}$. Furthermore, $(Z''\cap \Sigma)_{\mathsf{red}}$ is zero dimensional. We shall show that $n''=0$. 
	
	Assume that $n''>0$. Then the scheme-theoretic image of $Z''$ through the morphism $f\colon Y \to X$ is a non-empty one-dimensional closed subscheme $\Xi\subset X$. Let $H\subset X$ be an ample Cartier divisor which does not contain the singular point $\nu\in X$. Then, $H\cap \Xi\neq 0$. On the other hand, the projection formula yields 
	\begin{align}
		H\cap \Xi = H \cap f_\ast([Z'']) = f_\ast(f^\ast(H)\cap [Z'']) =n''\, f_\ast(f^\ast(H) \cap \Sigma) =0 \ ,
	\end{align}
	which leads to a contradiction.
\end{proof}

Let $N_1(Y/X)\subset N_1(Y)$ be the subgroup spanned by closures of one-cycles in $Y \smallsetminus \Sigma$. 
\begin{corollary}\label{cor:split-support}
	There is a splitting 
	\begin{align}\label{eq:numsplit}
		N_1(Y) \simeq \Z[\Sigma] \oplus N_1(Y/X) \ .
	\end{align}
	Moreover, given a purely one-dimensional sheaf $\calF$ on $Y$, one has 
	\begin{enumerate}\itemsep0.2cm
		\item $\mathsf{supp}(\calF)_{\mathsf{red}}=\Sigma$ if and only if $\ch_2(\calF)\in \N[\Sigma]$. 
		
		\item $\Sigma \nsubseteq \mathsf{supp}(\calF)$ if and only if $\ch_2(\calF)\in N_1(Y/X)$.
	\end{enumerate}
\end{corollary} 

\begin{proof}
	This statement from Lemma~\ref{lem:exclass} and the exact sequence 
	\begin{align}
		0\longrightarrow A_1(\Sigma) \longrightarrow A_1(Y) \longrightarrow A_1(Y \smallsetminus \Sigma)\longrightarrow 0\ .
	\end{align}
\end{proof}

\subsection{Serre duality}

We shall now prove a result which will help us to readapt part of the arguments in \cite{BS-Curve-counting-crepant, Maulik-Stable-pairs} to our setting.

\begin{lemma}\label{lem:trivial}
	For any thickening $\Sigma\subset \Sigma'\subset Y$, one has $\omega_Y\vert_{\Sigma'}\simeq \scrO_{\Sigma'}$.
\end{lemma}	

\begin{proof}
	We observe that the scheme-theoretic image $Z=f(\Sigma')$ is a zero-dimensional subscheme with set-theoretic support $\nu$ since $f$ is proper. Moreover, by Lemma~\ref{lem:contraction-A}, $\omega_X$ is locally free and $\omega_Y = f^\ast\omega_X$, where $\omega_X$ is the canonical bundle of $X$. Since $f\vert_{\Sigma'}\colon \Sigma'\to X$ factors through the closed immersion $Z\to X$, we get $\omega_Y\vert_{\Sigma'}\simeq \scrO_{\Sigma'}$. 
\end{proof}

Lemma \ref{lem:trivial} implies the following.
\begin{corollary}\label{cor:pairing} 
	Let $\calF$ and $\calG$ be coherent sheaves on $Y$ so that $\calG$ is one-dimensional, with $\ch_2(\calG)\in \N[\Sigma]$. Then Serre duality yields isomorphisms 
	\begin{align}
		\Ext^i_Y(\calF, \calG) \simeq \Ext^{3-i}_Y(\calG, \calF) 
	\end{align}
	for all $i \in \Z$. 
\end{corollary} 

\begin{proof}
	By Corollary \ref{cor:split-support}, under the present assumptions $\calG$ is scheme theoretically supported on a scheme theoretic thickening of $\Sigma$. Then Corollary \ref{cor:pairing} follows by Lemma~\ref{lem:trivial} and Serre duality. 
\end{proof}

\subsection{Rigidity and filtrations}

Let $\calH$ be the connected component of the Hilbert scheme of curves on $Y$ which contains $\Sigma$.
\begin{definition}\label{def:rigid-curve} 
	We say that $\Sigma$ is \textit{rigid} if $\calH$ is zero-dimensional, i.e., it is isomorphic to $\Spec(R)$, with $R$ a local Artin ring over $\C$. 
\end{definition}

\begin{proposition}\label{prop:rigid-curve-B}
	Under Assumption~\ref{assumption:f}, the curve $\Sigma$ is rigid.
\end{proposition}

\begin{proof}
	Suppose $\Sigma$ is not rigid. Then $\dim\calH\geq 1$. Let $\calZ\subset Y \times \calH$ be the universal curve and let $o\subset \calH $ be the closed point corresponding to $\Sigma \subset Y$. Since $\calH$ is projective and at least one-dimensional, it contains at least one closed point $p\in \calH$, such that $p \neq o$. Then $\calZ_p\subset Y$ is a closed subscheme of $Y$ which does not coincide with $\Sigma\subset Y$. 
	
	Let $\calE\coloneqq \scrO_{\calZ_p}$ and let $\calT \subset \calE$ be the maximal zero-dimensional subsheaf. Since $\calZ$ is a flat family, note that 
	\begin{align}\label{eq:chern-classes}
		\ch_i(\calE) = \ch_i(\scrO_\Sigma)
	\end{align}
	for $2\leq i \leq 3$. Then, Lemma~\ref{lem:exclass} implies that $\calE/\calT$ is a purely one-dimensional sheaf on $Y$ with scheme-theoretic support on $\Sigma$. Hence it is isomorphic to the pushforward of a line bundle $\calL$ on $\Sigma$. Moreover, the composition 
	\begin{align}
		\scrO_Y \longrightarrow \calE \longrightarrow i_{\Sigma,\, \ast}\calL 
	\end{align}
	is surjective, where $\scrO_Y\to \calE$ is the canonical epimorphism. The relations \eqref{eq:chern-classes} imply that $\calL \simeq \scrO_\Sigma$ and $\chi(\calT)=0$, i.e., $\calT=0$. Then, $\calE \simeq \scrO_\Sigma$. Furthermore, since $\End_Y(\scrO_\Sigma)\simeq \C$, there exists an isomorphism $\calE \to \scrO_\Sigma$ so that the following diagram is commutative 
	\begin{align}
		\begin{tikzcd}[ampersand replacement=\&]
			\& \scrO_Y\ar[dr] \ar[dl] \& \\
			\calE \ar[rr] \& \& \scrO_\Sigma 
		\end{tikzcd}\ ,
	\end{align}
	where the vertical maps are surjective. This implies that $\calZ_p$ coincides with $\Sigma$ as subschemes of $Y$, hence it leads to a contradiction. 
\end{proof}

The following result will be proved in Appendix~\ref{sec:theorem-filtration}.
\begin{theorem}\label{thm:filtration}
	Under Assumption~\ref{assumption:f}, there exists a unique chain of subschemes 
	\begin{align}
		\Sigma = \Sigma_1 \subset \cdots \cdots \subset \Sigma_n \ ,
	\end{align}
	with $n\geq 2$, which determines a filtration of the form
	\begin{align}\label{eq:ideal-filtration-A}
		0=\calJ_{n+1}\subset \calJ_n\subset \cdots \subset \calJ_1 = \calI_\Sigma 
	\end{align}
	so that
	\begin{align}\label{eq:ideal-filtration-B}
		\begin{split}
			\calI_\Sigma \calJ_i \subset \calJ_{i+1}\subset \calJ_i\ , \quad \calJ_i/\calJ_{i+1} \simeq i_{\Sigma,\, \ast}\scrO_{\Sigma}\ , \\[4pt]
			\calJ_i/\calI_\Sigma \calJ_i\simeq  i_{\Sigma,\, \ast}\scrO_{\Sigma}\oplus i_{\Sigma,\, \ast}\scrO_{\Sigma}(2)\ , \quad \calJ_{i+1}/\calI_\Sigma \calJ_i\simeq  i_{\Sigma,\, \ast}\scrO_{\Sigma}(2)\ ,
		\end{split}
	\end{align}	
	for $1\leq i \leq n-1$, and 
	\begin{align}\label{eq:ideal-filtration-C}
		\calJ_n /\calI_\Sigma \calJ_n \simeq i_{\Sigma,\, \ast}\scrO_{\Sigma}(1)^{\oplus 2}\ .
	\end{align}
	
	Moreover, the structure sheaves of the subschemes $\Sigma_k \subset Y$, with $1\leq k \leq n$, fit into nontrivial extensions 
	\begin{align}\label{eq:sch-sequence-X}
		0\longrightarrow i_{\Sigma,\, \ast}\scrO_\Sigma \longrightarrow i_{\Sigma_k,\, \ast}\scrO_{\Sigma_k} \longrightarrow i_{\Sigma_{k-1},\, \ast}\scrO_{\Sigma_{k-1}} \longrightarrow 0 \ , \\[4pt] \label{eq:sch-sequence-Y}
		0\longrightarrow i_{\Sigma_{k-1},\, \ast}\scrO_{\Sigma_{k-1}} \longrightarrow i_{\Sigma_k,\, \ast}\scrO_{\Sigma_k} \longrightarrow i_{\Sigma,\, \ast}\scrO_{\Sigma} \longrightarrow 0\ 
	\end{align}
	where the morphisms $i_{\Sigma_k,\, \ast}\scrO_{\Sigma_k} \longrightarrow i_{\Sigma_{k-1},\, \ast}\scrO_{\Sigma_{k-1}}$ and $i_{\Sigma_k,\, \ast}\scrO_{\Sigma_k} \longrightarrow i_{\Sigma,\, \ast}\scrO_{\Sigma}$ are the canonical surjections. 
\end{theorem}

\begin{remark}
	In the above theorem, all possible values of $n\geq 2$ can appear, as we shall show in \S\ref{subsec:examples-section-A}, where we discuss some explicit examples. 
\end{remark}

\begin{remark}	
	Furthermore, the above theorem was inspired by \cite[Theorem~3.1]{Spherical_twists}. In the modern language of Donovan-Wemyss' contraction algebras \cite{DW16}, in \textit{loc.cit.} Toda observed that the contraction algebra is $\C[t]/(t^n)$, which is uniserial, and so every module has a unique filtration by simples. The ``maximal'' filtration possible is given by the module $\C[t]/(t^n)$ itself, and every other module has a smaller filtration. From this point of view, $n$ coincides with Donovan-Wemyss' noncommutative width and Reid's width \cite{Reid-Minimal-Models}. The theory of contraction algebras would be essential to get a result corresponding to Theorem~\ref{thm:filtration} for other flops.\footnote{This observation was explained to us by Michael Wemyss.}	
\end{remark}

Furthermore, one more condition will be assumed to hold: 
\begin{assumption}\label{assumption:divisor}
	There exists a smooth connected Weil divisor $W\subset X$ so that its strict transform $S\subset Y$ is a smooth connected surface intersecting $\Sigma$ transversely at a single closed point $p\in \Sigma$. Moreover, the restriction $f_S\coloneqq f\vert_S\colon S\to W$ is an isomorphism. 
\end{assumption}

From now on, we assume that Assumptions~\ref{assumption:f} and \ref{assumption:divisor} hold.
\begin{proposition}\label{prop:dirimg-A} 
	One has 
	\begin{align}\label{eq:dirimg-B}
		\R^0 f_\ast i_{\Sigma,\, \ast}\scrO_\Sigma\simeq \scrO_\nu\quad\text{and} 
		\quad \R^i f_\ast i_{\Sigma,\, \ast}\scrO_\Sigma =0
	\end{align}
	for $i\geq 1$,
	\begin{align}\label{eq:dirimg-BC} 
		\R^0 f_\ast(i_{\Sigma,\, \ast}\scrO_\Sigma\otimes \scrO_Y(-kS)) = 0
	\end{align}
         for $k \geq 1$, 
	\begin{align}\label{eq:dirimg-BB} 
		\R^1 f_\ast(i_{\Sigma,\, \ast}\scrO_\Sigma\otimes \scrO_Y(-S)) = 0
	\end{align}
	and  
	\begin{align}\label{eq:dirimg-C} 
		\R^i f_\ast \calI_\Sigma =0
	\end{align}
	for $i \geq 1$.
\end{proposition} 

\begin{proof}
	Since $f\vert_{Y\smallsetminus \Sigma}\colon Y\smallsetminus \Sigma \to X\smallsetminus \{\nu\}$ is an isomorphism, all direct images $\R^i f_\ast(i_{\Sigma,\, \ast}\scrO_{\Sigma}\otimes \scrO_Y(aS))$ are set-theoretically supported at $\nu$ for all $a\in\Z$. Let $U\subset X$ be an affine open neighborhood of $\nu$, let $Y_U \coloneqq f^{-1}(U)$ and let $f_U\coloneqq Y_U \to U$ be the restriction of $f$ to $Y_U$. Note that $\Sigma \subset Y_U$. Then it suffices to show that relations \eqref{eq:dirimg-B}, \eqref{eq:dirimg-BC}, and \eqref{eq:dirimg-BB} hold for $f_U$. This follows from \cite[Proposition~III.8.5]{Har77}. Thus, the relations \eqref{eq:dirimg-B}, \eqref{eq:dirimg-BC} and \eqref{eq:dirimg-BB} hold. 

	Let us prove the vanishing result \eqref{eq:dirimg-C}. We have a long exact sequence 
	\begin{align}
		0\longrightarrow f_\ast \calI_\Sigma \longrightarrow f_\ast\scrO_Y\longrightarrow f_\ast i_{\Sigma,\, \ast}\scrO_\Sigma \longrightarrow \R^1 f_\ast \calI_\Sigma \longrightarrow \cdots 
	\end{align}
	associated to the exact sequence 
	\begin{align}
		0\longrightarrow \calI_\Sigma \longrightarrow \scrO_Y \longrightarrow i_{\Sigma,\, \ast} \scrO_\Sigma \longrightarrow 0\ . 
	\end{align}
	Relations \eqref{eq:dirimg-A} and \eqref{eq:dirimg-B} imply 
	\begin{align}
		\R^i f_\ast \calI_\Sigma =0\ ,
	\end{align}
	for $i\geq 2$. Moreover, $f_\ast\scrO_Y=\scrO_X$ and the morphism $f_\ast \scrO_Y\to f_\ast i_{\Sigma,\, \ast}\scrO_\Sigma$ is the pushforward of the canonical surjection $\scrO_Y \to i_{\Sigma,\, \ast}\scrO_\Sigma$. Since $f_\ast i_{\Sigma,\, \ast}\scrO_\Sigma\simeq \scrO_\nu$, the pushforward is surjective as well. Therefore $\R^1 f_\ast \calI_\Sigma =0$. 
\end{proof}

\begin{corollary}\label{cor:vanishing-A} 
	Let $\Sigma_k$ be the subscheme introduced in Theorem~\ref{thm:filtration}. 
	For any $1\leq k \leq n$ one has 
	\begin{align}\label{eq:dirimg-D} 
		\R^i f_\ast i_{\Sigma_k,\, \ast}\scrO_{\Sigma_k} =0\ ,
	\end{align}
	for $i \geq 1$, and 
	\begin{align}\label{eq:tor-vanishing-A} 
		\mathsf{Tor}_i(i_{S,\, \ast}\scrO_S, i_{\Sigma_k,\, \ast}\scrO_{\Sigma_k})=0\ ,
	\end{align}
	for $ i \geq 1$.
\end{corollary} 

\begin{proof}
	By construction, one has the exact sequences \eqref{eq:sch-sequence-X}:
	\begin{align}\label{eq:sch-sequence-A'}
		0\longrightarrow i_{\Sigma,\, \ast}\scrO_\Sigma \longrightarrow i_{\Sigma_{k+1},\, \ast}\scrO_{\Sigma_{k+1}} \longrightarrow i_{\Sigma_k,\, \ast}\scrO_{\Sigma_{k}} \longrightarrow 0 \ , 
	\end{align}
	for $1\leq k \leq n-1$. We prove the vanishing results \eqref{eq:dirimg-D} and  \eqref{eq:tor-vanishing-A} inductively. For $k=1$, we have $\Sigma_1=\Sigma$. Hence, we observe that it holds for $k = 1$. Then, the inductive step follows from the exact sequence \eqref{eq:sch-sequence-A'}. Therefore, the first claim holds.
	
	Since $\scrO_\Sigma$ is purely one-dimensional and the set theoretic intersection $S\cap \Sigma=\{p\}$, the second vanishing result follows immediately from the exact sequence 
	\begin{align}
		0\longrightarrow \scrO_Y(-S) \longrightarrow \scrO_Y \longrightarrow i_{S,\, \ast}\scrO_S\longrightarrow 0\ .
	\end{align}
	The inductive step follows straightforwardly by using the short exact sequence \eqref{eq:sch-sequence-A'}.
\end{proof}

For each $1\leq k \leq n$, let $T_k$ be the scheme-theoretic intersection of $\Sigma_k$ and $S$ in $Y$ and let $i_{T_k}\colon T_k \to Y$ denote the canonical closed immersion. By Assumption~\ref{assumption:divisor}, this is a zero-dimensional subscheme of $S$ with set-theoretic support $p$. Let also $\calQ_k\coloneqq f_\ast i_{\Sigma_k,\, \ast}\scrO_{\Sigma_k}$ for $1\leq k \leq n$. Then the following holds:
\begin{corollary}\label{cor:dirimg-B} 
	 Each morphism 
	\begin{align}
		\phi_k\colon \calQ_k\longrightarrow f_\ast i_{T_k,\, \ast}\scrO_{T_k}\
	\end{align}
	determined by the canonical epimorphism $i_{\Sigma_k,\, \ast}\scrO_{\Sigma_k}\to i_{T_k,\, \ast}\scrO_{T_k}$ is an isomorphism. Moreover, one has exact sequences 
	\begin{align}\label{eq:Q-sequence-X}
		0\longrightarrow \calQ_1 \longrightarrow \calQ_k \longrightarrow\calQ_{k-1}\longrightarrow 0 \ , \\[4pt] \label{eq:Q-sequence-Y}
		0\longrightarrow \calQ_{k-1} \longrightarrow \calQ_k \longrightarrow \calQ_1 \longrightarrow 0\ ,
	\end{align}
	induced by the exact sequences \eqref{eq:sch-sequence-X} and \eqref{eq:sch-sequence-Y} 
	respectively.
\end{corollary} 

\begin{proof}
	Note the canonical isomorphism 
	\begin{align}
		i_{T_k,\, \ast}\scrO_{T_k}\simeq i_{\Sigma_k,\, \ast} \scrO_{\Sigma_k}\otimes i_{S,\, \ast}\scrO_S. 
	\end{align}
	As a consequence of Equation~\eqref{eq:tor-vanishing-A}, the canonical exact sequence 
	\begin{align}
		i_{\Sigma_k,\, \ast}\scrO_{\Sigma_k}\otimes \scrO_Y(-S) \longrightarrow i_{\Sigma_k,\, \ast}\scrO_{\Sigma_k}\longrightarrow i_{\Sigma_k,\, \ast} \scrO_{\Sigma_k}\otimes i_{S,\, \ast}\scrO_S \longrightarrow 0
	\end{align}
	is exact to the left. Therefore, it suffices to prove the vanishing result 
	\begin{align}
		\R^1 f_\ast (i_{\Sigma_k,\, \ast}\scrO_{\Sigma_k}\otimes \scrO_Y(-S)) =0 
	\end{align}
	for $i\geq 0$. Given the exact sequence \eqref{eq:sch-sequence-A'}, this follows by induction from Equation~\eqref{eq:dirimg-BB}. 
	
	The second claim follows by pushforward from the exact sequences \eqref{eq:sch-sequence-X} and \eqref{eq:sch-sequence-Y} using the vanishing result \eqref{eq:dirimg-D}. 
\end{proof}

For each $1\leq k \leq n$, let $Z_k \subset W$ be the closed subscheme identified with $T_k\subset S$ by the isomorphism $f\vert_S\colon S \to W$. Let $i_{Z_k}\colon Z_k \to X$  denote the canonical closed immersion. Moreover, since $f_\ast\scrO_Y=\scrO_X$, note that the canonical epimorphism $\scrO_Y \to \scrO_{\Sigma_k}$ induces a morphism $\scrO_X \to \calQ_k$ for all $1\leq k \leq n$. Then one further has:
\begin{corollary}\label{cor:dirimg-C}
	For each $1\leq k \leq n$ there is a commutative diagram 
	\begin{align}\label{eq:QZdiag-A}
		\begin{tikzcd}[ampersand replacement=\&]
			\& \scrO_X\ar[dr] \ar[dl] \& \\
			\calQ_k\ar[rr] \&\& i_{Z_k, \ast}\scrO_{Z_k} 
		\end{tikzcd}
	\end{align}
	where the arrow $\scrO_X \to i_{Z_k, \ast}\scrO_{Z_k}$ is the canonical epimorphism and the bottom arrow is an isomorphism. Moreover, the morphisms $\scrO_X \to \calQ_k$, for $2\leq k \leq n$, fit into commutative diagrams 
		\begin{align}
		\begin{tikzcd}[ampersand replacement=\&]
			\& \scrO_X\ar[dr] \ar[dl] \& \\
			\calQ_k\ar[rr] \&\& \calQ_{k-1} 
		\end{tikzcd}
		\quad \text{and}\quad
		\begin{tikzcd}[ampersand replacement=\&]
				\& \scrO_X \ar[dr] \ar[dl] \& \\
			\calQ_{k} \ar[rr] \&\& \calQ_1  
		\end{tikzcd}\ ,
	\end{align}
	where the bottom arrows are the right epimorphisms in the exact sequences \eqref{eq:Q-sequence-X} and \eqref{eq:Q-sequence-Y}. 
\end{corollary} 

\begin{proof}
	For each $1\leq k \leq n$ one has a natural commutative diagram 
	\begin{align}
		\begin{tikzcd}[ampersand replacement=\&]
			\scrO_Y \ar[r]\ar[d] \&  i_{S,\, \ast}\scrO_S \ar[d]\\
			i_{\Sigma_k,\, \ast}\scrO_{\Sigma_k} \ar[r] \& i_{T_k,\, \ast}
			\scrO_{T_k}
		\end{tikzcd}\ ,
	\end{align}
	where all arrows are surjective. Since the restriction $f_S\colon S\to W$ is an isomorphism by assumption, this yields by pushforward a second commutative diagram 
	\begin{align}
		\begin{tikzcd}[ampersand replacement=\&]
			\scrO_X \ar[r]\ar[d] \& i_{W, \ast} \scrO_W \ar[d]\\
			\calQ_k \ar{r}{\phi_k}\&   f_\ast i_{T_k,\, \ast}\scrO_{T_k}
		\end{tikzcd}\ ,
	\end{align}
	where the top horizontal arrow is the natural epimorphism and bottom horizontal arrow is the isomorphism obtained in Corollary~\ref{cor:dirimg-B}. Moreover, by construction, there is a canonical isomorphism $f_\ast i_{T_k,\, \ast}\scrO_{T_k}\simeq i_{Z_k\, \ast}\scrO_{Z_k}$. This yields diagram \eqref{eq:QZdiag-A}.

	The second statement follows by applying $f_\ast$ to the natural commutative diagrams 
	\begin{align}
		\begin{tikzcd}[ampersand replacement=\&]
			\& \scrO_Y\ar[dr] \ar[dl] \& \\
			i_{\Sigma_k, \ast} \scrO_{\Sigma_k} \ar[rr] \&\& i_{\Sigma_{k-1}, \ast} \scrO_{\Sigma_{k-1}}
		\end{tikzcd}
		\quad\text{and}\quad
		\begin{tikzcd}[ampersand replacement=\&]
				\& \scrO_Y \ar[dr] \ar[dl] \& \\
			i_{\Sigma_k, \ast} \scrO_{\Sigma_k}  \ar[rr] \&\& i_{\Sigma_1, \ast} \scrO_{\Sigma_1} 
		\end{tikzcd} \ ,
	\end{align}
	respectively.
\end{proof}

\section{Semistable sheaves on the exceptional locus}\label{sec:semistable-exceptional-locus}

The goal of this section is to provide a complete classification for semistable pure one-dimen\-sio\-nal sheaves $\calF$ on $Y$ with $\ch_2(\calF)=r[\Sigma]$ and $\chi(\calF)=rd$ for $r,d\in \Z$ and $r\geq 1$, working under Assumptions~\ref{assumption:f} and \ref{assumption:divisor}. By Lemma~\ref{lem:exclass}, any such coherent sheaf will be set theoretically supported on $\Sigma$, hence, by \cite[Proposition~2.8]{Orlov-formal}, the results obtained in this section depend only on the formal completion of $Y$ along $\Sigma$. Since the set-theoretic support is fixed, the main difficulty lies in understanding the $\scrO_Y$-module structure of semistable sheaves\footnote{Note that this moduli problem admits an equivalent formulation in terms of Higgs bundles on $\PP^1$ in situations where $\Sigma \subset Y$ is a section of a projection $\pi\colon Y \to \PP^1$. This is the case for example for the threefolds $Y^\pm$ in Proposition~\ref{prop:example-A}. Then, the assignment $\calF \mapsto \pi_\ast \calF$ yields a one-to-one correspondence between semistable sheaves $\calF$ and  semistable Higgs bundles $(\calE, \Phi)$ on $\PP^1$, where $\Phi=(\Phi_1, \Phi_2)$ is a two component Higgs field $\Phi\colon \calE \to \calE\otimes \left(\scrO_{\PP^1}\oplus \scrO_{\PP^1}(-2)\right)$ so that $\Phi\wedge \Phi=0$ and $\Phi_1^{n+1}=0$. We will not prove this statement because it is not used anywhere the paper.}. In particular, Lemma~\ref{lem:stability-C} proves that all such semistable sheaves are scheme-theoretically supported on the unique thickening $\Sigma_n$ of $\Sigma$ constructed in Theorem~\ref{thm:filtration}. Then, the main result of this section proves an equivalence between the moduli stack of semistable sheaves $\calF$ on $Y$ as above and the moduli stack of $\scrO_{T_n}$-modules of length $r$ (cf.\ Proposition~\ref{prop:stability-F} and Corollary~\ref{cor:stability-G}).

\medskip
 
Recall that $T_n\subset Y$ is the scheme theoretic intersection of $\Sigma_n$ and $S$ in $Y$, and it is furthermore zero dimensional. 

\subsection{Semistable sheaves and filtrations}

Let $\calF$ be a nonzero purely one-dimensional sheaf on $Y$ with $\ch_2(\calF)= r [\Sigma]$, with $r>0$. First, note that $\calF$ is set-theoretically supported on $\Sigma$ by Corollary~\ref{cor:split-support}. We define the \textit{slope} of $\calF$ as
\begin{align}
	\mu_Y(\calF)\coloneqq \frac{\chi(\calF)}{r}\ .
\end{align}
Note that the slope does not depend on a specific choice of a polarization. We consider the corresponding slope-semistability, which does not depend on a specific polarization either. The sheaf $\calF$ will be simply called \textit{(semi)stable} if it is \textit{$\mu_Y$-slope (semi)stable}. 

For ease of exposition, let $\catCoh^{\mathsf{pure}}_\Sigma(Y)\subset \catCoh(Y)$ denote the full subcategory of purely one-dimensional sheaves on $Y$ set-theoretically supported on $\Sigma$.

The first goal of this section is to prove that any stable sheaf in $\catCoh^{\mathsf{pure}}_\Sigma(Y)$ is isomorphic to the pushforward of a line bundle on $\Sigma$. We start with some preliminary results. 
\begin{lemma}\label{lem:stability-A} 
	For any nonzero sheaf $\calF\in \catCoh^{\mathsf{pure}}_\Sigma(Y)$ one has 
	\begin{align}
			\mu_Y(\calF\otimes \scrO_Y(S)) = \mu_Y(\calF)+1\ .
	\end{align}
	Therefore $\calF$ is (semi)stable if and only if $\calF\otimes \scrO_Y(S)$ is (semi)stable.
\end{lemma}

\begin{proof}
	Since $\calF$ is purely one-dimensional and $S\cap \Sigma=\{p\}$, one has 
	\begin{align}	
		\Tor_i( i_{S,\, \ast} \scrO_S,\calF)=0
	\end{align}
	for $i\geq 1$. By tensoring by $\calF$ the exact sequence 
	\begin{align}
		0\longrightarrow \scrO_Y(-S) \longrightarrow \scrO_Y \longrightarrow i_{S,\, \ast}\scrO_S \longrightarrow 0 
	\end{align}
	and using the Grothendieck-Riemann-Roch theorem, one gets
	\begin{align}
		\chi(\calF\otimes\scrO_Y(S)) =\chi(\calF) + \chi(\calF\otimes \scrO_Y(S)\otimes i_{S,\, \ast} \scrO_S)= \chi(\calF) + S\cdot \ch_2(\calF)\ .
	\end{align}
\end{proof}

\begin{lemma}\label{lem:stable-mmA} 
	Let $\calL_1$ and $\calL_2$ be line bundles on $\Sigma$ so that $\chi(\calL_1) < \chi(\calL_2)$. Suppose that $\calE$ is an extension of the form 
	\begin{align}\label{eq:lb-ext-Y}
		0\longrightarrow i_{\Sigma,\, \ast}\calL_1\longrightarrow \calE \longrightarrow i_{\Sigma,\, \ast} \calL_2 \longrightarrow 0 
	\end{align}
	on $Y$. Then $\calE$ is isomorphic to the pushforward of a locally free sheaf on $\Sigma$. 
\end{lemma} 

\begin{proof}
	Let $e\in \Ext^1_Y(i_{\Sigma,\,\ast}\calL_2, i_{\Sigma,\, \ast}\calL_1)$ be the extension class associated to the exact sequence \eqref{eq:lb-ext-Y}. Let $p\in \Sigma$ be an arbitrary closed point. Recall that the standard local to global spectral sequence yields an exact sequence 
	\begin{align}
		\begin{tikzcd}[ampersand replacement=\&, row sep=tiny]
				0\ar{r}\& H^1(\calExt^0_Y(i_{\Sigma,\, \ast}\calL_2, i_{\Sigma,\, \ast}\calL_1))\ar{r}\& \Ext^1_Y(i_{\Sigma,\, \ast}\calL_2, i_{\Sigma,\, \ast}\calL_1)\\
				{} \ar{r}{f} \& H^0(\calExt^1_Y(i_{\Sigma,\, \ast}\calL_2, i_{\Sigma,\, \ast}\calL_1))\ar{r}\& 0
		\end{tikzcd}\ ,
	\end{align}
	where the germ of $f(e)$ at $p$ is the extension class associated to the extension of stalks
	\begin{align}\label{eq:stalk-ext}
		0\longrightarrow  (i_{\Sigma,\, \ast}\calL_1)_p\longrightarrow \calE_p \longrightarrow (i_{\Sigma,\, \ast} \calL_2)_p \longrightarrow 0\ . 
	\end{align}
	In particular if $f(e)=0$, it follows that the extension \eqref{eq:stalk-ext} is trivial at all $p\in \Sigma$, which implies that $\calE$ is scheme-theoretically supported on $\Sigma$. This implies the claim. 
	
	In conclusion, it suffices to prove that 
	\begin{align}\label{eq:zero-loc-ext}
		H^0(Y, \calExt^1_Y(i_{\Sigma,\, \ast}\calL_2, i_{\Sigma,\, \ast}\calL_1))=0\ .
	\end{align}
	Since $\Sigma\simeq \PP^1$, using Assumption~\ref{assumption:divisor}, one has $\calL_i \simeq \scrO_Y(d_i S)\otimes \scrO_\Sigma$ for some $d_i \in \Z$ so that $d_1< d_2$. Then one has 
	\begin{align}
		\calExt^1_Y(i_{\Sigma,\, \ast}\calL_2, i_{\Sigma, \, \ast}\calL_1) \simeq \calExt^1_Y(i_{\Sigma,\, \ast}\scrO_\Sigma, i_{\Sigma,\, \ast}\scrO_\Sigma) \otimes \scrO_Y((d_1-d_2)S)\ . 
	\end{align}
	Moreover, the exact sequence 
	\begin{align}
		0 \longrightarrow \calI_\Sigma \longrightarrow \scrO_Y \longrightarrow i_{\Sigma,\, \ast}\scrO_\Sigma \longrightarrow 0
	\end{align}
	yields isomorphisms 
	\begin{align}
		\calExt^1_Y(i_{\Sigma,\, \ast}\scrO_\Sigma, i_{\Sigma,\, \ast}\scrO_\Sigma) \simeq \calExt^0_Y(\calI_\Sigma, i_{\Sigma,\, \ast}\scrO_\Sigma) \simeq i_{\Sigma,\, \ast}\big(\scrO_\Sigma\oplus \scrO_\Sigma(-2))\big)\ ,
	\end{align}
	since $\Sigma$ is a $(0,-2)$ curve on $Y$. 
	Therefore 
	\begin{align}
		\calExt^1_Y(i_{\Sigma,\, \ast}\calL_2, i_{\Sigma,\, \ast}\calL_1) \simeq i_{\Sigma,\, \ast}\big( \scrO_\Sigma(d_1-d_2)\oplus \scrO_\Sigma(d_1-d_2-2)\big)\ .
	\end{align}
	Since $d_1<d_2$, this implies the vanishing result \eqref{eq:zero-loc-ext}.
\end{proof}

We have the following result.
\begin{proposition}\label{prop:stable-mmB}
	Suppose that $\calF$ is a nonzero stable purely one-dimensional sheaf on $Y$ with set-theoretic support on $\Sigma$. Then $\calF$ is isomorphic to the pushforward of a line bundle on $\Sigma$. 
\end{proposition} 

\begin{proof}
	Since $\Sigma \simeq \PP^1$, it suffices to prove that $\calF$ is scheme-theoretically supported on $\Sigma$. Let $\ch_2(\calF)=r[\Sigma]$, with $r>0$. If $r=1$ there is nothing to prove. 

	By induction, suppose the claim holds for all stable sheaves $\calF'$ with $\ch_2(\calF') = r'[\Sigma]$, with $1\leq r'\leq r-1$. Using Lemma~\ref{lem:stability-A}, one can assume that 
	\begin{align}\label{eq:chi-ineq-A} 
		0\leq \chi(\calF) \leq r-1
	\end{align}
	without loss of generality. 

	Let $\calT\subset \calF\otimes i_{\Sigma,\, \ast}\scrO_\Sigma$ be the maximal zero-dimensional subsheaf and let $\calG \coloneqq \calF\otimes i_{\Sigma,\, \ast}\scrO_\Sigma/\calT$. Then there is an exact sequence 
	\begin{align}
		0\longrightarrow \calE \longrightarrow \calF \longrightarrow \calG \longrightarrow 0\ , 
	\end{align}
	with $\calE$ and $\calG$ purely one-dimensional sheaves on $Y$ supported on $\Sigma$. Moreover, by construction, $\calG$ is the pushforward of a locally free sheaf on $\Sigma$. 
Note also that $\calF\otimes i_{\Sigma,\, \ast}\scrO_\Sigma$ is a nonzero $\scrO_Y$-module
with $\ch_2(\calF\otimes i_{\Sigma,\, \ast}\scrO_\Sigma) \neq 0$ by Nakayama's lemma. Therefore 
 $\ch_2(\calG)\neq 0$. 
	
	If $\calE=0$, it follows that $\calF$ is scheme theoretically supported on $\Sigma$. Since 
	it is stable by assumption, the claim follows. 
	
	Suppose $\calE\neq 0$. Since $\ch_2(\calG)\neq 0$, it follows that $\ch_2(\calE)=s[\Sigma]$ for some $1\leq s \leq r-1$. Let 
	\begin{align}
		0=\calH_0\subset \calH_1 \subset \cdots \subset \calH_k =\calE
	\end{align}
	be the Harder-Narasimhan filtration of $\calE$. Using the inductive hypothesis, each successive quotient $\calH_{i+1}/\calH_i$ is $S$-equivalent to a sheaf of the form $i_{\Sigma,\, \ast}(\C^{s_i} \otimes \scrO_\Sigma(d_i))$ with $s_i\geq 1$. Furthermore, since $\calF$ is stable of slope
	\begin{align}
		0\leq \mu(\calF) \leq \frac{r-1}{r} \ ,
	\end{align}
	one has 
	\begin{align}\label{eq:chi-ineq-B}
		d_k < \cdots < d_1 \leq -1\ . 
	\end{align}
	Therefore the Harder-Narasimhan filtration admits a refinement 
	\begin{align}
		0=\calE_0\subset \calE_1 \subset \cdots \subset \calE_s =\calE 
	\end{align}
	so that each quotient $\calE_i/\calE_{i-1}\simeq i_{\Sigma,\, \ast}\calL_i$ for some line bundles $\calL_i$ on $\Sigma$ so that 
	\begin{align}
		\chi(\calL_s) \leq \chi(\calL_{s-1}) \leq \cdots \leq \chi(\calL_1) \leq 0\ .
	\end{align}
	Moreover, $\calG$ is the pushforward of a direct sum of line bundles 
	\begin{align}
		\bigoplus_{j=1}^\ell \calM_j 
	\end{align}
	so that $\chi(\calM_j) \geq 1$.
	
	Note the exact sequence 
	\begin{align}\label{eq:quot-seq-A}
		0\longrightarrow \calE/\calE_{s-1} \longrightarrow \calF/\calE_{s-1} \longrightarrow \calG \longrightarrow 0\ ,
	\end{align}
	where $\calE/\calE_{s-1} \simeq i_{\Sigma,\, \ast}\calL_s$ and $\chi(\calL_s) < \chi(\calM_j)$ for all $1\leq j \leq \ell$. Then Lemma~\ref{lem:stable-mmA} shows that $\calG_1\coloneqq \calF/\calE_{s-1}$ is the pushforward of a 
	locally free sheaf on $\Sigma$. 
	
	Proceeding by induction, suppose that $\calG_i \coloneqq \calF/\calE_{s-i}$ is the pushforward of a locally free sheaf on $\Sigma$ for some $1\leq i \leq s-1$. Then one has the exact sequence 
	\begin{align}\label{eq:quot-seq-B}
		0\longrightarrow \calE_{s-i}/\calE_{s-i-1} \longrightarrow \calF/\calE_{s-i-1} \longrightarrow \calG_i \longrightarrow 0\ ,
	\end{align}
	where 
	\begin{align}
		\calE_{s-i}/\calE_{s-i-1} \simeq  \calL_{s-i}
	\end{align}
	with $\chi(\calL_{s-i}) \leq 0$. At the same time, $\calG_i$ is again isomorphic to the pushforward of a direct sum of line bundles 
	\begin{align}
		\bigoplus_{j=1}^{\ell_i} \calM_{i,j} 
	\end{align}
	with $\chi(\calM_{i,j})\geq 1$. Then Lemma~\ref{lem:stable-mmA} shows that the central term in extension~\eqref{eq:quot-seq-B} is the pushforward of a locally free $\scrO_\Sigma$-module. This proves the inductive step. 
	
	In conclusion, $\calF$ is scheme-theoretically supported on $\Sigma$, as claimed. 
\end{proof}

\begin{corollary}\label{cor:stability-B} 
	A nonzero sheaf $\calE\in \catCoh^{\mathsf{pure}}_\Sigma(Y)$ is semistable if and only if it admits a filtration 
	\begin{align}
		0 = \calE_0\subset \calE_1 \subset \cdots \subset \calE_m=\calE 
	\end{align}
	so that $\calE_j/\calE_{j-1} \simeq i_{\Sigma,\, \ast}\calL$, with $1\leq j\leq m$, for some line bundle $\calL$ on $\Sigma$ so that $\chi(\calL) = \mu_Y(\calE)$. 
\end{corollary}

\begin{remark}	
	Given an arbitrary contraction $f\colon Y\to X$, with $X$ affine, Davison proved in \cite[Proposition~4.3]{Davison_Refined_Jacobi} a result similar to Corollary~\ref{cor:stability-B} for semistable coherent sheaves $\calF \in \catCoh^{\mathsf{pure}}_\Sigma(Y)$ with $\chi(\calF) = 0$. In contrast, our result does not require any condition on the Euler characteristic of the semistable sheaves on $Y$.
\end{remark}

\begin{corollary}\label{cor:refined-HN-filtration} 
	Let $\calF$ be a nonzero sheaf in $\catCoh^{\mathsf{pure}}_\Sigma(Y)$. Then $\calF$ has a filtration 
	\begin{align}\label{eq:refHN}
		0=\calF_0 \subset \calF_1 \subset \cdots \subset \calF_m = \calF 
	\end{align}
	so that 
	\begin{align}
		\calF_i/\calF_{i-1} \simeq i_{\Sigma,\, \ast}\calL_i 
	\end{align}
	for $1\leq i \leq m$, with $\calL_1, \ldots, \calL_m$ line bundles on $\Sigma$ satisfying
	\begin{align}
		\chi(\calL_1) \geq  \cdots \geq  \chi(\calL_m)\ .
	\end{align}
	Moreover, 
	\begin{align}
		\mu_{Y\textrm{-}\mathsf{max}}(\calF) = \chi(\calL_1)\quad\text{and} \quad \mu_{Y\textrm{-}\mathsf{min}}(\calF) = \chi(\calL_m)\ .
	\end{align}
\end{corollary} 

\begin{proof}
	Let 
	\begin{align}
		0\subset \calG_1 \subset \cdots \subset \calG_h =\calF
	\end{align}
	be the Harder-Narasimhan filtration of $\calF$ with respect to $\mu_Y$-slope stability. Then each subquotient $\calG_k/\calG_{k-1}$ has a filtration as in Corollary~\ref{cor:stability-B}. Then the Harder-Narasimhan filtration admits a non-unique refinement, so the claim follows.
\end{proof}

Recall that $\Sigma_n$ denotes the closed subscheme of $Y$ defined by the ideal sheaf $\calJ_n \subset \scrO_Y$ from Formula~\eqref{eq:ideal-filtration-A} in Theorem~\ref{thm:filtration}. The next result shows that the scheme-theoretic support of any semistable sheaf in $\catCoh^{\mathsf{pure}}_\Sigma(Y)$ is contained in $\Sigma_n$.

\begin{lemma}\label{lem:stability-C}
	Any nonzero semistable sheaf $\calF\in \catCoh^{\mathsf{pure}}_\Sigma(Y)$  is
	 scheme-theoretically supported on $\Sigma_n$. 
\end{lemma} 

\begin{proof}
	By Corollary~\ref{cor:stability-B}, the sheaf $\calF$ admits a Jordan-Hölder filtration 
	\begin{align}\label{eq:JH}
		0\subset \calF_1 \subset \cdots \subset \calF_j =\calF 
	\end{align}
	so that each successive quotient is isomorphic to
	\begin{align}
		\calF_i/\calF_{i-1} \simeq i_{\Sigma,\, \ast}\calL\ .
	\end{align}
for a  line bundle $\calL$ on 
	$\Sigma$, with 
	\begin{align}
		\chi(\calL) = \mu_Y(\calF)\ .
	\end{align}

	If $j=1$, i.e., if $\calF$ is stable, the claim follows from Proposition~\ref{prop:stable-mmB}. Assume that $j>1$. The proof will proceed by induction on $j$, keeping $\calL$ fixed. Suppose that the claim holds for all semistable sheaves $\calF'$ with slope 
	\begin{align}
		\mu_Y(\calF') = \chi(\calL)\ ,
	\end{align}
	which have $j'\leq j-1$ factors in their Jordan-Hölder filtrations. Let $\calF$ be a semistable sheaf with 
	\begin{align}
		\mu_Y(\calF) = \chi(\calL)\ ,
	\end{align}
	which has a Jordan-Hölder filtration of length $j$. Then $\calF$ fits into an exact sequence
	\begin{align}
		0\longrightarrow \calF'\longrightarrow \calF \longrightarrow \calG \longrightarrow 0\ , 
	\end{align}
	where $\calG\simeq i_{\Sigma,\, \ast}\calL$ and $\calF'$ is semistable, with a Jordan-Hölder filtration of length $j-1$. 
	
	Let $\sigma_n\colon \calJ_n \to \scrO_Y$ be the defining section of $\Sigma_n$. For any coherent sheaf $\calE$ on $Y$, the morphism corresponding to the multiplication by $\sigma_n$ will be denoted by $\mu_\calE \coloneqq \id_\calE\otimes \sigma_n \colon \calE \otimes \calJ_n \to \calE$. We have the commutative diagram 
	\begin{align}
		\begin{tikzcd}[ampersand replacement=\&]
			\& \calF'\otimes \calJ_n \ar[r] \ar{d}{\mu_{\calF'}} \& \calF\otimes \calJ_n \ar[r] \ar{d}{\mu_\calF} \& \calG \otimes \calJ_n \ar{d}{\mu_\calG} \ar[r]\& 0\\
			0\ar[r] \& \calF'\ar[r] \& \calF \ar[r] \& \calG  \ar[r] \& 0
			\end{tikzcd}\ ,
	\end{align}
	where the rows are exact. Since $\calG\simeq i_{\Sigma,\, \ast}\calL$, we get $\mu_\calG=0$. Moreover, $\mu_{\calF'}=0$ by the inductive hypothesis. Then the snake lemma yields the exact sequence
	\begin{align}
		\begin{tikzcd}[ampersand replacement=\&]
					\cdots \arrow{r} \& \calG \otimes \calJ_n \arrow{r}{\delta} \& \calF' \arrow{r} \& \mathsf{Coker}(\mu_\calF) \arrow{r} \& \calG \arrow{r} \& 0
		\end{tikzcd}\ .
	\end{align}
	Now, we claim that $\calG\otimes \calJ_n$ is semistable of slope $\mu_Y(\calF')+1$. Indeed, this follows from condition \eqref{eq:ideal-filtration-C} in Theorem~\ref{thm:filtration}, which implies that 
	\begin{align}
		i_\Sigma^\ast\calJ_n \simeq  \scrO_{\Sigma}(1)^{\oplus 2}\ .
	\end{align}
	Since $\calG\simeq i_{\Sigma,\, \ast}\calL$, one obtains 
	\begin{align}
		\calG\otimes \calJ_n \simeq i_{\Sigma,\, \ast} (\calL\otimes \scrO_\Sigma(1)^{\oplus 2})\ ,
	\end{align} 
	which implies the claim.

	Since $\calG\otimes \calJ_n$ and $\calF'$ are semistable, with $\mu_Y(\calG)= \mu_Y(\calF')+1$, one has again $\Hom_Y(\calG\otimes \calJ_n , \calF') =0$. Therefore $\delta = 0$, which implies that $\mathsf{Coker}(\mu_\calF)$ has the same topological invariants as $\calF$. This implies that $\mathsf{Im}(\mu_\calF)=0$, which proves the inductive step. 
\end{proof}

For future reference, note the following consequence of Lemma~\ref{lem:stability-C}. 
\begin{corollary}\label{cor:chi-bound} 
	Let $\calF$ be a coherent sheaf on $Y$ with $\ch_2(\calF) = m [\Sigma]$, $m \geq 1$. Assume that there exists a morphism $s\colon \scrO_Y \to \calF$ so that the cokernel $\mathsf{Coker}(s)$ is zero-dimensional. Then $\chi(\calF) \geq m$. 
\end{corollary} 

\begin{proof}
	Clearly, it suffices to prove the claim under the assumption that $s\colon \scrO_Y\to \calF$ is surjective. By Lemma~\ref{lem:exclass}, $\calF$ is set-theoretically supported on $S$. Let $\calF \twoheadrightarrow \calG$ be the last subquotient in the Harder-Narasimhan filtration of $\calF$ with respect to slope stability. Then the composition 
	\begin{align}
		\begin{tikzcd}[ampersand replacement=\&]
			\scrO_Y \ar{r}{s}\& \calF \ar{r}\& \calG 
		\end{tikzcd}
	\end{align}
	is surjective. Since $\calG$ is semistable, it is scheme-theoretically supported on $\Sigma_n$ by Lemma~\ref{lem:stability-C}. Therefore the above morphism factors through a surjective morphism
	\begin{align}
		i_{\Sigma_n,\ast} \scrO_{\Sigma_n} \longrightarrow \calG\ . 
	\end{align}
	Moreover, since $i_{\Sigma,\ast}\scrO_\Sigma$ is stable of slope $1$, the exact sequences \eqref{eq:sch-sequence-X} show that $i_{\Sigma_n,\ast} \scrO_{\Sigma_n}$ is semistable of slope $1$. Therefore $\mu(\calG)\geq 1$. Then the claim follows from the defining properties of the Harder-Narasimhan filtration.
\end{proof}

Finally, recall that the direct image $\calQ_k=f_\ast\scrO_{\Sigma_k}$ is isomorphic to the structure sheaf $\scrO_{Z_k}$ of a zero-dimensional subscheme $Z_k \subset W$ by Corollary~\ref{cor:dirimg-C}. Furthermore, the exact sequences \eqref{eq:Q-sequence-Y} yield exact sequences 
\begin{align}\label{eq:sch-sequence-Z}
	0\longrightarrow \scrO_{Z_{k-1}} \longrightarrow \scrO_{Z_k} \longrightarrow \scrO_\nu \longrightarrow 0\ ,
\end{align}
for all $2\leq k \leq n$, where the epimorphisms $\scrO_{Z_k} \to \scrO_\nu$ are canonical. Conversely, the following holds:

\begin{proposition}\label{prop:invimg-A} 
	The evaluation map $\ev_k\colon f^\ast\calQ_k \to \scrO_{\Sigma_k}$ is an isomorphism for all $1\leq k \leq n$. Moreover, one has a commutative diagram with exact rows 
	\begin{align}\label{eq:evdiag-A}
		\begin{tikzcd}[ampersand replacement=\&]
			0\ar[r] \& f^\ast\scrO_{Z_{k-1}} \ar[r] \ar{d}{\ev_{k-1}} \& 
			f^\ast\scrO_{Z_k} \ar[r]   \ar{d}{\ev_{k}}        \& f^\ast \scrO_\nu \ar[r]   \ar{d}{\ev_{1}}     \& 0\\
			 0\ar[r] \& \scrO_{\Sigma_{k-1}}\ar[r] \& \scrO_{\Sigma_k}\ar[r] \& \scrO_\Sigma \ar[r] \& 0
		\end{tikzcd}\ .
	\end{align}
\end{proposition}

\begin{proof} 
	Let $\calE_k \subset \scrO_{\Sigma_k}$ denote the image of the evaluation map, which is a purely one-dimensional sheaf with set-theoretic support on $\Sigma$. The induced morphism $f_\ast \calE_k \to \calQ_k$ is a tautological isomorphism, while Corollary~\ref{cor:dirimg-B} implies that 
	\begin{align}
		\chi(\calQ_k) = r_k
	\end{align}
	where $\ch_2(\scrO_{\Sigma_k}) = r_k [\Sigma]$. In conclusion, 
	\begin{align}\label{eq:ev-chi-A}
		\chi(f_\ast\calE_k) = r_k \ .
	\end{align}
	
	Let 
	\begin{align}\label{eq:ev-filt-A}
		0\subset \calE_{k,1} \subset \cdots \subset \calE_{k, s_k} = \calE_k 
	\end{align}
	be the filtration constructed in Corollary~\ref{cor:refined-HN-filtration}, where 
	\begin{align}
		\ch_2(\calE_k) = s_k [\Sigma]\ .
	\end{align}
	Each successive quotient $\calE_{k, i}/\calE_{k, i-1}$ is isomorphic to the pushforward of a line bundle $\calL_{k,i}$ on $\Sigma$ so that 
	\begin{align}
		\chi(\calL_{k,s_k}) \leq \cdots\leq  \chi(\calL_{k,1})\ . 
	\end{align}
	Moreover, the exact sequences \eqref{eq:sch-sequence-X} and \eqref{eq:sch-sequence-Y} imply that $\scrO_{\Sigma_k}$ admits a filtration so that all subquotients are isomorphic to $i_{\Sigma,\, \ast}\scrO_\Sigma$. Therefore $\scrO_{\Sigma_k}$ is semistable of slope $1$ by Corollary~\ref{cor:stability-B}. This implies that 
	\begin{align}
		\chi(\calL_{k,s_k}) \leq \cdots \leq \chi(\calL_{k,1}) \leq 1\ .
	\end{align}
	
	 Filtration \ref{eq:ev-filt-A} yields a second filtration 
	\begin{align}\label{eq:ev-filt-B}
		0\subset f_\ast \calE_{k,1} \subset \cdots \subset f_\ast\calE_{k,s_k} = f_\ast \calE_k 
	\end{align} 
	so that for each $1\leq i \leq s_k$ there is an exact sequence 
	\begin{align}
		0\longrightarrow f_\ast \calE_{k,i-1}  \longrightarrow f_\ast \calE_{k,i} \longrightarrow f_\ast i_{\Sigma,\ast}\calL_{k,i} \to \cdots 
	\end{align}
	Moreover, Proposition~\ref{prop:dirimg-A} shows that 
	\begin{align}
		f_\ast i_{\Sigma,\ast}\calL_{k,i}\simeq  \begin{cases}
			\scrO_\nu & \text{if } \chi(\calL_{k,i}) =1\ , \\ 
			0 & \text{otherwise}\ . 
		\end{cases} 
	\end{align}
	In particular $f_\ast i_{\Sigma,\ast}\calL_{k,i}$ is zero dimensional, which implies 
	\begin{align}
		\chi(f_\ast \calE_{k,i}/f_\ast \calE_{k,i-1} ) \leq \chi(f_\ast i_{\Sigma,\ast}\calL_{k,i}) \leq 1
	\end{align}
	for all $1\leq i \leq s_k$. 	Then one obtains 
	\begin{align}
		\chi(f_\ast\calE_k) \leq \ell_k
	\end{align}
	where $0\leq \ell_k \leq s_k$ is the number of subquotients of the filtration \eqref{eq:ev-filt-A} of Euler characteristic 1. Since $\ell_k \leq s_k \leq r_k$ by construction, Equation \eqref{eq:ev-chi-A} implies $\ell_k = s_k=r_k$. Hence, $\calL_{k,i} \simeq \scrO_\Sigma$ for all $1\leq i \leq s_k$. Therefore $\calE_k$ has the same topological invariants as $\scrO_{\Sigma_k}$, which implies that $\calE_k = \scrO_{\Sigma_k}$. This proves the first claim. 
	
	In order to prove the second claim, note that diagram \eqref{eq:evdiag-A} is naturally commutative and the bottom row is exact. Since the vertical arrows are isomorphisms, the top row is also exact. 
\end{proof} 

\begin{corollary}\label{cor:invimg-B}
	For each $1\leq k \leq n$ one has a cartesian diagram 
	\begin{align}\label{eq:sch-diagram-A} 
		\begin{tikzcd}[ampersand replacement=\&]
			\Sigma_k \ar[swap]{d}{\phi_k} \arrow{r}{i_{\Sigma_k}}  \&  Y \ar{d}{f} \\
			Z_k \arrow{r}{i_{Z_k}} \& X 
		\end{tikzcd} \ ,
	\end{align}
	where $\phi_k$ is flat. 
\end{corollary}

\begin{proof} 
	As shown in Corollary~\ref{cor:dirimg-C}, the pushforward of the canonical morphism $\scrO_Y\to  \scrO_{\Sigma_k}$ coincides with the canonical morphism $\scrO_X\to \scrO_{Z_k}$. Then one obtains a tautological commutative diagram 
	\begin{align}
	\begin{tikzcd}[ampersand replacement=\&]
		f^\ast\scrO_X \ar[r]\ar[d]  \& \scrO_Y \ar[d] \\
		 f^\ast \scrO_{Z_k}\arrow{r}{\ev_k}\& \scrO_{\Sigma_k}
	\end{tikzcd}\ .
	\end{align}
	By Proposition~\ref{prop:invimg-A}, the evaluation map in the above diagram is an isomorphism. Hence $\Sigma_k\subset Y$ coincides with the scheme theoretic inverse image $f^{-1}(Z_k)$. 
	
	In order to prove that $\phi_k$ is flat, note that $Z_k$ is isomorphic to the spectrum of a local artinian ring over $\C$, since it is zero dimensional and it has a unique closed point $\nu$. Using the local criterion of flatness, it suffices to prove that the natural morphism 
	\begin{align}
		\phi_k^\ast \calI_\nu \longrightarrow \phi_k^\ast\scrO_{Z_k} 
	\end{align}
	is injective, where $\calI_\nu\subset \scrO_{Z_k}$ is the defining ideal sheaf of $\nu$ in $Z_k$. 
	
	Since the diagram \eqref{eq:sch-diagram-A} is cartesian, the exact sequence 
	\begin{align}
		0\longrightarrow f^\ast  \scrO_{Z_{k-1}} \longrightarrow f^\ast  \scrO_{Z_k} \longrightarrow f^\ast \scrO_\nu \longrightarrow 0 
	\end{align}
	in diagram \eqref{eq:evdiag-A} is the pushforward of the exact sequence 
	\begin{align}
		\phi_k^\ast\scrO_{Z_{k-1}} \longrightarrow \phi_k^\ast\scrO_{Z_k} \longrightarrow \phi_k^\ast\scrO_\nu \longrightarrow 0
	\end{align}
	obtained by pulling back the exact sequence
	\begin{align}
		\begin{tikzcd}[ampersand replacement=\&]
			0\ar{r}\&  \scrO_{Z_{k-1}} \ar{r}{\xi_k} \&\scrO_{Z_k} \ar{r}\& \scrO_\nu \ar{r}\& 0
		\end{tikzcd}
	\end{align}
	to $\Sigma_n$. Therefore the natural morphism 
	\begin{align}
		\phi_k^\ast\scrO_{Z_{k-1}} \longrightarrow \phi_k^\ast\scrO_{Z_k} 
	\end{align}
	is injective. Since the morphism $\xi_k\colon\scrO_{Z_{k-1}} \to \scrO_{Z_k}$ maps $\scrO_{Z_{k-1}}$ isomorphically onto $\calI_\nu\subset \scrO_{Z_k}$, this proves the claim. 
\end{proof}

\subsection{Moduli stacks}\label{subsect:moduli-stacks}

For any $r, c\in \Z$ with $r\geq 1$, let $\calM^{\mathsf{ss}}(Y; r[\Sigma], c)$ denote the moduli stack of semistable one-dimensional sheaves $\calF$ on $Y$ with $\ch_2(\calF)=r[\Sigma]$ and $\chi(\calF)=c$. By Corollary~\ref{cor:stability-B} this stack is empty unless $c=rd$ for some $d\in\Z$. From now on, we assume that $c$ is of this form. For $d=1$ the moduli stack $\calM^{\mathsf{ss}}(Y; r[\Sigma], r)$ will be denoted by $\calM^{\mathsf{ss}}(Y; r)$ for simplicity. 

Let $n$ be the integer appearing in the filtration \eqref{eq:ideal-filtration-A} of $\calI_\Sigma$. As in Corollary \ref{cor:dirimg-B},  let $T_n \subset S$ be the scheme-theoretic intersection $\Sigma_n \cap S$ in $Y$. Let $\calM(T_n; r)$ be the moduli stack of coherent $\scrO_{T_n}$-modules of length $r$. The goal of this section is to prove that the assignment 
\begin{align}
	\calF \longmapsto i_S^\ast\calF
\end{align}
yields an equivalence of stacks $\calM^{\mathsf{ss}}(Y; r[\Sigma],dr)\xrightarrow{\sim} \calM(T_n; r)$ for any $d\in \Z$. 

First note that, by Lemma~\ref{lem:stability-A}, one has $\calM^{\mathsf{ss}}(Y; r[\Sigma],r) \simeq \calM^{\mathsf{ss}}(Y; r[\Sigma],dr)$. Hence, we can set $d=1$ without loss of generality. Then one has: 

\begin{proposition}\label{prop:stability-F} 
	The assignment $\calF \mapsto f_\ast \calF$ yields an equivalence between the category of semistable sheaves $\calF$ on $Y$, with $\ch_2(\calF)=r[\Sigma]$ and $\chi(\calF)=rd$, and the category of $\scrO_{Z_n}$-modules of length $r$. The inverse functor maps $\calQ$ to 
	$f^\ast\calQ$.
\end{proposition} 

\begin{proof}
	By Corollary~\ref{cor:stability-B}, any semistable sheaf $\calF$ on $Y$ with $\ch_2(\calF)=r[\Sigma]$ and $\chi(\calF)=rd$ admits a length $r$ filtration so that all subquotients are isomorphic to $i_{\Sigma, \ast}\scrO_\Sigma$. Moreover, the scheme-theoretic support of $\calF$ is also contained in $\Sigma_n$ by Lemma~\ref{lem:stability-C}. Using Proposition~\ref{prop:dirimg-A}, this implies that $f_\ast\calF$ is scheme-theoretically supported on $Z_n$ and has length $r$. 
	
	Conversely, any length $r$ sheaf $\calQ$ on $X$ on $Z_n$ admits a length $r$ filtration so that all subquotients are isomorphic to $\scrO_\nu$. Since diagram \eqref{eq:sch-diagram-A} is cartesian, one has 
	\begin{align}
		f^\ast i_{Z_n,\ast}\calQ \simeq i_{\Sigma_n,\ast} \phi_n^\ast \calQ\ .
	\end{align}
	Since $\phi_n$ is flat, $\phi_n^\ast\calQ_n$ has a length $r$ filtration so that each subquotient is isomorphic to $\phi_n^\ast\scrO_\nu\simeq \scrO_\Sigma$. Therefore $f^\ast\calQ_n$ is a semistable sheaf on $Y$ satisfying the conditions of Proposition~\ref{prop:stability-F}. 	
\end{proof}

Now recall that the restriction $f\vert_S\colon S \to W$ maps $T_n \subset S$ isomorphically onto $Z_n \subset W$. Then Proposition~\ref{prop:stability-F} yields:
\begin{corollary}\label{cor:stability-G} 
	The assignment $\calF\mapsto i_S^\ast\calF$ yields and equivalence between the category of semistable sheaves $\calF$ on $Y$, with $\ch_2(\calF)=r[\Sigma]$ and $\chi(\calF)=rd$, and the category of $\scrO_{T_n}$-modules of length $r$. The inverse functor maps $\calQ$ to $f^\ast i_{Z_n,\,\ast} t_{n,\, \ast} \calQ$.
	
	In particular, we obtain an equivalence of stacks $\tau\colon \calM^{\mathsf{ss}}(Y; r[\Sigma],rd)\xrightarrow{\sim} \calM(T_n; r)$ for any $d\in \Z$.
\end{corollary}

\begin{remark}
	Let $f\colon Y\to X$ be a contraction with $X$ affine. In \cite[\S4.3]{Davison_Refined_Jacobi}, Davison showed that the moduli stack of semistable sheaves on $Y$ with zero Euler characteristic is a quotient stack of the form $Z/G$ with $Z$ an open subscheme of a Quot scheme and $G$ a general linear group. Moreover, the underlying reduced subscheme of the corresponding coarse moduli space consists of a single point. A similar result, at the level of coarse moduli spaces, was proved by Katz in \cite{Katz-Genus-zero} (see also \cite[Theorem~4.6]{Davison_Refined_Jacobi}) for semistable sheaves on $Y$ of Euler characteristic one. In contrast, Corollary~\ref{cor:stability-G} holds for semistable sheaves of arbitrary Euler characteristic.
\end{remark}

\section{$f$-stable pairs}\label{sec:stable-pairs}

In this section, we introduce $f$-stable pairs in our framework using the definition of \cite{BS-Curve-counting-crepant}, and construct the corresponding moduli stack. Furthermore, we show that the moduli stack $\fSP(Y; \beta, n)$ of $f$-stable pairs $(\calF, s)$ with $\ch_2(\calF)=\beta$ and $\chi(\calF)=n$ is a finite type algebraic space over $\C$. In order to provide some context, note that $f$-stable pairs are defined in \cite{BS-Curve-counting-crepant} for more general crepant resolutions of three-fold singularities. Here we specialize this definition to contractions $f\colon Y\to X$ of $(0,-2)$ curves as in \S\ref{sec:geometric-setup}, and, using the results of \S\ref{sec:semistable-exceptional-locus}, we prove explicit construction results for moduli stacks of such objects. 

We shall assume that only Assumption~\ref{assumption:f} holds in this section.

\subsection{Torsion pairs}\label{sec:torsion-pairs}

Let $\catCoh_{\leqslant i}(Y)$ and $\catCoh_0(Y)$ be the subcategories of the abelian category $\catCoh(Y)$ of coherent sheaves on $Y$ formed by those sheaves of dimension $\leq i$, for $i=1, 2$, and zero, respectively.
We define similarly $\catCoh_0(X)$.
Following \cite[\S2]{BS-Curve-counting-crepant}, we introduce the following torsion pair. Let
\begin{align}
	\scrT_f\coloneqq \{\calF\in \catCoh(Y)\,\vert\,& \R^\bullet f_\ast \calF \in \catCoh_0(X) \subset \catDb(X) \}\\[4pt] 
	\scrF_f \coloneqq \{\calF\in \catCoh(Y)\, \vert\,& \Hom(\calT, \calF)=0 \text{ for any }\calT\in \scrT_f\} \ .
\end{align}

\begin{lemma}\label{lem:T_f_1_dim}
	One has $\scrT_f \subseteq \catCoh_{\leqslant 1}(Y)$.
\end{lemma}

\begin{proof}
	Let $\calF \in \scrT_f$. Since $f$ is birational, we see that $\calF$ is $0$-dimensional away from the exceptional locus of $f$. Since the latter is $1$-dimensional, we immediately conclude that $\calF \in \catCoh_{\leqslant 1}(Y)$.
\end{proof}

\begin{lemma}[{\cite[Lemma~13]{BS-Curve-counting-crepant}}]
	The pair $(\scrT_f, \scrF_f)$ is a torsion pair in $\catCoh(Y)$.
\end{lemma}

\begin{remark}
	Note that $\scrT_f\cap \catCoh_{\leqslant 1}(Y)$ and $\scrF_f\cap \catCoh_{\leqslant 1}(Y)$ correspond to the categories $\calP_f, \calQ_f$ introduced in \cite{BS-Curve-counting-crepant}. Moreover, in \textit{loc.cit.}, the authors proved that these categories form a torsion pair in $\catCoh_{\leqslant 1}(Y)$. We are able to remove the extra condition of the dimension of the support because in our case $X^{\mathsf{sing}}$ is zero-dimensional and the exceptional locus of $f$ is one dimensional, and we can apply verbatim the arguments in the proof of \cite[Lemma~13]{BS-Curve-counting-crepant}.
\end{remark}
We denote by $\tau_\scrA$ the $t$-structure on $\catPerf(Y)$ obtained by tilting the standard $t$-structure with respect to the torsion pair $(\scrT_f, \scrF_f)$ and we denote by $\scrA$ the heart of $\tau_\scrA$. $\scrA$ is the abelian category whose objects are complexes $E\in \catPerf(Y)$ such that $\calH^{-1}(E)\in \scrF_f$, $\calH^0(E)\in \scrT_f$, and $\calH^i(E)\simeq 0$ for $i\neq -1, 0$.

\subsection{Characterization of $f$-torsion and $f$-torsion-free objects}\label{sec:f-objects}

Some specific results on the structure of $\scrT_f$ and $\scrF_f$ in the  present context are recorded in the following.
\begin{lemma}\label{lem:support-sequence}
	Any one-dimensional pure coherent sheaf $\calF \in \catCoh_{\leqslant 1}(Y)$ fits into a unique exact sequence 
	\begin{align}\label{eq:pure-sequence-A} 
		0\longrightarrow \calF_\Sigma \longrightarrow \calF  \longrightarrow \calF_{Y/X} \longrightarrow 0 \ ,
	\end{align}
	with $\calF_\Sigma$ and $\calF_{Y/X}$ pure one-dimensional coherent sheaves so that
	\begin{enumerate}\itemsep0.2cm 
		\item $(\mathsf{Supp}(\calF_{Y/X})\cap \Sigma)_{\mathsf{red}}$ is zero-dimensional, and 
		\item $\calF_\Sigma$ is set-theoretically supported on $\Sigma$.
	\end{enumerate} 
\end{lemma} 
\begin{proof}
	Let $Z_\calF$ be the scheme-theoretic support of $\calF$ and $Z\subset Z_\calF$ be the (unique) maximal closed subscheme so that $\Sigma \nsubseteq Z$. Let $\calF'\coloneqq i_{Z,\, \ast} i_Z^\ast \calF$ and let $\calF_{Y/X}$ be the quotient of $\calF'$ by its maximal zero-dimensional subsheaf. Then, one obtains an exact sequence of the form \eqref{eq:pure-sequence-A}.
	
	In order to prove uniqueness, suppose 
	\begin{align}
		0\longrightarrow \calE\longrightarrow\calF \longrightarrow \calG \longrightarrow 0 
	\end{align}
	is a second exact sequence satisfying properties $(1)$ and $(2)$ in Lemma \ref{lem:support-sequence}. Since $\calE$ is a purely one dimensional, with set theoretic support on $\Sigma$, one has $\Hom_Y(\calE,\calF_{Y/X})=0$. Hence the injective morphism $\calE\to\calF$ factors through the injective morphism $\calF_\Sigma \to\calF$. This yields a commutative diagram 
	\begin{align}
			\begin{tikzcd}[ampersand replacement=\&]
			0\ar[r] \& \calE \ar[r] \ar{d}{f} \& \calF \ar[r] \ar{d}\& \calG \ar[r] \ar{d}{g}\& 0 \\
			0\ar[r] \& \calF_\Sigma\ar[r] \& \calF \ar[r] \& \calF_{Y/X} \ar[r] \& 0
		\end{tikzcd}\ ,
	\end{align}
	with $f$ injective and $g$ surjective. Then the snake lemma yields an isomorphism $\ker(g) \to \mathsf{Coker}(f)$. However, $\mathsf{Coker}(f)$ is set theoretically supported on $\Sigma$ while, by assumption, $\calG$ is pure one-dimensional, and the intersection of its set-theoretical support with $\Sigma$ is zero dimensional. This implies that $\ker(g)$ and $\mathsf{Coker}{f}$ are identically zero, hence $f$ and $g$ are isomorphisms. 
\end{proof}

Corollary~\ref{cor:split-support} and Lemma~\ref{lem:support-sequence} yields the following. 
\begin{corollary}\label{cor:split-chern-class} 
	Let $\calF$ be a one-dimensional pure coherent sheaf on $Y$ with $\ch_2(\calF)=r[\Sigma] + \beta$, for some $r\in \N$ and $\beta \in N_1(Y/X)$. Then, $\ch_2(\calF_\Sigma) = r[\Sigma]$  and $\ch_2(\calF_{Y/X})=\beta$. 
\end{corollary} 

\begin{lemma}\label{lem:pushfwd}
	\hfill
	\begin{enumerate}\itemsep0.2cm
		\item \label{item:pushfwd-1} Let $\calF$ be a nonzero one-dimensional pure coherent sheaf on $Y$ so that $\ch_2(\calF) \in N_1(Y/X)$. Then $\R^k f_\ast\calF=0$, for all $k\geq 1$, and $\R^0 f_\ast\calF\vert_{X^{\mathsf{reg}}}$ is a nonzero one-dimensional pure coherent sheaf, where $X^{\mathsf{reg}} \coloneqq X \smallsetminus \{\nu\}$.
		
		\item \label{item:pushfwd-2} Let $\calF$ be a one-dimensional pure coherent sheaf so that $\ch_2(\calF) \in \Z_{>0}[\Sigma]$. Then $\R^kf_\ast\calF=0$, for all $k\geq 2$, and $\R^kf_\ast\calF$ is a zero-dimensional sheaf with support on the singular locus $\{\nu\}$ for $0\leq k\leq 1$. Moreover, $\R^0 f_\ast\calF=0$ if and only if $\mu_{Y\textrm{-}\mathsf{max}}(\calF) \leq 0$ and $\R^1 f_\ast\calF=0$ if and only if $\mu_{Y\textrm{-}\mathsf{min}}(\calF) \geq 0$.
	\end{enumerate} 
\end{lemma}

\begin{proof}
	We start by proving \eqref{item:pushfwd-1}. Let $Z_\calF$ be the scheme-theoretic support of $\calF$. By Corollary~\ref{cor:split-support}, under the current assumptions, $Z_\calF$ is a purely one-dimensional closed subscheme of $Y$ so that $(Z_\calF \cap \Sigma)_{\mathsf{red}}$ is zero-dimensional. Let $\imath\colon Z_\calF \to Y$ denote the canonical closed embedding. Then $\calF = \imath_\ast \calG$ for a one-dimensional pure coherent sheaf $\calG$ on $Z_\calF$. Moreover, $\R^j\imath_\ast\calG=0$ for all $j \geq 1$. This implies that 
	\begin{align}
		\R^k f_\ast\calF = \R^k(f\circ \imath)_\ast \calG 
	\end{align}
	for all $k \geq 0$. Under the current assumptions, all fibers of $f\circ \imath\colon Z_\calF \to X$ are zero-dimensional. Therefore \cite[Theorem~19.8.5 (relative dimensional cohomology vanishing)]{Vakil_AG} shows that $\R^k(f\circ \imath)_\ast\calG =0$ for all $k\geq 1$. The second statement follows by flat base change since $f$ maps $\pi^{-1}(X^{\mathsf{reg}}) =Y\smallsetminus \Sigma$ isomorphically onto $X^{\mathsf{reg}}$.

	Now we prove \eqref{item:pushfwd-2}. The first part follows again from \cite[Theorem~19.8.5 (relative dimensional cohomology vanishing)]{Vakil_AG} by analogy to \eqref{item:pushfwd-1}. In order to prove the second part, let $U\subset X$ be an affine open subscheme containing $\nu$ and let $f_U\colon Y_U \to U$ denote the restriction of $f$ to $Y_U\coloneqq f^{-1}(U)$. By Corollary~\ref{cor:split-support}, any one-dimensional pure coherent sheaf $\calF$ with $\ch_2(\calF)\in \Z[\Sigma]$ is supported on $\Sigma\subset Y_U$. By flat base change, this implies that all direct images $\R^k f_\ast \calF$ are set-theoretically supported at $\nu$. Hence it suffices to prove the statement \eqref{item:pushfwd-2} for $f_U$. This follows from Corollary~\ref{cor:refined-HN-filtration} using \cite[Proposition~III.8.5]{Har77}.
\end{proof}

Lemmas~\ref{lem:support-sequence} and \ref{lem:pushfwd} further imply the following.
\begin{proposition}\label{prop:torsion-vs-torsion-free} 
	Let $\calF$ be a coherent sheaf on $Y$, with one-dimensional support so that $\ch_2(\calF) \neq 0$, let $\calT$ be its torsion subsheaf, and let 
	\begin{align}\label{eq:torsion-sequence}
		0\longrightarrow \calT \longrightarrow \calF \longrightarrow \calG \longrightarrow 0
	\end{align} 
	be the corresponding short exact sequence, where $\calG$ is a one-dimensional pure coherent sheaf on $Y$.
	Let also 
	\begin{align}\label{eq:pure-sequence}
		0\longrightarrow \calG_\Sigma\longrightarrow \calG\longrightarrow \calG_{Y/X} \longrightarrow 0
	\end{align}
	be the short exact sequence obtained in Lemma~\ref{lem:support-sequence}. Then, we have the following:
	\begin{enumerate}\itemsep0.2cm
		\item \label{item:torsion-vs-torsion-free-1} $\calF\in \scrT_f$ if and only $\calG_{Y/X}=0$ and $\mu_{Y\textrm{-}\mathsf{min}}(\calG_\Sigma) \geq 0$. 
		\item \label{item:torsion-vs-torsion-free-2} $\calF\in \scrF_f$  if and only $\calT=0$ and $\calG_\Sigma$ is either zero or $\mu_{Y\textrm{-}\mathsf{max}}(\calG_\Sigma) < 0$. 
	\end{enumerate}  
\end{proposition}

\begin{proof}
	We start by proving \eqref{item:torsion-vs-torsion-free-1}. Recall that $\calF\in \scrT_f$ if and only if $\R^0 f_\ast\calF$ is zero-dimensional and $\R^k f_\ast\calF=0$ for $k\geq 1$.

	We prove the ``only if'' direction. The exact sequences \eqref{eq:torsion-sequence} and \eqref{eq:pure-sequence} yield the long exact sequences
	\begin{align}\label{eq:long-dirim-A} 
		0&\longrightarrow \R^0 f_\ast\calT \longrightarrow \R^0 f_\ast \calF \longrightarrow \R^0 f_\ast \calG \longrightarrow 0\longrightarrow \R^1 f_\ast\calF\longrightarrow \R^1f_\ast \calG\longrightarrow 0\\[4pt] \label{eq:long-dirim-B} 
		0&\longrightarrow \R^0 f_\ast \calG_\Sigma \longrightarrow \R^0 f_\ast \calG\longrightarrow \R^0 f_\ast \calG_{Y/X}\longrightarrow \R^1 f_\ast \calG_\Sigma \longrightarrow \R^1 f_\ast\calG\longrightarrow \R^1 f_\ast \calG_{Y/X}\longrightarrow 0
	\end{align}
	Since $\R^0 f_\ast \calF$ is zero-dimensional and $\R^k f_\ast\calF=0$, for $k\geq 1$, the first of the above sequences implies that $\R^0 f_\ast \calG$ is also zero-dimensional and  $\R^1 f_\ast \calG=0$. By Lemma~\ref{lem:pushfwd}, all direct images are $\R^k f_\ast \calG_\Sigma$, for $k\geq 1$, are zero-dimensional sheaves supported at the singular point $\nu$. Then the second of the above sequences implies that $\R^0 f_\ast \calG_{Y/X}$ is zero-dimensional as well, hence $\calG_{Y/X}=0$ by Lemma~\ref{lem:pushfwd}--\eqref{item:pushfwd-1}. This further implies that $\R^1 f_\ast \calG_\Sigma=0$ since $\R^1 f_\ast \calG=0$. Then $\mu_{Y\textrm{-}\mathsf{min}}(\calG_\Sigma) \geq 0$ by Lemma~\ref{lem:pushfwd}--\eqref{item:pushfwd-2}. 

	Conversely, by Lemma~\ref{lem:pushfwd}--\eqref{item:pushfwd-2}, one has $\R^1 f_\ast \calG_\Sigma=0$ and $\R^0 f_\ast \calG_\Sigma$ zero-dimensional. Since $\calG_{Y/X}=0$, the exact sequence \eqref{eq:long-dirim-B} implies $\R^1 f_\ast \calG=0$. Then the exact sequence \eqref{eq:long-dirim-A} shows that  $\R^1 f_\ast \calF=0$ and $\R^0 f_\ast\calF$ is zero-dimensional. 

	Now, we prove \eqref{item:torsion-vs-torsion-free-2}. Recall that $\calF$ belongs to $\scrF_f$ if and only $\Hom_Y(\scrT_f,\calF)=0$. 

	First, we deal with the ``only if'' direction. Clearly, $\calT=0$ since $\catCoh_0(Y) \subset \scrT_f$. Hence $\calF=\calG$. Let $\calE\subset \calG_\Sigma$ be the first step in the Harder-Narasimhan filtration of $\calG_\Sigma$, i.e., the maximal destabilizing subsheaf of $\calG_\Sigma$. By \eqref{item:torsion-vs-torsion-free-1}, one has $\calE\in \scrT_f$ if and only if $\mu_Y(\calE) \geq 0$. Hence one must have $\mu_{Y\textrm{-}\mathsf{max}}(\calG_\Sigma)=\mu_Y(\calE) <0$ in order to avoid a contradiction. 

	Conversely, since $\calT=0$, one has $\calF=\calG$. Note that \eqref{item:torsion-vs-torsion-free-1} implies that any coherent sheaf in $\scrT_f$ is set-theoretically supported on the union of $\Sigma$ with a zero-dimensional closed subscheme. On the other hand, $\calG_{Y/X}$ is a one-dimensional pure coherent sheaf and $(\mathsf{Supp}(\calG_{Y/X})\cap \Sigma)_{\mathsf{red}}$ is zero-dimensional. Hence, $\Hom_Y(\scrT_f, \calG_{Y/X})=0$. Thus, $\Hom_Y(\scrT_f, \calG_\Sigma)\simeq \Hom_Y(\scrT_f, \calF)$.
	
	Now, if $\calG_\Sigma=0$, it follows that $\calF=\calG_{Y/X}\in \scrF_f$. Otherwise, assume that $\mu_{Y\textrm{-}\mathsf{max}}(\calG_\Sigma) < 0$. By, \eqref{item:torsion-vs-torsion-free-1} any pure coherent sheaf $\calE\in \scrT_f$ has $\mu_{Y\textrm{-}\mathsf{min}}(\calE) \geq 0$. Hence, $\Hom_Y(\calE, \calG_\Sigma)=0$. Since $\calG_\Sigma$ is pure, it follows that $\Hom_Y(\scrT_f, \calG_\Sigma)=0$. Therefore, the assertion follows.
\end{proof}

Using Corollary~\ref{cor:refined-HN-filtration}, one obtains the following consequences of Proposition~\ref{prop:torsion-vs-torsion-free}.
\begin{corollary}\label{cor:torsion-vs-torsion-free} 
	\hfill
	\begin{itemize}\itemsep0.2cm
		\item Assume that $\calF\in \scrT_f$. Then, $H^i(Y, \calF) =0$ for $i\geq 1$. 
		
		\item Assume that $\calF\in \scrF_f$ with $\ch_2(\calF)\in \Z_{>0}[\Sigma]$. Let $D\subset Y$ be a smooth effective Cartier divisor which intersects $\Sigma$ transversely at a single point. Then $\Hom_Y(\scrO_Y(-kD), \calF)=0$ for $0\leq k \leq 1$. 
	\end{itemize}
\end{corollary} 
 
 We conclude this section with another characterization of $f$-torsion-free sheaves.
 \begin{proposition}\label{prop:torsion-free-criterion}
 	Let $\calE$ be a coherent sheaf on $Y$. Let $\calT\subset \calE$ be the maximal zero-dimensional subsheaf of $\calE$. Then $\calE \in \scrF_f$ if and only if $\calT=0$ and $\Hom_Y(i_{\Sigma,\, \ast}\scrO_\Sigma(-1), \calE) =0$. 
 \end{proposition} 
 
 \begin{proof}
 	The ``only if'' direction is obvious since $\calT$ and $i_{\Sigma,\, \ast}\scrO_\Sigma(-1)$ belong to $\calT_f$. 
 	
 	In order to prove the ``if'' direction, first note that $\calE$ does not admit any zero-dimensional subsheaves. Let $\calF\subset \calE$ be the maximal purely one-dimensional subsheaf of $\calE$ and let $\calF_\Sigma \subset \calF$ be the maximal subsheaf of $\calF$ with set theoretic support on $\Sigma$, as in Lemma~\ref{lem:support-sequence}. Given a pure one-dimensional coherent sheaf $\calG \in \scrT_f$, any nonzero morphism $\calG \to \calE$ factors through the inclusion $\calF_\Sigma \subset \calE$. Because of Proposition~\ref{prop:torsion-vs-torsion-free}--\eqref{item:torsion-vs-torsion-free-1}, it suffices to prove that $\mu_{Y\textrm{-}\mathsf{max}}(\calF_\Sigma)<0$ since this implies $\Hom_Y(\calG, \calF_\Sigma)=0$ for all $\calG\in \scrT_f$. 
 	
 	Suppose $\mu_{Y\textrm{-}\mathsf{max}}(\calF_\Sigma)\geq 0$ and let $\calF'_\Sigma \subset \calF_\Sigma$ be the semistable sheaf of maximal slope. Then by Corollary~\ref{cor:stability-B}, there is a line bundle $\calL$ on $\Sigma$ such that $i_{\Sigma,\, \ast}\calL \subset \calF'_\Sigma$ of degree $\geq -1$. This implies $\Hom_Y(i_{\Sigma,\, \ast}\scrO_\Sigma(-1), \calF'_\Sigma)\neq 0$, which further implies $\Hom_Y(i_{\Sigma,\, \ast}\scrO_\Sigma(-1),\calE)\neq 0$. Hence contradiction. 
 \end{proof}
 
\subsection{Openness of the torsion pair $(\scrT_f, \scrF_f)$}\label{subsec:open-tor-free} 

The next goal is to show that the torsion pair $(\scrT_f,\scrF_f)$ is open in the sense of \cite[Definition~A.2]{AB-Bridgeland-stable}). 

\begin{proposition}\label{prop:open-torsion} 
	The subcategory $\scrT_f$ satisfies the openness condition in flat families. 
\end{proposition} 

\begin{proof}
	Lemma~\ref{lem:exclass} and Proposition~\ref{prop:torsion-vs-torsion-free}--\eqref{item:torsion-vs-torsion-free-1} imply that a coherent sheaf $\calF$ belongs to $\scrT_f$ if and only if $\calF$ is of dimension at most one, with $\ch_2(\calF) \in \Z[\Sigma]$, and $\mu_{Y\textrm{-}\mathsf{min}}(F) \geq 0$. The first condition is open and closed. The second condition is open by the basic properties of Harder-Narasimhan filtrations in flat families. 
\end{proof}

\begin{proposition}\label{prop:open-torsion-free}
	The subcategory $\scrF_f$ satisfies the openness condition in flat families. 	
\end{proposition}

\begin{proof}
	Let $Z$ be a parameter scheme and let $\scrE$ be a $Z$-flat family of coherent sheaves on $Y$. Given a point $z\in Z$, Proposition~\ref{prop:torsion-free-criterion} shows that $\scrE_z\in (\scrF_f)_z$ if and only if the following conditions hold simultaneously
	\begin{enumerate}\itemsep0.2cm 
		\item \label{item:condition-1} $\scrE_z$ has no zero-dimensional subsheaves on $Y_z$, and 
		\item \label{item:condition-2} $\Hom_{Y_z}(i_{\Sigma_z,\, \ast}\scrO_{\Sigma_z}(-1), \scrE_z) =0$. 
	\end{enumerate} 
	The same arguments as in the proof of \cite[Example~A.4-(1)]{AB-Bridgeland-stable} show that the set of points $z\in Z$ satisfying condition \eqref{item:condition-1} is open. As for \eqref{item:condition-2}, it is enough to apply Corollary~\ref{cor:openness_orthogonality} (with $N = 0$, $F \coloneqq p_Y^\ast i_{\Sigma, \, \ast} \scrO_{\Sigma}(-1)$, where $p_Y\colon Z\times Y\to Y$ is the projection, and $G \coloneqq \scrE$) to deduce that it is an open condition. The conclusion follows.
\end{proof}

\begin{corollary}\label{cor:open-torsion-pair}
	The torsion pair $(\scrT_f, \scrF_f)$ is open.
\end{corollary}

\begin{remark}
	In \cite[Proposition~3.4]{Padurariu-stable-pairs}, Pădurariu proved that the torsion pair $(\scrT_f\cap \catCoh_{\leqslant 1}(Y),\scrF_f\cap \catCoh_{\leqslant 1}(Y))$ is open by reinterpreting it as a torsion pair arising from the choice of a stability condition $\mu\colon \catCoh_{\leqslant 1}(Y)\to (-\infty, \infty]\times (-\infty, \infty]$ and an element $s\in (-\infty, \infty]\times (-\infty, \infty]$ (see \S3.1 of \textit{loc.cit.} and \cite[Lemma~51]{BS-Curve-counting-crepant}). In our framework, however, since we did not restrict ourselves to $ \catCoh_{\leqslant 1}(Y)$, we adopted a different approach to prove the openness of $(\scrT_f, \scrF_f)$ as shown in the proofs of Propositions~\ref{prop:open-torsion} and \ref{prop:open-torsion-free}.
\end{remark}

Now, let $\scrA_{\mathsf{tor}}$ be the smallest full abelian subcategory of $\scrA$ closed under extensions and containing $\calF[1]$ for coherent sheaves $\calF \in \catCoh_{\leqslant 2}(Y) \cap \scrF_f$. 
\begin{lemma}\label{lem:characterization_Ator}
	For $E \in \scrA$ the following statements are equivalent:
	\begin{enumerate}\itemsep=0.2cm
		\item \label{item:characterization_Ator-1} $E \in \scrA_{\mathsf{tor}}$;
		
		\item \label{item:characterization_Ator-2} $\calH^{-1}(E) \in \catCoh_{\leqslant 2}(Y) \cap \scrF_f$.
						
		\item \label{item:characterization_Ator-3} $\mathsf{rk}(E) = 0$.
	\end{enumerate}
\end{lemma}

\begin{proof}
	First notice that for $E \in \scrA$ we have $\calH^0(E) \in \catCoh_{\leqslant 1}(Y)$ by Lemma~\ref{lem:T_f_1_dim}.
	From here, the equivalence \eqref{item:characterization_Ator-2} $\Leftrightarrow$ \eqref{item:characterization_Ator-3} follows automatically.
	The implication \eqref{item:characterization_Ator-1} $\Rightarrow$ \eqref{item:characterization_Ator-3} follows from the additivity of the rank.
	Finally, let us prove that \eqref{item:characterization_Ator-3} $\Rightarrow$ \eqref{item:characterization_Ator-1}. Now, all the terms of the fiber sequence
	\begin{align}
		\calH^{-1}(E)[1] \longrightarrow E \longrightarrow \calH^{0}(E) 
	\end{align}
	belong to $\scrA$, and therefore that the above can be interpreted as a short exact sequence in $\scrA$. Moreover, $\calH^{-1}(E)[1]$ and $\calH^{0}(E)$ belong to $\scrA_{\mathsf{tor}}$, hence also $E$ since $\scrA_{\mathsf{tor}}$ is closed under extensions.
\end{proof}

\subsection{Definition of $f$-stable pairs}\label{subsec:f-stable-pairs}

Let us recall the notion of Pandharipande-Thomas stable pairs.
\begin{defin}
	A \textit{Pandharipande-Thomas stable pair} on $Y$ is a pair $(\calF, s)$, where $\calF$ is a pure one-dimensional sheaf on $Y$ and $s\colon\scrO_Y \to \calF$ is a section so that $\mathsf{Coker}(s)$ is zero-dimensional. \hfill$\oslash$
\end{defin} 

Now, mimicking \cite[Definition~14]{BS-Curve-counting-crepant}, we introduce the following.
\begin{definition}
	An \textit{$f$-stable pair} on $Y$ is a pair $(\calF, s)$, where $\calF$ is an object of $\scrF_f\cap \catCoh_{\leqslant 1}(Y)$ and $s\colon\scrO_Y \to \calF$ is a section so that $\mathsf{Coker}(s)$ belongs to $\scrT_f$.
\end{definition} 

Proposition~\ref{prop:torsion-vs-torsion-free} yields the following:
\begin{corollary}\label{cor:f-stable-pair}
	\hfill
	\begin{itemize}\itemsep0.2cm
		\item Let $(\calF,s)$ be an $f$-stable pair on $Y$. Then, $\calF$ is a one-dimensional pure coherent sheaf. 
		\item Let $(\calF, s)$ be a Pandharipande-Thomas stable pair on $Y$ such that $\calF_\Sigma=0$. Then $(\calF, s)$ is a $f$-stable pair.
	\end{itemize}
\end{corollary}

\begin{proposition} \label{prop:stable_pairs_different_formulations}
	For a fiber sequence
	\begin{align}\label{eq:stable_pairs_as_triples}
		\begin{tikzcd}[column sep = 16pt, ampersand replacement = \&]
			\scrO_Y \arrow{r}{s} \& \calF \arrow{r} \& E
		\end{tikzcd}
	\end{align}
	in $\catPerf(Y)$, the following statements are equivalent:
	\begin{enumerate}\itemsep=0.2cm
		\item \label{item:stable_pairs_as_triples-1} $\calF\in \scrF_f\cap \catCoh_{\leqslant 1}(Y)$ and the morphism $s \colon \scrO_Y \to \calF$ is such that $\mathsf{Coker}(s)\in \scrT_f$, i.e., the pair $(\calF, s)$ is a $f$-stable pair;
		
		\item \label{item:stable_pairs_as_triples-2} $E$ belongs to $\scrA$, $\calF[1]$ belongs to $\scrA_{\mathsf{tor}}$, and $\ch_1(\calF)=0$.
	\end{enumerate}
\end{proposition}

\begin{proof}	
	We first prove that \eqref{item:stable_pairs_as_triples-1} $\Rightarrow$ \eqref{item:stable_pairs_as_triples-2}. Since $\calF\in \scrF_f\cap \catCoh_{\leqslant 1}(Y)$, we have $\ch_1(\calF)=0$ and $\calF[1] \in \scrA_{\mathsf{tor}} \subseteq \scrA$ by definition. We need only to show that $E \in \scrA$. Taking the long exact sequence of cohomology sheaves associated to \eqref{eq:stable_pairs_as_triples} we obtain
	\begin{align}\label{eq:stable_pairs_long_exact_sequence}
		0 = \calH^{-1}(\calF) \longrightarrow \calH^{-1}(E) \longrightarrow \calH^{0}(\scrO_Y) \longrightarrow \calH^{0}(\calF) \longrightarrow \calH^0(E) \longrightarrow 0 \ .
	\end{align}
	The central terms are canonically identified with $\scrO_Y$ and $\calF$, respectively. Thus, our assumption guarantees that $\calH^0(E)\simeq \mathsf{Coker}(s)\in \scrT_f$. On the other hand, $\scrO_Y\in \scrF_f$, hence $\calH^{-1}(E)\in \scrF_f$ because of the defining properties of a torsion pair (cf.\ \cite[Lemma~3.2]{DPS-torsion-pairs}). Therefore, $E \in \scrA$. 
	
	We now prove that \eqref{item:stable_pairs_as_triples-2} $\Rightarrow$ \eqref{item:stable_pairs_as_triples-1}. Since $\calF[1] \in \scrA_{\mathsf{tor}}$,  the sheaf $\calH^0(\calF) \simeq \calH^{-1}(\calF[1])$ belongs to $\scrF_f\cap \catCoh_{\leqslant 2}(Y)$, while $\calH^1(\calF) \simeq \calH^0(\calF[1])$. Passing to the long exact sequence of cohomology sheaves associated to \eqref{eq:stable_pairs_as_triples}, we obtain
	\begin{align}
		0 = \calH^1(\scrO_Y) \longrightarrow \calH^1(\calF) \longrightarrow \calH^1(E) = 0 \ . 
	\end{align}
	Thus, $\calF \simeq \calH^0(\calF)$. Since $\ch_1(\calF)=0$, we get that $\calF$ belongs to $\catCoh_{\leqslant 1}(Y)$. Finally, the sequence \eqref{eq:stable_pairs_long_exact_sequence} canonically identifies the cokernel of $\scrO_Y \to \calF$ with $\calH^0(E)$, which is an object of $\scrT_f$ because $E \in \scrA$.
\end{proof}

Now, we introduce the derived stack of $f$-stable pairs, following \cite[\S~II.4.4]{DPS-torsion-pairs}. First, let $\derivedPerf^\dagger(Y;\scrO_Y)$ be the derived stack parametrizing fiber sequences of the form
\begin{align}\label{eq:fiber-sequence}
	\mathsf{E} \coloneqq ( \scrO_Y \longrightarrow \calF \longrightarrow E ) \ .
\end{align}
$\derivedPerf^\dagger(Y;\scrO_Y)$ is a geometric derived stack locally of finite presentation over $\C$. We set
\begin{align}
	\partial_0(\mathsf{E}) \coloneqq E \qquad \text{and} \qquad \partial_1(\mathsf{E}) \coloneqq \calF \ . 
\end{align}

Let $\derivedCoh(Y,\tau_\scrA)$ be the derived moduli stack of $\tau_\scrA$-flat objects of $\catPerf(Y)$. Thanks to Corollary~\ref{cor:open-torsion-pair}, \cite[Theorem~A.8]{AB-Bridgeland-stable}, and \cite[Theorem~I.2.81]{DPS-torsion-pairs} $\derivedCoh(Y,\tau_\scrA)$ is an open substack of $\derivedPerf(Y)$, hence it is a geometric derived stack locally of finite presentation over $\C$. Thanks to Lemma~\ref{lem:characterization_Ator}, we define the derived substack $\derivedCoh_{\mathsf{tor}}(Y,\tau_\scrA)$ of $\derivedCoh(Y,\tau_\scrA)$ parametrizing objects in $\scrA_{\mathsf{tor}}$ as the open substack $\derivedCoh(Y,\tau_\scrA)_{\mathsf{rk}=0}$ consisting of those objects having vanishing rank. Thus, also $\derivedCoh_{\mathsf{tor}}(Y,\tau_\scrA)$ is a geometric derived stack locally of finite presentation over $\C$.

Motivated by Proposition~\ref{prop:stable_pairs_different_formulations}, we introduce the following definition.\footnote{Here and in Definition~\ref{def:C-framed-moduli-stack}, we prefer to define our relevant moduli stacks as derived stacks, following the approach outlined in \cite[\S~II.4.4]{DPS-torsion-pairs}.} 
\begin{definition}
	The \textit{derived moduli stack $\derivedfSP(Y)$ of $f$-stable pairs on $Y$} is the fiber product
	\begin{align}\label{eq:derivedfSP}
		\begin{tikzcd}[ampersand replacement=\&]
			\derivedfSP(Y) \arrow{r} \arrow{d} \& \derivedPerf^\dagger(Y;\scrO_Y) \arrow{d}{\partial_1[1] \times \partial_0} \\
			\derivedCoh_{\mathsf{tor}}(Y,\tau_\scrA)_{\ch_1=0} \times \derivedCoh(Y,\tau_\scrA) \arrow{r} \& \derivedPerf(Y) \times \derivedPerf(Y)
		\end{tikzcd} \ ,
	\end{align}
	where $\derivedCoh_{\mathsf{tor}}(Y,\tau_\scrA)_{\ch_1=0}$ is the open and closed substack of $\derivedCoh_{\mathsf{tor}}(Y,\tau_\scrA)$ consisting of those objects having vanishing first Chern class.
\end{definition}
By construction, $\derivedfSP(Y)$ is a geometric derived stack, locally of finite presentation over $\C$. Let $\fSP(Y)\coloneqq \trunc{(\derivedfSP(Y))}$ be the classical truncation. It is an algebraic stack, locally of finite presentation over $\C$. From now on, we shall focus on $\fSP(Y)$ since the derived structure will not play any role in our main results.

Note that
\begin{align}
	\fSP(Y)=\bigsqcup_{\beta\in N_1(Y),\, n\in \Z}\, \fSP(Y; \beta, n)\ ,
\end{align}
where $\fSP(Y; \beta, n)$ parametrizes those $f$-stable pairs $(\calF, s)$ for which $\ch_2(\calF)=\beta$ and $\chi(\calF)=n$. Now, we can argue as in \cite[Lemma~46]{BS-Curve-counting-crepant} to show that flat families of $f$-stable pairs with fixed Chern classes are bounded. Therefore, we obtain the following.
\begin{lemma}
	$\fSP(Y; \beta, n)$ is a finite type algebraic stack over $\C$.
\end{lemma}
Now, the same arguments as in the proof of \cite[Lemma~23]{BS-Curve-counting-crepant} shows that the automorphism group of a $f$-stable pair is trivial (cf. \cite[Proposition~3.3]{Padurariu-stable-pairs}). Hence, the canonical morphism from the inertia stack of $\fSP(Y; \beta, n)$ to $\fSP(Y; \beta, n)$ is an equivalence. Thus,  \cite[Tag~04SZ]{stacks-project} yields the following.
\begin{proposition}\label{prop:fSP_algebraic_space}
	$\fSP(Y; \beta, n)$ is a finite type algebraic space over $\C$.
\end{proposition} 

\begin{remark}
	In \cite[\S3]{Padurariu-stable-pairs}, Pădurariu constructed the classical moduli stack of $f$-stable pairs as an open substack of Lieblich's \textit{master moduli stack} of gluable perfect complexes on $Y$. He proved that it is a proper algebraic space over $\C$ and defined a perfect obstruction theory on it. In our approach, we first constructed a derived moduli stack of $f$-stable pairs as an open substack of the derived stack of perfect complexes on $Y$, and then recover the corresponding classical moduli stack by taking the truncation. While the derived structure of our moduli stacks does not play a role in proving results about their Euler characteristics, we believe our definitions could be of independent interest. In particular, this approach not only provides a canonical perfect obstruction theory on the classical stack but also establishes the proper framework to explore the \textit{categorification}\footnote{By ``categorification'', we mean the study of (bounded) derived categories of coherent sheaves on these derived moduli stacks. In recent years, it has become evident that categorification of enumerative invariants of moduli stacks and spaces involves derived categories of coherent sheaves on their corresponding \textit{derived} moduli stacks—see, e.g., \cite{PS-categorified, Toda-Quot, Toda-conifold, Toda-DT}.} of the results presented in this paper.
\end{remark}

\section{$C$-framed $f$-stable pairs}\label{sec:framed-stable-pairs} 

In this section we will work in the framework provided by Assumptions~\ref{assumption:f} and \ref{assumption:divisor}. Therefore we are given a contraction $f\colon Y \to X$ of a rigid $(0,-2)$ curve $\Sigma$ and a Weil divisor $W\subset X$ so that its inverse image $S=f^{-1}(W)$ is a smooth connected effective divisor on $Y$ intersecting $\Sigma$ transversely at a single point $p$. Given a reduced irreducible curve $C\subset S$ containing $p$, we introduce the notion of $C$-framed $f$-stable pairs on $Y$ (cf.\ Definition~\ref{def:C-framed-pair}) and study their moduli stacks. In particular we will prove that they are algebraic spaces locally of finite type over $\C$ (cf.\ Proposition~\ref{prop:C-framed-algebraic-space}). The main result of this section is (cf.\ Theorem~\ref{thm:C-framed-Hilbert-A} and Proposition~\ref{prop:C-framed-Hilbert-C}):
\begin{theorem}
	Let $\fSP_C(Y; m,\ell)$ denote the moduli stack of $C$-framed $f$-stable pairs $(\calF, s)$ on $Y$ with $\ch_2(\calF)=[C]+m[\Sigma]$ and $\chi(\calF) =\ell$. Let $\FHilb^k_{T_n}(C;m)$ denote the Flag Hilbert scheme parametrizing flags of ideal sheaves $\calI\subset \calI'\subset \scrO_C$ so that $\calI'/\calI$ is the pushforward of a length $m$ zero-dimensional sheaf on $T_n$, and $\chi(\scrO_C/\calI)=k$. Then, for any $m,k\in \N$, there exists an equivalence of algebraic spaces
	\begin{align}
		\begin{tikzcd}[ampersand replacement=\&]
			\fSP_C(Y; m, k +\chi(\scrO_C)) \arrow{r}{\sim} \& \FHilb^k_{T_n}(C;m)
		\end{tikzcd}\ .
	\end{align}
\end{theorem}
The main technical ingredient in the proof of the above result is Theorem~\ref{thm:Pi}, which provides a structural result for $C$-framed $f$-stable pairs in terms of Pandharipande-Thomas stable pairs on $C$ and certain coherent sheaves set-theoretically supported on $\Sigma$. This result is very specific to the contractions $f\colon Y \to X$ considered in this paper, relying heavily on the structure result for moduli stacks of semistable sheaves proven in Proposition~\ref{prop:stability-F}. We also note that no such results are currently known known for other types of curve contractions in threefolds with the exception of the case of $(-1,-1)$ curves studied in \cite{DHS-HOMFLY}. The main application of the above theorem, is to bridge the gap between $C$-framed stable pairs on $Y$ and the Hilbert scheme of $C$, generalizing the results obtained in \textit{loc. cit.}

\subsection{$C$-framed $f$-stable pairs}\label{subsec:C-framing}

In this section, we shall assume that $f\colon Y\to X$ is a threefold contraction satisfying only Assumption~\ref{assumption:f}. Assumption~\ref{assumption:divisor} will be imposed in \S\ref{subsec:C-framed-moduli-stack} to establish a relation between the moduli stack of $C$-framed $f$-stable pairs and Flag Hilbert schemes.

Before giving the definition of $C$-framed $f$-stable pairs on $Y$ and proving some structural results for these objects, we introduce some preliminary results on extensions of coherent sheaves on $Y$ in the next section.

\subsubsection{Some results on extensions of sheaves}\label{subsect:extension-results} 

In this section, let $S\subset Y$ be a smooth irreducible effective divisor $S\subset Y$ which intersects $\Sigma$ transversely exactly at one point. Let $\zeta \in H^0(Y, \scrO_Y(S))$ be a defining section. Let $\calF_0$ and $\calF_1$ be purely one-dimensional coherent sheaves on $Y$, so that $\calF_1$ is scheme-theoretically supported on $S$ and $(\mathsf{Supp}(\calF_0)\cap S)_{\mathsf{red}}$ is zero-dimensional. First note the following vanishing result.
\begin{lemma}\label{lem:tor} 
	Under the present assumptions, we get
	\begin{align}
		\Tor_k(i_{S,\, \ast}\scrO_S\otimes\scrO_Y(S), \calF_0)=0
	\end{align}
	for all $k\geq 1$. Moreover
	\begin{align}
		\chi(\calF_0\otimes i_{S,\, \ast}\scrO_S\otimes\scrO_Y(S)) = S\cdot \ch_2(\calF_0)\ .
	\end{align}
\end{lemma}

\begin{proof}
	The first part follows immediately from the exact sequence 
	\begin{align}
		0\longrightarrow \scrO_Y \longrightarrow \scrO_Y(S) \longrightarrow i_{S,\, \ast}\scrO_S\otimes\scrO_Y(S) \longrightarrow 0\ ,
	\end{align}
	Since $\calF_0$ is a one-dimensional pure coherent sheaf and $(\mathsf{Supp}(\calF_0)\cap S)_{\mathsf{red}}$ is zero-dimensional. The second part follows from Grothendieck-Riemann-Roch theorem. 
\end{proof}

Now assume that
\begin{align}
	0\longrightarrow\calF_0 \longrightarrow \calF \longrightarrow \calF_1 \longrightarrow 0
\end{align}
is an extension of coherent sheaves on $Y$. Then multiplication by $\zeta$ yields the following commutative diagram 
\begin{align}\label{eq:extensions-diagram-A}
	\begin{tikzcd}[ampersand replacement=\&]
		0\ar[r] \& \calF_0 \ar[r] \ar{d}{\zeta_{\calF_0}} \& \calF \ar[r] \ar{d}{\zeta_\calF}\& \calF_1 \ar[r] \ar{d}{0}\& 0 \\
		0\ar[r] \& \calF_0\otimes \scrO_Y(S) \ar[r] \& \calF\otimes \scrO_Y(S) \ar[r] \& \calF_1 \otimes\scrO_Y(S) \ar[r] \& 0
	\end{tikzcd}\ .
\end{align}
Lemma~\ref{lem:tor} shows that $\zeta_{\calF_0}$ is injective, hence $\mathsf{Coker}(\zeta_{\calF_0}) =  \calF_0\otimes \scrO_S(S)$, where $\scrO_S(S)\coloneqq i_{S,\, \ast}\scrO_S\otimes \scrO_Y(S)$. Therefore the snake lemma yields a morphism 
\begin{align}\label{eq:delta-F}
	\delta_\calF\colon \calF_1 \longrightarrow \calF_0\otimes \scrO_S(S)\ . 
\end{align}
Conversely, given a morphism $\delta\colon \calF_1 \to \calF_0\otimes \scrO_S(S)$, let $\calE_\delta$ be defined by the exact sequence 
\begin{align}
	\begin{tikzcd}[ampersand replacement=\&]
		0\ar[r] \& \calE_\delta \ar[r] \& \calF_1 \oplus \calF_0\otimes \scrO_Y(S) \ar{r}{\genfrac{(}{)}{0pt}{}{\delta}{g}}  \& \calF_0\otimes \scrO_S(S)\ar[r]\& 0
	\end{tikzcd}\ ,
\end{align}
where $g\colon \calF_0\otimes \scrO_Y(S)\to \calF_0\otimes \scrO_S(S)$ is the canonical epimorphism. The next result follows from the functorial properties of the snake lemma. The detailed proof will be omitted.
\begin{lemma}\label{lem:extensions-A} 
	Under the above assumptions, the assignments $\calF \mapsto \delta_\calF$ and $\delta\mapsto \calE_\delta$ determine an isomorphism
	\begin{align}\label{eq:extensions-isomorphism}
		\varepsilon\colon \Ext^1_Y(\calF_1, \calF_0)\longrightarrow \Hom_S(\calF_1,\calF_0\otimes \scrO_S(S))\ .
	\end{align}
	Furthermore, one also has:
	\begin{enumerate}\itemsep0.2cm
		\item \label{item:extensions-A-1}  Given a subsheaf $\calF_1'\subset \calF_1$, the pull-back of an extension class 
		$e\in  \Ext^1_Y(\calF_1, \calF_0)$ to $\Ext^1_Y(\calF_1',\calF_0)$ is trivial if and only if $\calF_1'\subset \ker(\varepsilon(e))$. 
		
		\item  \label{item:extensions-A-2} Given a purely one-dimensional quotient $\calF_0 \to \calF_0'$, the pushforward of an extension class $e\in \Ext^1_Y(\calF_1, \calF_0)$ to $\Ext^1_Y(\calF_1,\calF_0')$ is trivial if and only if 
		the composition 
		\begin{align}
			\begin{tikzcd}[ampersand replacement=\&]
				\calF_1 \arrow{r}{\varepsilon(e)} \& \calF_0\otimes \scrO_S(S) \arrow{r} \&  \calF_0'\otimes \scrO_S(S)
			\end{tikzcd}
		\end{align}
		is identically zero. 
	\end{enumerate}
\end{lemma}
Set $\calF_1'\coloneqq \ker(\delta_\calF)$ and $\calQ\coloneqq \mathsf{Im}(\delta_\calF) \subset \calF_0\otimes \scrO_S(S)$, where $\delta_\calF$ is defined in Formula~\eqref{eq:delta-F}.
\begin{lemma}\label{lem:extensions-B} 
	Under the above assumptions, suppose $\calF_0$ is set-theoretically supported on $\Sigma$. Then, one has a canonical commutative diagram 
	\begin{align}\label{eq:extensions-diagram-B}
		\begin{tikzcd}[ampersand replacement=\&]
			\&\& 0\ar[d] \& 0\ar[d] \&\\
			\&\& \calF'_1 \ar[d] \ar{r}{\id} \&  \calF_1' \ar[d] \& \\
			0\ar[r] \& \calF_0\ar[r] \ar{d}{\id} \& \calF\ar[r] \ar[d]\ar[d] \& \calF_1\ar[r] \ar[d] \& 0 \\
			0\ar[r] \& \calF_0\ar[r] \& \calF_0'\ar[r]  \ar[d]\& \calQ \ar[r] \ar[d]\& 0 \\
			\& \& 0 \& 0\&
		\end{tikzcd}
	\end{align}
	with exact rows and columns, where $\calF_0'$ is a one dimensional pure coherent sheaf on $Y$, set-theoretically supported on $\Sigma$. 
\end{lemma} 

\begin{proof}
	Using the snake lemma in diagram \eqref{eq:extensions-diagram-A}, one obtains an isomorphism $\ker(\zeta_\calF) \xrightarrow{\sim} \ker(\delta_\calE)$ since $\zeta_{\calF_0}$ is injective. This yields the following commutative diagram 
	\begin{align}
		\begin{tikzcd}[ampersand replacement=\&]
			\& \& \calF'_1 \ar[d] \ar{r}{\id} \&  \calF_1' \ar[d] \& \\
			0\ar[r] \& \calF_0\ar[r] \& \calF\ar[r]  \& \calF_1\ar[r]  \& 0
		\end{tikzcd}\ ,
	\end{align}
	where the vertical arrows are injective. Applying the snake lemma to the above diagram, one further obtains a diagram as in Equation \eqref{eq:extensions-diagram-B}, where $\calF_0'\coloneqq \calF/\calF_1'$. 
	
	Clearly, $\calF_0'$ is set-theoretically supported on $\Sigma$ since $\calF_0$ and $\calQ$ are supported on $\Sigma$. Furthermore, since $\calQ$ is set-theoretically supported on the intersection $S\cap \Sigma$, it follows that the restriction of $\calF_0'$ to $Y\smallsetminus S$ is purely one dimensional. Hence the maximal zero-dimensional subsheaf $\calT\subset \calF_0'$ is set-theoretically supported on $S\cap \Sigma$. 
	
	Assume that $\calT\neq 0$ and set $\calF_0''\coloneqq \calF_0'/\calT$. Then one obtains a second commutative diagram 
	\begin{align}
		\begin{tikzcd}[ampersand replacement=\&]
			0\ar[r] \& \calF_1'\ar[r] \ar[d]\& \calF\ar[r]  \ar{d}{\id} \ar[d]\& \calF_0'\ar[r] \ar[d] \& 0 \\
			0\ar[r] \& \calF_1''\ar[r] \& \calF\ar[r] \& \calF_0'' \ar[r] \& 0 
		\end{tikzcd}\ ,
	\end{align}
	with exact rows, where the right vertical arrow is the canonical surjection. Then the snake lemma shows that the left vertical arrow is injective, and it also yields an isomorphism $\calT \xrightarrow{\sim} \calF_1''/\calF_1'$. Since $\mathsf{Supp}(\calT) \subset S\cap \Sigma$ and $\calF_1'$ is scheme-theoretically supported on $S$, it follows that $\calF_1''$ is set-theoretically supported on $S$. Since $\calF_0$ is purely one-dimensional, set-theoretically supported on $\Sigma$, we have that $\Hom_Y(\calF_1'', \calF_0) =0$. Therefore the composition 
	\begin{align}
		\calF_1''\longrightarrow \calF\longrightarrow \calF_1
	\end{align}
	is injective, leading to the commutative diagram 
	\begin{align}
		\begin{tikzcd}[ampersand replacement=\&]
			\& \& \calF''_1 \ar[d] \ar{r}{\id} \& \calF_1''\ar[d] \& \\
			0\ar[r] \& \calF_0\ar[r] \& \calF\ar[r] \& \calF_1\ar[r]  \& 0
		\end{tikzcd}\ .
	\end{align}
	Then Lemma~\ref{lem:extensions-A} implies that $\delta_\calF(\calF_1'') =0$, while by construction 
	\begin{align}
		\delta_\calF(\calF_1'')\simeq \calF_1''/\calF'_1\simeq \calT\ ,
	\end{align}
	since $\calF_1'=\ker(\delta_\calF)$. This leads to a contradiction. 
\end{proof}

\subsubsection{Definition of $C$-framed $f$-stable pairs}

Let $C\subset Y$ be a reduced, irreducible, purely one-dimensional, closed subscheme so that $(C\cap \Sigma)_{\mathsf{red}}$ is zero-dimensional.

Let $\calF$ be a one-dimensional pure coherent sheaf on $Y$ and let 
\begin{align}
	0\longrightarrow \calF_\Sigma \longrightarrow \calF\longrightarrow\calF_{Y/X}\longrightarrow 0
\end{align}
be the canonical support sequence constructed in Lemma~\ref{lem:support-sequence}. Assume that $\ch_2(\calF)=m[\Sigma] + [C]$. Then, by Corollary~\ref{cor:split-chern-class}, one has $\ch_2(\calF_\Sigma)=m [\Sigma]$ and $\ch_2(\calF_{Y/X}) =[C]$. 
\begin{definition}\label{def:C-framed-pair} 
	We say that $\calF$ is \textit{$C$-framed} if $\calF_{Y/X}$ is scheme-theoretically supported on $C$. 
	
	Moreover, a pair $(\calF,s\colon \scrO_Y\to \calF)$, with $s$ an arbitrary section, is called \textit{$C$-framed} if $\calF$ is $C$-framed. 
\end{definition} 

\subsubsection{Structural result for $C$-framed $f$-stable pairs}\label{subsec:structural-results} 

\begin{theorem}\label{thm:Pi}
	Let $(\calF, s\colon \scrO_Y \to \calF)$ be a $C$-framed $f$-stable pair with $\ch_2(\calF)= [C] + m [\Sigma]$, with $m\geq 0$. Then, there exists a unique commutative diagram
	\begin{align}\label{eq:Pi-diagram-A}
		\begin{tikzcd}[ampersand replacement=\&]
			\& \& \scrO_Y \ar{d}{s} \ar{r}{\id} \&  \scrO_Y \ar{d}{s_1} \& \\
			0\ar[r] \& \calF_0\ar[r]\& \calF\ar[r]  \& \calF_1\ar[r] \& 0
		\end{tikzcd}
	\end{align}
	such that $(\calF_1, s_1\colon \scrO_Y\to \calF_1)$ is a Pandharipande-Thomas stable pair with $\calF_1$ scheme-theoretically supported on $C$, $\calF_0$ is $S$-equivalent to  $\C^m\otimes i_{\Sigma,\,\ast}\scrO_\Sigma(-2)$, the associated morphism $\delta_\calF\colon \calF_1 \to \calF_0\otimes \scrO_S(S)$, via Lemma~\ref{lem:extensions-A}, is surjective, and 
	\begin{align}
		\delta_\calF \circ s_1 =0\ .
	\end{align}
	
	Conversely, let $(\calF_1, s_1\colon \scrO_Y\to \calF_1)$ be a Pandharipande-Thomas stable pair with $\calF_1$ scheme-theoretically supported on $C$ and suppose $\calF_0$ is 
	$S$-equivalent to $\C^m\otimes i_{\Sigma,\,\ast}\scrO_\Sigma(-2)$, where $m\in \Z$, with $m\geq 0$. Let
	\begin{align}\label{eq:Pi-stable-F}
		\begin{tikzcd}[ampersand replacement=\&]
			0 \arrow{r} \& \calF_0 \arrow{r}{f_0} \& \calF  \arrow{r}{f_1} \& \calF_1 \arrow{r} \& 0
		\end{tikzcd}
	\end{align}
	a short exact sequence such that the associated morphism $\delta_\calF\colon \calF_1 \to \calF_0\otimes \scrO_S(S)$, via Lemma~\ref{lem:extensions-A}, is surjective, and 
	\begin{align}\label{eq:zero-composition} 
		\delta_\calF \circ s_1 =0\ .
	\end{align}
	Then there exists a unique section $s\colon \scrO_Y \to \calF$ so that $f_1 \circ s= s_1$ and $(\calF,s)$ is a $C$-framed $f$-stable pair. 
\end{theorem}

We start by proving the ``only if'' direction. 
\begin{lemma}\label{lem:Pi-A}
	There exists a unique commutative diagram of the form \eqref{eq:Pi-diagram-A} such that 
	\begin{enumerate} \itemsep0.2cm
		\item $(\calF_1, s_1)$ is a Pandharipande-Thomas stable pair on $Y$; and 
		\item one has a second commutative diagram 
		\begin{align}\label{eq:Pi-diagram-C}
			\begin{tikzcd}[ampersand replacement=\&]
				\& \& \mathsf{Im}(s) \ar[d] \ar{r}{\sim} \& \mathsf{Im}(s_1)\ar[d] \& \\
				0\ar[r] \& \calF_0\ar[r]\& \calF\ar[r]  \& \calF_1\ar[r] \& 0
			\end{tikzcd}
		\end{align}
		where the top horizontal arrow is an isomorphism. In particular $\delta_\calF\circ s_1=0$. 
	\end{enumerate} 
\end{lemma}

\begin{proof}
	First note that $s\neq 0$ since $\mathsf{Coker}(s)$ must belong to $\scrT_f$. 
	
	Since $\calF$ is $C$-framed and an object in $\scrF_f$, by Proposition~\ref{prop:torsion-vs-torsion-free}--\eqref{item:torsion-vs-torsion-free-2}, $\calF$ is a pure coherent sheaf, which fits in a unique exact sequence 
	\begin{align}
		0\longrightarrow \calF_\Sigma \longrightarrow \calF \longrightarrow \calF_{Y/X} \longrightarrow 0\ ,
	\end{align}
	where $\calF_\Sigma$ is either a one-dimensional pure coherent sheaf on $Y$ set-theoretically supported on $\Sigma$, with $\mu_{Y\textrm{-}\mathsf{max}}(\calF_\Sigma) < 0$ or zero, and $\calF_{Y/X}$ is a one-dimensional pure coherent sheaf on $Y$ scheme-theoretically supported on $C$ (in particular, $(\mathsf{Supp}(\calF_{Y/X})\cap \Sigma)_{\mathsf{red}}$ is zero-dimensional). Set $\calF_0\coloneqq\calF_\Sigma$ and $\calF_1\coloneqq\calF_{Y/X}$.
	
	Since the fundamental cycle of $\calF$ is $C+m\Sigma$, Corollary~\ref{cor:split-support} shows that the fundamental cycle of $\calF_1$ is $C$. Since $C$ is reduced irreducible and $\calF_1$ is purely one dimensional, this implies that $\calF_1$ is the pushforward to $Y$ of a rank one torsion-free sheaf on $C$. This further shows that the composition $s_1 \coloneqq f_1 \circ s$ is either zero or generically surjective. Moreover, Corollary~\ref{cor:torsion-vs-torsion-free} shows that $H^0(Y, \calF_0)=0$. Since $s\neq 0$, this rules out the case $s_1=0$, hence $s_1$ is generically surjective. This means that $(\calF_1,s_1)$ is a Pandharipande-Thomas stable pair on $Y$, which further implies that $\mathsf{Im}(s_1) \simeq \scrO_C$ (cf.\ \cite[Lemma~1.6]{PT-Curve-counting}. 
	
	Assume that $\calF_0\neq 0$. Applying the snake lemma to diagram~\eqref{eq:Pi-diagram-A}, one obtains the exact sequence 
	\begin{align}\label{eq:coboundary}
		\begin{tikzcd}[ampersand replacement=\&]
			0\arrow{r} \& \ker(s) \arrow{r} \& \ker(s_1)\arrow{r}{\delta} \& \calF_0 \arrow{r} \&\cdots
		\end{tikzcd}
	\end{align}
	where $\ker(s_1) \simeq \calI_C$, where $\calI_C$ is the ideal sheaf of $C$ (cf.\ \cite[Formula~(1.7)]{PT-Curve-counting}). Since $C$ is a divisor in the smooth surface $S$, one has an exact sequence 
	\begin{align}
		0\longrightarrow \scrO_Y(-S) \longrightarrow \calI_C\longrightarrow i_{S,\, \ast} \scrO_S(-C)\longrightarrow 0 \ .
	\end{align}
	Since $S$ intersects $\Sigma$ transversely exactly once, Corollary~\ref{cor:torsion-vs-torsion-free} shows that $\Hom_Y(\scrO_Y(-S), \calF_0) =0$. Furthermore, since $\calF_0$ is purely one dimensional, set-theoretically supported on $\Sigma$, one also has $\Hom_Y( \scrO_S(-C),\calF_0)=0$. Then the above exact sequence implies that $\Hom_Y(\calI_C, \calF_0)=0$ as well. In particular, $\delta$ in \eqref{eq:coboundary} vanishes. This implies that the induced map $\mathsf{Im}(s) \to \mathsf{Im}(s_1)$ is an isomorphism, hence $\mathsf{Im}(s) \cap \calF_0 =0$ as subsheaves of $\calF$. This yields the diagram \eqref{eq:Pi-diagram-C}. 
	
	Now, the epimorphism $\calF\to \calF_1$ restricts to an isomorphism $\mathsf{Im}(s) \to \mathsf{Im}(s_1)$. Since $\mathsf{Im}(s)$ is a subsheaf of $\calF$, this implies that the restriction of the extension 
	\begin{align}
		0\longrightarrow \calF_0 \longrightarrow \calF \longrightarrow \calF_1 \longrightarrow 0 
	\end{align}
	to $\mathsf{Im}(s_1) \subset \calF_1$ admits a splitting. Hence the associated extension class is trivial, and Lemma~\ref{lem:extensions-A} implies that $\mathsf{Im}(s_1) \subset \ker(\delta_\calF)$. Thus, the relation $\delta_\calF\circ s_1=0$.
\end{proof}

Let $\calF_1'\coloneqq\ker(\delta_F)$. By Lemma~\ref{lem:extensions-B}, the quotient $\calF_0'\coloneqq \calF/\calF_1'$ is a pure one-dimensional coherent sheaf on $Y$ set-theoretically supported on $\Sigma$. 
\begin{lemma}\label{lem:Pi-C}
	The following hold:
	\begin{enumerate}\itemsep0.2cm
		\item \label{item:Pi-C-1} $\calF_0$ and $\calF_0'$ are slope-semistable, $S$-equivalent to 
		\begin{align}
			\calF_0 \simeq \C^m\otimes i_{\Sigma,\, \ast }\scrO_\Sigma(-2) \quad\text{and}\quad \calF'_0 \simeq \C^m \otimes i_{\Sigma,\, \ast}\scrO_\Sigma(-1)\ .
		\end{align}
		
		\item \label{item:Pi-C-2} The morphism $\delta_\calF\colon \calF_1\to \calF_0\otimes \scrO_S(S)$ is surjective. 
		
		\item \label{item:Pi-C-3} The section $s\colon\scrO_Y\to \calF$ factors through a section $s_1'\colon\scrO_Y \to \calF_1'$.
	\end{enumerate}
\end{lemma}

\begin{proof}
	Let $\calG \coloneqq \mathsf{Coker}(s)$ and let $\calG'\coloneqq \calG/\calT$ be the quotient by the maximal zero-dimensional subsheaf of $\calG$. Since $\calG$ belongs to $\scrT_f$, Proposition~\ref{prop:torsion-vs-torsion-free}--\eqref{item:torsion-vs-torsion-free-1} shows that $\calG'$ is set theoretically supported on $\Sigma$ and $\mu_{Y\textrm{-}\mathsf{min}}(\calG) \geq 0$. Furthermore, diagram~\eqref{eq:Pi-diagram-C} in Lemma~\ref{lem:Pi-A}  yields an exact sequence 
	\begin{align}
		0\longrightarrow\calF_0 \longrightarrow\calG \longrightarrow\calG_1\longrightarrow 0
	\end{align}
	where $\calG_1\coloneqq\mathsf{Coker}(s_1)$. Under the current assumptions, $\calG_1$ is a zero-dimensional sheaf supported at $p\in C$. Since $\calF_0$ is purely one-dimensional, the composition 
	\begin{align}
		\calT \longrightarrow \calG \longrightarrow \calG_1 
	\end{align}
	is injective. 
	
	Let $\calK\coloneqq \ker(\calF\to \calG)$, which is isomorphic to $\mathsf{Im}(s_1)\simeq \scrO_C$ by Lemma~\ref{lem:Pi-A}. Since $\calF_0'$ is purely one-dimensional, supported on $\Sigma$, we get $\Hom_Y(\calK, \calF_0') =0$.	Therefore $\calK$ is a subsheaf of $\calF_1'\subset \calF$. This yields a commutative diagram 
	\begin{align}
		\begin{tikzcd}[ampersand replacement=\&]
			0\ar[r] \& \calK \ar[r]\ar[d]\&\calF\ar[r] \ar{d}{\id}\&\calG \ar[r] \ar[d]\&0 \\
			0\ar[r] \& \calF_1' \ar[r]\&\calF\ar[r] \&\calF_0' \ar[r] \&0
		\end{tikzcd}\ ,
	\end{align}
	where the left vertical arrow is injective and the right vertical arrow 
	is surjective. Since $\calF_0'$ is purely one-dimensional by Lemma~\ref{lem:extensions-B}, the resulting surjective morphism $\calG \to \calF_0'$ factors through a surjective morphism $\calG' \to \calF_0'$. Since $\mu_{Y\textrm{-}\mathsf{min}}(\calG') \geq 0$, this implies
	\begin{align}
		\mu_{Y\textrm{-}\mathsf{min}}(\calF_0')\geq 0\ .
	\end{align}
	By Corollary~\ref{cor:stability-B}, each Harder-Narasimhan subsequent quotient of $\calF'_0$ is $S$-equivalent to 
	\begin{align}
		\C^r\otimes i_{\Sigma, \, \ast}\scrO_\Sigma(a) 
	\end{align}
	for some $r, a\in \Z$ with $r\geq 1$. The above inequality implies that $\chi(\scrO_{\Sigma}(a))\geq 0$, i.e., $a\geq -1$. Therefore one finds $\chi(\calF'_0) \geq 0$. Moreover, equality holds if and only if $\calF_0'$ is semistable, $S$-equivalent to $\C^m \otimes i_{\Sigma,\, \ast}\scrO_\Sigma(-1)$.
	
	On the other hand, since $\calF\in\scrF_f\cap\catCoh_{\leqslant 1}(Y)$, Proposition~\ref{prop:torsion-vs-torsion-free}--\eqref{item:torsion-vs-torsion-free-2} shows that $\mu_{Y\textrm{-}\mathsf{max}}(\calF_0) < 0$. Again, Corollary~\ref{cor:stability-B} shows that each Harder-Narasimhan subsequent quotient of $\calF_0$ is $S$-equivalent to $\C^r\otimes i_{\Sigma,\, \ast} \scrO_\Sigma(a)$ for some $r, a\in \Z$ with $r\geq 1$. The above inequality implies that $\chi(\scrO_{\Sigma}(a))\leq -1$, i.e., $a\leq -2$. Therefore one finds $\chi(\calF_0) \leq -m$. Moreover, equality holds if and only if $\calF_0$ is slope-semistable in the $S$-equivalence class
	of $\C^m \otimes i_{\Sigma,\, \ast}\scrO_{\Sigma}(-2)$. However, the bottom row of diagram~\eqref{eq:extensions-diagram-B} shows that 
	\begin{align}\label{eq:FFineq} 
		0\leq \chi(\calF_0')  = \chi(\calF_0) + \chi(\calQ)  \leq \chi(\calF_0) + \chi(\calF_0\otimes \scrO_S(S))  = \chi(\calF_0) +m\ .
	\end{align}
	Then one obtains the simultaneous inequalities 
	\begin{align}
		\chi(\calF_0) \leq -m \quad\text{and}\quad \chi(\calF_0) \geq -m\ , 
	\end{align}
	hence $\chi(\calF_0) = -m$. Thus, $\calF_0$ is slope-semistable in the $S$-equivalence class of $\C^m \otimes i_{\Sigma, \, \ast}\scrO_{\Sigma}(-2)$. Furthermore, inequality \eqref{eq:FFineq} implies that 
	\begin{align}
		\chi(\calF_0') = 0\quad\text{and} \quad \calQ = \calF_0\otimes \scrO_S(S)\ .
	\end{align}
	This further implies that $\calF_0'$ is semistable, $S$-equivalent to $\C^m \otimes \scrO_\Sigma(-1)$. Then $H^0(Y, \calF_0')=0$, hence $s$ factors through $s_1'\colon\scrO_Y \to \calF_1'$. 
\end{proof}

\begin{proof}[Proof of Theorem~\ref{thm:Pi}]
	The ``only if'' direction follows from Lemmas~\ref{lem:Pi-A} and \ref{lem:Pi-C}.
	
	We prove now the ``if'' direction. Since $\delta_\calF\circ s_1=0$, Lemma~\ref{lem:extensions-A}--\eqref{item:extensions-A-1} shows that the restriction of the extension \eqref{eq:Pi-stable-F} to $\mathsf{Im}(s_1) \subset \calF_1$ is trivial, which means that there exists an injective morphism $\mathsf{Im}(s_1)\to \calF$ which fits in the commutative diagram 
	\begin{align}
		\begin{tikzcd}[ampersand replacement=\&]
			\& \& \mathsf{Im}(s_1) \ar[d] \ar{r}{\id} \&  \mathsf{Im}(s_1) \ar[d] \& \\
			0\ar[r] \& \calF_0\ar[r] \& \calF\ar[r]  \& \calF_1\ar[r]  \& 0 
		\end{tikzcd}\ .
	\end{align}
	Then let $s\colon \scrO_Y\to \calF$ be defined by the composition 
	\begin{align}
		\scrO_Y\longrightarrow  \mathsf{Im}(s_1)\longrightarrow \calF\ .
	\end{align}
	Given any section $s'\colon\scrO_Y\to \calF$ so that $f_1 \circ s'=s_1$, one has $s'-s\in H^0(Y, \calF_0)$, which is zero since $\calF_0\simeq \C^m\otimes i_{\Sigma,\, \ast}\scrO_\Sigma(-2)$. Therefore $s$ is unique.
	
	Now, we need to show that $\calG\coloneqq \mathsf{Coker}(s)$ belongs to $\scrT_f$. We shall make use of Proposition~\ref{prop:torsion-vs-torsion-free}--\eqref{item:torsion-vs-torsion-free-1}. Snake lemma yields the exact sequence 
	\begin{align}
		0\longrightarrow\calF_0 \longrightarrow\calG \longrightarrow\mathsf{Coker}(s_1)\longrightarrow 0 \ .
	\end{align}
	Let $\calT\subset \calG$ be the maximal zero-dimensional subsheaf of $\calG$. Since $\calF_0$ is purely one dimensional, the composition 
	\begin{align}
		\calT \longrightarrow \calG \longrightarrow \mathsf{Coker}(s_1)
	\end{align}
	is injective and one obtains the exact sequence 
	\begin{align}
		0\longrightarrow \calF_0\longrightarrow \calG/\calT\longrightarrow \mathsf{Coker}(s_1)/\calT \longrightarrow 0\ .
	\end{align}
	This implies that $\calG/\calT$ is set-theoretically supported on $\Sigma$ and $\ch_2(\calG/\calT) = m [\Sigma]$.
	
	Let $\calG/\calT\to \calG'\neq 0$ be the minimal slope-semistable quotient provided by the Harder-Narasimhan filtration. Note that $\calG'$ is set-theoretically supported on $\Sigma$ and $\ch_2(\calG') = m' [\Sigma]$, with $m'>0$. Let $\calF_0'\subset \calG'$ be the image of the composition $\calF_0 \to \calG/\calT \to \calG'$. Then, note the commutative diagram 
	\begin{align}
		\begin{tikzcd}[ampersand replacement=\&]
			0\ar[r] \& \calF_0\ar[r] \ar[d] \& \calG/\calT \ar[r]  \ar[d] \& \mathsf{Coker}(s_1)/\calT \ar[r] \ar[d] \& 0 \\
			0\ar[r]\& \calF_0' \ar[r]\& \calG'\ar[r]\& \calG'/\calF'_0\ar[r] \&0
		\end{tikzcd}\ ,
	\end{align}
	where the rows are exact and the vertical maps are surjective. This shows that $\calG'/\calF'_0$ is zero dimensional since $\mathsf{Coker}(s_1)/\calT$ is zero-dimensional. In particular, $\calF_0'\neq 0$. 
	
	In order to finish the proof, one has to show that $\mu_Y(\calG')\geq 0$. By Corollary~\ref{cor:stability-B}, $\calG'$ is $S$-equivalent to $\C^{m'} \otimes i_{\Sigma,\, \ast}\scrO_\Sigma(a)$. Suppose $\mu_Y(\calG') <0$, i.e., $a\leq -2$. Since $\calF_0'\neq 0$, the case $a\leq -3$ is ruled out since $\calG'$ is semistable and receives a nonzero morphism from $\calF_0$ which is $S$-equivalent to $\C^m \otimes 
	i_{\Sigma,\, \ast}\scrO_\Sigma(-2)$. Hence $a=-2$. Since both $\calF_0$ and $\calG'$ are semistable sheaves of equal slope, it follows that $\calG'/\calF_0'$ is either a nonzero semistable sheaf of the same slope or zero. Since $\calG'/\calF_0'$ was already shown to be zero-dimensional, it follows that $\calG'/\calF_0'=0$. This yields a commutative diagram 
	\begin{align}
		\begin{tikzcd}[ampersand replacement=\&]
			0\ar[r] \& \calF_0\ar[r] \ar[d] \& \calF\ar[r]  \ar[d] \& \calF_1\ar[r]  \& 0 \\
			\& \calG' \ar{r}{\id} \& \calG' \& \&
		\end{tikzcd}\ ,
	\end{align}
	where the vertical arrows are surjective. Then Lemma~\ref{lem:extensions-A}--\eqref{item:extensions-A-2} shows that composition 
	\begin{align}\label{eq:composition-map}
		\begin{tikzcd}[ampersand replacement=\&]
			\calF_1\arrow{r}{\delta_\calF} \& \calF_0\otimes \scrO_S(S) \arrow{r} \&  \calG'\otimes \scrO_S(S)
		\end{tikzcd}
	\end{align}
	is identically zero. However, since $\calF_0\to \calG'$ is surjective, and $\delta_\calF$ is also surjective, the composition \eqref{eq:composition-map} has to be surjective. This leads to a contradiction since $\calG'\neq 0$. Therefore $\mu_{Y\textrm{-}\mathsf{min}}(\calG)\geq 0$, hence $G\in \scrT_f$. 
\end{proof}

\subsection{Moduli stacks of $C$-framed $f$-stable pairs}\label{subsec:C-framed-moduli-stack} 

Let $C$ be a reduced irreducible curve in a smooth irreducible effective divisor $S\subset Y$, which intersects $\Sigma$ transversely at a single point $p\in C^{\mathsf{sing}}$. The goal of this section is to construct a geometric derived stack parametrizing $C$-framed $f$-stable pairs on $Y$, as introduced in Definition~\ref{def:C-framed-pair}, under only Assumption~\ref{assumption:f}. This will require a detailed structural analysis for these objects.

Furthermore, we establish a connection to Flag Hilbert schemes, assuming Assumption~\ref{assumption:divisor} as well.

\subsubsection{Moduli stacks}\label{subsubsection:moduli-C-framed} 

Denote by $\derivedextCoh(Y)$ the derived moduli stack parametrizing short exact sequences
\begin{align}\label{eq:ses}
	0\longrightarrow\calE_1\longrightarrow \calE_2 \longrightarrow \calE_3 \longrightarrow 0 
\end{align}
of coherent sheaves on $Y$. We denote by $p_i\colon \derivedextCoh(Y)\to \derivedCoh(Y)$ the map sending the short exact sequence of the form \eqref{eq:ses} to $\calE_i$ for $i=1, 2, 3$.

Fix $m\in \N$. Let $\derivedCoh^{\mathsf{pure}}(C;1)$ be the open substack of $\derivedCoh(C)$ parametrizing pure coherent sheaves on $C$ of rank one and  let $\derivedCoh^{\mathsf{ss}}_{\mu=-1}(Y; m[\Sigma])$ is the moduli stack of semistable sheaves on $Y$ of slope $-1$ and second Chern class $m[\Sigma]$. Consider the fiber product
\begin{align}
	\begin{tikzcd}[ampersand replacement=\&]
		\widehat{\derivedextCoh}(Y;m)\ar{r} \ar{d} \& \derivedextCoh(Y)\ar{d}{p_3\times p_1} \\
		\derivedCoh^{\mathsf{pure}}(C;1) \times \derivedCoh^{\mathsf{ss}}_{\mu=-1}(Y; m[\Sigma]) \ar{r} \& \derivedCoh(Y)\times \derivedCoh(Y)
	\end{tikzcd}\ ,
\end{align}
where the map $\derivedCoh^{\mathsf{pure}}(C;1)\to  \derivedCoh(Y)$ is induced by the pushforward with respect to the canonical closed embedding of $C$ into $Y$. Now, we introduce $\derivedfSP_C(Y; m)$ as the fiber product
\begin{align}
	\begin{tikzcd}[ampersand replacement=\&]
		\derivedfSP_C(Y;m) \ar{r}\ar{d} \& \derivedfSP(Y) \ar{d}\\
		\widehat{\derivedextCoh}(Y;m)\ar{r}{p_2[1]}  \& \derivedCoh_{\mathsf{tor}}(Y,\tau_\scrA)_{\ch_1=0} 
	\end{tikzcd}\ ,
\end{align}
where the right vertical map is induced by the fiber product \eqref{eq:derivedfSP}. Note that the target of $p_2[1]$ is exactly $\derivedCoh_{\mathsf{tor}}(Y,\tau_\scrA)_{\ch_1=0} $ because of Proposition~\ref{prop:torsion-vs-torsion-free}--\eqref{item:torsion-vs-torsion-free-2}.

By construction, $\derivedfSP_C(Y;m)$ is a geometric derived stack locally of finite presentation over $\C$. Moreover, by Theorem~\ref{thm:Pi}, it is the stack parameterizing $C$-framed $f$-stable pairs $(\calF, s)$ on $Y$ such that $\ch_2(\calF)=[C]+m[\Sigma]$.
 It decomposes into
\begin{align}
	\derivedfSP_C(Y;m)\coloneqq \bigsqcup_{\ell\in \Z} \derivedfSP_C(Y;m, \ell)
\end{align}
with respect to the Euler characteristic. We denote the truncation of $\derivedfSP_C(Y; m)$ by $\fSP_C(Y; m)$.

\begin{definition}\label{def:C-framed-moduli-stack}
	Fix $m\in \N$ and $n\in \Z$. The derived moduli stack of \textit{$C$-framed $f$-stable pairs $(\calF, s)$ on $Y$, of second Chern class $\ch_2(\calF)=[C]+m[\Sigma]$ and $\chi(\calF)=\ell$} is $\derivedfSP_C(Y;m, \ell)$. We denote by $\fSP_C(Y;m, \ell)$ its truncation.
	
	The derived moduli stack $\fSP_C(Y)$ of \textit{$C$-framed $f$-stable pairs $(\calF, s)$ on $Y$} is
	\begin{align}
		\derivedfSP_C(Y)\coloneqq \bigsqcup_{m\in \N}\, \derivedfSP_C(Y;m)\ .
	\end{align}
	We denote by $\fSP_C(Y)$ its truncation.
\end{definition}

\begin{proposition}\label{prop:C-framed-algebraic-space}
For any $m, \ell\in \Z$, $m \geq 1$, the stack $\fSP_C(Y; m, \ell)$ is an algebraic space of finite type over $\C$. Moreover $\fSP_C(Y; m)$ and $\fSP_C(Y)$ are algebraic spaces locally of finite type over $\C$.
\end{proposition}

\begin{proof}
	Note that the map $\fSP_C(Y; m, \ell)\to \fSP(Y; m, \ell)$ induced by $p_2[1]$ is represented by algebraic spaces since $p_2$ is represented by Quot schemes. Since $\fSP(Y; m, n)$ is an algebraic space by Proposition~\ref{prop:fSP_algebraic_space}, the claim follows.
\end{proof}

\subsubsection{Relation to Flag Hilbert schemes of $C$}\label{subsect:relation-hilbert-scheme}

In this section, we work under Assumptions~\ref{assumption:f} and~\ref{assumption:divisor}. Our main goal is to show that the algebraic moduli space $\fSP_C(Y)$ is isomorphic to a Flag Hilbert scheme of $C$. 

Let $T_n$ be the scheme-theoretic intersection $S \cap \Sigma_n$. Let $\SP_C(\ell)$ denote the moduli space of Pandharipande-Thomas stable pairs on $C$, i.e., pairs $(\calE,s)$ with $\calE$ a torsion free coherent sheaf on $C$ with $\chi(\calE) =\ell$, and $s\colon \scrO_C\to \calE$ a generically surjective morphism. In particular, since $C$ is reduced, irreducible, this implies that $\mathsf{rk}(\calE) = 1$. Furthermore, $\SP_C(\ell)$ is empty for $\ell< \chi(\scrO_C)$. Let 
\begin{align}
	\SP_C \coloneqq \bigsqcup_{\ell \geqslant \chi(\scrO_C)} \SP_C(\ell)\ . 
\end{align}

Let $(\scrE, \sigma)$ denote the universal stable pair on $\SP_C\times C$ and let $\scrQ\coloneqq \mathsf{Coker}(\sigma)$. Abusing notation we will also denote by $\scrE$ its pushforward to $\SP_C\times Y$.
\begin{lemma}\label{lem:Q-flat} 
	The sheaf $\scrQ$ is flat over $\SP_C$. 
\end{lemma} 

\begin{proof}
	By construction, one has the universal exact sequence 
	\begin{align}
		\begin{tikzcd}[ampersand replacement=\&]
			0 \arrow{r} \& \scrO_{\SP_C\times C} \arrow{r}{\sigma} \& \scrE\arrow{r} \& \scrQ \arrow{r} \& 0
		\end{tikzcd} \ ,
	\end{align}
	where $\sigma$ restricts to an injection on fibers. This implies that $\scrQ$ is $\SP_C$-flat.
\end{proof}

Now recall that $C$ is contained as closed subscheme in the smooth surface $S\subset Y$ and $T_n$ is the scheme theoretic intersection of $\Sigma_n$ and $S$ in $Y$. Let $i_{C,S}\colon \SP_C\times C \to \SP_C\times S$ be the canonical closed immersion. For any $\ell\in \Z$, let $q\colon \Quot_{T_n}(\scrQ/ \SP_C; m,\ell)\to \SP_C(\ell)$ be the relative Quot scheme whose functor of points assigns to any map $\phi\colon Z \to \SP_C(\ell)$ the set of quotients 
\begin{align}
	g\colon (q\times \imath_C)^\ast\scrQ\longrightarrow \scrG\ ,
\end{align}
so that $i_{C,S,\, \ast}\scrG$ is a $Z$-flat family of length $m$ sheaves on $T_n$. Note that $\Quot_{T_n}(\scrQ/ \SP_C; m,\ell)$ is empty for $\ell < m+\chi(\scrO_C)$. Therefore one can write $\ell = k+m+\chi(\scrO_C)$ with $k \in \N$. Set 
\begin{align}
	\Quot_{T_n}(\scrQ/ \SP_C; m)=\bigsqcup_{k\in \N} \Quot_{T_n}(\scrQ/ \SP_C; m,k+m+\chi(\scrO_C))\ .
\end{align}
The next goal is to construct an equivalence $\fSP_C(Y; m) \to \Quot_{T_n}(\scrQ/ \SP_C; m)$. First, we need the following \textit{family variant} of Lemma~\ref{lem:extensions-A} and Theorem~\ref{thm:Pi}, which we shall prove now. 

Let $Z$ be an arbitrary parameter scheme and let 
\begin{align}
	\begin{tikzcd}[ampersand replacement=\&]
		\& Z\times Y \ar{dr}{p} \ar{dl}[swap]{\pi} \& \\ 
		Z \& \& Y 
	\end{tikzcd}
\end{align}
denote the canonical projections. Let also $\zeta\in H^0(Y, \scrO_Y(S))$ be a defining section of $S$. 
\begin{lemma}\label{lem:C-framed}
	Let 
	\begin{align}
		0\longrightarrow \scrF_0\longrightarrow\scrF \longrightarrow\scrF_1 \longrightarrow 0 
	\end{align}
	be an exact sequence of $Z$-flat families of one-dimensional pure coherent sheaves on $Y$, such that $\scrF_1$ is $Z$-flat family of pushforwards of rank one pure coherent sheaves on $C$ and $\scrF_0$ is a flat family of slope $(-1)$ semistable sheaves on $Y$ with second Chern class $m[\Sigma]$. Then, multiplication by $\pi^\ast\zeta$ determines a commutative diagram 
	\begin{align}\label{eq:extensions-diagram-C}
		\begin{tikzcd}[ampersand replacement=\&]
			\& 0 \ar[d] \&\&\& \\
			0\ar[r] \& \scrF_0 \ar[r] \ar{d}{\zeta_{\scrF_0}} \& \scrF \ar[r] \ar{d}{\zeta_\scrF} \& \scrF_1 \ar[r] \ar{d}{0} \& 0 \\
			0\ar[r] \& \scrF_0\otimes p^\ast\scrO_Y(S) \ar[r] \ar[d] \& \scrF\otimes p^\ast\scrO_Y(S) \ar[r] \& \scrF_1 \otimes p^\ast \scrO_Y(S) \ar[r] \& 0 \\
			\& \scrF_0 \otimes p^\ast\scrO_S(S) \ar[d] \& \& \& \\
			\& 0\& \& \& 
		\end{tikzcd}
	\end{align}
	where the rows and the left column are exact. In particular one has a connecting morphism  
	\begin{align}
		\delta_\scrF \colon \scrF_1 \longrightarrow \scrF_0 \otimes p^\ast\scrO_S(S)\ ,
	\end{align}
	where $\scrF_0 \otimes p^\ast\scrO_S(S)$ is $Z$-flat. Moreover, for any $z\in Z$, the restriction of $\delta_\scrF$ to $\{z\}\times Y$ coincides with the analogous morphism constructed in Lemma~\ref{lem:extensions-A}. 
\end{lemma} 

\begin{proof}
	The first part is clear since the pull-back $\pi^\ast\colon \catCoh(Y) \to \catCoh(Z\times Y)$ is  exact and $\pi^\ast \calG$ is $Z$-flat for any coherent sheaf $\calG$ on $Y$. By construction, the restriction $\delta_\scrF\vert_{\{z\}\times Y}$ coincides to connecting morphism $\delta_{\scrF_z}$ constructed in Lemma~\ref{lem:extensions-A} for any point $z\in Z$. The latter is surjective by Theorem~\ref{thm:Pi}, hence $\delta_{\scrF}$ is also surjective. Moreover, $\scrF_0 \otimes q^\ast\scrO_S(S)$ is $Z$-flat since both factors are $Z$-flat. 
\end{proof}

The stack $\fSP_C(Y; m)$ assigns to an arbitrary parameter scheme $Z$ the groupoid of $Z$-flat families of diagrams
\begin{align}\label{eq:flat-diagram-A}
	\begin{tikzcd}[ampersand replacement = \&]
		{} \& {} \& \scrO_{Z\times Y}\ar{d}{s} \& {} \& {}\\
		0 \ar{r} \& \scrF_0 \ar{r} \& \scrF \ar{r}  \& \scrF_1\ar{r}\& 0
	\end{tikzcd}\ ,
\end{align}
satisfying the following conditions:
\begin{enumerate}\itemsep0.2cm
	\item $\scrF_0$ is a flat family of purely one-dimensional slope $(-1)$ semistable sheaves on $Y$ with second Chern class $m[\Sigma]$, 
	\item $\scrF_1$ is a flat family of purely one dimensional sheaves on $Y$ with second Chern class $[C]$, which has scheme theoretic support $Z\times C$ as a sheaf on $Z\times Y$.
	\item For any point $z\in Z$, the restriction of the pair $(\scrF,s\colon \scrO_{Z\times Y}\to \scrF)$ to $z\times Y$ is a $C$-framed $f$-stable pair.
\end{enumerate} 
Using the same notation as in Lemma~\ref{lem:C-framed}, the main result of this section is the following. 

\begin{theorem}\label{thm:C-framed-Hilbert-A}
Suppose Assumptions~\ref{assumption:f} and~\ref{assumption:divisor} hold. 
Then, using the isomorphism $\boldsymbol{\zeta}\colon \scrO_{T_n}\otimes \scrO_S(S) \xrightarrow{\sim} \scrO_{T_n}$ induced by the defining section $\zeta \in H^0(Y, \scrO_Y(S))$, the assignment 
	\begin{align}
		(\scrF_0,\scrF_1,\scrF,s)\longmapsto (\scrF_1, \pi^\ast\xi\circ \delta_\scrF)
	\end{align} 
	determines an equivalence of algebraic spaces
	\begin{align}\label{eq:C-framed-Hilbert-A}
		\begin{tikzcd}[ampersand replacement=\&]
			\fSP_C(Y; m, k+\chi(\scrO_C))\arrow{r}{\sim} \& \Quot_{T_n}(\scrQ/ \SP_C; m,k+m +\chi(\scrO_C))
		\end{tikzcd}\ .
	\end{align}
\end{theorem} 

\begin{proof}
	First, by Lemma~\ref{lem:stability-A}, one has $\calM^{\mathsf{ss}}(Y; m) \simeq \calM^{\mathsf{ss}}(Y; m[\Sigma],-m)$. Let $\calM(T_n; m)$ be the moduli stack of length $m$ sheaves on $T_n$. Then Corollary~\ref{cor:stability-G} proves that one has an equivalence 
	\begin{align}\label{eq:tau-isomorphism} 
		\begin{tikzcd}[ampersand replacement=\&]
			\tau\colon \calM^{\mathsf{ss}}(Y; m[\Sigma],-m)\arrow{r}{\sim} \& \calM(T_n; m)
		\end{tikzcd}
	\end{align}
	which acts as $\calF_0 \mapsto \calF_0 \otimes \pi^\ast\scrO_S$ on flat families. Then the existence of a functor as claimed in the assertion follows from Lemma~\ref{lem:C-framed} and the ``only if'' direction of Theorem~\ref{thm:Pi}. In particular, the latter shows that the restriction $\delta_{\scrF_z}$ is surjective for any point $z\in Z$. The fact that this is an equivalence follows from the ``if'' direction of Theorem~\ref{thm:Pi}. 
\end{proof}

Theorem~\ref{thm:C-framed-Hilbert-A} yields the following.
\begin{corollary}\label{cor:C-framed-Hilbert-B} 
	For any $k,m\in \N$, the motivic classes of $\fSP_C(Y; m, k+\chi(\scrO_C))$ and $\Quot_{T_n}(\scrQ/ \SP_C; m+k+\chi(\scrO_C))$ in the Grothendieck ring of algebraic spaces coincide. 
\end{corollary} 

Furthermore, \cite[Proposition~B.8]{PT-BPS} yields:
\begin{proposition}\label{prop:C-framed-Hilbert-C} 
	Let $\FHilb^k_{T_n}(C;m)$ be the Flag Hilbert scheme parametrizing flags of ideal sheaves $\calI\subset \calI'\subset \scrO_C$ so that $\calI'/\calI$ is the pushforward of a length $m$ zero-dimensional sheaf on $T_n$, and $\chi(\scrO_C/\calI') =k$.  Then, there is an isomorphism 
	\begin{align}
		\begin{tikzcd}[ampersand replacement=\&]
			\Quot_{T_n}(\scrQ/ \SP_C; m,m+k+\chi(\scrO_C))\ar{r}{\sim} \& \FHilb^k_{T_n}(C;m)
		\end{tikzcd}
	\end{align}
	 mapping the data $(s_1\colon\scrO_Y\to \calF_1, g\colon \calQ\to \calG)$ to 
	\begin{align}
		\calI\coloneqq \calH om_C(i_C^\ast\calF_1, \scrO_C)\quad \text{and} \quad \calI'\coloneqq \calH om_C(\calF_1', \scrO_C) \ ,
	\end{align}
	where $\calF_1'$ is the kernel of the composition $i_C^\ast\calF_1 \to \calQ \xrightarrow{g} \calG$.
\end{proposition} 

\section{$C$-framed wallcrossing identity}\label{sec:framed-wallcrossing-identity}

In this section $f\colon Y\to X$ is a threefold contraction satisfying Assumptions~\ref{assumption:f} and \ref{assumption:divisor}. As in \S\ref{subsec:C-framed-moduli-stack}, let $\fSP_C(Y; m,\ell)$ denote the moduli space of $C$-framed $f$-stable pairs $(\calF, s)$ on $Y$ with $\ch_2(\calF)= m[\Sigma]+[C]$ and $\chi(\calF) =\ell$, where $m, \ell\in \Z$, with $m\geq 0$. A construction completely analogous to the one in \S\ref{sec:framed-stable-pairs} yields a moduli space $\SP_C(Y;m,\ell)$ of $C$-framed Pandharipande-Thomas stable pairs $(\calF,s)$ on $Y$ with $\ch_2(\calF)= m[\Sigma]+[C]$ and $\chi(\calF) =\ell$, where $m, \ell\in \Z$, with $m\geq 0$. 

\begin{remark}\label{rem:SPzero}
	Note that $\SP_C(Y;0,\ell)$ coincides with the moduli space of of Pandharipan\-de-Thomas stable pairs $(\calF,s)$ on $Y$ with $\calF$ scheme theoretically supported on $C$.
\end{remark}

We first define generating functions for topological Euler numbers of moduli spaces of $C$-framed ($f$-)stable pairs. Let $\C\{q,Q\}$ denote the vector space of bi-infinite
formal power series 
\begin{align}
	\sum_{m,\ell\in \Z} c_{\ell,m} q^\ell Q^m 
\end{align}
with complex coefficients. Let $\Lambda_q$ denote the ring of Laurent formal power series in $q$, i.e., power series with complex coefficients of the form 
\begin{align}
	\sum_{\ell\geq N} c_\ell q^\ell
\end{align}
for some $N \in \Z$. For any formal variable $T$ we denote by $\Lambda_q[[T]]$ the ring of formal power series of the form 
\begin{align}
	\sum_{m\in\N} f_m(q) T^m
\end{align}
where $f_m(q) \in \Lambda_q$ for all $m \in \N$. Note the natural injective linear maps 
\begin{align}
	\Lambda_q[[Q^{\pm 1}]]\longrightarrow \C\{q,Q\}\ ,
\end{align}
as well as the natural ring monomorphisms 
\begin{align}
	\C[[q,Q^{\pm 1}]] \longrightarrow \Lambda_q[[Q^{\pm 1}]]\ .
\end{align}

Now define the following generating functions 
\begin{align}\label{eq:fPT}
 	f\textrm{-}\mathsf{PT}_C(Y; q, Q)&\coloneqq  \sum_{\ell\in \Z} \sum_{m\in \N} \chi(\fSP_C(Y; m, \ell)) q^\ell Q^m \ , \\[4pt] \label{eq:PT}
	 \mathsf{PT}_C(Y; q, Q)&\coloneqq  \sum_{\ell\in \Z} \sum_{m\in \N} \chi(\SP_C(Y; m, \ell)) q^\ell Q^m
 \end{align}
as elements of $\C\{q,Q\}$. The same boundedness arguments as in \cite[\S3.2]{PT-Curve-counting} and \cite[\S7]{BS-Curve-counting-crepant} show that both generating functions belong to $\Lambda_q[[Q]]$. 

Furthermore, set
\begin{align} \label{eq:PTex}
	\mathsf{PT}_{\mathsf{ex}}(Y; q,Q)\coloneqq \sum_{\ell\in \Z} \sum_{m\in \N} \chi(\SP(Y; m[\Sigma], \ell)) q^\ell Q^m\ .
\end{align}
By Corollary~\ref{cor:chi-bound}, this is an element of $\C[[q,Q]]$. Moreover, it is invertible since, by definition, the term of bi-degree $(0,0)$ is $1$.

The main result of this section is the following.
\begin{theorem}\label{thm:C-framed-identity} 
	The following identity holds in $\Lambda_q[[Q]]$ under Assumptions~\ref{assumption:f} and \ref{assumption:divisor}.
	\begin{align}\label{eq:f-pairs-identity-B} 
		f\textrm{-}\mathsf{PT}_C(Y; q,Q) = \frac{\mathsf{PT}_C(Y; q,Q)}{\mathsf{PT}_{\mathsf{ex}}(Y; q,Q)}\ .
	\end{align}
\end{theorem}

We postpone the proof of the theorem to Appendix~\ref{sec:proof-C-framed-identity}. We shall proceed by analogy to \cite[\S4]{Bridgeland-Curve-Counting} and \cite[\S6]{BS-Curve-counting-crepant}. In \S\ref{subsec:motHall} and \S\ref{subsec:hall-identity} will derive the relevant identities in the infinite type motivic Hall algebra of stack functions. As explained in \S\ref{subsec:proof}, these identities result in a fake proof of identity \eqref{eq:f-pairs-identity-B}. The actual proof is then obtained using Laurent truncations in complete analogy to \cite[\S7]{BS-Curve-counting-crepant}. 

\section{$C$-framed flop identity}\label{sec:flop-identity}

Consider a \textit{threefold flop diagram}
\begin{align}\label{eq:flop-diagram-A}
	\begin{tikzcd}[ampersand replacement=\&]
		Y^+ \ar[swap]{dr}{f^+} \ar[dashed]{rr}{\phi}\& \& Y^- \ar{dl}{f^-}\\
		\& X \&
	\end{tikzcd}\ ,
\end{align}
where  $f^\pm \colon Y^\pm\to X$ are contractions satisfying the conditions of Assumption~\ref{assumption:f} and $\phi$ is a canonical isomorphism $\phi\colon Y^+\smallsetminus \Sigma^+\xrightarrow{\sim} Y^-\smallsetminus \Sigma^-$. Recall that $\nu$ denotes the singular point of $X$.
\begin{remark}
	Note that given a morphism $f\colon Y \to X$ as in Assumption~\ref{assumption:f}, results of \cite{Kawamata-Crepant-Blowup,Mori-Flip-Theorem, Kollar-Flops} imply that a flop of the form \eqref{eq:flop-diagram-A}, with e.g. $f=f^+$, always exists and it is moreover unique by \cite[Corollary~6.4]{KM-Birational}. 
\end{remark}
Let us assume that $f^+\colon Y^+\to X$ satisfies Assumption~\ref{assumption:divisor}. In particular, there exists a Weil divisor $W\subset X$ for which the strict transform of $W$ is a smooth divisor $S^+\subset Y^+$ so that the restriction $f^+\vert_{S^+}$ is an isomorphism onto $W$. 

Let $C\subset W$ be a reduced irreducible curve on $W$ so that $\nu \in C^{\mathsf{sing}}$. The strict transform of $C$ with respect to $f^+$ is a reduced irreducible curve $C^+\subset S^+$, also mapped isomorphically to $C$ by $f^+$. Let $S^-$, $C^-$ be the strict transforms of $W,C$ respectively under the contraction $f^-\colon Y^-\to X$. Then $C^-$ is a reduced irreducible curve on $Y^-$. An extra assumption will be made in the following, namely:
\begin{assumption}\label{assumption:surface} 
	The strict transform $S^-$ of $W$ is a reduced irreducible divisor on $Y$ which coincides with the scheme-theoretic inverse image $S^-=(f^-)^{-1}(W)$. Moreover, the singular locus of $S^-$ is zero-dimensional.
\end{assumption}

\begin{remark}
	In \S\ref{sec:examples}, we shall construct an example of a flop transition between projective threefolds satisfying Assumptions~\ref{assumption:f}, \ref{assumption:divisor}, and \ref{assumption:surface}. 
\end{remark}

The goal of this section is to prove an identity relating the Euler character invariants of the moduli space of stable pairs on $C^-$ to those of the moduli space of $C^+$-framed $f^+$-stable pairs on $Y^+$. 

\subsection{Transformation of numerical classes}\label{subsec:transformation-classes} 

Let $NS(Y^\pm)$ be the \textit{Neron-Severi group} of $Y^\pm$, i.e., the image of $\ch_1\colon \Pic(Y^\pm)\to A_2(Y^\pm)$, and let $N^1(Y^\pm)$ be the group of \textit{numerical equivalence classes} of Cartier divisors on $Y^\pm$, i.e., the quotient of $NS(Y^\pm)$ by its torsion subgroup. $N^1(Y^\pm)$ is a finitely generated free abelian group. One has an isomorphism 
\begin{align}
	\begin{tikzcd}[ampersand replacement=\&]
		\phi^1_\ast \colon N^1(Y^+) \arrow{r}{\sim} \& N^1(Y^-)
	\end{tikzcd}\ ,
\end{align}
determined by the chain of isomorphisms 
\begin{align}
	\begin{tikzcd}[ampersand replacement=\&]
		A_2(Y^+) \arrow{r}{\sim} \& A_2(Y^+ \smallsetminus \Sigma^+)\arrow{r}{\sim} \&
		A_2(Y^-\smallsetminus \Sigma^-) \arrow{r}{\sim} \& A_2(Y^-)
	\end{tikzcd}\ .
\end{align}
More explicitly, this isomorphism is determined by the assignment 
\begin{align}
	D^+\longmapsto D^+ \cap Y^+ \smallsetminus \Sigma^+ \longmapsto \overline{D^+ \cap Y^+ \setminus \Sigma^+}\ ,
\end{align}
where $\overline{D^+ \cap Y^+ \setminus \Sigma^+}$ denotes the scheme-theoretic closure of $D^+ \cap Y^+ \setminus \Sigma^+$ in $Y^-$. 

Using the non-degenerate bilinear form $N^1(Y^\pm) \times N_1(Y^\pm) \to \Z$, one obtains an isomorphism $\phi_{1,\, \ast}\colon N_1(Y^+) \xrightarrow{\sim} N_1(Y^-)$, where $\phi_{1,\, \ast} = ((\phi^1_\ast)^{\vee})^{-1}$. More explicitly, given any class $\beta \in N_1(Y^+)$, the image $\phi_{1,\, \ast}(\beta) \in N_1(Y^-)$ is uniquely determined by the relations
\begin{align}
	[D^+]\cdot \beta = \phi^1_\ast([D^+])\cdot \phi_{1,\ast}(\beta) \ ,
\end{align}
with $[D^+]\in N^1(Y^+)$ arbitrary. Recall also that $\phi_\ast([\Sigma^+]) = -[\Sigma^-]$. 

Let $f_{S^-}\colon S^- \to W$ be natural projection.  Assumption~\ref{assumption:surface} implies that the scheme-theoretic inverse image $(f^-)^{-1}(C)$ coincides with the inverse image $f_{S^-}^{-1}(C)$, regarded as a closed subscheme of $Y$ contained in $S^-\subset Y$. 
Moreover, since $S^-$ is the strict transform of $S^+$, and $S^+\cdot \Sigma^+=1$ by Assumption~\ref{assumption:divisor}, one has $S^-\cdot \Sigma^-=-1$. Therefore $\Sigma^-$ is contained in $S^-$ as a closed subscheme. This implies that $(f_{S^-})^{-1}(C)$ is the scheme-theoretic union of two irreducible components:
\begin{align}
	(f_{S^-})^{-1}(C) = C^-\cup Z_{\Sigma^-}\ ,
\end{align}
where $(Z_{\Sigma^-})_{\mathsf{red}}=\Sigma^-$. Let $m_{\Sigma^-}\in \N$ be defined by $\ch_2(\scrO_{Z_{\Sigma^-}}) = m_{\Sigma^-} [\Sigma^-]$ in $N_1(Y^-)$.

\begin{lemma}\label{lem:flop-isomorphim}
	Under Assumptions~\ref{assumption:f}, \ref{assumption:divisor}, and \ref{assumption:surface}, 
 the following identity holds in $N_1(Y^-)$.
	\begin{align}
			\phi_{1,\, \ast}([C^+] ) = [C^-] + m_{\Sigma^-} [\Sigma^-] \ .
	\end{align}
\end{lemma} 

\begin{proof}
	By construction, $\phi$ identifies the open curves $C^{+,\, \circ} \coloneqq C^+\cap (Y^+\smallsetminus \Sigma^+)$ and $C^{-,\, \circ}\coloneqq C^-\cap (Y^-\smallsetminus \Sigma^-)$. Furthermore, $C^\pm$ is the scheme-theoretic closure of $C^{\pm,\, \circ}$ and $(C^\pm \cap\Sigma^\pm)_{\mathsf{red}}$ is zero-dimensional. The image $\phi_{1,\, \ast}([C^+])$ is uniquely determined by the intersection numbers 
	\begin{align}
		\gamma^- \cdot \phi_{1,\, \ast}([C^+]) = (\phi^1_\ast)^{-1}(\gamma^-) \cdot [C^+]\ , 
	\end{align}
	with $\gamma^-\in N^1(Y^-)$. Without loss of generality, we can assume that $\gamma^-$ is sufficiently ample so that the following conditions hold for any sufficiently generic divisor $D^-\in \gamma^-$:
	\begin{enumerate}\itemsep0.2cm
		\item \label{item:a} the scheme-theoretic intersection of $D^-$ and $S^-$ is a smooth projective curve $\Gamma^-\subset (S^-)^{\mathsf{reg}}$ which intersects $\Sigma^-$ transversely at finitely many closed points $p_1,\ldots, p_d$ contained in $(S^-)^{\mathsf{reg}}\smallsetminus C^-$, and 
		\item \label{item:b} $(\Gamma^-\cap C^-)_{\mathsf{red}}$ is a zero-dimensional set contained in $(C^-)^{\mathsf{reg}} \smallsetminus \Sigma^-$.
	\end{enumerate} 
	
	Let $D\subset X$ be the scheme theoretic image of $D^-$ through $f^-\colon Y^-\to X$ and let $D^+\subset Y^+$ be the strict transform of $D$. 
	Then $(\phi^1_\ast)^{-1}(\gamma^-)=[D^+]$. Note also that 
	\begin{align}
		[D^+]\cdot [\Sigma^+] = - [D^-]\cdot [\Sigma^-]<0 \ ,
	\end{align}
	hence $\Sigma^+$ is contained in $D^+$ as a closed subscheme. 
	
	Let $\sigma\colon \widetilde{S}^-\to S^-$ be a minimal resolution of singularities, and let $\widetilde{\Sigma}^-$, $\widetilde{C}^-$ be the strict transforms of $\Sigma^-$ and $C^-$, which are divisors on $\widetilde{S}^-$. Then the inverse image $\sigma^{-1}(f_{S^-})^{-1}(C)$ is a scheme-theoretic union of effective divisors 
	\begin{align}
		\sigma^{-1}(f_{S^-})^{-1}(C) = \widetilde{C}^- \cup\, m_{\Sigma^-}\widetilde{\Sigma}^-\, \cup \Delta\ , 
	\end{align}
	where $(\Delta)_{\mathsf{red}}$ is contained in the exceptional locus of $\sigma$. Since $\Gamma^-$ is entirely contained in the smooth locus of $S^-$, its strict transform $\widetilde{\Gamma}^-$ coincides with the inverse image $\sigma^{-1}(\Gamma^-)$ and it is canonically isomorphic to $\Gamma^-$. Furthermore it is also disjoint from the exceptional locus of $\sigma$. 
	
	Let $\Gamma\subset W$ be the scheme theoretic image of $\Gamma^-$. Clearly, $\Gamma$ coincides with the scheme theoretic image of $\widetilde{\Gamma}^-$ via the the proper morphism $f_{S^-}\circ \sigma\colon \widetilde{S}^- \to W$. Applying the projection formula for 
	$f_{S^-}\circ \sigma$, one finds 
	\begin{align}
		[\Gamma] \cdot [C]= m_{\Sigma^-}\, [\widetilde{\Gamma}^-]\cdot [\widetilde{\Sigma}] + [\widetilde{\Gamma}^-] \cdot [\widetilde{C}^-] \ .
	\end{align}
	Given assumptions \eqref{item:a} and \eqref{item:b} above, since $\Sigma^-$ and $C^-$ are subschemes of $S^-$, this further yields 
	\begin{align}
		[\Gamma] \cdot [C] = m_{\Sigma^-} [D^-]\cdot [\Sigma^-] + [D^-]\cdot [C^-]\ ,
	\end{align}
	where the right-hand-side is a sum of intersection products in $Y^-$. In order to finish the proof, recall that the induced projection $f_{S^+}\colon S^+ \to S$ is an isomorphism under the current assumptions. Set $\Gamma^+\coloneqq (f_{S^+})^{-1}(\Gamma)$ and note that $\Gamma^+$ is canonically isomorphic to $\Gamma$. The same holds for the $C^+\subset S^+$. Then one obtains 
	\begin{align}
		[D^+]\cdot [C^+] = [\Gamma^+]\cdot [C^+]  = [\Gamma]\cdot [C]  = [D^-]\cdot (m_{\Sigma^-}[\Sigma^-] + [C^-])\ . 
	\end{align}
	In conclusion, one obtains an identity 
	\begin{align}
		\gamma^-\cdot \phi_{1,\, \ast}([C^+]) = [D^-]\cdot (m_{\Sigma^-}[\Sigma^-] + [C^-])\ ,
	\end{align}
	for any ample divisor class $\gamma_-$ satisfying conditions \eqref{item:a} and \eqref{item:b} above. This proves the claim. 
\end{proof}

\subsection{Flop identity}\label{subsec:flop-identity}  

Let $\mathsf{PT}_{C^\pm}(Y^\pm)$ and $\mathsf{PT}_{\mathsf{ex}}(Y^\pm)$ be the generating functions \eqref{eq:PT} and \eqref{eq:PTex}, respectively, associated to $f^\pm\colon Y^\pm \to X$. Set
\begin{align}
	\mathsf{PT}'_{C^\pm}(Y^\pm; q,Q)=
	\frac{\mathsf{PT}_{C^\pm}(Y^\pm; q,Q)}{\mathsf{PT}_{\mathsf{ex}}(Y^\pm; q,Q)}\ .
\end{align}
As in Theorem~\ref{thm:C-framed-identity}, these are elements of $\Lambda_q[[Q]]$. Recall that $\C\{q,Q\}$ denotes the vector space of bi-formal power series $\sum_{\ell, m \in \Z} c_{\ell,m}q^\ell Q^m$. Then we have:
\begin{proposition}\label{prop:C-framed-flop-A} 
	The following identity holds in $\C\{q,Q\}$ under Assumptions~\ref{assumption:f}, \ref{assumption:divisor}, and \ref{assumption:surface}. 
	\begin{align}\label{eq:C-framed-flop-A} 
		Q^{m_{\Sigma^-}}\, \mathsf{PT}'_{C^+}(Y^+; q, Q^{-1}) =\mathsf{PT}'_{C^-}(Y^-; q, Q) \ .
	\end{align}
	In particular, both sides are polynomial functions of $Q$ with coefficients in $\Lambda_q$.
\end{proposition}

\begin{proof}
	The proof is completely analogous to the proof of \cite[Proposition~2.4]{Maulik-Stable-pairs}, which is a $C$-framed variant of  \cite[Theorem~3.26 and Corollary~3.27]{Calabrese-DT-Flops}\footnote{An equivalent proof of this result is provided also in \cite{Toda_Curve_counting_flops}.}, and it also uses Lemma~\ref{lem:flop-isomorphim}. One should also keep in mind that in the present context the derived equivalence $\Phi\colon \catDb(Y^+) \xrightarrow{\sim} \catDb(Y^-)$ required in \cite[\S3.9]{Calabrese-DT-Flops} is a special case of \cite[Theorem~C]{VdB-Flops}. 
	
	The proofs of \cite[Proposition~2.4]{Maulik-Stable-pairs} and \cite[Theorem~3.26 and Corollary~3.27]{Calabrese-DT-Flops} rely on wallcrossing identities in a certain completion of the motivic Hall algebra of the stack $\Coh(\catCoh_{\leqslant 1}(Y)$, which translate into generating function identities via an integration map. The completion used in this proces was introduced in \cite{Bridgeland-Curve-Counting}, and it is briefly reviewed in Appendix~\ref{subsec:motHall} for completeness. In the present framework, the motivic wallcrossing identities are identical to those of \textit{loc. cit.}, the main difference residing in the fact that the
	threefolds  $Y^\pm$ are not $K$-trivial. As explained in \S\ref{subsec:motHall}, in general, this precludes the construction of an integration map defined on the whole motivic Hall algebra. However, using the local $K$-triviality condition proven in Lemma~\ref{lem:trivial},  Proposition~\ref{prop:integrationB} provides a certain truncation of the integration map, which applies to all wallcrossing identities in \cite{Maulik-Stable-pairs}, \cite{Calabrese-DT-Flops}. Note that a similar construction was employed in \cite{ST-Hilbert} for non-Calabi-Yau threefolds, using analogous arguments.
\end{proof}

As shown in Theorem~\ref{thm:C-framed-identity}, one also has the identity 
\begin{align}
	f^+\textrm{-}\mathsf{PT}_{C^+}(Y^+; q,Q)= \mathsf{PT}'_{C^+}(Y^+; q,Q)\ .
\end{align}
Therefore Theorem~\ref{thm:C-framed-identity} and Proposition~\ref{prop:C-framed-flop-A} yield:\begin{corollary}\label{cor:C-framed-flop-B} 
The following identity holds in $\Lambda_q$ under Assumptions~\ref{assumption:f}, \ref{assumption:divisor}, and \ref{assumption:surface}.
	\begin{align}\label{eq:C-framed-flop-B}
		 \Big[f^+\textrm{-}\mathsf{PT}_{C^+}(Y^+; q, Q^{-1})\Big]\Big\vert_{Q^{-m_{\Sigma^-}}}=q^{\chi(\scrO_{C^-})}\sum_{k\in \N}\chi(\SP_{C^-}^k(Y^-)) q^k\ ,
	\end{align}
	where the right-hand-side is the coefficient of $f^+\textrm{-}\mathsf{PT}_{C^+}$ of degree $Q^{-m_{\Sigma^-}}$. 
	
	Here, $\SP_{C^-}^k(Y^-)\coloneqq\SP_{C^-}(Y^-; 0, k+\chi(\scrO_{C^-}))$, i.e., it is the moduli space of stable pairs $(\calF, s)$ on $Y^-$ such that $\calF$ is scheme-theoretically supported on $C^-$, $\ch_2(\calF)=[C]$, and $\chi(\calF)=k+\chi(\scrO_{C^-})$.
\end{corollary} 

Using Proposition~\ref{prop:C-framed-Hilbert-C}, the left-hand-side of identity \eqref{eq:C-framed-flop-B} can be rewritten as a generating function of Euler numbers for the Flag Hilbert schemes $\FHilb_{T_n}^k(C^+; m)$ considered in \S\ref{subsect:relation-hilbert-scheme}. Recall that $\FHilb_{T_n}^k(C^+; m)$ parametrizes flags $\calI\subset \calI'\subset \scrO_{C^+}$ of ideal sheaves such that $\chi(\scrO_{C^+}/\calI')=k$ and $\calI'/\calI$ is a length $m$ zero-dimensional sheaf of $T_n$, where $T_n$ is the scheme-theoretic intersection of $S^+$ and $\Sigma^+_n$ in $Y^+$. Then, Proposition~\ref{prop:C-framed-Hilbert-C} immediately yields the identity 
\begin{align}\label{eq:CfrflopC} 
	\Big[f^+\textrm{-}\mathsf{PT}_{C^+}(Y^+; q, Q^{-1})\Big]\Big\vert_{Q^{-m_{\Sigma^-}}}= q^{\chi(\scrO_{C^+})-m_{\Sigma^-}} \sum_{k\in \N} \chi(\FHilb_{T_n}^k(C^+; m_{\Sigma^-})) q^k\ . 
\end{align}
Therefore, we obtain:
\begin{corollary}\label{cor:C-framed-flop-C}
	The following identity holds in $\Lambda_q$ under Assumptions~\ref{assumption:f}, \ref{assumption:divisor}, and \ref{assumption:surface}.
	\begin{align}\label{eq:C-framed-flop-C}
		q^{\chi(\scrO_{C^+})} \sum_{k\in \N} \chi(\FHilb_{T_n}^k(C^+; m_{\Sigma^-})) q^k= q^{\chi(\scrO_{C^-})}\sum_{k\in \N}\chi(\SP_{C^-}^k(Y^-)) q^k\ .
	\end{align}
\end{corollary}

We can also provide a punctual version of identity \eqref{eq:C-framed-flop-C} as follows. Let $(C^-\cap \Sigma^-)_{\mathsf{red}}=\{ y_1, \ldots, y_d\}$. For each $1\leq i \leq d$, let $\SP_{C^-,\, y_i}^k(Y^-)\subset \SP_{C^-}^k(Y^-)$ be the reduced closed subscheme defined by the condition that $\mathsf{Coker}(s)$ is set-theoretically supported at $y_i$. 

Recall that $(C^+\cap \Sigma^+)_{\mathsf{red}}$ consists of a single point $p\in (C^+)^{\mathsf{sing}}$. Let $\Hilb_p^{k+m_{\Sigma^-}}(C^+)$ be the punctual Hilbert scheme of $C^+$ parametrizing length $k+m_{\Sigma^-}$ zero-dimensional subschemes with set-theoretic support at $p$. Let $\FHilb_{T_n, p}^k(C^+; m_{\Sigma^-})\subset \FHilb_{T_n}^k(C^+; m_{\Sigma^-})$ be the inverse image of $\Hilb_p^{k+m_{\Sigma^-}}(C^+)$ via the canonical forgetful morphism $\FHilb_{T_n}^k(C^+; m_{\Sigma^-}) \to \Hilb^{k+m_{\Sigma^-}}(C^+)$ sending the flag $\calI\subset \calI'\subset \scrO_{C^+}$ to $\calI\subset \scrO_{C^+}$. Then, by an inductive stratification argument, Corollary~\ref{cor:C-framed-flop-C} further yields:
\begin{corollary}\label{cor:C-framed-flop-D} 
	The following identity holds in $\Lambda_q$ under Assumptions~\ref{assumption:f}, \ref{assumption:divisor}, and \ref{assumption:surface}.
	\begin{align}\label{eq:C-framed-flop-D} 
		q^{\chi(\scrO_{C^+})} \sum_{k\in \N} \chi(\FHilb_{T_n, p}^k(C^+; m_{\Sigma^-}))q^k = q^{\chi(\scrO_{C^-})} \prod_{i=1}^d \sum_{k\in \N} \chi(\SP_{C^-,\, y_i}^k(Y^-))q^k\ .
	\end{align}
\end{corollary}

Finally, we can also derive the following identity when $C^-$ has locally complete intersection singularities.
\begin{corollary}\label{cor:C-framed-flop-E} 
	In addition to Assumptions \ref{assumption:f}, \ref{assumption:divisor}, and \ref{assumption:surface}, suppose that $C^-$ has locally complete intersection singularities. Then, the following identity holds: 
	\begin{align}\label{eq:C-framed-flop-E} 
		q^{\chi(\scrO_{C^+})} \sum_{k\in \N}  \chi(\FHilb_{T_n, p}^k(C^+; m_{\Sigma^-})) q^k = q^{\chi(\scrO_{C^-})} \prod_{i=1}^d \sum_{k\in \N} \chi(\Hilb_{y_i}^k(C^-))q^k\ .
	\end{align}
\end{corollary}

\begin{proof}
	By assumption, the dualizing sheaf of $C^-$ is a line bundle. Then  \cite[Proposition~B.5]{PT-BPS} yields an isomorphism $\SP_{C^-}^k(Y^-) \simeq \Hilb^k(C^-)$. Then identity \eqref{eq:C-framed-flop-E} follows from identity \eqref{eq:C-framed-flop-D}.
\end{proof}

\section{Examples and explicit results}\label{sec:examples}

This section presents a concrete construction of a flop transition between projective threefolds satisfying Assumptions~\ref{assumption:f}, \ref{assumption:divisor}, and \ref{assumption:surface}. Furthermore, we prove the following results.

First, fix $r,s,n\in \N$, with $r>s>n\geq 2$ and $r, s$, as well as $r,n$, coprime. Consider the planar curve singularity $C_{r,s}$ given by the equation $x^r=w^s$ and corresponding to the $(s, r)$ torus knot. Denote by $p$ the origin of $\C^2$. Let $T_n$ be the closed subscheme given by the equations $x=0$ and $w^n=0$. For any $t,m\geq 1$, let 
\begin{align}
	Z_{t,m}(q,x)  \coloneqq  \prod_{i=1}^t (1+xq^i + \cdots + (xq^i)^{m})
\end{align}
and let $Z_{t,m}^{(k)}(q)$ denote the coefficient of $x^k$ in $Z_{t,m}(q,x)$. In \S\ref{subsec:torus-invariant-+} we prove the following:
\begin{theorem}
	Let $\FHilb_{T_n, p}^k(C_{r,s}; s)$ be the punctual Flag Hilbert scheme parametrizing flags of ideal sheaves $\calI_1\subset \calI_2\subset \scrO_{C_{r,s}}$ so that $\scrO_{C_{r,s}}/\calI_1$ is set-theoretic supported at $p$, the ideal $\calI_2$ has colength $k$, and $\calI_2/\calI_1$ is the pushforward of a length $s$ zero-dimensional sheaf on $T_n$. Then 
	\begin{align}
		\sum_{k\geq 0} \chi(\FHilb_{T_n, p}^k(C_{r,s}; s)) q^k =  \frac{1}{1-q^s}\sum_{m=0}^{n-1} \sum_{k=s-n}^{s-m-1} q^{mr} Z^{(k)}_{r-1,n}(q)\ .
	\end{align}
\end{theorem} 
Moreover, consider the complete intersection curves of the form 
\begin{align}
	C_{r,t,n}\colon\begin{cases}
		xv - w^n =0 \\[2pt]
		x^{r-t}- v^{t} =0 
	\end{cases}
\end{align}
in $\C^3$, with $r,t,n\geq 2$ with $r>t$ and $(r,nt)$ coprime. In \S\ref{subsec:explicit-example-space-curve}, combining the above theorem with the results derived in the previous sections, we obtain:
\begin{theorem}
	For any $\ell \geq 0$, let $\Hilb_o^\ell(C_{r,t,n})$ be punctual Hilbert scheme parametrizing length $\ell$ zero-dimensional subschemes of $C_{r,t,n}$ with support at the origin $o$. Then 
	\begin{align}
		\sum_{\ell\geq 0} \chi(\Hilb_o^\ell(C_{r,t,n})) q^\ell = \frac{q^{-nt(t-1)/2} }{1-q^{tn}}\sum_{m=0}^{n-1} \sum_{k=tn-n}^{tn-m-1} q^{mr} Z^{(k)}_{r-1,n}(q) \ .
	\end{align}
\end{theorem} 

\subsection{Pagoda flop transition for complete intersection threefolds}\label{subsec:examples-section-A}

Let $P$ be the toric fourfold $(\C^6\setminus\Delta)/\C^\ast \times \C^\ast$ where the torus action has weights 
\begin{align}
	\begin{array}{cccccc}
		z_0 & z_1 & z_2 & z_3 & z_4 & z_5 \\
		1& 0 & 0 & 0 & 1 & -n \\
		0 & 1 & 1 & 1 & 0 & 1
	\end{array} 
\end{align}
and 
\begin{align}
	\Delta \coloneqq \{z_0=z_4=0\} \cup \{z_1=z_2=z_3=z_5=0\}\ . 
\end{align}
Let $X \subset P$ be the hypersurface given by
\begin{align}
	2z_1z_2-z_3^2 + z_4^{2n} z_5^2 =0\ .
\end{align}
\begin{lemma}\label{lem:example-O} 
	$X$ has a unique singular point at $z_i=0$, with $1\leq i \leq 4$. 
\end{lemma} 

\begin{proof}
	Let $P^\circ\subset P$ denote the open subscheme  $\{ z_0\neq 0, \ z_5\neq 0\}$. Note that $P^\circ \simeq \C^4$ with toric affine coordinates
\begin{align}\label{eq:toric-coordinates} 
	\begin{cases}
		x= z_0^{}z_5^{-1} z_1\ , \\ 
		y= z_0^{}z_5^{-1} z_2\ , \\ 
		z= z_0^{}z_5^{-1} z_3\ , \\ 
		w= z_0^{-1} z_4\ .
	\end{cases}
\end{align}
Then the intersection $X^\circ \coloneqq X \cap P^\circ$ is the hypersurface given by
\begin{align}
	2xy -z^2 + w^{2n} =0\ .
\end{align}
This is precisely the threefold singularity studied in \cite[\S5, Remark~5.13]{Reid-Minimal-Models}. In particular, it has a unique singular point at the origin
which is furthermore a cDV singularity. In order to finish the proof, one has to check that $X$ is smooth at all points in the complement of $X^\circ$. This follows by straightforward computations in toric affine charts. 
\end{proof}

Let $Z^\pm\subset P$ be the closed subschemes given by
\begin{align}
	 \begin{cases}
	 	z_1 = 0\ ,\\ 
	 	z_3 - z_4^n z_5 =0\ ,
	 \end{cases}  \quad \text{and} \quad 
	\begin{cases}
		z_1 =0 \ ,\\ 
		z_3 + z_4^n z_5 = 0\ .
	\end{cases}
\end{align}
respectively. Let $P^{\pm}\to P$ be the blow-up of $P$ along $Z^\pm$, which is again smooth. Let $Y^{\pm}\subset P^{\pm}$ denote the strict transform of $X$ in $P^\pm$ and let $f^\pm \colon Y^\pm \to X$ denote the canonical projections. Let $\Sigma^\pm\coloneqq (f^\pm)^{-1}(\nu)$ be the scheme-theoretic inverse images of the singular point $\nu \in X$. Note that both $\Sigma^\pm$ coincide with the rational curve given by
\begin{align}
	z_i=0\ \text{ for } \ 1\leq i \leq 4 
\end{align}
as subschemes of $P\times \PP^1$. Then the following holds.
\begin{proposition}\label{prop:example-A} 
	$Y^\pm$ is smooth, the projection  $f^\pm \colon Y^\pm \to X$ satisfies Assumption~\ref{assumption:f}.
\end{proposition} 

\begin{proof}
	Note that $Y^\pm$ are given by the following equations in $P\times \PP^1$
	\begin{align}\label{eq:Y+-equations} 
		Y^+&\colon \begin{cases}
			z_1w_1 - (z_3-z_4^nz_5) w_2 =0\ , \\ 
			2z_2w_2 - (z_3+z_4^nz_5) w_1 =0 \ ,
		\end{cases}\\ \label{eq:Y--equations}
		Y^-&\colon  \begin{cases}
			2z_1w_1 - (z_3+z_4^nz_5) w_2 =0\ , \\
			z_2w_2 - (z_3-z_4^nz_5) w_1 =0\ .
		\end{cases} 
	\end{align}
	This implies that the exceptional loci $\Sigma^\pm$ of the projections $f^\pm$ coincide simultaneously with the projective line 
	\begin{align}
		z_1=0, \ z_2 =0, \ z_3 =0, \ z_4 =0 \ ,
	\end{align}
	as subschemes of $P\times \PP^1$. Since $X\smallsetminus \{\nu\}$ is smooth, $Y^\pm \smallsetminus \Sigma_\pm$ is also smooth. Moreover, let $Y_\pm^{\circ} \coloneqq (f^\pm)^{-1}(P^\circ)$, where $P^\circ \subset P$ is the open subscheme  $\{ z_0\neq 0, \ z_5\neq 0\}$ used in the proof of Lemma~\ref{lem:example-O}. Then the resulting projections $(Y^\pm)^\circ \to X^\circ$ coincide with the two canonical Atiyah-Brieskorn resolutions of $X^\circ$, as observed in  \cite[Remark~5.13]{Reid-Minimal-Models}. Hence $Y^\pm$ is smooth and $f^\pm$ satisfies Assumption~\ref{assumption:f}.
\end{proof}

\begin{remark}
	Let $\Sigma_k \subset P\times \PP^1$, with $1\leq k \leq n$, be the closed subschemes 
	\begin{align}
		z_i=0 \ \text{ for } \ 1\leq i \leq 3\ , \quad  z_4^k =0\ .
	\end{align}
	Equations~\eqref{eq:Y+-equations} and \eqref{eq:Y--equations} show that $\Sigma_k$ is contained as a closed subscheme in both $Y^\pm$ for all $1\leq k \leq n$. Let $\calJ_k \subset \scrO_{Y^\pm}$ be the associated ideal sheaves. Clearly, one has a filtration 
	\begin{align}
		\calJ_n \subset \cdots \subset \calJ_1=\calI_\Sigma
	\end{align}
	as \eqref{eq:ideal-filtration-A}, where $\calI_\Sigma$ is the ideal sheaf of $\Sigma$. Moreover, conditions \eqref{eq:ideal-filtration-B} and  \eqref{eq:ideal-filtration-C} are easily verified by local computations using the affine coordinates \eqref{eq:toric-coordinates}. 
	
	The above filtration coincides with that of Theorem~\ref{thm:filtration} because of the uniqueness result proved in Theorem~\ref{thm:filtration}.
\end{remark}

Furthermore, let $W \subset X$ be the Weil divisor defined by the equations 
\begin{align}\label{eq:Weil-divisor}
	z_2=0\ , \quad z_3-z_4^nz_5 =0\ .
\end{align}
Let $S^\pm\subset Y^\pm$ denote its strict transforms. Then the following holds by straightforward computations.
\begin{proposition}\label{prop:example-B} 
	\hfill
	\begin{itemize}\itemsep0.2cm
		\item  The strict transform $S^+$ of $W$ under $f^+\colon Y^+\to X$ is determined by the equations 
		\begin{align}\label{eq:plus-surface} 
			\begin{cases}
				z_2=0\ , \\ 
				w_1=0\ , \\ 
				z_3-z_4^nz_5 =0\ .
			\end{cases}
		\end{align}
		In particular it is smooth irreducible and has intersects $\Sigma^+$ transversely at  the unique point $p\in Y^+$ determined by 
		\begin{align}
			z_i=0\ \text{ for }\ i=1, \ldots, 4\ , \ w_1=0\ .
		\end{align}
	
		\item The strict transform $S^-$ coincides with the scheme-theoretic inverse image  $(f^-)^{-1}(W)$ and it is determined by the equations 
		\begin{align}\label{eq:minus-surface} 
			\begin{cases}
				z_2=0\ , \\
				z_3-z_4^nz_5=0\ , \\ 
				z_1w_1 -z_4^nz_5w_2=0\ .
			\end{cases}
		\end{align}
		In particular, it has a unique singular point $q$ determined by 
		\begin{align}
			z_i=0\ \text{ for }\ i=1, \ldots, 4\ , \ w_1=0\ ,
		\end{align}
		which is Kleinian singularity of type $A_{n-1}$. 
	\end{itemize}
\end{proposition}

\begin{corollary}
	Assumptions\ref{assumption:divisor} for $f^+$ and Assumption~\ref{assumption:surface} are satisfied. 
\end{corollary}

\subsection{Torus invariant plane curves in $Y^+$} \label{subsec:torus-invariant-+}

Now let $C\subset W$ be the curve 
\begin{align}\label{eq:plane-curve-A}
	z_1^r - z_4^s =0\ ,
\end{align}
where $r,s\geq 2$ are coprime integers. As in the proof of Lemma~\ref{lem:example-O}, let $P^\circ\subset P$ the affine toric chart $z_0\neq 0$ and $z_5\neq 0$, with coordinates $(x,y,z,w)$ defined in \eqref{eq:toric-coordinates}. Recall that the defining equation of $X^\circ = X\cap P^\circ$ is 
\begin{align}
	2xy - z^2 + w^{2n} =0\ . 
\end{align}
Let $(Y^+)^{\circ}\coloneqq (f^+)^{-1}(X^\circ)$ and note that it has the natural open cover $U^+\coloneqq   (Y^+)^{\circ}\cap \{w_1\neq 0\}$ and $V^+\coloneqq  (Y^+)^{\circ}\cap \{w_2\neq 0\}$, where $[w_1:w_2]$ are homogeneous coordinates on $\PP^1$ as in equations \eqref{eq:Y+-equations} and \eqref{eq:Y--equations}. Let $u= w_1^{-1}w_2$ and $v = w_2^{-1} w_1$ be the natural affine coordinates on $\PP^1$. Then 
\begin{align}
	U^+ \coloneqq \Spec\, \C[u,y,w]\quad\text{and}\quad V^+ \coloneqq \Spec\, \C[v,x,w]
\end{align}
and $U^+\smallsetminus \{u=0\}$ is identified with $V^+\smallsetminus \{v=0\}$ via the transition functions 
\begin{align}\label{eq:trfct} 
	uv =1\ , \quad xv = w^n - yu\ .
\end{align}
Moreover, the following holds by local computations.
\begin{lemma}\label{lem:plus-curve-B} 
	\hfill
	\begin{itemize}\itemsep0.2cm
		\item Let $W^\circ$ and $C^\circ$ be the scheme-theoretic intersection of $W$ and $C$, respectively, with the open subscheme $P^\circ\subset P$. Then $W^\circ$ is given by 
		\begin{align}
			y=0\ , \ z-w^n =0\ ,
		\end{align}
		hence it is isomorphic to the affine plane $\Spec\, \C[x,w]$. The affine curve $C^\circ\subset W^\circ$ is defined by 
		\begin{align}
			x^r - w^s=0\ .
		\end{align}
	
		\item The scheme-theoretic intersection $S^+\cap V^+$ is the affine plane $v=0$ and the strict transform $C^+\subset S^+$ is locally given by 
		\begin{align}
			v=0\ , \ x^r - w^s =0
		\end{align}
		in $V^+$. In particular, $C^+$ has a unique intersection point $p$ with $\Sigma^+$, which coincides with the origin in $V^+$. Furthermore, the subscheme $T_n \subset S^+$ coincides with the closed subscheme 
		\begin{align}
			x=0\ , \ w^n =0
		\end{align}
		in $S^+\cap V^+$.
	\end{itemize}
\end{lemma}

Let $\Hilb_p^k(C^+)$ be the punctual Hilbert scheme parametrizing length $k$ zero-dimensional subschemes of $C^+$ with set theoretic support $\{p\}$. Let $\Quot_{T_n, p}^k(C^+; s)$ be the relative Quot scheme over $\Hilb_p^k(C^+)$ parametrizing quotients $\calI\twoheadrightarrow \calQ$, with $\calQ$ a coherent $\scrO_{T_n}$-module of length $s$. Then, 
\begin{align}\label{eq:Quot-Flag}
	\Quot_{T_n, p}^k(C^+; s) \simeq \FHilb_{T_n}^k(C^+; s)\ .
\end{align}

Set 
\begin{align}
	Z_{r,s,n}(q) \coloneqq \sum_{k\geq 1} \chi(\Quot_{T_n, p}^k(C^+; s)) q^k \ .
\end{align}

An explicit formula for this generating function will be derived below by localization with respect to the torus action $\C^\times \times C^+\to C^+$, given by
\begin{align}
	\lambda \times (x,w) \longmapsto (\lambda^s x, \lambda^r w)\ .
\end{align}
By analogy to \cite[Lemma~17]{OS-HOMFLY}, the fixed points will be in one-to-one correspondence to partitions, or Young diagrams, satisfying certain conditions. In order to fix conventions, Young diagrams will be identified with finite subsets of $\N \times \N$ satisfying the conditions 
\begin{align}
	(i,j) \in \mu\ , \ i\geq 1 \ \Rightarrow \ (i-1,j)\in \mu
\end{align}
and 
\begin{align}
	(i,j) \in \mu\ , \ j\geq 1 \ \Rightarrow  \ (i, j-1) \in \mu\ ,
\end{align}
as shown below
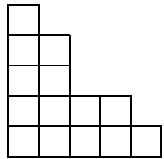
\begin{align}
	{\young(\hfil,\hfil\hfil,\hfil\hfil,\hfil\hfil\hfil\hfil,\hfil\hfil\hfil\hfil\hfil)}
\end{align}
By convention, the partition associated to such a diagram is defined by lengths $\mu_1 \geq \cdots \geq \mu_{\ell(\mu)}$ of its rows. The number of rows, $\ell(\mu)$, will be called the length of $\mu$ and the size $\vert \mu\vert$ of $\mu$ will be defined as its total number of boxes. For each $1\leq i \leq \ell(\mu)$, let $m_i$ denote the multiplicity of the $i$-th row of $\mu$. 

Set
\begin{align}
	\Quot_{T_n, p}(C^+; s)\coloneqq \bigsqcup_{k\geq 0} \Quot_{T_n, p}^k(C^+; s)\ .
\end{align}
Proceeding by analogy to \cite[Lemma~17]{OS-HOMFLY}, the following holds. 
\begin{lemma}\label{lem:plus-curve-A} 
	The set of fixed points in $\Quot_{T_n, p}(C^+; s)$ is isomorphic to the set $\sfP(r,s,n)$ of partitions satisfying the following conditions:
	\begin{enumerate}\itemsep0.2cm 
		\item \label{item:plus-curve-A-1} $\ell(\mu)=s$ and $m_i \leq n$ for $1\leq i \leq \ell(\mu)$, 
		\item \label{item:plus-curve-A-2} $\mu_1 - \mu_{\ell(\mu)} \leq r$, and 
		\item \label{item:plus-curve-A-3} if $\mu_1 - \mu_{\ell(\mu)}=r$, then $m_1 + m_{\ell(\mu)} \leq n$. 
	\end{enumerate} 
\end{lemma}

\begin{proof}
	As shown in \cite[Lemma~17]{OS-HOMFLY}, the set of fixed points is in one-to-one correspondence to Young diagrams $\mu$ satisfying conditions 
	\begin{itemize}\itemsep0.2cm
		\item $\ell(\mu)\leq s$, and 
		\item if $\ell(\mu) =s$, then $\mu_1-\mu_{\ell(\mu)}\leq r$.
	\end{itemize} 
	The zero dimensional subscheme $Z_\mu \subset C^+$ associated to the partition $\mu$ is determined by
	\begin{align}
		H^0(C^+, \scrO_{Z_\mu}) \simeq \C \langle x^{i-1} w^{j-1},\  (i,j) \in \mu\rangle \ ,
	\end{align}
	equipped with ring structure induced by the natural surjection 
	\begin{align}
		\C[x,w]/(x^r-w^s)\rightarrow H^0(C^+,\scrO_{Z_\mu})\ . 
	\end{align}
	In particular the kernel of the multiplication map $H^0(C^+, \scrO_{Z_\mu}) \xrightarrow{\cdot x} H^0(C^+,\scrO_{Z_\mu})$ is given by 
	\begin{align}
		K_\mu \simeq \C\langle x^{\mu_j-1} w^j,\  (i,j) \in \mu\rangle\ . 
	\end{align}
	The monomials spanning $K_\mu$ are in one-to-one correspondence to the set of boxes of $\mu$ marked with $\bullet$  below. 
	\begin{align}
		{\young(\bullet,\hfil\hfil\bullet,\hfil\hfil\bullet,\hfil\hfil\hfil\hfil\bullet,\hfil\hfil\hfil\hfil\bullet)}
	\end{align}
	The fixed points in the relative Quot scheme $\Quot_{T_n, p}(C^+; s)$ are in one-to-one correspondence to pairs $Z_{\nu}\subset Z_\mu$, so that 
	the kernel of the natural epimorphism $H^0(C^+,\scrO_{Z_\mu}) \to H^0(C^+, \scrO_{Z_\nu})$ is a length $s$ submodule of $K_\mu$ annihilated by $w^n$. Such pairs are in one-to-one correspondence to Young diagrams $\mu$ satisfying conditions \eqref{item:plus-curve-A-1} -- \eqref{item:plus-curve-A-3} in the statement. 
\end{proof}

\begin{corollary}
	We have
	\begin{align}\label{eq:combinatorial-formula-Z}
		Z_{r,s,n} = \sum_{\mu \in \sfP(r,s,n)} q^{\vert\mu\vert}\ .
	\end{align}
\end{corollary}

\begin{example}\label{ex:partitions-A}
	For illustration, some examples of partitions satisfying conditions \eqref{item:plus-curve-A-1} -- \eqref{item:plus-curve-A-3} in Lemma~\ref{lem:plus-curve-A} are shown below for $(r,s,n) = (7,5,2)$.
	\begin{align}\label{eq:a}
		&{\young(\hfil,\hfil\hfil\hfil,\hfil\hfil\hfil,\hfil\hfil\hfil\hfil\hfil,\hfil\hfil\hfil\hfil\hfil)}\\[2pt] \label{eq:b}
		&{\young(\hfil,\hfil\hfil\hfil,\hfil\hfil\hfil,\hfil\hfil\hfil\hfil\hfil,\hfil\hfil\hfil\hfil\hfil\hfil\hfil\hfil)} 
	\end{align}
	Note that in both cases one has $\mu'_1-\mu'_5\leq 7$ and all rows have multiplicity at most 
	$2$, as required by conditions \eqref{item:plus-curve-A-1} and \eqref{item:plus-curve-A-2}. Moreover, for case \eqref{eq:b} one also has 
	$m_1 + m_5 \leq 2$ as required by condition \eqref{item:plus-curve-A-3}, since $\mu'_1-\mu'_5=7$. This condition is not required in case \eqref{eq:a}, where $\mu'_1-\mu'_5=4<7$. 
\end{example} 

\begin{example}\label{partitionsB} 
	Consider 
	\begin{align}\label{eq:c}
		&{\young(\hfil,\hfil\hfil\hfil,\hfil\hfil\hfil,\hfil\hfil\hfil,\hfil\hfil\hfil\hfil\hfil)}\\[2pt] \label{eq:d}
		&{\young(\hfil,\hfil\hfil\hfil,\hfil\hfil\hfil,\hfil\hfil\hfil\hfil\hfil\hfil\hfil\hfil,\hfil\hfil\hfil\hfil\hfil\hfil\hfil\hfil)}
	\end{align}
	Cases \eqref{eq:c} and \eqref{eq:d} do not satisfy all required conditions in Lemma~\ref{lem:plus-curve-A} for $(r,s,n) = (7,5,2)$. For case \eqref{eq:c}, one has $m_2=m_3=m_4=3>2$, in contradiction to condition \eqref{item:plus-curve-A-2}. For case \eqref{eq:d}, one has $\mu'_1-\mu'_5=7$ while $m_1+m_7=3>2$, in contradiction 
	with condition \eqref{item:plus-curve-A-2}. 
\end{example}

Set $d(r,s,n)\coloneqq\min\{ \vert\mu\vert \, :\, \mu\in \sfP(r,s,n)\}$. For future reference note the following. 
\begin{lemma}\label{lem:minimal-size} 
	Assume that $s=nk$, where $k\in \Z$, with $1\leq k< r$. Then $d(r,s,n) = nk(k+1)/2$. 
\end{lemma} 

\begin{proof}
	It will be shown inductively, that under the current assumptions, the partition of minimal size is given by 
	\begin{align}\label{eq:minimal-size-A}
		\mu^{(k)}_{ni+j} =  k-i
	\end{align}
	for  $0\leq i \leq k-1$ and $1\leq j \leq n$. In other words $\mu^{(k)}$ has $n$ rows of length $i$ for any $1\leq i \leq k$. Indeed, in this case, $\vert \mu^{(k)}\vert= nk(k+1)/2$, which implies the claim. 
	
	Assume that $k=1$. By definition, any partition $\mu\in \sfP(r,s,n)$ has $s$ rows, i.e., $\mu_s\geq 1$. Then the partition of minimal size in $\sfP(r,n,n)$ is indeed given by
	\begin{align}
		\mu_1 = \cdots = \mu_n = 1\ .
	\end{align}
	Suppose that the claim is true for all $1\leq k \leq m$, for some $1\leq m \leq r-1$. Let $\mu\in \sfP(r,n(m+1),n)$ be the partition of minimal size. 
	By the induction hypothesis, in order for  $\mu^{(m+1)}$ to have minimal size allowed 
	by conditions \eqref{item:plus-curve-A-1} -- \eqref{item:plus-curve-A-3} in Lemma~\ref{lem:plus-curve-A}, the subpartition of $\mu^{(m+1)}$ determined by the rows $n+1$ through $n(m+1)$ coincides with $\mu^{(m)}$. Therefore one must have 
	\begin{align}
		\mu^{(m+1)}_{i+n} = \mu^{(m)}_{i} \ ,
	\end{align}
	for $ 1\leq i \leq nm$. Furthermore, condition \eqref{item:plus-curve-A-2} in Lemma~\ref{lem:plus-curve-A} implies that the first $n$ rows, must have at least $m+1$ boxes. Hence the minimal configuration is achieved for  
	\begin{align}
		\mu^{(\ell+1)}_1=\cdots = \mu^{(\ell+1)}_{n} = m+1\ . 
	\end{align}
	Note that condition \eqref{item:plus-curve-A-3} in Lemma~\ref{lem:plus-curve-A} is trivially satisfied since $\mu_1^{(m+1)} - \mu_{n(m+1)}^{(m+1)} =m \leq r-1$. 
\end{proof}

The main goal of this section is to derive an explicit formula for the generating function \eqref{eq:combinatorial-formula-Z}. As a first step, for any $t,m\geq 1$ let $\sfP(t,m)$ denote the set of partitions $\nu=(\nu_1\geq \cdots \geq \nu_{\ell(\nu)} >0)$ satisfying the following conditions 
\begin{itemize}\itemsep0.2cm
	\item $\nu_{1}\leq t$, 
	\item the multiplicity any part of $\nu$ is at most $m$.
\end{itemize}
Set
\begin{align}
	Z_{t,m}(q,x) \coloneqq \sum_{\nu \in \sfP(t,m)} x^{\ell(\nu)} q^{\vert\nu\vert}\ .
\end{align}
\begin{lemma}\label{lem:part-sum} 
	One has 
	\begin{align}\label{eq:part-sum-A} 
		Z_{t,m}(q,x)   =  \prod_{i=1}^t (1+xq^i + \cdots + (xq^i)^{m})\ .
	\end{align}
\end{lemma} 

\begin{proof}
	It is straightforward to check that there is a one-to-one correspondence between partitions in $\sfP(t,m)$ and monomials in the expansion of the right hand side of equation \eqref{eq:part-sum-A}. 
\end{proof}

\begin{theorem}\label{thm:plus-curve} 
	For any $k\geq 0$, let $Z_{t,n}^{(k)}(q)$ denote the coefficient of $x^k$ in $Z_{t,n}(q,x)$. Then 
	\begin{align}\label{eq:generating-function-A} 
		Z_{r,s,n}(q) =  \frac{1}{1-q^s}\sum_{m=0}^{n-1} \sum_{k=s-n}^{s-m-1} q^{mr} Z^{(k)}_{r-1,n}(q)\ .
	\end{align}
\end{theorem} 

\begin{proof}
	Note the decomposition
	\begin{align}
		\sfP(r,s,n) = \sfP_{\leq r-1}(r,s,n) \cup \ \sfP_r(r,s,n)\ ,
	\end{align}
	where $\sfP_{\leq r-1}(r,s,n)$ consists of partitions $\mu \in \sfP_r(r,s,n)$ with $0\leq \mu_1-\mu_s\leq r-1$, while $\sfP_{r}(r,s,n)$ consists of 
	partitions $\mu \in \sfP(r,s,n)$ with $\mu_1-\mu_s= r$. This yields a decomposition 
	\begin{align}
		Z_{r,s,n}(q) = Z_{\leq r-1}(q) + Z_{r}(q)\ .
	\end{align}
	Moreover, as shown below, any partition $\mu\in \sfP_{\leq r-1}(r,s,n)$ decomposes uniquely as a union of a rectangular Young diagram $\lambda$ consisting of the empty boxes, and a second Young diagram $\rho$ consisting of the boxes marked with $\bullet$. 
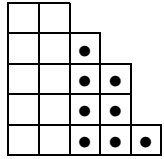
	\begin{align}
		\young(\hfil\hfil,\hfil\hfil\bullet,\hfil\hfil\bullet\bullet,\hfil\hfil\bullet\bullet,\hfil\hfil\bullet\bullet\bullet)
	\end{align}
	The resulting partitions $(\lambda, \rho)$ must satisfy the conditions below, which follow from properties 
	\eqref{item:plus-curve-A-1} -- \eqref{item:plus-curve-A-3} in Lemma~\ref{lem:plus-curve-A}:
	\begin{itemize}\itemsep0.2cm
		\item $\lambda$ is nonempty and $\ell(\lambda) =s$, 
		\item  $s-n \leq \ell(\rho)\leq s-1 $ and  $\rho_1\leq r-1$, and
		\item all rows of $\rho$ have multiplicity at most $n$.
	\end{itemize} 
	In particular,  $\rho \in \sfP(r-1,n)$. 
	
	Conversely, any pair $(\lambda, \rho)$ satisfying the three conditions listed above determines a unique partition $\mu \in \sfP_{\leq r-1}(r,s,n)$. This yields 
	\begin{align}\label{eq:partition-sum-C}
		Z_{\leq r-1}(q) = \frac{1}{1-q^s} \sum_{k=s-n}^{s-1} Z^{(k)}_{r-1,n}(q)\ .
	\end{align}
	
	Similarly, any Young diagram $\mu \in \sfP_r(r,s,n)$ decomposes uniquely into three Young diagrams $(\lambda, \nu,\rho)$ as shown below, 
	\begin{align}
		\young(\hfil\hfil,\hfil\hfil\bullet,\hfil\hfil\bullet\bullet,\hfil\hfil\bullet\bullet,\hfil\hfil\ast\ast\ast\ast,\hfil\hfil\ast\ast\ast\ast)
	\end{align}
	where the boxes belonging to $\lambda$ are empty, the boxes belonging to $\rho$ are marked with $\bullet$, and those belonging to $\nu$ are marked with $\ast$. Since $\mu\in \sfP_r(r,s,n)$, the following conditions must hold:
	\begin{itemize}\itemsep0.2cm
		\item $\lambda$ is nonempty rectangular partition with $\ell(\lambda)=s$, 
		\item $\nu$ is a nonempty rectangular partition with $m$ rows of length $r$ for some $1\leq m\leq s$, 
		\item  $s-n \leq \ell(\rho)\leq s-m-1 $ and  $\rho_1\leq r-1$, and
		\item all rows of $\rho$ have multiplicity at most $n$.
	\end{itemize} 
	In particular, again, $\rho\in \sfP(r-1,n)$. 
	
	Conversely, any triple $(\lambda, \nu, \rho)$ satisfying all the above conditions determines a unique diagram $\mu \in \sfP_r(r,s,n)$. This yields 
	\begin{align}\label{eq:partition-sum-D} 
		Z_{r}(q) = \frac{1}{1-q^s}\sum_{m=1}^{n-1} \sum_{k=s-n}^{s-m-1} q^{mr} Z^{(k)}_{r-1,n}(q)\ .
	\end{align}
\end{proof}

Summarizing, we have proved the following:
\begin{theorem}
	We have 
	\begin{align}
		\sum_{k\geq 0} \chi(\FHilb_{T_n, p}^k(C^+; s)) q^k =  \frac{1}{1-q^s}\sum_{m=0}^{n-1} \sum_{k=s-n}^{s-m-1} q^{mr} Z^{(k)}_{r-1,n}(q)\ . 
	\end{align}
\end{theorem} 

\subsection{The strict transforms in $Y^-$}\label{exsectB}

The goal of this section is to determine the strict transform $C^-\subset Y^-$ of the plane curve $C\subset W$ in equation \eqref{eq:plane-curve-A}. 

Let $(Y^-)^{\circ}\coloneqq (f^-)^{-1}(X^\circ)$ and note that it has an open cover $U^-\coloneqq (Y^-)^{\circ}\cap \{ w_1\neq 0\}$ and $V^-\coloneqq  (Y^-)^{\circ}\cap \{w_2\neq 0\}$, where $[w_1:w_2]$ are homogeneous coordinates on $\PP^1$ as in equations \eqref{eq:Y+-equations} and \eqref{eq:Y--equations}. Let $u= w_1^{-1}w_2$ and $v = w_2^{-1} w_1$ be the canonical affine coordinates on $\PP^1$. Then a simple computation shows that 
\begin{align}
	U^- \coloneqq \Spec\, \C[u,y,w]\quad \text{and} \quad V^- \coloneqq \Spec\, \C[v,x,w] \ ,
\end{align}
and $U^-\smallsetminus \{u=0\}$ is identified with $V^-\smallsetminus \{v=0\}$ via the transition functions 
\begin{align}\label{eq:trfct-2} 
	uv =1\ , \quad xv = w^n - yu\ .
\end{align}%
Let $\Gamma^-\coloneqq(f^-)^{-1}(C)$ be the scheme-theoretic image of $C$, and let $(C^-)^\circ \coloneqq C^- \cap (Y^-)^\circ$ be its strict transform. Since $C$ is irreducible, $\Gamma^-$ has two irreducible components, $C^-$ and $\Gamma_{\Sigma^-}$. By definition, the strict transform $C^-$ is also reduced, while $\Gamma_{\Sigma^-}$ is set-theoretically supported on $\Sigma^-$. Recall the multiplicity $m_{\Sigma^-}\in \Z$ defined by 
\begin{align}
	\ch_2(\scrO_{\Gamma_{\Sigma^-}}) = m_{\Sigma^-}\, \ch_2(\scrO_{\Sigma^-})
\end{align}
as sheaves on $Y^-$. Then, the following holds by local computations.
\begin{lemma}\label{lem:minus-curve-C} 
	For any coprime pair $r,s\geq 2$, one has  $m_{\Sigma^-} =s$. Furthermore, $(C^-\cap\Sigma^-)_{\mathsf{red}}$ consists of a single point $q$, which coincides with the origin in $V^-$. 
\end{lemma}

\begin{proof}
	The scheme-theoretic intersection $\Gamma^-\cap U^-$ is given by 
	\begin{align}
		y=0\ , \quad w^s(4^r u^r w^{nr-s} -1) =0 \ .
	\end{align}
	This is a plane curve consisting of two irreducible components,
	\begin{align}
		y=0\ , \quad w^s =0 \ ,
	\end{align}
	and 
	\begin{align}
		y=0\ , \quad 4^r u^{r-s} w^{nr-s} -1 = 0\ .
	\end{align}
	Moreover, $\Sigma^-\cap U^-$ is determined by 
	\begin{align}
		y=0\ , \quad w=0\ .
	\end{align}
	Hence the first component of $\Gamma^-\cap V^-$ is a length $s$ thickening of $\Sigma^-\cap U^-$, while the second is a smooth plane curve disjoint from $\Sigma^-\cap U^-$. This identifies the first component as $\Gamma_{\Sigma^-}\cap U^-$ and the second as $C^-\cap U^-$. In particular, $m_{\Sigma^-}=s$.
	
	In order to prove the second claim, note that the defining equations of $\Gamma^-\cap V^-$ in $V^-$ are 
	\begin{align}
		xv - w^n =0\ , \quad x^r -w^s =0\ . 
	\end{align}
	By restriction to the open subscheme $v\neq 0$, these equations are equivalent to 
	\begin{align}
		x - v^{-1}w^n=0\ , \quad w^{s}( v^r - w^{nr-s}) =0 \ .
	\end{align}
	This is a plane curve with two irreducible components
	\begin{align}
		x - v^{-1}w^n=0\ , \quad w^{s}=0 \ ,
	\end{align}
	and 
	\begin{align}
		x - v^{-1}w^n=0\ , \quad  v^r - w^{nr-s} =0\ . 
	\end{align}
	The first is a length $s$ thickening of $\Sigma^-\cap \{v\neq 0\}$, while the second is a smooth plane curve disjoint from $\Sigma^-$. Therefore $C^-\cap V^-$ is the scheme theoretic closure of the latter in $V^-$, i.e., the smallest closed subscheme of $V^-$ containing the second component as an open subscheme. Given the above equations, it is clear that $C^-$ contains the origin, which is the unique common point of $C^-$ and $\Sigma^-$. 
\end{proof}

\subsection{Euler numbers of Hilbert schemes of points of curves with local complete intersections singularities}\label{subsec:explicit-example-space-curve}

Corollary~\ref{cor:C-framed-flop-C}, Formula~\eqref{eq:Quot-Flag}, and Theorem~\ref{thm:plus-curve} yield the following:
\begin{corollary}
The following identity holds:
	\begin{align}
		q^{\chi(\scrO_{C^-})} \sum_{k\in \N} \chi(\Hilb_{q}^k(C^-)) = q^{\chi(\scrO_{C_+})}Z_{r,s,n}(q) \ .
	\end{align}
\end{corollary}

It is technically difficult to find an explicit presentation for the defining ideal of $C^-$ in $Y^-$ for general values of $(r,s,n)$. An explicit presentation will be determined in the following, subject to the assumption that $s$ is a multiple $n$. Moreover, in this case $C^-$ will be shown to be a local complete intersection in $Y^-$. 

\begin{lemma}\label{lem:minus-curve-A} 
	Let $a>b\geq 2$ be coprime integers so that $a+b,n$ are also coprime. Let $I = (xv-w^n, x^a-v^b)\subset \C[v,x,w]$. Then $I$ is a prime ideal. 
\end{lemma} 

\begin{proof}
	Let $f\colon \C[v,x,w]\to \C[t,w]$ be the ring homomorphism
	\begin{align}
		f(v) = t^b\ , \quad f(x)=t^a\ , \quad f(w) = w\ .
	\end{align}
	Let $J \subset \C[t,w]$ be the principal ideal $(t^{a+b}-w^n)$, which is prime under the current assumptions. Moreover, $I \subset f^{-1}(J)$. Hence it suffices to prove that $I = f^{-1}(J)$. 
	
	Suppose a polynomial $P\in \C[v,x,w]$ belongs to $f^{-1}(J)$, i.e., $P(t^b,t^a,w)= (t^{a+b}-w^n) Q(t,w)$ for some $Q(t,w)\in \C[t,w]$. Note that $P$ can be uniquely written as 
	\begin{align}
		P(v,x,w) = P_1(v,x,w)+  (xv-w^n) P_2(x,v,w)
	\end{align}
	with 
	\begin{align}
		P_1(v,x,w) = \sum_{k=0}^{n-1} P_{1,k}(v,x) w^k\ .
	\end{align}
	Therefore, $P_1(t^b,t^a,w)$ is a multiple of $t^{a+b}-w^n$, which holds if and only if it vanishes identically. Hence $P_{1,k}(t^b,t^a)$ vanishes identically for all $0\leq k \leq n-1$. Since $a,b$ are assumed coprime, the chinese remainder theorem implies that $P_1(x,v)$ is a multiple of $v^b-x^a$ as in the proof of \cite[Lemma~17]{OS-HOMFLY}. In conclusion, $P\in I$ as claimed above. 
\end{proof}

\begin{lemma}\label{lem:minus-curve-B}
	Assume that $s=ns'$, where $s'\in \Z$, with $s'\geq 1$. Then the defining equations of the strict transform $C^-$ in the affine chart $V^-$ are 
	\begin{align}\label{eq:minus-curve-A}
		xv - w^n =0\ , \quad x^{r-s'}-v^{s'} =0 \ .
\end{align}
\end{lemma} 

\begin{proof}
	The defining ideal of $\Gamma^-$ in $V^-\coloneqq \Spec\, \C[v,x,w]$ is $I \coloneqq (vx - w^n, \ x^r - w^{ns'})$ or, equivalently, $I = (vx - w^n,\ x^{s'} ( x^{r-s'} - v^{s'} ))$. We shall show that $I = I_0\cap I_1$, where 
	\begin{align}
		I_0\coloneqq (vx-w^n,\ x^{s'})\ , \quad I_1\coloneqq (vx-w^n, \ x^{r-s'} - v^{s'})\ .
	\end{align}
	Clearly, $I\subset I_0\cap I_1$, hence it suffices to prove the inverse inclusion. An element 
	\begin{align}
		(vx-w^n) A+ x^{s'} B \in I_0 
	\end{align}
	belongs to $I_1$ if and only if it can be written as 
	\begin{align}
		(vx-w^n) A+ x^{s'} B = (vx-w^n) A'+ (x^{r-s'}-v^{s'}) B'\ . 
	\end{align}
	Note that $B,B'$ can be uniquely written as 
	\begin{align}
		B \coloneqq B_1+ (xv-w^n) B_2\ , \quad B' \coloneqq B_1'+(xv-w^n)B_2'
	\end{align}
	with $\deg_w B_1, \deg_w B_1'\leq n-1$. Then it follows that 
	\begin{align}
		x^{s'}B_1 -  (x^{r-s'}-v^{s'}) B'_1 
	\end{align}
	is a multiple of $(vx-w^n)$. Therefore it must vanish identically, which proves the claim.
	
	In conclusion, $\Gamma_{V^-}$ is the scheme-theoretic union of the closed subschemes $\Gamma^0$ and $\Gamma^1$. Note that the first is set-theoretically supported on the closed subscheme given by $x=0$, $w=0$, i.e., $\Sigma^-\cap V^-$. Moreover, the second is reduced irreducible by Lemma~\ref{lem:minus-curve-A} and the only point in $(\Gamma_0 \cap \Gamma_0)_{\mathsf{red}}$ is the origin. This identifies $I_0, I_1$ as the defining ideals of $Z_{\Sigma^-}$ and $C^-$ in $V^-$, respectively. 
\end{proof}

The previous two lemmas imply the following.
\begin{proposition}
	For any integers $r,t,n\geq 2$ with $r> t$ and $(r,nt)$ coprime, let $C_{r,t,n}$ be the curve defined by 
	\begin{align}\label{eq:equations-C}
		\begin{cases}
			xv - w^n =0\ , \\ 
			x^{r-t}- v^{t} =0 \ , 
		\end{cases}
	\end{align}
	in $\C^3=\Spec\, \C[v,x,w]$. Then $C_{r,t,n}$ is local complete intersection in $\C^3$. Its scheme theoretic closure in $Y^-$ is also a local complete intersection, and coincides with the strict transform $C^-$.
\end{proposition}

In order to conclude this section, we the following explicit formula for the generating function of Euler numbers of punctual Hilbert schemes of points of $C_{r,t,n}$.
\begin{theorem}\label{thm:lci-curves} 
	For any integers $r,t,n\geq 2$ with $r> t$ and $(r,nt)$ coprime, let $C_{r,t,n}\subset \C^3$ be the curve defined by the equations \eqref{eq:equations-C}. Let $o$ be the origin in $\C^3$. Then 
	\begin{align}\label{eq:generating-function-B} 
		\sum_{k\in \N} \chi(\Hilb_o^k(C_{r,t,n})) q^k = q^{-nt(t-1)/2} Z_{r,tn,n}(q)\ . 
	\end{align}
\end{theorem} 

\begin{proof}
	Corollary~\ref{cor:C-framed-flop-C} and Theorem~\ref{thm:plus-curve} yield the identity 
	\begin{align}
		\sum_{k\in \N} \chi(\Hilb_o^k(C_{r,t,n}))) = q^{\chi(\scrO_{C^+})-\chi(\scrO_{C_{r,t,n}})}Z_{r,s,n}(q) \ ,
	\end{align}
	where $s=nt$. Note that the lowest degree term in the right hand side of the above identity is $1$. By Lemma~\ref{lem:minimal-size}, the lowest degree term of $Z_{r,tn,n}(q)$ is $q^{d(r,s,n)-s}$, where 	 $d(r,s,n)\coloneqq \min\{ \vert \mu\vert  \, :\, \mu\in \sfP(r,s,n)\}$. Therefore it follows that $\chi(\scrO_{C^+})-\chi(\scrO_{C_{r,t,n}})+d(r,s,n)-s=0$ and the identity claimed in Theorem \ref{thm:lci-curves} holds. 
\end{proof}

\appendix

\section{Proof of Theorem~\ref{thm:filtration}}\label{sec:theorem-filtration}

The main goal of this section is to prove the following:
\begin{theorem}\label{thm:filtration-Appendix}
	Under Assumption~\ref{assumption:f}, there exists a unique chain of subschemes 
	\begin{align}
		\Sigma = \Sigma_1 \subset \cdots \cdots \subset \Sigma_n \ ,
	\end{align}
	with $n\geq 2$, which determines a filtration of the form
	\begin{align}\label{eq:ideal-filtration-A-Appendix}
		0=\calJ_{n+1}\subset \calJ_n\subset \cdots \subset \calJ_1 = \calI_\Sigma 
	\end{align}
	so that
	\begin{align}\label{eq:ideal-filtration-B-Appendix}
		\begin{split}
			\calI_\Sigma \calJ_i \subset \calJ_{i+1}\subset \calJ_i\ , \quad \calJ_i/\calJ_{i+1} \simeq i_{\Sigma,\, \ast}\scrO_{\Sigma}\ , \\[4pt]
			\calJ_i/\calI_\Sigma \calJ_i\simeq  i_{\Sigma,\, \ast}\scrO_{\Sigma}\oplus i_{\Sigma,\, \ast}\scrO_{\Sigma}(2)\ , \quad \calJ_{i+1}/\calI_\Sigma \calJ_i\simeq  i_{\Sigma,\, \ast}\scrO_{\Sigma}(2)\ ,
		\end{split}
	\end{align}	
	for $1\leq i \leq n-1$, and 
	\begin{align}\label{eq:ideal-filtration-C-Appendix}
		\calJ_n /\calI_\Sigma \calJ_n \simeq i_{\Sigma,\, \ast}\scrO_{\Sigma}(1)^{\oplus 2}\ .
	\end{align}
\end{theorem}

The proof of Theorem~\ref{thm:filtration-Appendix} relies on the deformation theory of the $\Sigma\subset Y$ as a closed subscheme of $Y$, i.e., on the structure of the component of the Hilbert scheme of curves on $Y$ containing the point $[\scrO_Y\to \Sigma]$. The deformation theory of the structure sheaf $\scrO_\Sigma$ as a coherent sheaf on $Y$ was studied in detail in \cite{Spherical_twists} and, as special case of a more 
general setting, in \cite{DW16}. We summarize some of their results below for a flop diagram 
\begin{align}
	\begin{tikzcd}[ampersand replacement=\&]
		Y^+ \ar[swap]{dr}{f^+} \ar[dashed]{rr}{\phi}\& \& Y^- \ar{dl}{f^-}\\
		\& X \&
	\end{tikzcd}\ ,
\end{align}
satisfying Assumption~\ref{assumption:f}. In particular $X$ has a unique singular point $\nu$ so that the scheme theoretic inverse image $\Sigma^\pm=(f^\pm)^{-1}(\nu)$ is a $(0,-2)$ curve on $Y^\pm$. Let $\calM(Y^\pm, \Sigma^\pm)$ denote the connected component of the algebraic moduli space of simple sheaves on $Y^\pm$ containing the point $[\scrO_{\Sigma^\pm}]$. Then \cite[Corollary~3.3 and Example~3.10]{DW16} yield the following. 
\begin{theorem}\label{thm:moduli-sigma} 
	Under Assumption \ref{assumption:f}, given a diagram~\eqref{eq:flop-diagram-A}, there exists a unique integer $d \in \Z$, $d \geq 2$, so that the algebraic spaces $\calM(Y^\pm, \Sigma^\pm)$ are simultaneously isomorphic to $\Spec\ \C[u]/(u^d)$.  
\end{theorem} 

\begin{remark}
	By \cite[Theorem~3.15]{DW16}, the integer $d$ also coincides with the width of the curves $\Sigma^\pm \subset Y^\pm$ introduced in \cite[Definition~5.3]{Reid-Minimal-Models}. Moreover, as noted in \cite[Example~3.10]{DW16}, the formal completion of $X$ at the singular point is isomorphic to the formal completion of the singular hypersurface $xy + z^2 - w^{2d}=0$ at the origin.   
\end{remark} 

In the process of proving Theorem~\ref{thm:filtration}, we will also show that the moduli space $\calM(Y^\pm,\Sigma^\pm)$ is isomorphic to the connected component of the Hilbert scheme of curves on $Y^\pm$ containing the point $[\Sigma^\pm \hookrightarrow Y^\pm]$. From now on we will drop the labels $\pm$, working only on one side of the transition.

Let $\calH$ denote the connected component of the Hilbert scheme of curves on $Y$ containing the point $[\Sigma \hookrightarrow Y]$. Recall that under Assumption~\ref{assumption:f}, the curve $\Sigma$ is rigid by Proposition~\ref{prop:rigid-curve-B}. Our first goal is to show that $\calH \simeq \Spec\, \C[t]/(t^n)$ for some $n\geq 2$. First note: 
\begin{lemma}\label{lem:artin-ring} 
	Let $R$ be a local Artinian ring over $\C$ so that $\dim_\C (m/m^2)=1$, where $m \subset R$ is the maximal ideal. Then $R \simeq \C[t]/(t^n)$ for some $n \geq 2$. 
\end{lemma}

\begin{proof}
	By Nakayama's lemma, $m$ has a single generator $t$ as an $R$-module. Since $R$ is Artinian, and $m/m^2$ is one dimensional, there exists an integer $n \geq 2$ so that $t^{n-1}\neq 0$ and $t^n=0$. 
\end{proof}

\begin{lemma}\label{lem:ext-curve} 
	We have $\calH\simeq \Spec\, \C[t]/(t^n)$ for some $n\geq 2$.
\end{lemma} 

\begin{proof}
	Since $\Sigma$ is rigid, $\calH \simeq \Spec(R)$, with $R$ a local Artinian ring over $\C$. The Zariski tangent space to $\calH$ ay $[\scrO_\Sigma]$ is isomorphic to $\Hom_Y(\calI_\Sigma, \scrO_\Sigma)$, where $\calI_\Sigma \subset \scrO_Y$ is the defining ideal of $\Sigma$. 
	Moreover,
	\begin{align}
		\Hom_Y(\calI_\Sigma, \scrO_\Sigma)\simeq \Hom_Y(\calI_\Sigma/\calI_\Sigma^2, \scrO_\Sigma) \simeq 
		\Hom_{\Sigma^-}(\calN_{\Sigma/Y}^\vee, \scrO_\Sigma) \ ,
	\end{align}
	where $\calN_{\Sigma/Y}\simeq \scrO_\Sigma \oplus \scrO_\Sigma(-2)$ is the normal bundle to $\Sigma$ in $Y$. Therefore
	\begin{align}
		\Hom_Y(\calI_\Sigma, \scrO_\Sigma) \simeq H^0(\Sigma,\scrO_\Sigma) \simeq \C\ . 
	\end{align}
	Then the claim follows from Lemma~\ref{lem:artin-ring}.
\end{proof}

Let $n$ be the integer determined in Lemma~\ref{lem:ext-curve}. By analogy to \cite[Theorem~3.1]{Spherical_twists}, for any $k\in \Z$, with $k\geq 1$, let $R_k \coloneqq \C[t]/(t^k)$ and note the canonical exact sequences 
\begin{align}\label{eq:ring-sequence-A}
	0\longrightarrow R_1 \longrightarrow R_k \longrightarrow R_{k-1} \longrightarrow 0 \ , \\ \label{eq:ring-sequence-B}
	0\longrightarrow R_{k-1} \longrightarrow R_k \longrightarrow R_1 \longrightarrow 0\ .
\end{align}
Let $\calH_k \subset \calH$ be the closed subscheme defined by the projection $R_n \twoheadrightarrow R_k$. Let $Y_k \coloneqq Y\times \calH_k$
and let $\pi_k\colon Y_k \to Y$ be the canonical projection. We have $\calH_1\subset \calH_2\subset \cdots \subset \calH_n= \calH$ and $Y=Y_1\subset Y_2\subset \cdots \subset Y_n$. We denote the corresponding closed embeddings by $i_{k-1, k}$ and $j_{k-1, k}$, for $k=2, \ldots, n$, respectively.

Let $\calZ_k\subset Y_k=Y \times \calH_k$ be the universal closed subscheme, let $\kappa_k$ be the corresponding closed embedding, and let $\scrJ_k\subset \scrO_{Y_k}$ be its defining ideal. Note that $\kappa_{k,\, \ast}\scrO_{\calZ_k}$ and $\scrJ_k$ are flat over $\calH_k$. Set $\bfE_k \coloneqq \kappa_{k, \, \ast}\scrO_{\calZ_k}$. Note that $\bfE_n\vert_{Y_1}=i_{\Sigma,\, \ast}\scrO_\Sigma$. 

For any $1\leq k\leq n$, set $\calE_k \coloneqq \pi_{k,\,\ast}\bfE_k$.  In particular, $\calE_1 = i_{\Sigma,\,\ast}\scrO_\Sigma$. Then one has the exact sequences of $\scrO_Y$-modules
\begin{align}\label{eq:sh-sequence-A} 
0\longrightarrow \calE_1 \longrightarrow \calE_k \longrightarrow \calE_{k-1} \longrightarrow 0\ ,\\	\label{eq:sh-sequence-B} 
0\longrightarrow \calE_{k-1} \longrightarrow \calE_k \longrightarrow \calE_1 \longrightarrow 0\ .
\end{align}
As in the proof of \cite[Theorem~3.1]{Spherical_twists}, the infinitesimal deformation theory of the family of sheaves $\bfE_k$ over $\calH_k$ is encoded in the long exact sequence 
\begin{align}\label{eq:long-ext-sequence}
	\begin{tikzcd}[ampersand replacement=\&, row sep=tiny]
		\cdots \ar{r} \&  \Hom_Y(\calE_k, \calE_1) \ar{r} \&  \Hom_Y(\calE_{k-1}, \calE_1) \ar{r} \&  \Ext_Y^1(\calE_1, \calE_1) \& \\
		\ar{r} \&  \Ext_Y^1(\calE_k, \calE_1) \ar{r}{\xi_k} \&   \Ext^1_Y(\calE_{k-1}, \calE_1) \ar{r}{\delta_k} \& 
		\Ext^2_Y(\calE_1, \calE_1) \ar{r} \&  \cdots 
	\end{tikzcd}
\end{align}
derived from \eqref{eq:sh-sequence-A}. Let $e_k \in \Ext^1_Y(\calE_{k-1}, \calE_1)$, for $2\leq k \leq n$, be the extension class associated to the sequence \eqref{eq:sh-sequence-B}. Then the following holds by \cite[Proposition~3.13]{Hol_Casson}:
\begin{lemma}\label{lem:A} 
	For any $2\leq k \leq n-1$, one has $e_k = \xi_k(e_{k+1})$, hence also $\delta_k(e_k) =0$. 
\end{lemma} 

Now, we shall show that $e_k \neq 0$ for all $2\leq k \leq n$.
\begin{lemma}\label{lem:B}
	The extension class in $e_2 \in \Ext^1_Y(i_{\Sigma,\,\ast}\scrO_\Sigma, i_{\Sigma,\,\ast}\scrO_\Sigma)$ associated to the short exact sequence \eqref{eq:sh-sequence-A} for $k=2$ is nonzero. 
\end{lemma} 

\begin{proof}
	Note the commutative diagram of $\scrO_{Y_2}$-modules 
	\begin{align}\label{eq:def-diagram-A}
		\begin{tikzcd}[ampersand replacement=\&]
			0\ar[r] \& (j_{1, 2})_\ast\scrO_{Y_1} \ar[r] \ar[d]\& \scrO_{Y_2} \ar[r]\ar[d]\& (j_{1, 2})_\ast\scrO_{Y_1} \ar[r] \ar[d]\& 0 \\
			0\ar[r] \& (j_{1, 2})_\ast \kappa_{1,\, \ast}\scrO_{\calZ_1} \ar[r] \& \kappa_{2,\, \ast}\scrO_{\calZ_2} \ar[r] \& (j_{1, 2})_\ast \kappa_{1,\, \ast}\scrO_{\calZ_1}\ar[r] \& 0
		\end{tikzcd}\ ,
	\end{align}
	where the rows are exact and the vertical arrows are the canonical epimorphisms. Applying $\pi_{2\ast}$, i.e., forgetting the $R_2$-module structure, one obtains an analogous diagram 
	\begin{align}\label{eq:def-diagram-B}
		\begin{tikzcd}[ampersand replacement=\&]
			0\ar[r] \& \scrO_Y \ar{r}{\imath} \ar{d}{f}\& \scrO_Y \oplus \scrO_Y \ar{r}{q} \ar{d}{g} \ar[bend right=50, swap]{l}{p} \& \scrO_Y \ar[r] \ar[d]\& 0 \\
			0\ar[r] \& i_{\Sigma, \, \ast}\scrO_\Sigma  \ar{r}{\jmath}\& \calE_2 \ar[r] \& i_{\Sigma, \, \ast}\scrO_\Sigma \ar[r] \& 0
		\end{tikzcd}
	\end{align}
	in $\catCoh(Y)$, where the maps $\imath, q$ are canonical, i.e.,
	\begin{align}
		\imath\coloneqq \left(\begin{matrix}
			\id\\ 0
		\end{matrix}\right)\quad\text{and}\quad q\coloneqq  \left( \begin{matrix}
			\id & 0
		\end{matrix}\right)\ . 
	\end{align}
	Moreover, one has the canonical splitting 
	\begin{align}	
		p \coloneqq  \left( \begin{matrix}
			\id & 0
		\end{matrix}\right)\ . 
	\end{align}
	Note that the $R_2$-module structure is encoded in the endomorphisms $\phi\colon \scrO_Y\oplus \scrO_Y\to \scrO_Y\oplus \scrO_Y$,  
	\begin{align}
		\phi_2\coloneqq \left(\begin{matrix}
			0 & \id \\ 0 & 0 
		\end{matrix}\right)\ ,
	\end{align}
	and $\psi_2\colon \calE_2 \to \calE_2$ satisfying 
	\begin{align}
		\psi_2\circ g = g \circ \phi_2\ , \quad \psi_2^2 =0\ ,\\
		\phi_2 \circ \imath =0\ , \quad \psi_2 \circ \jmath=0\ .
	\end{align}
	
	Suppose the bottom row of the commutative diagram \eqref{eq:def-diagram-B} is a trivial extension, i.e., there exists an isomorphism $\xi\colon \calE_2 \to i_{\Sigma, \, \ast}\scrO_\Sigma \oplus i_{\Sigma, \, \ast}\scrO_\Sigma$. Note that 
	\begin{align}
		\End_Y(i_{\Sigma,\,\ast}\scrO_\Sigma, i_{\Sigma,\,\ast}\scrO_\Sigma)\simeq \mathsf{End}_\C(\C^2),\qquad		
		\Aut_Y(i_{\Sigma,\,\ast}\scrO_\Sigma, i_{\Sigma,\,\ast}\scrO_\Sigma)\simeq \mathsf{GL}(2, \C)\ , 
	\end{align}
	since $\Hom_Y(i_{\Sigma,\,\ast}\scrO_\Sigma, i_{\Sigma,\,\ast}\scrO_\Sigma)\simeq \C$. Then there exists an isomorphism $\eta\colon i_{\Sigma, \, \ast}\scrO_\Sigma\oplus i_{\Sigma, \, \ast}\scrO_\Sigma\to i_{\Sigma, \, \ast}\scrO_\Sigma\oplus i_{\Sigma, \, \ast}\scrO_\Sigma$ so that 
	\begin{align}
		\eta\circ(\xi\circ  \psi_2\circ \xi^{-1})\circ \eta^{-1} =  \left(\begin{matrix} 0 & \id \\ 0 & 0 \end{matrix}\right)\ .
	\end{align}
	Moreover, since $\psi_2\circ\jmath=0$, one also has 
	\begin{align}
		\eta\circ(\xi \circ \psi_2\circ \xi^{-1}) \circ \eta^{-1}\circ (\eta \circ \xi\circ \jmath) =  \eta\circ \xi\circ \psi_2\circ\jmath =0 \ .
	\end{align}
	Hence 
	\begin{align}
		\eta\circ \xi\circ\jmath= \left(\begin{matrix} \lambda \\ 0\end{matrix}\right)
	\end{align}
	for some $\lambda \in \C$, with $\lambda \neq 0$. Then one obtains a commutative diagram 
	\begin{align}\label{eq:def-diagram-C}
		\begin{tikzcd}[ampersand replacement=\&]
			0\ar[r] \& \scrO_Y \ar{r}{\imath} \ar{d}{f'}\& \scrO_Y \oplus \scrO_Y \ar{r}{q}\ar{d}{g'} \ar[bend right=50, swap]{l}{p}\& 
			\scrO_Y \ar[r] \ar[d]\& 0 \\
			0\ar[r] \& i_{\Sigma, \, \ast}\scrO_\Sigma  \ar{r}{\jmath'}\& i_{\Sigma, \, \ast}\scrO_\Sigma \oplus i_{\Sigma, \, \ast}\scrO_\Sigma \ar[r] \ar[bend right=50, swap]{l}{s}\& i_{\Sigma, \, \ast}\scrO_\Sigma \ar[r] \& 0 
		\end{tikzcd}\ ,
	\end{align}
	where 
	\begin{align}
		\jmath' = \lambda^{-1} (\eta\circ \xi \circ \jmath)\ , \quad g' = \eta\circ \xi\circ g\ , \quad f' =\lambda f \ ,
	\end{align}
	and 
	\begin{align}
		s =  \left(\begin{matrix}
			\id & 0
		\end{matrix}\right)\ . 
	\end{align}
	In particular, $s\circ g' = f'\circ p$. At the same time the $R_2$-module structure on $i_{\Sigma, \, \ast}\scrO_\Sigma \oplus i_{\Sigma, \, \ast}\scrO_\Sigma$ is given by the Jordan block $\psi_2'\coloneqq \eta\circ(\xi\circ \psi_2\circ \xi^{-1})\circ \eta^{-1}$. This implies that the canonical quotient 
	\begin{align}
		\scrO_{\calZ_2}\twoheadrightarrow (\kappa_{1,2})_\ast\scrO_{\calZ_1} \ ,
	\end{align}
	where $\kappa_{1,2}$ denotes the inclusion of $\calZ_1$ in $\calZ_2$, is isomorphic to the pull-back of the canonical quotient 
	\begin{align}
		\scrO_Y \twoheadrightarrow i_{\Sigma, \, \ast}\scrO_\Sigma 
	\end{align}
	via the canonical projection $Y_2 \to Y_1$. Then, the universality property of the Hilbert scheme implies that the canonical closed immersion $\calH_2 \to \calH_n$ factors through the closed immersion $\calH_1 \to \calH_n$. Hence it leads to a contradiction. 
\end{proof}

Lemmas~\ref{lem:A} and \ref{lem:B} imply the following result.
\begin{corollary}\label{cor:A} 
	For any $1\leq k \leq n-1$, one has $e_k \neq 0$. 
\end{corollary} 

Next, we shall show that $\delta_n (e_n) \neq 0$, i.e., the $\calH_n$-flat deformation $\bfE_n$ does not admit any further extensions. 
\begin{lemma}\label{lem:C} 
	Let $R_{n+1}\coloneqq \C[t]/(t^{n+1})$, $\calH_{n+1} \coloneqq \Spec\, R_{n+1}$ and $Y_{n+1}\coloneqq Y \times \calH_{n+1}$. Assume that there exists a coherent sheaf $\bfE$ on $Y_{n+1}$, flat over $\calH_{n+1}$, which fits in an exact sequence of $\scrO_{Y_{n+1}}$-modules 
	\begin{align}\label{eq:sh-sequence-D}
		0\longrightarrow (j_{1, n+1})_\ast i_{\Sigma, \, \ast}\scrO_\Sigma \longrightarrow \bfE \longrightarrow (j_{n, n+1})_\ast \kappa_{n,\, \ast}\scrO_{\calZ_n}\longrightarrow 0 \ .
	\end{align}
	Then, there exists a surjective morphism $\scrO_{Y_{n+1}} \to \bfE$ which fits in a commutative diagram 
	\begin{align}
		\begin{tikzcd}[ampersand replacement=\&]
			0\ar[r] \& (j_{1, n+1})_\ast\scrO_Y \ar[r] \ar[d] \& \scrO_{Y_{n+1}} \ar[r]\ar[d] \& (j_{n, n+1})_\ast\scrO_{Y_n} \ar[r] \ar[d]\& 0 \\
			0\ar[r] \&  (j_{1, n+1})_\ast i_{\Sigma, \, \ast}\scrO_{\Sigma}\ar[r] \& \bfE \ar[r]\& (j_{n, n+1})_\ast \kappa_{n, \, \ast}\scrO_{\calZ_n}\ar[r]\& 0 
		\end{tikzcd}
	\end{align}
	where the left and right vertical maps are the canonical surjections. In particular, $\bfE$ is isomorphic to the structure sheaf of an $\calH_{n+1}$-flat closed subscheme $\calZ_{n+1} \subset Y_{n+1}$.
\end{lemma} 

\begin{proof}
	First, note that $\calZ_n\subset Y_n$ is a closed subscheme of $Y_{n+1}$ through the canonical closed immersion $Y_n \to Y_{n+1}$.  Hence one has a natural surjective morphism $\scrO_{Y_{n+1}} \to (j_{n, n+1})_\ast \kappa_{n,\, \ast}\scrO_{\calZ_n}$ which factors through $\scrO_{Y_n}\to \kappa_{n,\, \ast} \scrO_{\calZ_n}$.
	
	The exact sequence \eqref{eq:sh-sequence-D} yields the exact sequence 
	\begin{align}
		0\longrightarrow H^0(\Sigma, \scrO_\Sigma)  \longrightarrow H^0(Y_{n+1}, \bfE) \longrightarrow H^0(\calZ_n, \scrO_{\calZ_n}) \longrightarrow H^1(\Sigma,\scrO_\Sigma) \longrightarrow 0\ , 
	\end{align}
	where $H^1(\Sigma,\scrO_\Sigma) =0$. Hence, the map $H^0(Y_{n+1},\bfE)\to H^0(\calZ_n, \scrO_{\calZ_n})$ is surjective. Therefore, pick a lift $\scrO_{Y_{n+1}} \to \bfE$ corresponding to $\scrO_{Y_{n+1}} \to (j_{n, n+1})_\ast \kappa_{n,\, \ast}\scrO_{\calZ_n}$.
	
	One obtains a commutative diagram 
	\begin{align}\label{eq:ext-diagram}
		\begin{tikzcd}[ampersand replacement=\&]
			0\ar[r] \& (j_{1, n+1})_\ast\scrO_Y  \ar[r] \ar{d}{f}\& \scrO_{Y_{n+1}} \ar[r]\ar{d}{g}\& (j_{n, n+1})_\ast\scrO_{Y_n} \ar[r] \ar[d]\& 0 \\
			0\ar[r] \& (j_{1, n+1})_\ast i_{\Sigma,\, \ast}\scrO_{\Sigma} \ar[r] \& \bfE \ar[r]\& (j_{n, n+1})_\ast\kappa_{n, \, \ast}\scrO_{\calZ_n}\ar[r]\& 0
		\end{tikzcd}\ ,
	\end{align}
	where the right vertical map is the canonical surjection. Let $\bfQ \coloneqq \mathsf{Coker}(g)$. The snake lemma yields an exact sequence 
	\begin{align}\label{eq:coker-sequence}
		\begin{tikzcd}[ampersand replacement=\&]
			\cdots \ar{r} \&  \mathsf{Coker}(f) \ar{r}{p} \& \mathsf{Coker}(g) \ar{r} \& 0 
		\end{tikzcd}\ .
	\end{align}
	Since $H^0(\Sigma, \scrO_\Sigma) \simeq \C$, the left vertical arrow $f$ can be either surjective or identically zero. If it is surjective, the exact sequence \eqref{eq:coker-sequence} shows that $\mathsf{Coker}(g)=0$, hence the claim follows. 
	
	Suppose that $f$ is identically zero. Then $\mathsf{Coker}(f) \simeq \scrO_\Sigma$, and the map $p$ in \eqref{eq:coker-sequence} is either identically zero or an isomorphism. If it is identically zero, one obtains again $\mathsf{Coker}(g)=0$ and the claim follows. 
	
	Suppose that $p$ is an isomorphism. Then the composition  
	\begin{align}
		\begin{tikzcd}[ampersand replacement=\&]
			\bfE\ar{r} \& \mathsf{Coker}(g) \ar{r}{p^{-1}} \& \scrO_\Sigma 
		\end{tikzcd}
	\end{align}
	is a splitting of the bottom row in diagram \eqref{eq:ext-diagram}. Hence $\bfE \simeq (j_{n, n+1})_\ast \kappa_{n,\, \ast}\scrO_{\calZ_n}\oplus 
	(j_{1, n+1})_\ast i_{\Sigma,\, \ast}\scrO_\Sigma$. This leads to a contradiction since $\bfE$ is flat over $\calH_{n+1}$ by assumption, while the direct summands violate the local criterion of flatness. 
\end{proof}
Using the universality property of the Hilbert scheme, Lemma \ref{lem:C} yields:
\begin{corollary}\label{cor:B} 
	One has $\delta_n(e_n) \neq 0$. 
\end{corollary} 

Using Lemma \ref{lem:trivial}, the same arguments as in the proof of \cite[Theorem~3.1]{Spherical_twists} yield the following:
\begin{corollary}\label{cor:C} 
	One has
	\begin{align}\label{eq:ext-vanishing}
		\Ext_Y^1(\calE, \calE_n) =0 \quad\text{and}\quad \Ext_Y^2(\calE_n, \calE) =0\ .
	\end{align}
\end{corollary}

Using Lemmas~\ref{lem:B}, \ref{lem:C}, and Corollaries~\ref{cor:A} and \ref{cor:B}, we now prove:
\begin{proposition}\label{prop:hilb-sigma} 
	Using the notation in Theorem~\ref{thm:moduli-sigma}, the natural morphism $\varphi\colon \calH\to \calM(Y, \Sigma)$ is an isomorphism. In particular, $d=n$. 
\end{proposition} 

\begin{proof} 
	For ease of notation set $\calM(Y,\Sigma) = \calM$. By Theorem~\ref{thm:moduli-sigma} and Lemma~\ref{lem:ext-curve}, we have $\calM\simeq \Spec(R)$ and $\calH\simeq \Spec(S)$ where 
	\begin{align}
		R \coloneqq \C[u]/(u^d)\quad\text{and} \quad S \coloneqq \C[t]/(t^n)
	\end{align}
	with $d,n\geq 2$. In particular $R,S$ are local Artinian rings over $\C$ with maximal ideals $m_R\coloneqq (u)$ and $m_S\coloneqq (t)$. Note that  
	\begin{align}
		\dim(m_R/m_R^2) = \dim(m_S/m_S^2) = 1\ . 
	\end{align}
	By Lemma~\ref{lem:B}, the induced map $\varphi_\ast \colon T_{m_S}\calH \to T_{m_R}(\calM)$ between Zariski tangent spaces is non-zero. Since both tangent spaces are one dimensional, it follows that $\varphi_\ast $ is an isomorphism. Then \cite[Tag~0E8M]{stacks-project} shows that the ring homomorphism $f\colon R\to S$ associated to $\varphi$ is surjective. Therefore $\varphi$ is a closed immersion. Moreover, since $\varphi_\ast $ is an isomorphism, one has $f(u) = t \xi$ for some unit $\xi \in S$. Since multiplication by $\xi^{-1}$ is an automorphism of $S$,  the composition 
	\begin{align}
		\begin{tikzcd}[ampersand replacement=\&]
			R \ar{r}{\varphi} \& S\ar{r}{\xi^{-1}} \& S
		\end{tikzcd}
	\end{align}
	is a surjective ring homomorphism mapping $u\mapsto t$. 

	In order to conclude the proof, note that $d\geq n$ since $f$ is surjective. Hence it suffices to prove that $d=n$. Suppose $d>n$. Let $\calH' \subset \calM$ be the closed subscheme determined by the ideal $(t^{n+1})\subset R $. By construction, one has a chain of closed immersions 
	\begin{align}
		\calH \longrightarrow \calH' \longrightarrow \calM 
	\end{align}
	where the morphism $\calH\to \calH'$ is induced by the composition of $\xi^{-1}$ with the canonical ring homomorphism $\C[t]/(t^{n+1}) \to \C[t]/(t^n)$. This leads to a contradiction via Lemma~\ref{lem:C} and Corollary~\ref{cor:B}. 
\end{proof} 

It will be shown next by induction that each $\calE_k$ is isomorphic to the structure sheaf of a scheme theoretic thickening $\Sigma\subset \Sigma_k \subset Y$. The starting point is the exact sequence 
\begin{align}
	0\longrightarrow i_{\Sigma,\, \ast}\scrO_\Sigma\longrightarrow \calE_2 \longrightarrow i_{\Sigma,\, \ast}\scrO_\Sigma \longrightarrow 0\ , 
\end{align}
which yields a long exact sequence 
\begin{align}
	0\longrightarrow \Hom_Y(\scrO_Y, i_{\Sigma,\, \ast}\scrO_\Sigma) \longrightarrow \Hom_Y(\scrO_Y, \calE_2)\longrightarrow \Hom_Y(\scrO_Y, i_{\Sigma,\, \ast}\scrO_\Sigma) \longrightarrow 0
\end{align}
since $H^1(\Sigma, \scrO_\Sigma)=0$. Pick a lift $f_2\colon\scrO_Y \to \calE_2$ of the canonical morphism $f_1\colon\scrO_Y \to i_{\Sigma,\, \ast}\scrO_\Sigma$. 
\begin{proposition}\label{prop:step-two} 
	The morphism $f_2$ is surjective, hence $\calE\simeq \scrO_{\Sigma_2}$ for a closed subscheme $\Sigma_2\subset Y$. Moreover, its defining ideal sheaf $\calJ_2\subset \scrO_Y$ fits into an exact sequence 
	\begin{align}\label{eq:ideal-flag-A}
		0\longrightarrow \calJ_2 \longrightarrow \calI_\Sigma \longrightarrow i_{\Sigma,\, \ast}\scrO_\Sigma \longrightarrow 0\ .
	\end{align}
\end{proposition}

\begin{proof}
	One has a commutative diagram 
	\begin{align}\label{eq:e-two-diagram}
		\begin{tikzcd}[ampersand replacement=\&]
			\& \& \scrO_Y \ar{d}{f_2}\ar{r}{\id} \& \scrO_Y\ar{d}{f_1}\\
			0\ar[r] \& i_{\Sigma,\, \ast}\scrO_\Sigma \ar[r] \& \calE_2 \ar[r] \&i_{\Sigma,\, \ast}\scrO_{\Sigma} \ar[r] \& 0
		\end{tikzcd}\ .
	\end{align}
	Note that $\calE_2$ is purely one-dimensional, set-theoretically supported on $\Sigma$. Hence $\mathsf{Im}(f_2)$ is the structure sheaf of a purely one-dimensional closed subscheme $\Sigma_2\subset Y$ with set theoretic support on $\Sigma$. Let $\calJ_2\subset \scrO_Y$ be its defining ideal. Applying the snake lemma to the diagram \eqref{eq:e-two-diagram}, one obtains a long exact sequence 
	\begin{align}\label{eq:e-two-diagram-2}
		\begin{tikzcd}[ampersand replacement=\&]
			0\ar{r}\& \calJ_2 \ar{r}\& \calI_\Sigma \ar{r}\& i_{\Sigma,\, \ast}\scrO_\Sigma \ar{r}{c} \& \calE_2/i_{\Sigma_2,\, \ast}\scrO_{\Sigma_2} \ar{r} \& 0 
		\end{tikzcd}\ .
	\end{align}
	Suppose that the morphism $c$ is nonzero. Then $c$ is an isomorphism and the composition 
	\begin{align}
		\begin{tikzcd}[ampersand replacement=\&]
			\calE_2 \ar{r} \& \calE_2/i_{\Sigma_2,\, \ast}\scrO_{\Sigma_2} \ar{r}{c^{-1}} \& i_{\Sigma,\, \ast}\scrO_\Sigma
		\end{tikzcd}
	\end{align}
	determines a splitting of the bottom row of diagram \eqref{eq:e-two-diagram}. This implies $e_2=0$, leading to a contradiction (cf.\ Lemma~\ref{lem:B}). Hence $c=0$, which implies $\calE/i_{\Sigma_2,\, \ast}\scrO_{\Sigma_2} =0$, as claimed. The exact sequence \eqref{eq:ideal-flag-A} follows from the bottom row of diagram \eqref{eq:e-two-diagram-2}.
\end{proof}

Proceeding by induction, the same arguments of the previous proposition shows: 
\begin{proposition}\label{prop:inductive-step} 
	Assume that the sheaf $\calE_{k-1}$, is isomorphic to the structure sheaf of a purely one-dimensional scheme-theoretic thickening $\Sigma \subset \Sigma_{k-1} \subset Y$ for some $3\leq k \leq n$. Then there exists a commutative diagram 
	\begin{align}\label{eq:e-two-diagram-3}
		\begin{tikzcd}[ampersand replacement=\&]
			\&  \& \scrO_Y \ar{r}{\id}\ar{d}{f_k} \& \scrO_Y \ar{d}{f_{k-1}} \&  \\
			0\ar{r} \& i_{\Sigma,\, \ast}\scrO_{\Sigma} \ar[r] \& \calE_k \ar[r]\& i_{\Sigma_{k-1},\, \ast}\scrO_{\Sigma_{k-1}} \ar[r]\& 0 
		\end{tikzcd}\ ,
	\end{align}
	where the right vertical arrow is the canonical epimorphism and the left vertical arrow is surjective. In particular $\calE_k\simeq i_{\Sigma_k, \ \ast}\scrO_{\Sigma_k}$ for a closed subscheme $\Sigma_k\subset Y$ and its defining ideal sheaf $\calJ_k\subset \scrO_Y$ fits in an exact sequence 
	\begin{align}\label{eq:ideal-flag-B}
		0\longrightarrow \calJ_k \longrightarrow \calJ_{k-1} \longrightarrow i_{\Sigma,\, \ast}\scrO_\Sigma \longrightarrow 0\ .
	\end{align}
\end{proposition}

In conclusion, Propositions~\ref{prop:step-two} and \ref{prop:inductive-step} provide a canonical chain of scheme-theoretic thickenings 
\begin{align}
	\Sigma_n \subset \cdots \subset \Sigma_1 =\Sigma 
\end{align}
whose structure sheaves fit in the exact sequences 
\begin{align}\label{eq:sch-sequence-A}
	0\longrightarrow i_{\Sigma,\, \ast}\scrO_\Sigma \longrightarrow i_{\Sigma_k,\, \ast}\scrO_{\Sigma_k} \longrightarrow i_{\Sigma_{k-1},\, \ast}\scrO_{\Sigma_{k-1}} \longrightarrow 0.
	\end{align}
Furthermore, note that the exact sequences \eqref{eq:sh-sequence-B} translate into exact sequences 
\begin{align} 
	\label{eq:sch-sequence-B}
	0\longrightarrow i_{\Sigma_{k-1},\, \ast}\scrO_{\Sigma_{k-1}} \longrightarrow i_{\Sigma_k,\, \ast}\scrO_{\Sigma_k} \longrightarrow i_{\Sigma,\, \ast}\scrO_{\Sigma} \longrightarrow 0,\ 
\end{align}
where  $\scrO_{\Sigma_k} \longrightarrow i_{\Sigma,\, \ast}\scrO_{\Sigma}$ is the canonical epimorphism induced by the closed immersion $\Sigma \subset \Sigma_k$. 

By construction, the extension class associated to each sequence in  \eqref{eq:sch-sequence-A} is nontrivial. As shown below, the same holds for the sequences \eqref{eq:sch-sequence-B} as well.

\begin{lemma}\label{lem:non-triv-extension} 
	The extension class associated to each sequence in \eqref{eq:sch-sequence-B} is nontrivial. 
\end{lemma} 

\begin{proof}
	Suppose that the extension class associated to \eqref{eq:sch-sequence-B} is trivial. This yields a surjective morphism $\scrO_Y \twoheadrightarrow i_{\Sigma,\, \ast}\scrO_{\Sigma}\oplus  i_{\Sigma_{k-1},\, \ast}\scrO_{\Sigma_{k-1}}$. By restriction to any closed point $p\in \Sigma$ one obtains a  surjective morphism $\C \to \C \oplus V$, where $V$ is a nonzero vector space. Hence, we get a contradiction. 
\end{proof}

Moreover, one has the following result. 
\begin{lemma}\label{lem:hom-space} 
	One has $\Hom_Y(i_{\Sigma,\, \ast}\scrO_{\Sigma},i_{\Sigma_k,\, \ast}\scrO_{\Sigma_k}) \simeq \C$ for all $1\leq k \leq n$.
\end{lemma} 

\begin{proof}
	The exact sequence \eqref{eq:sch-sequence-B} yields the long exact sequence 
	\begin{align}
		\begin{tikzcd}[ampersand replacement=\&, row sep=tiny]
			0\ar{r}\& \Hom_Y(i_{\Sigma,\, \ast}\scrO_\Sigma, i_{\Sigma_{k-1},\, \ast}\scrO_{\Sigma_{k-1}}) \ar{r}\& \Hom_Y(i_{\Sigma,\, \ast}\scrO_\Sigma, i_{\Sigma_k,\, \ast}\scrO_{\Sigma_k}) \ar{r}\& {} \\
			\& \Hom_Y(i_{\Sigma,\,\ast}\scrO_\Sigma,i_{\Sigma,\,\ast}\scrO_\Sigma) \ar{r}{\delta} \& \Ext^1_Y(i_{\Sigma,\, \ast}\scrO_\Sigma, i_{\Sigma_{k-1},\, \ast}\scrO_{\Sigma_{k-1}}) \ar{r}\& \cdots 
		\end{tikzcd}
	\end{align}
	where $\delta(1)$ is the associated extension class. Lemma~\ref{lem:non-triv-extension} shows that $\delta(1)\neq 0$. Hence 
	\begin{align}
		\Hom_Y(i_{\Sigma,\, \ast}\scrO_\Sigma, i_{\Sigma_{k-1},\, \ast}\scrO_{\Sigma_{k-1}})  \simeq  \Hom_Y(i_{\Sigma,\, \ast}\scrO_\Sigma, i_{\Sigma_k,\, \ast}\scrO_{\Sigma_k}) 
	\end{align}
	and the claim follows by induction. 
\end{proof}

Recall the following basic result: 
\begin{lemma}\label{lem:hom-dimension} 
	Let $Z\subset Y$ be a purely one-dimensional closed subscheme and let $i$ be the corresponding closed embedding. Then 
	\begin{align}
		\calExt^i_Y(i_\ast\scrO_Z,\scrO_Y) =0
	\end{align}
	for $i=0, 1, 3$. 
	
	Let $\calI_Z\subset \scrO_Y$ be the ideal sheaf of $Z$. Then,
	\begin{align}
		\calExt_Y^i(\calI_Z, \scrO_Y)=0
	\end{align}
	for $i=2, 3$, 
	\begin{align}
		\calExt_Y^0(\calI_Z, \scrO_Y)\simeq \scrO_Y \ , \text{ and } 	\calExt_Y^1(\calI_Z, \scrO_Y)\simeq \calExt_Y^2(i_\ast\scrO_Z,\scrO_Y)\ .
	\end{align}
	In particular, $\calI_Z$ has homological dimension one, i.e., it admits a locally free resolution of amplitude $[-1, \ 0]$. 
\end{lemma} 

\begin{proof}
	The first part follows from \cite[Proposition~1.1.6-(i) and Proposition~1.1.10]{HL_Moduli}. 
	Using the isomorphisms $\calExt_Y^0(\scrO_Y, \scrO_Y)\simeq \scrO_Y$, while $\calExt_Y^i(\scrO_Y, \scrO_Y)=0$, $i\geq 1$, the second part follows from long exact sequence obtained by applying $\calHom_Y(-, \scrO_Y)$ to the exact sequence 
	\begin{align}\label{eq:ses-Z}
		0\longrightarrow \calI_Z \longrightarrow \scrO_Y \longrightarrow i_\ast\scrO_Z\longrightarrow 0\ ,
	\end{align}
	In particular, one obtains $\calExt^i(\calI_Z, \scrO_Y)=0$ for all $i\geq 2$ which implies that $\calI_Z$ has homological dimension one. 
\end{proof}

\begin{lemma}\label{lem:der-dual} 
	Let $(i_{\Sigma,\, \ast}\scrO_\Sigma)^\vee \coloneqq \R\calHom_Y(i_{\Sigma,\, \ast}\scrO_\Sigma, \scrO_Y)$ be the derived dual of $\scrO_\Sigma$. Then 
	\begin{align}
		(i_{\Sigma,\, \ast}\scrO_\Sigma)^\vee \simeq i_{\Sigma,\, \ast}\scrO_\Sigma(-2)[-2] \ .
	\end{align}
\end{lemma} 

\begin{proof}
	Since $i_{\Sigma,\, \ast}\scrO_\Sigma$ is a purely one-dimensional sheaf, Lemma~\ref{lem:hom-dimension} implies that $\calExt^i_Y(i_{\Sigma,\, \ast}\scrO_\Sigma, \scrO_Y)=0$ for $i=0, 1, 3$. Therefore, we have
	\begin{align}
		(i_{\Sigma,\, \ast}\scrO_\Sigma)^\vee \simeq \calExt^2_Y(i_{\Sigma,\, \ast}\scrO_\Sigma, \scrO_Y)[-2]\ ,
	\end{align}
	where $\calExt^2_Y(i_{\Sigma,\, \ast}\scrO_\Sigma, \scrO_Y)$ is reflexive, by \cite[Proposition~1.1.10]{HL_Moduli}. Hence it is in particular purely one-dimensional. Moreover, $\calExt^2_Y(i_{\Sigma,\, \ast}\scrO_\Sigma, \scrO_Y)$ is also set-theoretically supported on $\Sigma$. Since 
	\begin{align}
		\ch_2((i_{\Sigma,\, \ast}\scrO_\Sigma)^\vee) = \ch_2(\scrO_\Sigma) = [\Sigma] \quad \text{and}\quad \chi((i_{\Sigma,\, \ast}\scrO_\Sigma)^\vee) = -\chi(\scrO_\Sigma) =-1\ ,
	\end{align}
	one gets $\calExt^2_Y(i_{\Sigma,\, \ast}\scrO_\Sigma, \scrO_Y)\simeq i_{\Sigma,\, \ast}\scrO_\Sigma(-2)$.
\end{proof}

\begin{lemma}\label{lem:loc-ext-A} 
	One has
	\begin{align}
		\calExt^1_Y(i_{\Sigma,\, \ast}\scrO_\Sigma, i_{\Sigma,\, \ast}\scrO_\Sigma) &\simeq i_{\Sigma,\, \ast}\scrO_\Sigma \oplus i_{\Sigma,\, \ast}\scrO_\Sigma(-2)\ , \\
		\calExt^2_Y(i_{\Sigma,\, \ast}\scrO_\Sigma, i_{\Sigma,\, \ast}\scrO_\Sigma) &\simeq i_{\Sigma,\, \ast}\scrO_\Sigma(-2)\ , \\
		\calExt^3_Y(i_{\Sigma,\, \ast}\scrO_\Sigma, i_{\Sigma,\, \ast}\scrO_\Sigma) &\simeq 0 \ .
	\end{align}		
	Moreover, 
	\begin{align}
		\calExt^1_Y(\calI_\Sigma, i_{\Sigma,\, \ast}\scrO_\Sigma) \simeq i_{\Sigma,\, \ast}\scrO_\Sigma(-2)\ .
	\end{align}
\end{lemma} 

\begin{proof}
	Recall that $\calExt_Y^0(\scrO_Y,  i_{\Sigma,\, \ast}\scrO_\Sigma)\simeq  i_{\Sigma,\, \ast}\scrO_\Sigma$, while $\calExt_Y^i(\scrO_Y,  i_{\Sigma,\, \ast}\scrO_\Sigma)=0$ for $i\geq 1$. 
	By applying $\calHom_Y(-,  i_{\Sigma,\, \ast}\scrO_\Sigma)$ to the exact sequence \eqref{eq:ses-Z} and considering the corresponding long exact sequence, we obtain isomorphisms
	\begin{align}\label{eq:ext-isom}
		\begin{tikzcd}[ampersand replacement=\&]
			\calExt^{i-1}_Y(\calI_\Sigma, i_{\Sigma,\, \ast}\scrO_\Sigma) \ar{r}{\sim} \& \calExt^i_Y(i_{\Sigma,\, \ast}\scrO_\Sigma, i_{\Sigma,\, \ast}\scrO_\Sigma)
		\end{tikzcd}
	\end{align}
	for all $1\leq i \leq 3$. For $i=0$, one has 
	\begin{align}
		\calExt^0_Y(\calI_\Sigma, i_{\Sigma,\, \ast}\scrO_\Sigma) \simeq i_{\Sigma,\, \ast}\calH om_{\Sigma^-}(\calI_\Sigma\otimes \scrO_\Sigma, 
		\scrO_\Sigma) \simeq i_{\Sigma,\, \ast} \calH om_{\Sigma^-}( \calN_{\Sigma/Y}, \scrO_\Sigma) \simeq i_{\Sigma,\, \ast}\scrO_\Sigma\oplus i_{\Sigma,\, \ast}\scrO_\Sigma(-2)\ .
	\end{align}
	This proves the first identity. 
	
	In order to prove the second identity, note the isomorphism
	\begin{align}
		\R\calH om_Y(i_{\Sigma,\, \ast}\scrO_\Sigma,i_{\Sigma,\, \ast}\scrO_\Sigma) \simeq (i_{\Sigma,\, \ast}\scrO_\Sigma)^\vee\otimes^\LL i_{\Sigma,\, \ast}\scrO_\Sigma 
	\end{align}
	in $\catDb(Y)$. Using Lemma~\ref{lem:der-dual} this yields an isomorphism 
	\begin{align}
		\calExt_Y^2(i_{\Sigma,\, \ast}\scrO_\Sigma,i_{\Sigma,\, \ast}\scrO_\Sigma) \simeq i_{\Sigma,\, \ast}\scrO_\Sigma(-2)\ .
	\end{align}
	
	The third identity follows from Lemma~\ref{lem:hom-dimension} and Formula~\eqref{eq:ext-isom}. Since $\calI_\Sigma$ has homological dimension one, it has a two term locally free resolution, 
	which implies that $\calExt^2(\calI_\Sigma, i_{\Sigma,\, \ast}\scrO_\Sigma) =0$. 	
	
	The fourth identity follows from the second using the isomorphism \eqref{eq:ext-isom} with $i=2$. 
\end{proof}

\begin{lemma}\label{lem:loc-ext-coh} 
	One has $H^1(Y, \calExt_Y^0(i_{\Sigma_k,\, \ast}\scrO_{\Sigma_k}, i_{\Sigma,\, \ast}\scrO_\Sigma)) =0$ for any $1\leq k\leq n$.
\end{lemma} 

\begin{proof}
	The exact sequence \eqref{eq:sch-sequence-B} yields the exact sequence 
	\begin{align}
		\begin{tikzcd}[ampersand replacement=\&, row sep=tiny]
			0\ar{r}\& \calExt_Y^0(i_{\Sigma,\, \ast}\scrO_{\Sigma}, i_{\Sigma,\, \ast}\scrO_\Sigma) \ar{r}\& \calExt_Y^0(i_{\Sigma_k,\, \ast}\scrO_{\Sigma_k}, i_{\Sigma,\, \ast}\scrO_\Sigma) \ar{r}\& {}\\
			\& \calExt_Y^0(i_{\Sigma_{k-1},\, \ast}\scrO_{\Sigma_{k-1}}, i_{\Sigma,\, \ast}\scrO_\Sigma) \ar{r}{\delta} \& \calExt_Y^1(i_{\Sigma,\, \ast}\scrO_{\Sigma}, i_{\Sigma,\, \ast}\scrO_\Sigma) \ar{r} \& \cdots 
		\end{tikzcd}\ ,
	\end{align}
	where $\calExt_Y^1(i_{\Sigma,\, \ast}\scrO_{\Sigma}, i_{\Sigma,\, \ast}\scrO_\Sigma) = i_{\Sigma,\, \ast}\scrO_\Sigma \oplus i_{\Sigma,\, \ast}\scrO_{\Sigma}(-2)$ by Lemma~\ref{lem:loc-ext-A}. 
	
	It will be shown by induction that $\calExt_Y^0(i_{\Sigma_k,\, \ast}\scrO_{\Sigma_k}, i_{\Sigma,\, \ast}\scrO_\Sigma)\simeq V_k \otimes i_{\Sigma,\, \ast}\scrO_\Sigma$  for some
	vector spaces $V_k$ for all $1 \leq k \leq n$, which implies the claim. This clearly holds for $k=1$. 
	
	Suppose that $\calExt_Y^0(i_{\Sigma_{k-1},\, \ast}\scrO_{\Sigma_{k-1}}, i_{\Sigma,\, \ast}\scrO_\Sigma)\simeq V_{k-1}\otimes i_{\Sigma,\, \ast}\scrO_\Sigma$, with $V_{k-1}$ a vector space. Since $\Hom_Y(i_{\Sigma,\, \ast}\scrO_\Sigma, i_{\Sigma,\, \ast}\scrO_\Sigma(-2))=0$, this implies that $\mathsf{Im}(\delta) \subset i_{\Sigma,\, \ast}\scrO_\Sigma$. Hence $\mathsf{Im}(\delta)$ is either identically zero or isomorphic to $i_{\Sigma,\, \ast}\scrO_\Sigma$. At the same time, there is a surjective morphism 
	\begin{align}
		V_{k-1}\otimes i_{\Sigma,\, \ast}\scrO_\Sigma \twoheadrightarrow \mathsf{Im}(\delta)\ . 
	\end{align}
	Since $\mathsf{Im}(\delta)$ is either identically zero or isomorphic to $i_{\Sigma,\, \ast}\scrO_\Sigma$, the kernel of this epimorphism is isomorphic to $V'_{k-1} \otimes i_{\Sigma,\, \ast}\scrO_\Sigma$, with $V'_{k-1} \subset V_{k-1}$ a vector subspace. Then one is left with a short exact sequence 
	\begin{align}
		0\longrightarrow i_{\Sigma,\, \ast}\scrO_\Sigma \longrightarrow \calExt_Y^0(i_{\Sigma_k,\, \ast}\scrO_{\Sigma_k}, i_{\Sigma,\, \ast}\scrO_\Sigma) \longrightarrow V'_{k-1} \otimes i_{\Sigma,\, \ast}\scrO_\Sigma \longrightarrow 0\ .
	\end{align}
	Since $\calExt_Y^0(i_{\Sigma_k,\, \ast}\scrO_{\Sigma_k}, i_{\Sigma,\, \ast}\scrO_\Sigma)$ is a $\scrO_\Sigma$-module, and $\Ext^1_\Sigma(\scrO_\Sigma, \scrO_\Sigma)=0$, the inductive step follows. 
\end{proof}

Now note that, using Corollary \ref{cor:pairing}, Corollary~\ref{cor:C} yields
\begin{align}\label{eq:ext-vanishing-A} 
	\Ext^1_Y(i_{\Sigma_n,\, \ast}\scrO_{\Sigma_n}, i_{\Sigma,\, \ast}\scrO_\Sigma) =0\quad \text{and} \quad \Ext^2_Y(i_{\Sigma_n,\, \ast}\scrO_{\Sigma_n}, i_{\Sigma,\, \ast}\scrO_{\Sigma}) =0\ ,
\end{align}
via Serre duality.
Given the above vanishing results, Lemmas~\ref{lem:hom-space} and \ref{lem:loc-ext-coh} yield:
\begin{lemma}\label{lem:loc-ext-B} 
	One has isomorphisms 
	\begin{align}
		\calExt^1_Y(i_{\Sigma_n,\, \ast}\scrO_{\Sigma_n}, i_{\Sigma,\, \ast}\scrO_{\Sigma})&\simeq V_1 \otimes i_{\Sigma,\, \ast}\scrO_{\Sigma}(-1)\ , \\
		\calExt^2_Y(i_{\Sigma_n,\, \ast}\scrO_{\Sigma_n}, i_{\Sigma,\, \ast}\scrO_{\Sigma}) &\simeq i_{\Sigma,\, \ast}\scrO_\Sigma(-2) \oplus (V_2\otimes i_{\Sigma,\, \ast}\scrO_{\Sigma}(-1))
	\end{align}
	for some vector spaces $V_1, V_2$. 
\end{lemma} 

\begin{proof}
	The second term of the standard local to global spectral sequence for $\Ext_Y^\bullet(i_{\Sigma_n,\, \ast}\scrO_{\Sigma_n}, i_{\Sigma,\, \ast}\scrO_{\Sigma})$ reads 
	\begin{align}
		E_2^{p,q} = H^q(Y, \calExt_Y^p(i_{\Sigma_n,\, \ast}\scrO_{\Sigma_n}, i_{\Sigma,\, \ast}\scrO_{\Sigma}))\ , 
	\end{align}
	for any $p,q\geq 0$. Since all local Ext sheaves are set-theoretically supported on $\Sigma$, all terms with $q\geq 2$ vanish identically. Since $Y$ is a smooth threefold, all terms with $p\geq 4$ also vanish identically. Furthermore Lemmas~\ref{lem:loc-ext-A} and \ref{lem:loc-ext-coh} show that $E_2^{0,1} =0=E_2^{3,0}$. 
	
	Hence one has a first quadrant spectral sequence 
	\begin{align}
		\begin{tikzcd}[ampersand replacement=\&, column sep=huge, row sep=huge]
			0 \ar{drr}{d_2^{0,1}} \& E_2^{1,1} \ar{drr}{d_2^{1,1}} \& E_2^{2,1} \& E^{3,1} \\
			E_2^{0,0} \& E_2^{1,0} \& E_2^{2,0} \& 0
		\end{tikzcd}
	\end{align}
	where all other terms are zero. In particular, all differentials $d_r^{p,q}$, with $r\geq 2$, are identically zero. Then the first vanishing condition in \eqref{eq:ext-vanishing-A} implies that $E_2^{1,0}=0$, i.e., one has
	\begin{align}
		H^0(Y, \calExt_Y^1(i_{\Sigma_n,\, \ast}\scrO_{\Sigma_n}, i_{\Sigma,\, \ast}\scrO_{\Sigma}))=0\ .
	\end{align}
	This further implies that  $\calExt_Y^1(i_{\Sigma_n,\, \ast}\scrO_{\Sigma_n}, i_{\Sigma,\, \ast}\scrO_{\Sigma})$ must be purely one-dimensional. Since it is also a $\scrO_\Sigma$-module, it must be the pushforward of a locally free sheaf $\calF$ on $\Sigma$ so that $H^0(\Sigma, \calF)=0$. Hence $\calF$ is either zero or
	\begin{align}
		\mu_{\mathsf{max}}(\calF) \leq 0\ . 
	\end{align}
	Here, the slope of a coherent sheaf $\calE$ on $\Sigma$ is defined by $\mu(\calE)= \chi(\calE)/ \mathsf{rk}(\calE)$ provided that $\mathsf{rk}(\calE)>0$.
	
	Similarly, the second vanishing condition in \eqref{eq:ext-vanishing-A} implies that $H^1(\Sigma, \calF)=0$. If $\calF$ is nonzero, this yields  $\mu_{\mathsf{min}}(\calF) \geq 0$. In conclusion, $\calF \simeq V_1\otimes \scrO_\Sigma(-1)$ for some vector space $V_1$. 
	
	Moreover, the second vanishing condition in \eqref{eq:ext-vanishing-A} also implies that 
	\begin{align}
		H^0(Y, \calExt^2_Y(i_{\Sigma_n,\, \ast}\scrO_{\Sigma_n}, i_{\Sigma,\, \ast}\scrO_{\Sigma})) =0\ .
	\end{align}
	Therefore $\calExt^2_Y(i_{\Sigma_n,\, \ast}\scrO_{\Sigma_n}, i_{\Sigma,\, \ast}\scrO_{\Sigma})$ is the pushforward of a locally free sheaf $\calG$ on $\Sigma$. If $\calG\neq 0$, it must satisfy 
	\begin{align}
		\mu_{\mathsf{max}}(G) \leq 0\ .
	\end{align}
	Finally, by Serre duality, Lemma~\ref{lem:hom-space} yields 
	\begin{align}
		\Ext^3_Y(i_{\Sigma_n,\, \ast}\scrO_{\Sigma_n}, i_{\Sigma,\, \ast}\scrO_{\Sigma})\simeq \C\ .
	\end{align}
	Since $E_2^{1,1}=0=E_2^{p,0}$, with $p\geq 3$, one obtains an isomorphism 
	\begin{align}
		\Ext^3_Y(i_{\Sigma_n,\, \ast}\scrO_{\Sigma_n}, i_{\Sigma,\, \ast}\scrO_{\Sigma})\simeq H^1(\Sigma, \calG)\ . 
	\end{align}
	This implies that $\calG\simeq i_{\Sigma,\, \ast}\scrO_\Sigma(-2) \oplus(V_2 \otimes i_{\Sigma,\, \ast}\scrO_\Sigma(-1))$ for some vector space $V_2$. 
\end{proof}

\begin{lemma}\label{lem:loc-ext-C} 
	One has 
	\begin{align}
		\calExt_Y^1(\calJ_k, i_{\Sigma,\, \ast}\scrO_\Sigma) \simeq i_{\Sigma,\, \ast}\scrO_\Sigma(-2)
	\end{align}
	for all $1\leq k \leq n$ and 
	\begin{align}
		\calExt_Y^0(\calJ_k, i_{\Sigma,\, \ast}\scrO_\Sigma) \simeq \begin{cases} 
			i_{\Sigma,\, \ast}\scrO_\Sigma \oplus i_{\Sigma,\, \ast}\scrO_\Sigma(-2) & \text{for } 1\leq k \leq n-1\ , \\[4pt]
			i_{\Sigma,\, \ast}\scrO_\Sigma(-1)\oplus i_{\Sigma,\, \ast}\scrO_\Sigma(-1) & \text{for }k =n \ ,
		\end{cases}
	\end{align}
	where $\calJ_k\subset \scrO_Y$ is the ideal sheaf of $\Sigma_k$ for $1\leq k\leq n$. 
\end{lemma} 

\begin{proof}
	The exact sequence 
	\begin{align}\label{eq:J-sequence}
		0\longrightarrow \calJ_k \longrightarrow \calJ_{k-1} \longrightarrow i_{\Sigma,\, \ast}\scrO_\Sigma\longrightarrow 0
	\end{align}
	yields the long exact sequence 
	\begin{align}\label{eq:loc-ext-sequence}
		\begin{tikzcd}[ampersand replacement=\&, row sep=tiny]
			0 \ar{r} \&  \calExt^0_Y(i_{\Sigma,\, \ast}\scrO_\Sigma,i_{\Sigma,\, \ast}\scrO_\Sigma) \ar{r}\& \calExt^0_Y(\calJ_{k-1}, i_{\Sigma,\, \ast}\scrO_\Sigma) \ar{r} \& \calExt^0_Y(\calJ_k, i_{\Sigma,\, \ast}\scrO_\Sigma)\\
			{} \ar{r}{\delta_0} \& \calExt^1_Y(i_{\Sigma,\, \ast}\scrO_\Sigma, i_{\Sigma,\, \ast}\scrO_\Sigma) 
			\ar{r} \& \calExt^1_Y(\calJ_{k-1}, i_{\Sigma,\, \ast}\scrO_\Sigma)\ar{r}\&\calExt^1_Y(\calJ_k,i_{\Sigma,\, \ast}\scrO_\Sigma)\\
			{} \ar{r}{\delta_1} \& \calExt^2_Y(i_{\Sigma,\, \ast}\scrO_\Sigma, i_{\Sigma,\, \ast}\scrO_\Sigma) \ar{r}\& 0
		\end{tikzcd}
	\end{align}
	for any $1\leq k\leq n$. Lemma~\ref{lem:loc-ext-A} shows that there are isomorphisms
	\begin{align}
		\calExt^1_Y(i_{\Sigma,\, \ast}\scrO_\Sigma, i_{\Sigma,\, \ast}\scrO_\Sigma)\simeq i_{\Sigma,\, \ast}\scrO_\Sigma\oplus i_{\Sigma,\, \ast}\scrO_\Sigma(-2)\quad\text{and} \quad \calExt^2_Y(i_{\Sigma,\, \ast}\scrO_\Sigma, i_{\Sigma,\, \ast}\scrO_\Sigma)\simeq i_{\Sigma,\, \ast}\scrO_\Sigma(-2)\ .
	\end{align}
	Next, it will be shown by induction that 
	\begin{align}\label{eq:J-iso}
		\calExt_Y^1(\calJ_k, i_{\Sigma,\, \ast}\scrO_\Sigma) \simeq i_{\Sigma,\, \ast}\scrO_\Sigma(-2) 
	\end{align}
	for all $1\leq k \leq n$. For $k=1$, the claim is proven in Lemma~\ref{lem:loc-ext-A}. Suppose 
	\begin{align}
		\calExt^1_Y(\calJ_{k-1}, i_{\Sigma,\, \ast}\scrO_\Sigma)\simeq i_{\Sigma,\, \ast}\scrO_\Sigma(-2)
	\end{align}
	for some $2\leq k \leq n$. The the exact sequence \eqref{eq:loc-ext-sequence} yields an exact sequence 
	\begin{align}	\label{eq:loc-ext-sequence-B}
		\begin{tikzcd}[ampersand replacement=\&, row sep=tiny]
			\cdots \ar{r} \& i_{\Sigma,\, \ast}\scrO_\Sigma \oplus i_{\Sigma,\, \ast}\scrO_\Sigma(-2) \ar{r}{f} \& i_{\Sigma,\, \ast}\scrO_\Sigma(-2) \&\\
			{} \ar{r}\& \calExt^1_Y(\calJ_k, i_{\Sigma,\, \ast}\scrO_\Sigma) \ar{r}{\delta} \& i_{\Sigma,\, \ast}\scrO_\Sigma(-2) \ar{r}\& 0
		\end{tikzcd}\ .
	\end{align}
	Since $\Hom_Y(i_{\Sigma,\, \ast}\scrO_\Sigma, i_{\Sigma,\, \ast}\scrO_\Sigma(-2))=0$, it follows that $\mathsf{Im}(f)$ is either $0$ or $i_{\Sigma,\, \ast}\scrO_\Sigma(-2)$. Hence $\mathsf{Coker}(f)$ is either $i_{\Sigma,\, \ast}\scrO_\Sigma(-2)$ or $0$ respectively. Since $\calExt^1_Y(\calJ_k, i_{\Sigma,\, \ast}\scrO_\Sigma)$ is a $\scrO_\Sigma$-module, this implies that 
	\begin{align}\label{eq:J-ext-sigma} 
		\calExt^1_Y(\calJ_k, i_{\Sigma,\, \ast}\scrO_\Sigma)\simeq W_k\otimes i_{\Sigma,\, \ast}\scrO_\Sigma(-2)
	\end{align}
	with $W_k$ a vector space. 
	Lemma~\ref{lem:hom-space} shows that there isomorphisms 
	\begin{align}\label{eq:J-ext-isomorphism}
		\calExt^{i-1}_Y(\calJ_k, i_{\Sigma,\, \ast}\scrO_\Sigma) \simeq \calExt^i_Y(i_{\Sigma,\, \ast}\scrO_{\Sigma_k}, i_{\Sigma,\, \ast}\scrO_\Sigma)
	\end{align}
	for all $1\leq k\leq n$ and all $i\geq 1$. By Lemma~\ref{lem:loc-ext-B}, one has 
	\begin{align}
		\calExt^2_Y(i_{\Sigma_n,\, \ast}\scrO_{\Sigma_n}, i_{\Sigma,\, \ast}\scrO_\Sigma)\simeq i_{\Sigma,\, \ast}\scrO_\Sigma(-2) \oplus (V_2 \otimes 
		i_{\Sigma,\, \ast}\scrO_\Sigma(-1)) \ .
	\end{align}
	Therefore $W_n \simeq \C$ and $V_2=0$, hence 
	\begin{align}
		\calExt^1_Y(\calJ_n, i_{\Sigma,\, \ast}\scrO_\Sigma)\simeq i_{\Sigma,\, \ast}\scrO_\Sigma(-2)\ .
	\end{align}
	
	By Lemma~\ref{lem:loc-ext-A}, also $\calExt^2_Y(i_{\Sigma,\, \ast}\scrO_\Sigma, i_{\Sigma,\, \ast}\scrO_\Sigma)\simeq i_{\Sigma,\, \ast}\scrO_\Sigma(-2)$, hence returning to the sequence \eqref{eq:loc-ext-sequence} with $k=n$, and using Lemma~\ref{lem:loc-ext-A} and equation \eqref{eq:J-ext-sigma}, one then obtains a surjective morphism
	\begin{align}
		i_{\Sigma,\, \ast}\scrO_\Sigma\oplus i_{\Sigma,\, \ast}\scrO_\Sigma(-2) \longrightarrow W_{n-1}\otimes i_{\Sigma,\, \ast}\scrO_{\Sigma}(-2)\longrightarrow 0 \ ,
	\end{align}
	where $W_{n-1}$ is a nonzero vector space. Since $\Hom_Y(i_{\Sigma,\, \ast}\scrO_\Sigma,i_{\Sigma,\, \ast}\scrO_\Sigma(-2))=0$, this implies that
	$W_{n-1} \simeq \C$. Proceeding by descending induction on $\ell$, a completely analogous argument shows that $W_\ell \simeq \C$ for all $k \leq \ell \leq n$. In particular this proves the inductive step. In conclusion, 
	\begin{align}
		\calExt_Y^1(\calJ_k, i_{\Sigma,\, \ast}\scrO_\Sigma) \simeq i_{\Sigma,\, \ast}\scrO_\Sigma(-2) 
	\end{align}
	for all $1\leq k \leq n$. 
	
	To prove the second part, note that the exact sequence \eqref{eq:loc-ext-sequence} reduces to an exact sequence  
	\begin{align}\label{eq:loc-ext-sequence-C}
		0\longrightarrow i_{\Sigma,\, \ast}\scrO_\Sigma \longrightarrow \calExt_Y^0(\calJ_{k-1}, i_{\Sigma,\, \ast}\scrO_\Sigma) \longrightarrow 
		\calExt_Y^0(\calJ_k, i_{\Sigma,\, \ast}\scrO_\Sigma)\longrightarrow i_{\Sigma,\, \ast}\scrO_\Sigma \longrightarrow  0
	\end{align}
	for any $1\leq k \leq n$. By induction, this implies 
	\begin{align}
		\ch_2(\calExt_Y^0(\calJ_n, i_{\Sigma,\, \ast}\scrO_\Sigma)) =\ch_2(\calExt_Y^0(\calI_\Sigma, i_{\Sigma,\, \ast}\scrO_\Sigma)) =
		2[\Sigma]\ . 
	\end{align}
	Using equation \eqref{eq:J-ext-isomorphism} and Lemma~\ref{lem:loc-ext-B}, this implies that 
	\begin{align}
		\calExt^0_Y(\calJ_n, i_{\Sigma,\, \ast}\scrO_\Sigma)\simeq  i_{\Sigma,\, \ast}\scrO_\Sigma(-1)\oplus i_{\Sigma,\, \ast}\scrO_\Sigma(-1)\ . 
	\end{align}
	Then the exact sequence \eqref{eq:loc-ext-sequence-C} yields an exact sequence
	\begin{align}
		0\longrightarrow i_{\Sigma,\, \ast}\scrO_\Sigma\longrightarrow \calExt^0_Y(\calJ_{n-1}, i_{\Sigma,\, \ast}\scrO_\Sigma)\longrightarrow i_{\Sigma,\, \ast}\scrO_\Sigma(-2)\longrightarrow 0\ .
	\end{align}
	Since the central term is by definition the pushforward of a $\scrO_\Sigma$-module, it follows that 
	\begin{align}
		\calExt^0_Y(\calJ_{n-1}, i_{\Sigma,\, \ast}\scrO_\Sigma) \simeq i_{\Sigma,\, \ast}\scrO_\Sigma\oplus i_{\Sigma,\, \ast}\scrO_\Sigma(-2)\ .
	\end{align}
	Then by descending induction on $k$, the same exact sequence yields 
	\begin{align}
		\calExt^0_Y(\calJ_k, i_{\Sigma,\, \ast}\scrO_\Sigma) \simeq i_{\Sigma,\, \ast}\scrO_\Sigma\oplus i_{\Sigma,\, \ast}\scrO_\Sigma(-2)
	\end{align}
	for all $1\leq k\leq n-1$. 
\end{proof}

\begin{proposition}\label{prop:J-tensor} 
	One has
	\begin{align}
		\calJ_k \otimes i_{\Sigma,\, \ast}\scrO_\Sigma \simeq \begin{cases} 
			i_{\Sigma,\, \ast}\scrO_\Sigma \oplus i_{\Sigma,\, \ast}\scrO_\Sigma(2) & \text{for }\ 1\leq k \leq n-1\ ,\\[4pt]
			i_{\Sigma,\, \ast}\scrO_\Sigma(1)\oplus i_{\Sigma,\, \ast}\scrO_\Sigma(1) & \text{for }\ k =n \ .
		\end{cases}
	\end{align}
\end{proposition} 

\begin{proof}
	By Lemma~\ref{lem:hom-dimension}, the ideal sheaf $\calJ_k$  has a locally free resolution 
	\begin{align}
		\begin{tikzcd}[ampersand replacement=\&]
			\calV_{-1} \ar{r}{d}\& \calV_0
		\end{tikzcd}
	\end{align}
	of amplitude $[-1, \ 0]$. Then $\calJ_k \otimes^\LL i_{\Sigma,\, \ast}\scrO_\Sigma$ is isomorphic to the pushforward of the locally free  complex 
	\begin{align}\label{eq:tor-cpx-A}
		\begin{tikzcd}[ampersand replacement=\&]
			\calV_{-1}\vert_\Sigma \ar{r}{d\vert_\Sigma}\& \calV_0\vert_\Sigma 
		\end{tikzcd}
	\end{align}
	on $\Sigma$. Moreover, $\R \calHom_Y(\calJ_k, i_{\Sigma,\, \ast}\scrO_\Sigma)$ is isomorphic to the pushforward of the locally free complex of amplitude $[0, \ 1]$ 
	\begin{align}\label{eq:ext-cpx-A}
		\begin{tikzcd}[ampersand replacement=\&]
			\calV_0^\vee\vert_\Sigma \ar{r}{d^\vee\vert_\Sigma} \& \calV_{-1}^\vee\vert_\Sigma 
		\end{tikzcd}
	\end{align}
	on $\Sigma$. This complex is isomorphic to the dual of \eqref{eq:tor-cpx-A}, i.e., 
	\begin{align}\label{eq:ex-cpx-B} 
		\begin{tikzcd}[ampersand replacement=\&]
			(\calV_0\vert_\Sigma)^\ast \ar{r}{(d\vert_\Sigma)^\ast} \& (\calV_{-1}\vert_\Sigma)^\ast
		\end{tikzcd}
	\end{align}
	where $(-)^\ast$ denotes the dual on $\Sigma$. By Lemma~\ref{lem:loc-ext-C}, the cohomology groups of the complex \eqref{eq:ex-cpx-B}
	\begin{align}
		\calExt^0_Y(\calJ_k, i_{\Sigma,\, \ast}\scrO_\Sigma) \quad\text{and}\quad \calExt^1_Y(\calJ_k, i_{\Sigma,\, \ast}\scrO_\Sigma) 
	\end{align}
	are locally free as $\scrO_\Sigma$-modules. This implies that the cohomology groups of the complex \eqref{eq:tor-cpx-A} are 
	\begin{align}
		\calJ_k \otimes i_{\Sigma,\,\ast}\scrO_\Sigma &\simeq \calExt^0_Y(\calJ_k, i_{\Sigma,\, \ast}\scrO_\Sigma)^\ast\ ,\\
		\calT or_1^Y(\calJ_k, i_{\Sigma,\, \ast}\scrO_\Sigma) &\simeq \calExt^1_Y(\calJ_k, i_{\Sigma,\, \ast}\scrO_\Sigma)^\ast \ .
	\end{align}
	Then the claim follows from  Lemma~\ref{lem:loc-ext-C}.
\end{proof}

We are ready to prove:
\begin{proof}[Proof of the existence part of Theorem~\ref{thm:filtration-Appendix}]
	The existence of the chain of subschemes is proved in Propositions~\ref{prop:step-two} and \ref{prop:inductive-step}.
	
	The second isomorphism in the first row of \eqref{eq:ideal-filtration-B-Appendix} is clear by construction. The first isomorphism in the second row of \eqref{eq:ideal-filtration-B-Appendix} and \eqref{eq:ideal-filtration-C-Appendix} follow from Proposition~\ref{prop:J-tensor} via the exact sequence 
	\begin{align}
		0\longrightarrow \calI_\Sigma \calJ_k \longrightarrow \calJ_k \longrightarrow \calJ_k \otimes i_{\Sigma,\, \ast}\scrO_\Sigma \longrightarrow 0\ .
	\end{align}
	In particular, these isomorphisms imply 
	\begin{align}
		\Hom_Y(\calJ_k, i_{\Sigma,\, \ast}\scrO_\Sigma) \simeq \C\ ,
	\end{align}
	for $1\leq k \leq n-1$.
	
	It remains to show that the first isomorphism of the first row and the second isomorphism of the second row of \eqref{eq:ideal-filtration-B-Appendix} hold. One has a commutative diagram
	\begin{align}
		\begin{tikzcd}[ampersand replacement=\&]
			0\ar[r] \&  \calI_\Sigma \calJ_{k-1}\ar[r] \& \calJ_{k-1}\ar[r] \ar{d}{\id} \& i_{\Sigma,\, \ast}\scrO_\Sigma\oplus i_{\Sigma,\, \ast}\scrO_\Sigma(2)  \ar[r]\ar{d}{p} \& 0\\
			0\ar[r] \&  \calJ_k \ar[r] \& \calJ_{k-1}\ar[r] \& i_{\Sigma,\, \ast}\scrO_\Sigma \ar[r] \& 0
		\end{tikzcd}\ ,
	\end{align}
	where $p$ is surjective and $\ker(p)\simeq i_{\Sigma,\, \ast}\scrO_\Sigma(2)$.  This implies that $\calI_\Sigma \calJ_{k-1}\subset \calJ_k$ and $\calJ_k/\calI_\Sigma\calJ_{k-1} \simeq i_{\Sigma,\, \ast}\scrO_\Sigma(2)$ by the snake lemma. 
\end{proof}

We prove now the uniqueness part of the statement in Theorem~\ref{thm:filtration-Appendix}. We need first the following lemma.
\begin{lemma}\label{lem:rigid-B} 
	We have 
	\begin{align}\label{eq:j-toc-one} 
		\Hom_Y(\calJ_k, i_{\Sigma,\, \ast}\scrO_\Sigma) \simeq \C\ ,
	\end{align}
	for $1\leq k \leq n-1$, and 
	\begin{align}\label{eq:j-toc-two} 
		\Hom_Y(\calJ_n, i_{\Sigma,\, \ast}\scrO_\Sigma) = 0\ ,
	\end{align}
	for a filtration of the form \eqref{eq:ideal-filtration-A-Appendix} satisfying conditions \eqref{eq:ideal-filtration-B-Appendix} and \eqref{eq:ideal-filtration-C-Appendix}.
\end{lemma} 

\begin{proof}
	We shall prove Formula~\eqref{eq:j-toc-one} by using induction on $k$. Since $\Sigma\simeq \PP^1$ and $i_{\Sigma}^\ast\calJ_1\simeq \scrO_\Sigma\oplus \scrO_\Sigma(2)$ by assumption, the claim clearly holds for $k=1$. By applying $\Hom_Y(-, i_{\Sigma,\, \ast}\scrO_\Sigma)$ to the canonical exact sequence 
	\begin{align}
		0\longrightarrow \calJ_1\longrightarrow \scrO_Y \longrightarrow i_{\Sigma,\, \ast}\scrO_\Sigma\longrightarrow 0\ ,
	\end{align}
	we obtain the long exact sequence 
	\begin{multline}
		0\longrightarrow \Hom_{\Sigma^-}(\scrO_\Sigma, \scrO_\Sigma) \longrightarrow \Hom_Y(\scrO_Y, i_{\Sigma,\, \ast}\scrO_\Sigma) \longrightarrow \Hom_Y(\calJ_1, i_{\Sigma,\, \ast}\scrO_\Sigma)\\ 
		\longrightarrow \Ext^1_Y(i_{\Sigma,\, \ast}\scrO_\Sigma, i_{\Sigma,\, \ast}\scrO_\Sigma) \longrightarrow \Ext^1_Y(\scrO_Y, i_{\Sigma,\, \ast}\scrO_\Sigma)\longrightarrow \cdots 
	\end{multline}
	Since $\Sigma \simeq \PP^1$, this exact sequence yields 
	an isomorphism 
	\begin{align}
		\Ext^1_Y(i_{\Sigma,\, \ast}\scrO_\Sigma, i_{\Sigma,\, \ast}\scrO_\Sigma)\simeq \Hom_Y(\calJ_1, i_{\Sigma,\, \ast}\scrO_\Sigma)\ ,
	\end{align}
	hence 
	\begin{align}\label{eq:curve-ext} 
		\Ext^1_Y(i_{\Sigma,\, \ast}\scrO_\Sigma, i_{\Sigma,\, \ast}\scrO_\Sigma)\simeq\C\ .
	\end{align}
	In order to prove the inductive step, we apply the functor $\Hom_Y(-, i_{\Sigma,\, \ast}\scrO_\Sigma)$ to the short exact sequence 
	\begin{align}
		0\longrightarrow \calJ_{k+1}\longrightarrow \calJ_k\longrightarrow i_{\Sigma,\, \ast}\scrO_\Sigma \longrightarrow 0\ ,  
	\end{align}
	for $1\leq k \leq n-1$, yielding the long exact sequence 
	\begin{multline}
		0\longrightarrow \Hom_{\Sigma^-}(\scrO_\Sigma, \scrO_\Sigma) \longrightarrow \Hom_Y(\calJ_k, i_{\Sigma,\, \ast}\scrO_\Sigma)\\
		\longrightarrow \Hom_Y(\calJ_{k+1}, i_{\Sigma,\, \ast}\scrO_\Sigma) \longrightarrow \Ext^1_Y(i_{\Sigma,\, \ast}\scrO_\Sigma, i_{\Sigma,\, \ast}\scrO_\Sigma) \longrightarrow \cdots 
	\end{multline}
	Using the inductive hypothesis and Equation~\eqref{eq:curve-ext}, this exact sequence reduces to 
	\begin{align}
		0 \longrightarrow \Hom_Y(\calJ_{k+1}, i_{\Sigma,\, \ast}\scrO_\Sigma) \longrightarrow \C \longrightarrow \cdots 
	\end{align}
	If $k \leq n-2$, one must have $\Hom_Y(\calJ_{k+1}, i_{\Sigma,\, \ast}\scrO_\Sigma)\simeq \C$, since, by assumption, there is an exact sequence 
	\begin{align}
		0\longrightarrow \calJ_{k+2} \longrightarrow \calJ_{k+1} \longrightarrow i_{\Sigma,\, \ast}\scrO_\Sigma \longrightarrow 0\ ,
	\end{align}
	hence a nontrivial morphism $\calJ_{k+1} \longrightarrow i_{\Sigma,\, \ast}\scrO_\Sigma$. This proves the inductive step. 
	
	\medskip
	
	In order to prove \eqref{eq:j-toc-two}, note that $i_{\Sigma}^\ast(\calJ_n/\calI_\Sigma \calJ_n) \simeq i_{\Sigma}^\ast \calJ_n$, hence  the claim follows from condition \eqref{eq:ideal-filtration-C-Appendix}. 
\end{proof}

\begin{proof}[Proof of the uniqueness part of Theorem~\ref{thm:filtration-Appendix}]
	The proof is inductive. Let $\{\calJ_n\}$ and $\{\calJ'_{n'}\}$ be two filtrations satisfying the conditions \eqref{eq:ideal-filtration-B-Appendix} and \eqref{eq:ideal-filtration-C-Appendix}. By construction, $\calJ_1$ and $\calJ_1'$ coincide with $\calI_\Sigma$ as ideal sheaves in $\scrO_Y$. Using Formula~\eqref{eq:j-toc-one} one obtains a commutative diagram 
	\begin{align}
		\begin{tikzcd}[ampersand replacement=\&]
			0\ar[r] \&  \calJ_2 \ar[r] \ar[d] \& \calI_\Sigma \ar[r] \ar{d}{\id}\& i_{\Sigma,\, \ast}\scrO_\Sigma \ar{d}{\lambda \id} \ar[r]  \& 0\\
			0\ar[r] \&  \calJ'_2 \ar[r] \& \calI_\Sigma \ar[r] \& i_{\Sigma,\, \ast}\scrO_\Sigma \ar[r] \& 0 
		\end{tikzcd}\ ,
	\end{align}
	for some $\lambda\in \C^\times$. Then the snake lemma implies that the induced morphism $\calJ_2\to \calJ_2'$ is the identity. By reasoning in a analogous way, one can prove the inductive step. 
\end{proof}

\section{Proof of Theorem~\ref{thm:C-framed-identity}}\label{sec:proof-C-framed-identity}

In this section, we shall prove Theorem~\ref{thm:C-framed-identity}. 

\subsection{Motivic Hall algebras}\label{subsec:motHall} 

Introduced in \cite{Config-II}, motivic Hall algebras constitute the foundational framework for wallcrossing phenomena in Donaldson-Thomas theory \cite{Mot-DT, Gen-DT}. This section is a very brief review following the treatment of \cite{Bridgeland-motivic-Hall, Bridgeland-Curve-Counting}, aiming to explain how the general theory applies to the current framework.  

Let $K(\mathsf{St}/\C)$ be the Grothendieck ring of motivic equivalence classes of algebraic stacks of finite type over $\C$ with affine stabilizers. Let $K(\mathsf{Sch}/\C)$ be Grothendieck ring of motivic equivalence classes of schemes of finite type over $\C$. Let $\LL\in K(\mathsf{Sch}/\C)$ be the motivic class of $\A^1$.

Let $\scrC$ denote the abelian category $\catCoh(Y)$ and let $\scrC_{\leqslant 1}$ denote the full subcategory category $\catCoh_{\leqslant 1}(Y)$. As a complex vector space, the motivic Hall algebra $H(\scrC_{\leqslant 1})$ is spanned by all motivic equivalence classes of pairs $(\calX, f)$ where $\calX$ is an algebraic stack of finite type over $\C$ with affine stabilizers, and $f\colon \calX\to \Coh_{\leqslant 1}(Y)$ is a morphism of stacks. The algebra structure is defined by the natural convolution product, $\ast \colon H(\scrC_{\leqslant 1}) \otimes H(\scrC_{\leqslant 1}) \to H(\scrC_{\leqslant 1})$  induced by the diagram 
\begin{align}
	\begin{tikzcd}[ampersand replacement=\&]
		\Coh^{\mathsf{ext}}(\scrC_{\leqslant 1})\ar[d] \ar[r] \&  \Coh(\scrC_{\leqslant 1})\\
		\Coh( \scrC_{\leqslant 1})\times \Coh(\scrC_{\leqslant 1}  )   \& 
	\end{tikzcd}\ ,
\end{align}
where the top left corner is the moduli stack of three term exact sequences in $\scrC_{\leq 1}$. In addition, $H(\scrC_{\leqslant 1})$ has a natural $K(\mathsf{St}/\C)$-algebra structure defined by the assignment 
\begin{align}
	[\calX, f]\times [\calY] \longmapsto [\calX\times \calY, f \circ p_\calX] 
\end{align}
where $p_\calX\colon \calX\times \calY \to \calX$ is the canonical projection. We refer to \cite[\S4.1]{Bridgeland-Curve-Counting} for a detailed construction. For future reference, we also note that the above Hall algebra admits an infinite type variant, $H_\infty(\scrC_{\leqslant 1})$, spanned by motivic equivalence classes of pairs $(\calX, f)$ where $\calX$ is an algebraic stack, locally of finite type, with affine stabilizers. 

Let $\Delta \subset N_1(Y) \oplus \Z$ be the effective cone consisting of pairs $(\beta, n)$ so that either $\beta$ is a nonzero effective curve class or $\beta=0$ and $n\geq 0$. Note that the Hall algebras $H(\scrC_{\leqslant 1})$ and $H_\infty(\scrC_{\leqslant 1})$ have a natural $\Delta$-grading induced by the decomposition 
\begin{align}
	\Coh(\scrC_{\leqslant 1})= \bigsqcup_{\alpha \in \Delta}  \Coh(\scrC_{\leqslant 1}; \alpha)\ .
\end{align}

The main results of \cite{Bridgeland-Curve-Counting}, \cite{BS-Curve-counting-crepant}, and \cite{Calabrese-DT-Flops} are proven using a certain completion of the Hall algebra $H(\scrC_{\leqslant 1})$ defined in terms of Laurent subsets $S\subset \Delta$. By definition, a subset $S\subset \Delta$ is called Laurent if for any $\beta \in N_1(Y)$ so that $(\{\beta\}\times \Z) \cap S \neq \emptyset$ the set $\{ n \in \Z\, \vert\, (\beta, n) \in S\}$ is bounded below. Moreover, any $\Delta$-graded ring 
\begin{align}
	R = \bigoplus_{\alpha \in \Delta } R_\alpha 
\end{align}
admits a completion $R_\Phi$ consisting of formal sums 
\begin{align}
	\sum_{\alpha \in S} r_\alpha
\end{align}
where $r_\alpha \in R_\alpha$ and $S$ is a Laurent subset of $\Delta$. As shown in \cite[\S5.2]{Bridgeland-Curve-Counting}, the completion $R_\Phi$ has a natural topological ring structure. In particular, this construction yields a completion $H(\scrC_{\leqslant 1})_\Phi$, which has a natural $K(\mathsf{St}/\C)$-algebra structure. 

We next review the construction of the integration map, which is a central element in the wallcrossing theory of Donaldson-Thomas invariants. This is defined on a certain subquotient of $H(\scrC_{\leqslant 1})$ and extends by completion to a subquotient of $H(\scrC_{\leqslant 1})_\Phi$. Following \cite[\S5.1]{Bridgeland-motivic-Hall}, let $H_{\mathsf{reg}}(\scrC_{\leqslant 1})\subset H(\scrC_{\leqslant 1})$ be $K(\mathsf{Sch}/\C)[\LL^{-1}]$-submodule spanned by equivalence classes $[\calX, f]$ with $\calX$ a scheme of finite type. Then let 
\begin{align}\label{eq:scHall}
	H_{\mathsf{sc}}(\scrC_{\leqslant 1}) = H_{\mathsf{reg}}(\scrC_{\leqslant 1})/ (\LL-1) H_{\mathsf{reg}}(\scrC_{\leqslant 1})\ .
\end{align}
Then \cite[Theorem~5.1]{Bridgeland-motivic-Hall} shows that the Hall product $\ast$ induces a 
commutative algebra structure, as well as a Poisson bracket 
\begin{align}
	\{ [\calX, f],\, [\calY,g]\} = \frac{[\calX, f]\ast [\calY,g]-[\calY,g]\ast [\calX, f]}{\LL-1} \, \bmod (\LL-1)
\end{align}
on the quotient \eqref{eq:scHall}. This is by definition the semi-classical Hall algebra associated to the abelian category $\scrC_{\leqslant 1}$. 

Next, let $\C[\Delta]$ be the $\C$-vector space with basis $x^\alpha$, $\alpha \in \Delta$. Note that the assignment 
\begin{align}
	(x^\alpha_1, \ x^{\alpha_2}) \longmapsto x^{\alpha_1+\alpha_2} 
\end{align}
defines a commutative algebra structure on $\C[\Delta]$ so that the associated Poisson bracket is trivial. Then, by \cite[Theorem~5.2]{Bridgeland-motivic-Hall}, one has the following.
\begin{theorem}\label{thm:integrationA} 
	Suppose $Y$ is $K$-trivial. There exists a unique homomorphism $I \colon H_{\mathsf{sc}}(\scrC_{\leqslant 1}) \to \C[\Delta]$ so that 
	\begin{align}\label{eq:intmapA}
		I([\calX, f]) = \chi(\calX) \cdot x^\alpha 
	\end{align}
	for any stack function $[\calX, f]\in H_{\mathsf{reg}}((\scrC_{\leqslant 1})$ of degree $\alpha \in \Delta$, where $\chi(\calX)$ is the topological Euler characteristic of $\calX$. 
\end{theorem}

\begin{remark}\label{rmk:CYcondition} 
	Note that the only way $K$-triviality is used in the proof of  \cite[Theorem~5.2]{Bridgeland-motivic-Hall} is through Serre duality, which implies identity 
	\begin{align}
		\chi(\calE,\, \calF) = \dim\, \Ext^0_Y(\calE,\calF) - \dim\, \Ext^1_Y(\calE,\calF) - \left(\dim\, \Ext^0_Y(\calF,\calE) - \dim\, \Ext^1_Y(\calF,\calE) \right)
	\end{align}
	for any coherent sheaves $\calE, \calF$ on $Y$. 
\end{remark}

Although  $Y$ is not $K$-trivial in the present context, Theorem~\ref{thm:integrationA} admits a weaker version which suffices for our purposes. Note that one still has a unique linear map $I\colon H_{\mathsf{sc}}(\scrC_{\leqslant 1}) \to \C[\Delta]$ satisfying condition \eqref{eq:intmapA}, which is not in general a homomorphism of Poisson algebras. However, given Corollary~\ref{cor:pairing}, in complete analogy with \cite[Theorem~5.2]{Bridgeland-motivic-Hall}, one has:
\begin{proposition}\label{prop:integrationB} 
	Let $[\calX_i,f_i]$, $1\leq i \leq 2$ be homogeneous elements of $H_{\mathsf{reg}}(\scrC_{\leqslant 1})$ of degrees $\alpha_i \in \Delta$, $1\leq i \leq 2$. Suppose that either $\alpha_1$ or $\alpha_2$ belongs to $\N[\Sigma]\times \Z$. Then 
	\begin{align}\label{eq:intmapB} 
 		I(\{ [\calX_1,f_1],\ [\calX_2,f_2]  \}) = 0\ .
	 \end{align} 
\end{proposition}

Finally, as in \cite[\S5.3]{Bridgeland-Curve-Counting}, note that the construction of the semi-classical Hall algebra and the integration map carries over to the completion $H(\scrC_{\leqslant 1})_\Phi$. Hence, for $K$-trivial threefolds, one obtains a Poisson algebra homomorphism 
\begin{align}
	I_\Phi \colon H_{\mathsf{sc}}(\scrC_{\leqslant 1})_\Phi\longrightarrow \C[\Delta]_\Phi\ .
\end{align}
In the present context, this map is not a Poisson algebra homomorphism, it satisfies the identity 
\begin{align}
	I_\Phi(\{\xi_1,\, \xi_2\}) = 0 
\end{align}
for any elements 
\begin{align}
	\xi_i = \sum_{\alpha\in S_i} \xi_{i, \alpha} \in H_{\mathsf{reg}}(\scrC_{\leqslant 1})_\Phi
\end{align}
provided that either $S_1$ or $S_2$ is contained in $\N[\Sigma] \times \Z$. 

In conclusion, note that proofs of the wallcrossing identities in \cite{Bridgeland-Curve-Counting}, \cite{BS-Curve-counting-crepant}, and \cite{Calabrese-DT-Flops} employ the following strategy. First, the logical structure of the proof is delineated through wallcrossing identities in the Hall algebra $H_{\infty}(\scrC_{\leqslant 1})$ of infinite type. Next, it is shown that each of these identities yields, by a suitable truncation, an identity in the  completion $H(\scrC_{\leqslant 1})_\Phi$. The final step converts the resulting Laurent stack function identities into enumerative wallcrossing formulas through the integration map $I_\Phi$. We will apply the same strategy in the proof of Theorem~\ref{thm:C-framed-identity} below, using Proposition \ref{prop:integrationB} for the required properties of the integration map. 

\subsection{$C$-framed stack functions}\label{subsec:Cframedsf} 

The purpose of this section is to construct the stack functions needed in the proof of Theorem~\ref{thm:C-framed-identity}, using $C$-framed conditions as in \ref{def:C-framed-pair}. 

\subsubsection{Subcategories, pairs, and moduli stacks}\label{subsubsec:cat-mod} 

\begin{definition}\label{def:algebraic-sub-category}
	We say that a full subcategory $\scrB\subset \scrC_{\leqslant 1}$ is \textit{algebraic} if flat families of objects of $\scrB$ form an algebraic substack $\Coh(\scrB)\subset\Coh_{\leqslant 1}(Y)$, locally of finite type over $\C$. 
\end{definition} 
Examples include
\begin{align}
	\scrC_\Sigma\ , \quad \scrT_f\ ,\quad \scrF_f\cap \catCoh_{\leqslant 1}(Y)\ ,
\end{align}
where $\scrC_\Sigma\subset \scrC_{\leqslant 1}$ is the full abelian subcategory consisting of coherent sheaves $\calF$ defined by the condition $\ch_2(\calF) \in \N[\Sigma]$. Note that $\Coh(\scrT_f)$ and $\Coh_{\leqslant 1}(\scrF_f)$ are defined as in \cite[Construction~3.17]{DPS-torsion-pairs} and they are algebraic since the torsion pair $(\scrT_f, \scrF_f)$ is open by Corollary~\ref{cor:open-torsion-pair}.

Fixing topological invariants $\alpha\in N_{\geqslant 1}(Y)$, one obtains an open and closed substack $\Coh(\scrB;\alpha)\subset \Coh(\scrB)$. Similarly, fixing a curve class $\beta \in N_1(Y)$, one obtains an open and closed substack $\Coh(\scrB;\beta)\subset \Coh(\scrB)$ consisting of flat families of objects $\calF\in \scrB$ with $\ch_2(\calF)=\beta$. 

Furthermore, let $\Coh^\scrO_{\leqslant 1}(Y)$ be the moduli stack of pairs $(\calF,s)$, with $s\colon\scrO_Y\to \calF$ an arbitrary section. Given an algebraic full subcategory 
$\scrB \subset \scrC_{\leq 1}$, let  $\Coh^\scrO(\scrB)$ be the substack defined by the fiber product 
\begin{align}
	\begin{tikzcd}[ampersand replacement=\&]
		\Coh^\scrO(\scrB)\ar[r] \ar[d] \& \Coh^\scrO_{\leqslant 1}(Y)\ar[d] \\
		\Coh(\scrB)\ar[r] \& \Coh_{\leqslant 1}(Y) 
	\end{tikzcd}\ ,
\end{align}
where the vertical arrows are the natural forgetful morphisms. The stack $\Coh^\scrO_{\leqslant 1}(Y)$ is algebraic, locally of finite type over $\C$, and the projection to $\Coh_{\leqslant 1}(Y) $ is representable, of finite type\footnote{All these statements can be proved as in \cite[Lemma~2.4]{Bridgeland-Curve-Counting} or by using our derived approach via the derived stack $\derivedPerf(Y)$ of perfect complexes on $Y$. Both approaches work also to prove similar geometric properties of all the other stacks introduced in this section.}. Then, the same holds for $\Coh^\scrO(\scrB)$ and the projection to $\Coh(\scrB)$. The associated stack functions will be denoted by 
\begin{align}
	\mathbf{1}_\scrB\ ,\ \mathbf{1}^\scrO_\scrB\in H_{\infty}(\scrC_{\leqslant 1})\ ,
\end{align}
respectively. 

Furthermore, for any algebraic full subcategory $\scrB \subset \scrC_{\leqslant 1}$ 
let $\scrB_\Sigma \coloneqq \scrB \cap \scrC_\Sigma$ and set
\begin{align}
	\Coh(\scrB_\Sigma)\coloneqq \Coh(\scrB)\times_{\Coh_{\leqslant 1}(Y)} \Coh(\scrC_\Sigma)\ .
\end{align}
The associated stack functions will be denoted by $\mathbf{1}_{\scrB_\Sigma}$ and $\mathbf{1}^\scrO_{\scrB_\Sigma}$, respectively. 

\subsubsection{$C$-framing}\label{subsubsec:C-framed-stack-functions} 

\begin{definition}\label{def:C-framed-sub-category}
	Given an algebraic subcategory $\scrB\subseteq \scrC_{\leqslant 1}$ as in Definition~\ref{def:algebraic-sub-category}, the subcategory $\scrB_C$ of $C$-framed objects is the full subcategory of all sheaves $\calF\in \scrB$ so that 
	\begin{itemize}\itemsep0.2cm
		\item $\ch_2(\calF) = [C] +m [\Sigma]$ for some $m \geq 0$, and
		\item the fundamental cycle $[\calF]$ in the Chow variety coincides with $C + m \Sigma$. 
	\end{itemize}
	When $\scrB=\scrC_{\leqslant 1}$, we denote by $\scrC_C$ the subcategory $\scrB_C$.
\end{definition}
The goal of this section is to construct stack functions associated to $C$-framed subcategories by analogy to \S\ref{subsec:C-framed-moduli-stack}. 

Recall that $\Coh_{\leqslant 1}^{\mathsf{ext}}(Y)$ denotes the moduli stack of three term exact sequences $0\to \calE_1\to \calE \to \calE_2 \to 0$ in $\scrC_{\leqslant 1}$, and
\begin{align}
	\pi_1,\pi_2, \pi\colon \Coh_{\leqslant 1}^{\mathsf{ext}}(Y)\longrightarrow \Coh_{\leqslant 1}(Y)
\end{align}
denote the three natural projections. Moreover, $\Coh_C(Y) \subset \Coh_{\leqslant 1}(Y)$ is the closed substack of pure one-dimensional sheaves $\calF$ on $Y$ with scheme-theoretic support on $C$, so that $\ch_2(\calF)=[C]$. 
The stack $\Coh_C^{\mathsf{ext}}(Y)$ is defined by the fiber product 
\begin{align}
	\begin{tikzcd}[ampersand replacement=\&]
		\Coh_C^{\mathsf{ext}}(Y)\ar[r] \ar[d]\& \Coh_{\leqslant 1}^{\mathsf{ext}}(Y)\ar{d}{(\pi_1, \pi_2)}  \\
		\Coh(\scrC_\Sigma)\times \Coh_C(Y)\ar[r]\& \Coh_{\leqslant 1}(Y)\times \Coh_{\leqslant 1}(Y)
	\end{tikzcd}\ .
\end{align}
Then the stack function $\mathbf{1}_{\scrC_C}$ is defined by the composition
\begin{align}
	\mathbf{1}_{\scrC_C}\coloneqq \begin{tikzcd}[ampersand replacement=\&]
		\Big[\pi_C\colon \Coh_C^{\mathsf{ext}}(Y)\arrow{r} \&  \Coh_{\leqslant 1}^{\mathsf{ext}}(Y)\arrow{r}{\pi} \& \Coh_{\leqslant 1}(Y)\Big]
	\end{tikzcd}\ .
\end{align}

More generally, given any algebraic full subcategory $\scrB\subset \scrC_{\leqslant 1}$, let $\Coh_C^{\mathsf{ext}}(\scrB)$ be defined by the fiber product
\begin{align}
	\begin{tikzcd}[ampersand replacement=\&]
		\Coh_C^{\mathsf{ext}}(\scrB)\ar[r] \ar[d]\& \Coh_C^{\mathsf{ext}}(Y)\ar{d}{\pi_C}\\
		\Coh(\scrB) \ar[r]\& \Coh_{\leqslant 1}(Y)
	\end{tikzcd}\ .
\end{align}
Let also $\Coh_C^{\scrO,\, \mathsf{ext}}(\scrB)$ be the stack defined by the fiber product
\begin{align}
	\begin{tikzcd}[ampersand replacement=\&]
		\Coh_C^{\scrO,\, \mathsf{ext}}(\scrB) \ar[r] \ar[d]\& \Coh^\scrO_{\leqslant 1}(Y) \ar{d}{q}\\
		\Coh_C^{\mathsf{ext}}(\scrB)\ar{r}{\pi_C} \& \Coh_{\leqslant 1}(Y)
	\end{tikzcd}\ .
\end{align}
Clearly, $\Coh_C^{\mathsf{ext}}(Y)$ is algebraic, of locally finite type over $\C$ and the same holds for $\Coh_C^{\scrO,\, \mathsf{ext}}(Y)$. Similarly results hold for $\Coh_C^{\mathsf{ext}}(\scrB)$ and $\Coh_C^{\scrO,\, \mathsf{ext}}(\scrB)$. Then, let 
\begin{align}\label{eq:1-B-C}
	\mathbf{1}_{\scrB_C} & \coloneqq \Big[\Coh_C^{ \mathsf{ext}}(\scrB)\longrightarrow \Coh_{\leqslant 1}(Y)\Big]\ ,\\
\mathbf{1}^\scrO_{\scrB_C} & \coloneqq \Big[\Coh_C^{\scrO,\, \mathsf{ext}}(\scrB) \longrightarrow \Coh_{\leqslant 1}(Y)\Big]
\end{align}
denote the associated stack functions, where the structural maps are induced by the central projection $\pi\colon \Coh_{\leqslant 1}^{\mathsf{ext}}(Y) \to \Coh_{\leqslant 1}(Y)$. 

For future reference, note the following variant of Lemma~\ref{lem:support-sequence}.
\begin{lemma}\label{lem:sf-A}
	\hfill
	\begin{enumerate}\itemsep0.2cm
		\item \label{item:sf-A-1} Let $\calF \in \scrC_C$ be a $C$-framed one-dimensional coherent sheaf on $Y$. Then, there is a unique exact sequence 
		\begin{align}
			0\longrightarrow \calF_\Sigma \longrightarrow \calF \longrightarrow \calF_C \longrightarrow 0\ ,
		\end{align}
		where $\calF_\Sigma \in \scrC_\Sigma$ and $\calF_C$ is the pushforward of a rank one torsion-free sheaf on $C$.
		
		\item \label{item:sf-A-2} Moreover, if $\calF\in \scrC_C$ is the structure sheaf of a closed one-dimensional subscheme, then $\calF_C = \scrO_C$ and the composition 
		$\scrO_Y \to \calF \to \calF_C$ coincides with the natural epimorphism $\scrO_Y \to \scrO_C$. 
	\end{enumerate}
\end{lemma} 

\begin{proof}
	We start by proving \eqref{item:sf-A-1}. Let $\calF_C \coloneqq \calF\otimes \scrO_C/\calT$ be the quotient by the maximal zero-dimensional subsheaf, and let $\calF_\Sigma$ be the kernel of the natural surjection $\calF\to \calF_C$. Since $\calF$ is $C$-framed, the claim is obvious. 

	Now, we consider \eqref{item:sf-A-2}. Since $\scrO_Y \to \calF$ is surjective, the composition $\scrO_Y \xrightarrow{s} F \to \scrF_C$ must be also surjective. Since $\calF_C$ is the pushforward of a rank one torsion-free sheaf on $C$, the claim follows. 
\end{proof}

Using the notation in Lemma~\ref{lem:sf-A}, one defines a notion of $C$-framed section of a $C$-framed sheaf as follows. 
\begin{definition}\label{def:C=framed-s} 
	Let $\calF$ be a $C$-framed object of $\scrC_{\leqslant 1}$. We say that a section $s\colon \scrO_Y \to \calF$ is \textit{$C$-framed} if the composition 
	\begin{align}
		\begin{tikzcd}[ampersand replacement=\&]
			\scrO_Y \arrow{r}{s} \& \calF \arrow{r} \& \calF_C 
		\end{tikzcd}
	\end{align}
	is not identically zero. 
\end{definition} 

For any algebraic subcategory $\scrB\subset \scrC_{\leqslant 1}$, let $\tensor*[^\circ]{\Coh}{^{\scrO,\, \mathsf{ext}}_C}(\scrB) \subset \Coh_C^{\scrO,\, \mathsf{ext}}(\scrB)$ be the open substack consisting of pairs $(\calF,s)$ so that $s$ is $C$-framed. The  associated stack function will be denoted by 
\begin{align}\label{eq:circ-1-B-C}
	\tensor*[^\circ]{\mathbf{1}}{_{\scrB_C}} \coloneqq \Big[\tensor*[^\circ]{\Coh}{^{\scrO,\, \mathsf{ext}}_C}(\scrB) \longrightarrow \Coh_{\leqslant 1}(Y)\Big]\in H_\infty(\scrC_{\leqslant 1})\ .
\end{align}

\begin{remark}\label{rem:QC} 
	Note that in general $\calF \in \scrB_C$ does not imply that $\calF_C$ and $\calF_\Sigma$ belong to $\scrB$. However, this holds for $\scrB = \scrF_f$ since $\scrF_f$ is closed under subobjects, and $\calF_C\in \scrF_f$ for any sheaf $\calF\in \scrC_C$ since $\Hom_Y(\scrT_f,\calF_C)=0$ by support 
	conditions.
\end{remark} 

\subsubsection{$C$-framed Hilbert schemes}\label{subsubsec:C-framed-hilbert}

For any $m\geq 0$, let $\Hilb(Y;[C]+m[\Sigma])$ denote the Hilbert scheme of closed one-dimensional subschemes $W\subset Y$ with $\ch_2(\scrO_W) =[C]+m[\Sigma]$. Then, let $\calH_C(Y;m)$ be the stack defined by the fiber product 
\begin{align}
\begin{tikzcd}[ampersand replacement=\&]
		\calH_C(Y;m) \ar[d] \ar[r] \& \Hilb(Y; [C]+m[\Sigma]) \ar[d]\\
		\Coh_C^{ \mathsf{ext}}(Y; m)\ar{r}{\pi_C} \& \Coh_{\leqslant 1}(Y)
\end{tikzcd}\ ,
\end{align}
where the bottom horizontal arrow is induced by the central projection $\pi\colon \Coh_{\leqslant 1}^{ \mathsf{ext}}(Y) \to \Coh_{\leqslant 1}(Y)$. The resulting stack function will be denoted by 
\begin{align}\label{eq:C-Hilbert} 
	\mathbf{1}_{\calH_C} \coloneqq  \sum_{m \geq 0} \Big[\calH_C(Y;m)\longrightarrow \Coh_{\leqslant 1}(Y) \Big]\ ,
\end{align}
the structural morphisms being again induced by the central projection. 
\begin{remark}\label{rem:Hilbert-C} 
	Note that $\calH_C(Y;m)$ is an algebraic space of locally finite type. Moreover, by Lemma~\ref{lem:sf-A}--\eqref{item:sf-A-2}, the 
	groupoid associated to a parameter scheme $Z$ consists of $Z$-flat commutative diagrams 
\begin{align}
	\begin{tikzcd}[ampersand replacement=\&]
		\& \& \scrO_{Z\times Y} \ar[d] \ar[dr] \& \&\\
		0\ar[r] \& \scrF_\Sigma \ar[r] \& \scrF \ar[r] \& \scrO_{Z\times C} \ar[r] \& 0
	\end{tikzcd}\ ,
\end{align}
	where the row is exact, the morphism $ \scrO_{Z\times Y}\to \scrF$ is surjective, and the morphism $\scrO_{Z\times Y}\to \scrO_{Z\times C}$ is the canonical epimorphism.
\end{remark} 

\subsubsection{$C$-framed stable pairs}\label{subsubsec:C-framed-stable-pairs} 

Let $\tensor*[^{\mathsf{sp}}]{\Coh}{^{\scrO,\, \mathsf{ext}}_C}(Y)$ be the stack defined by the fiber product
\begin{align}
	\begin{tikzcd}[ampersand replacement=\&]
		\tensor*[^{\mathsf{sp}}]{\Coh}{^{\scrO,\, \mathsf{ext}}_C}(Y)\ar[r] \ar[d]\& \SP(Y)\ar{d}{q}\\
		\Coh_C^{\mathsf{ext}}(Y)\ar{r}{\pi} \& \Coh_{\leqslant 1}(Y)
	\end{tikzcd}\ ,
\end{align}
where $\SP(Y)$ denotes the moduli space of Pandharipande-Thomas stable pairs on $Y$. $\tensor*[^{\mathsf{sp}}]{\Coh}{^{\scrO,\, \mathsf{ext}}_C}(Y)$ is an algebraic space, locally of finite type over $\C$. 

We introduce the stack functions
\begin{align}
	\mathbf{1}_{\SP_C} &\coloneqq \Big[\tensor*[^{\mathsf{sp}}]{\Coh}{^{\scrO,\, \mathsf{ext}}_C}(Y)\longrightarrow \Coh_{\leqslant 1}(Y)\Big]\ ,\\
	\mathbf{1}_{\fSP_C} &\coloneqq \Big[\fSP_C(Y)\longrightarrow \Coh_{\leqslant 1}(Y)\Big]\ ,
\end{align}
where the moduli stack $\fSP_C(Y)$ has been introduced in \S\ref{subsec:C-framed-moduli-stack}.

\subsubsection{Exceptional Hilbert scheme} \label{subsubsec:exception-Hilbert-scheme}

Let $\calH_{\mathsf{ex}}(Y)$ denote the Hilbert scheme parametrizing closed subschemes $W\subset Y$ so that $\ch_2(\scrO_W) \in \N[\Sigma]$. By Corollary~\ref{cor:split-support}, this implies that the quotient $\scrO_W/\calT$ by the maximal zero-dimensional subsheaf of $\scrO_W$ is either zero or set-theoretically supported on $\Sigma$. The associated stack function will be also denoted by $\mathbf{1}_{\calH_{\mathsf{ex}}(Y)}$. 
 
\subsection{Motivic Hall algebra identities}\label{subsec:hall-identity}

\begin{lemma}\label{lem:sf-C}
	Let 
	\begin{align}
		0\longrightarrow \calF_1 \longrightarrow \calF\longrightarrow \calF_2 \longrightarrow 0 
	\end{align}
	be an exact sequence in $\scrC_{\leqslant 1}$ where $\calF_2$ belongs to $\scrC_\Sigma$. Let $\calT_1 \subset \calF_1\otimes i_{C, \, \ast}\scrO_C$ and $\calT\subset \calF\otimes i_{C, \, \ast}\scrO_C$ be the maximal zero-dimensional subsheaves, respectively. Set $\calF_{1,\, C}\coloneqq \calF_1\otimes i_{C, \, \ast}\scrO_C/\calT_1$ and $\calF_C \coloneqq \calF\otimes i_{C, \, \ast}\scrO_C/\calT$.
	\begin{enumerate}\itemsep0.2cm
		\item \label{item:sf-C-1} Assume that $\calF_1$ is $C$-framed. Then $\calF$ is $C$-framed as well, and the injection $\calF_1\to \calF$ yields an injection $\calF_{1,\, C} \to \calF_C$. Moreover, suppose $s_1\colon \scrO_Y \to \calF_1$ is a $C$-framed section. Then the composition 
		\begin{align}
			\begin{tikzcd}[ampersand replacement=\&]
				s\colon \scrO_Y \arrow{r}{s_1} \& \calF_1 \arrow{r} \& \calF
			\end{tikzcd}
		\end{align}
		is also $C$-framed. 
		\item \label{item:sf-C-2} Conversely, assume that $\calF$ is $C$-framed. Then $\calF_1$ is also $C$-framed and the injection $\calF_1\to \calF$ yields an injection $\calF_{1,\,C} \to  \calF_C$. 
	\end{enumerate}
\end{lemma} 

\begin{proof}
	We start by proving \eqref{item:sf-C-1}. The canonical morphism $f\colon \calF_1\otimes i_{C, \, \ast}\scrO_C \to \calF\otimes i_{C, \, \ast}\scrO_C$ yields a  commutative diagram 
\begin{align}
	\begin{tikzcd}[ampersand replacement=\&]
		0\ar[r] \&  \calT_1 \ar[r]\ar{d}{t} \& \calF_1\otimes i_{C, \, \ast}\scrO_C\ar[r] \ar{d}{f}\& \calF_{1,\, C}\ar{d}{g} \ar[r]\& 0\\
		0\ar[r] \&  \calT \ar[r] \& \calF\otimes i_{C, \, \ast}\scrO_C\ar[r] \& \calF_{C} \ar[r]\& 0
	\end{tikzcd}\ .
\end{align}
	Since $\ch_2(\calF) \in [C] + \N[\Sigma]$, Corollary~\ref{cor:split-support} shows that its fundamental cycle must be of the form $C' + m\Sigma$ with $[C']=[C]$. By construction, this implies that $\calF_C$ is either a purely one-dimensional sheaf with scheme-theoretic support on $C$ or identically zero. In order to rule out $\calF_C =0$ it suffices to show that $g$ is injective. 
	
	By assumption, $\calF_{1,\, C}$ is the pushforward of a rank one torsion-free sheaf on $C$. At the same time, the snake lemma yields the exact sequence 
	\begin{align}\label{eq:snake-sequence-A}
		0\longrightarrow \ker(t) \longrightarrow \ker(f) \longrightarrow \ker(g)\longrightarrow \mathsf{Coker}(t) \longrightarrow \mathsf{Coker}(f) \longrightarrow \mathsf{Coker}(g) \longrightarrow 0\ ,
	\end{align}
	where $\ker(t)$ and $\mathsf{Coker}(t)$ are zero dimensional. Moreover, one also has the long exact sequence 
	\begin{align}
		\begin{tikzcd}[ampersand replacement=\&]
			\cdots \arrow{r} \& \Tor_1(\calF_2, i_{C, \, \ast}\scrO_C) \arrow{r} \& \calF_1 \otimes i_{C, \, \ast}\scrO_C \arrow{r}{f} \& \calF \otimes i_{C, \, \ast}\scrO_C \arrow{r} \&\calF_2 \otimes i_{C, \, \ast}\scrO_C \arrow{r} \& 0
		\end{tikzcd}\ .
	\end{align}
	Since $\calF_2\in \scrC_\Sigma$, $(\mathsf{Supp}(\calF_2) \cap C)_{\mathsf{red}}$ is zero-dimensional. Hence $\Tor_1(\calF_2, i_{C, \, \ast}\scrO_C)$ is zero-dimensional, which implies that $\ker(f)$ is also zero-dimensional. In conclusion, the exact sequence \eqref{eq:snake-sequence-A} shows that 
	$\ker(g)$ must be zero-dimensional. Since $\calF_{1,\,C}$ is purely one-dimensional, this implies that $\ker(g)=0$, i.e., $g$ is injective as claimed. 
	
	The second part follows from the commutative diagram 
	\begin{align}
		\begin{tikzcd}[ampersand replacement=\&]
			\scrO_Y \ar{r}{s_1}\ar{d}{\id}\& \calF_1\ar[r] \ar[d]\& \calF_{1,\,C}\ar[d]\ar[r]\& 0\\
			\scrO_Y \ar{r}{s} \& \calF\ar[r] \& \calF_{C} \ar[r]\& 0
		\end{tikzcd}\ ,
	\end{align}
	where the vertical arrows are injective. 

	Now, we prove \eqref{item:sf-C-2}. Note that $\Hom_Y(\calF_2, \calF_C)=0$ since $\mathsf{Supp}(\calF_2) \cap C$ is zero-dimensional. Hence, one obtains an injection 
	\begin{align}
		\Hom_Y(\calF, \calF_C) \longrightarrow \Hom_Y(\calF_1, \calF_C) 
	\end{align}
	which implies that the composition $\calF_1 \to \calF\to \calF_C$ is not identically zero. Therefore its image $\calI_1\subset \calF_C$ is purely one-dimensional, scheme-theoretically supported on $C$, while $\calF_C/\calI_1$ is zero-dimensional. This further implies that the kernel $\calK_1\subset \calF_1$ of the map $\calF_1\to \calF_C$ has $\ch_2(\calK_1) \in \N[\Sigma]$, hence $\calK_1 \in \scrC_\Sigma$. 
	This shows that $\calF_1$ is also $C$-framed, and $\calI_1 = \calF_{1,\,C}$ by Lemma~\ref{lem:sf-A}.
\end{proof}

\begin{lemma}\label{lem:sf-D} 
	Let $\calF\in \scrC_{C}$ and let 
	\begin{align}
		0\longrightarrow \calF_\Sigma \longrightarrow \calF \longrightarrow \calF_C \longrightarrow 0 
	\end{align}
	be the exact sequence in Lemma~\ref{lem:sf-A}. Let $s\colon \scrO_Y \to \calF$ be a $C$-framed section and let $s_C\colon \scrO_Y \to \calF_C$ denote the composition $\scrO_Y\xrightarrow{s} \calF \to \calF_C$. Then $\mathsf{Im}(s) \subset \calF$ is isomorphic to the structure sheaf of a unique $C$-framed one-dimensional closed subscheme $Z\subset Y$, and $\mathsf{Coker}(s)$ belongs to $\scrC_\Sigma$.
\end{lemma} 

\begin{proof}
	Since $\mathsf{Im}(s)$ is a one dimensional quotient of $\scrO_Y$ it is clearly isomorphic to $\scrO_Z$ for a unique one-dimensional closed subscheme $Z\subset Y$. Moreover, under the current assumptions, $(\calF_C,s_C)$ is a stable pair, in particular, $\mathsf{Im}(s_C)= \scrO_C$. One also has a commutative diagram with exact rows
	\begin{align}
		\begin{tikzcd}[ampersand replacement=\&]
			0\ar[r]\& \calK \ar[r] \ar[d] \&  \mathsf{Im}(s)\ar[r]\ar[d]\& \mathsf{Im}(s_C)\ar[r] \ar[d] \& 0\\ 
			0\ar[r]\& \calF_\Sigma \ar[r] \& \calF \ar[r] \& \calF_C \ar[r] \& 0
		\end{tikzcd}\ ,
	\end{align}
	where $\calK$ is the kernel of the canonical epimorphism $\mathsf{Im}(s) \to \mathsf{Im}(s_C)$. Then, the snake lemma shows that
	\begin{itemize} \itemsep0.2cm
		\item the morphism $\calK \to \calF_\Sigma$ is injective, hence $\calK\in \scrC_\Sigma$, and 
		\item $\mathsf{Coker}(s)\in \scrC_\Sigma$ since $\calF_\Sigma \in \scrC_\Sigma$, and $\mathsf{Coker}(s_C)$ is zero-dimensional. 
	\end{itemize} 
	This proves the claim. 
\end{proof}

By using Lemmas~\ref{lem:sf-C} and \ref{lem:sf-D}, the next result is the $C$-framed analogue of \cite[Lemma~4.3]{Bridgeland-Curve-Counting}.
\begin{lemma}\label{lem:sf-DB}
	The following identity holds in the stack function algebra $H_\infty(\scrC_{\leqslant 1})$:
\begin{align}\label{eq:sf-identity-K}
	\tensor*[^\circ]{\mathbf{1}}{_{\scrC_C}} =\mathbf{1}_{\calH_C} \ast \mathbf{1}_{\scrC_\Sigma}\ .
\end{align}
\end{lemma}

Next note that any sheaf $\calF\in \scrC_{\leqslant 1}$ fits into a unique exact sequence 
\begin{align}\label{eq:PQ-sequence-A}
	0\longrightarrow \calP \longrightarrow \calF \longrightarrow \calQ \longrightarrow 0\ ,
\end{align}
with $\calP\in\scrT_f$ and $\calQ\in\scrF_f\cap \catCoh_{\leqslant 1}(Y)$. By analogy to \cite[Lemma~4.1]{Bridgeland-Curve-Counting} and \cite[Lemma~55]{BS-Curve-counting-crepant}, this yields:
\begin{lemma}\label{lem:sf-DB-2} 
	The following identity holds in $H_\infty(\scrC_{\leqslant 1})$
	\begin{align}\label{eq:sf-identity-G}
		\mathbf{1}_{\scrC_\Sigma} = \mathbf{1}_{\scrT_f} \ast \mathbf{1}_{(\scrF_f)_\Sigma}\ .
	\end{align}
\end{lemma}

Furthermore, given the exact sequence \eqref{eq:PQ-sequence-A}, one also has: 
\begin{lemma}\label{lem:sf-E} 
	\hfill
	\begin{enumerate}\itemsep0.2cm
		\item \label{item:sf-E-1} $\calF$ is $C$-framed if and only if $\calQ$ is $C$-framed. Moreover, if this is the case, 
		the surjection $\calF\to \calQ$ yields an isomorphism $\calF_C \to \calQ_C$. 
		\item \label{item:sf-E-2} Suppose $\calF$ is $C$-framed and let $s\colon \scrO_Y\to \calF$ be a section. Then $s$ is $C$-framed if and only if 
		the composition $\scrO_Y \xrightarrow{s} \calF \to \calQ$ is nonzero. 
	\end{enumerate}
\end{lemma} 

\begin{proof}
	We start by proving \eqref{item:sf-E-1}.  Suppose $\calF$ is $C$-framed and let 
	\begin{align}
		0\longrightarrow \calF_\Sigma \longrightarrow \calF \longrightarrow \calF_C \longrightarrow 0
	\end{align}
	be the exact sequence in Lemma~\ref{lem:sf-A}. At the same time, By Proposition~\ref{prop:torsion-vs-torsion-free}--\eqref{item:torsion-vs-torsion-free-1}, the set-theoretic support of the subsheaf $\calP\subset \calF$ in \eqref{eq:PQ-sequence-A} is either zero-dimensional or the union of $\Sigma$ and a zero-dimensional set. Therefore $\Hom_Y(\calP,\calF_C)=0$, which implies that the inclusion $\calP \subset \calF$ factors through $\calF_\Sigma \subset \calF$. Hence one obtains a commutative diagram 
	\begin{align}
		\begin{tikzcd}[ampersand replacement=\&]
			0\ar[r]\& \calP\ar[r]\ar{d}{\imath}\& \calF \ar[r]\ar{d}{\id}\&\calQ\ar{d}{\jmath}\ar[r]\& 0 \\
			0\ar[r]\&\calF_\Sigma \ar[r]\& \calF \ar[r]\& \calF_C \ar[r]\& 0
		\end{tikzcd}\ .
	\end{align}
	Then the snake lemma shows that $\jmath$ is surjective and the connecting homomorphism
	\begin{align}\label{eq:deltaisom}
		\delta\colon \ker(\jmath) \longrightarrow \mathsf{Coker}(\imath) 
	\end{align}
	is an isomorphism. 
	
	At the same time, Lemma \ref{lem:support-sequence} implies that
	$\calQ$ fits in an exact sequence 
		\begin{align}\label{eq:Qsuppseq} 
	0\to \calQ_\Sigma \to \calQ \to \calQ_C \to 0 
	\end{align}
	where 
	\begin{itemize} 
	\item $\calQ_\Sigma$ belongs to $\calA_\Sigma$, while
	\item $\calQ_C$ is purely one dimensional, the intersection of its set theoretic support with  $\Sigma$ being zero dimensional. 
	\end{itemize} 
	Since $\mathsf{Coker}(\imath)$ belongs to $\scrC_\Sigma$, the isomorphism \eqref{eq:deltaisom} implies that $\mathsf{Ker}(j)$ is contained in $\calQ_\Sigma$ as a subsheaf. Hence the epimorphism $\calQ\to \calQ_C$ factors through an epimorphism $\calF_C \to \calQ_C$. Since $\calF_C$ is by assumption the extension by zero of a rank one torsion free sheaf on $C$, and $\calQ_C$ is purely one dimensional, it follows that $\calQ_C$ is also the extension by zero of a rank one torsion free sheaf on $C$, or identically zero. In order to conclude the proof it suffices to prove that $\calQ_C$ is not identically zero. This follows from the exact sequences \eqref{eq:PQ-sequence-A} and \eqref{eq:Qsuppseq}, which, using Corollary~\ref{cor:split-support}, yield the identity 
	\begin{align}
		\ch_2(\calQ_C) = \ch_2(\calF_C) = [C]\ . 
	\end{align}

	Conversely, suppose $\calQ$ is $C$-framed and let 
	\begin{align}
		0\longrightarrow \calQ_\Sigma \longrightarrow \calQ\longrightarrow \calQ_C\longrightarrow 0 
	\end{align}
	be the exact sequence in Lemma~\ref{lem:sf-A}. Again, by 
	Proposition~\ref{prop:torsion-vs-torsion-free}--\eqref{item:torsion-vs-torsion-free-1} one has $\Hom_Y(\calP, \calQ_C)=0$. Hence one obtains a commutative diagram 
	\begin{align}
		\begin{tikzcd}[ampersand replacement=\&]
			0\ar[r]\& \calP\ar[r]\& \calF \ar[r]\ar{d}{f}\&\calQ\ar[d]\ar[r]\& 0 \\
			\&\& \calQ_C\ar{r}{\id}\& \calQ_C \ar[r]\& 0
		\end{tikzcd}\ ,
	\end{align}
	where the vertical arrows are surjective. This yields the exact sequence 
	\begin{align}
		0\longrightarrow \calP \longrightarrow \ker(f) \longrightarrow \calQ_\Sigma\longrightarrow 0\ ,
	\end{align}
	which shows that the central term belongs to $\scrC_\Sigma$. Therefore, $\calF$ is $C$-framed by Lemma~\ref{lem:sf-A}, and the surjection $\calF\to \calQ_C$ coincides with the canonical surjection $\calF\to \calF_C$. 

	Now, we prove \eqref{item:sf-E-2}. The first part shows that $\calQ$ is also $C$-framed and there is a commutative diagram 
	\begin{align}
		\begin{tikzcd}[ampersand replacement=\&]
			\calF \ar[r]\ar[d]\& \calQ\ar[d] \\
			\calF_C \ar[r]\&\calQ_C 
		\end{tikzcd}\ ,
	\end{align}
	where the bottom horizontal arrow is an isomorphism. Then it is clear that $s$ is $C$-framed if and only if the composition 
	\begin{align}
		\begin{tikzcd}[ampersand replacement=\&]
			\scrO_Y \arrow{r}{s} \& \calF \arrow{r} \& \calQ 
		\end{tikzcd}
	\end{align}
	is also $C$-framed. In order to finish the proof, note that any nonzero section $\scrO_Y\to \calQ$ must be $C$-framed since $H^0(Y, \calQ_\Sigma) =0$ by Proposition~\ref{prop:torsion-vs-torsion-free}--\eqref{item:torsion-vs-torsion-free-2}.
\end{proof}

Now consider the stack functions $\mathbf{1}_{(\scrF_f)_C}$ and $\tensor*[^\circ]{\mathbf{1}}{_{(\scrF_f)_C}}$ as introduced in \eqref{eq:1-B-C} and \eqref{eq:circ-1-B-C}, respectively, with $\scrB=(\scrF_f)_C$. Lemma~\ref{lem:sf-E}--\eqref{item:sf-E-1} yields the following result, which the $C$-framed variant of \cite[Lemma~55]{BS-Curve-counting-crepant}. 
\begin{corollary}\label{cor:sf-B}
	The following identities hold in $H_\infty(\scrC)$
	\begin{align}\label{eq:sf-identity-H}
	\mathbf{1}_{\scrC_C} = \mathbf{1}_{\scrT_f} \ast \mathbf{1}_{(\scrF_f)_C}\ ,\\[2pt] \label{eq:sf-identity-HB}
	\tensor*[^\circ]{\mathbf{1}}{_{\scrC_C}} = \mathbf{1}_{\scrT_f} \ast  \tensor*[^\circ]{\mathbf{1}}{_{(\scrF_f)_C}}\ . 
	\end{align}
\end{corollary} 

Next recall that $\Coh_C^{\scrO,\, \mathsf{ext}}(\scrF_f)$ is the stack defined by the cartesian diagram 
\begin{align}
	\begin{tikzcd}[ampersand replacement=\&]
		\Coh_C^{\scrO,\, \mathsf{ext}}(\scrF_f)\ar[r] \ar[d] \& \Coh^\scrO_{\leqslant 1}(Y)\ar[d]\\
		\Coh^{\mathsf{ext}}_C((\scrF_f))\ar[r]\& \Coh_{\leqslant 1}(Y)
	\end{tikzcd}\ ,
\end{align}
where the bottom horizontal arrow is determined by the central projection. Therefore, given a parameter scheme $Z$, the groupoid $\Coh_C^{\scrO,\, \mathsf{ext}}(\scrF_f)(Z)$ consists of $Z$-flat diagrams 
\begin{align}
	\begin{tikzcd}[ampersand replacement=\&]
			\& \&\scrO_{Z\times Y} \ar{d}{s}\& \&\\
		0\ar[r] \& \scrQ_\Sigma \ar[r] \& \scrQ \ar[r] \& \scrQ_C \ar[r] \& 0 
	\end{tikzcd}\ ,
\end{align}
where 
\begin{itemize} \itemsep0.2cm
	\item $\scrQ$ is a flat family of objects of $\scrF_f$, 
	\item $\scrQ_\Sigma$ is a flat family of objects of $\scrC_\Sigma$, 
	\item $\scrQ_C$ is the pushforward of a flat family of rank one torsion-free sheaves on $C$, and 
	\item $s$ is an arbitrary section.
\end{itemize} 

Let $\tensor*[^\circ]{\Coh}{_C^{\scrO,\, \mathsf{ext}}}(\scrF_f)$ denote the open substack defined by  $s\neq 0$.  Let $\tensor*[^\circ]{\mathbf{1}}{^\scrO_{(\scrF_f)_C}}$ be the associated stack function. By using Proposition~\ref{prop:torsion-vs-torsion-free} and Lemma~\ref{lem:sf-E}, we obtain the next result, which is the $C$-framed variant of \cite[Lemma~4.2]{Bridgeland-Curve-Counting}, \cite[Lemma~57]{BS-Curve-counting-crepant}. 
\begin{lemma}\label{lem:sf-F} 
	The following identity holds in $H_\infty(\scrC_{\leqslant 1})$
\begin{align}\label{eq:sf-identity-C}
	\tensor*[^\circ]{\mathbf{1}}{^\scrO_{\scrC_C}} = \mathbf{1}^\scrO_{\scrT_f}\ast \tensor*[^\circ]{\mathbf{1}}{^\scrO_{(\scrF_f)_C}} \ .
\end{align}
\end{lemma}

\begin{lemma}\label{lem:sf-FB} 
	Let $(\calF, s\colon \scrO_Y \to \calF)$ be a $C$-framed $f$-stable pair. Let 
	\begin{align}
		0\longrightarrow \calF_\Sigma \longrightarrow \calF \longrightarrow \calF_C \longrightarrow 0 
	\end{align}
	be the canonical sequence in Lemma~\ref{lem:sf-A}. Then $\calF_\Sigma \in (\scrF_f)_\Sigma$ and $s$ is $C$-framed. 
\end{lemma} 

\begin{proof}
	By assumption, $\calF\in\scrF_f\cap \catCoh_{\leqslant 1}(Y)$ is $C$-framed and $\mathsf{Coker}(s) \in \scrT_f$. Since $\scrF_f$ is closed under subobjects, $\calF_\Sigma \in \scrF_f$, hence $\calF_\Sigma \in (\scrF_f)_\Sigma$.  
	
	Next, let $s_C\colon \scrO_Y \to \calF_C$ denote the composition 
	\begin{align}
		\begin{tikzcd}[ampersand replacement=\&]
			\scrO_Y\arrow{r}{s} \& \calF \arrow{r} \& \calF_C
		\end{tikzcd}
	\end{align}
	and note the commutative diagram 
	\begin{align}
		\begin{tikzcd}[ampersand replacement=\&]
			0\ar[r]\& \mathsf{Im}(s)\ar[r]\ar[d]\& \calF \ar[r]\ar[d]\&\mathsf{Coker}(s)\ar[d]\ar[r]\& 0 \\
			0\ar[r]\& \mathsf{Im}(s_C)\ar[r]\& \calF_C \ar[r]\&\mathsf{Coker}(s_C)\ar[r]\& 0
		\end{tikzcd}\ .
	\end{align}
	The snake lemma shows that the left vertical arrow is surjective. Since $\mathsf{Coker}(s)$ belongs to $\scrT_f$, so does $\mathsf{Coker}(s_C)$ since  $\scrT_f$ is closed under quotients. At the same time $\mathsf{Coker}(s_C)$ is set-theoretically supported on $C$. This implies it must be zero-dimensional, hence $s_C \neq 0$. 
\end{proof} 

By using Lemmas~\ref{lem:sf-E} and \ref{lem:sf-FB}, one can prove the following $C$-framed variant of \cite[Lemma~60]{BS-Curve-counting-crepant}.
\begin{lemma}\label{lem:sf-H} 
	The following identity holds in $H_\infty(\scrC_{\leqslant 1})$
	\begin{align}\label{eq:sf-identity-E}
		\tensor*[^\circ]{\mathbf{1}}{^\scrO_{(\scrF_f)_C}}  = \mathbf{1}_{\fSP_C} \ast \mathbf{1}_{(\scrF_f)_\Sigma} \ .
	\end{align}
\end{lemma}

Finally, by analogy to \cite[Lemma~59]{BS-Curve-counting-crepant}, one also has the following.
\begin{lemma}\label{lem:sf-G} 
	The following identity holds in $H_\infty(\scrC_{\leqslant 1})$:
	\begin{align}\label{eq:sf-identity-D}
		\mathbf{1}^\scrO_{\scrT_f} = \mathbf{1}_{\calH_{\mathsf{ex}}} \ast \mathbf{1}_{\scrT_f} \ .
	\end{align}
\end{lemma}

\subsection{Proof of Theorem~\ref{thm:C-framed-identity}}\label{subsec:proof}

First, using the DT/PT correspondence proven in \cite{ST-Hilbert}, it suffices to prove the identity 
\begin{align}\label{eq:framed-pairs-identity-C} 
	f\textrm{-}\mathsf{PT}_C(Y)= \frac{\mathsf{DT}_C(Y)}{\mathsf{DT}_{\mathsf{ex}}(Y)}\ , 
\end{align}
where
\begin{align}
	\mathsf{DT}_C(Y) &\coloneqq \sum_{m\geq 0}\sum_{n\in \Z} \chi(\calH_C(Y; m,n)) q^n Q^m\ , \\[4pt]
	\mathsf{DT}_{\mathsf{ex}}(Y) & \coloneqq \sum_{m\geq 0}\sum_{n\in \Z} \chi(\calH_{\mathsf{ex}}(Y; m, n)) q^n Q^m\ ,
\end{align}
where $\calH_C(Y; m,n)$ and $\calH_{\mathsf{ex}}(Y; m, n)$ are introduced in \S\ref{subsubsec:C-framed-hilbert} and \S\ref{subsubsec:exception-Hilbert-scheme}, respectively.
The proof of the above identity follows the strategy outlined in \S\ref{subsec:motHall}. 
Heuristically, this follows from an wall crossing identity in the infinite type 
Hall algebra $H_\infty(\scrC_{\leqslant 1})$ to be derived below. A rigorous proof is then 
obtained by converting each step of the heuristic derivation into identities in the completion 
$H(\scrC_{\leqslant 1})_\Phi$ and using a suitable integration map. 

The wall crossing identity in $H_\infty(\scrC_{\leqslant 1})$ is obtained as follows. 
First, equation \eqref{eq:sf-identity-G} yields
\begin{align}\label{eq:first-identity}
	\mathbf{1}_{\calH_C} \ast \mathbf{1}_{\scrC_\Sigma} = \mathbf{1}_{\calH_C}\ast \mathbf{1}_{\scrT_f} \ast \mathbf{1}_{(\scrF_f)_\Sigma}\ .
\end{align}
On the other hand, recall the identity \eqref{eq:sf-identity-K}:
\begin{align}
	\mathbf{1}_{\calH_C} \ast \mathbf{1}_{\scrC_\Sigma}=\tensor*[^\circ]{\mathbf{1}}{_{\scrC_C}} \ .
\end{align}
Using identity \eqref{eq:sf-identity-C} in the right-hand-side, one further obtains
\begin{align}
	\mathbf{1}_{\calH_C} \ast \mathbf{1}_{\scrC_\Sigma} = \mathbf{1}^\scrO_{\scrT_f}\ast \tensor*[^\circ]{\mathbf{1}}{^\scrO_{(\scrF_f)_C}} \ .
\end{align}
Now, by identity \eqref{eq:sf-identity-D}, this yields 
\begin{align}
	\mathbf{1}_{\calH_C} \ast \mathbf{1}_{\scrC_\Sigma} = \mathbf{1}_{\calH_{\mathsf{ex}}} \ast \mathbf{1}_{\scrT_f} \ast \tensor*[^\circ]{\mathbf{1}}{^\scrO_{(\scrF_f)_C}} \ .
\end{align}
Finally, using identity \eqref{eq:sf-identity-E} in the right-hand-side one obtains 
\begin{align}\label{eq:second-identity}
	\mathbf{1}_{\calH_C} \ast \mathbf{1}_{\scrC_\Sigma} = \mathbf{1}_{\calH_{\mathsf{ex}}} \ast \mathbf{1}_{\scrT_f} \ast \mathbf{1}_{\fSP_C} \ast \mathbf{1}_{(\scrF_f)_\Sigma} \ .
\end{align}

By comparing \eqref{eq:first-identity} and \eqref{eq:second-identity}, we obtain 
\begin{align}
	 \mathbf{1}_{\calH_C}\ast \mathbf{1}_{\scrT_f} \ast \mathbf{1}_{(\scrF_f)_\Sigma} = \mathbf{1}_{\calH_{\mathsf{ex}}} \ast \mathbf{1}_{\scrT_f} \ast \mathbf{1}_{\fSP_C} \ast \mathbf{1}_{(\scrF_f)_\Sigma} \ .
\end{align}
Assuming $\mathbf{1}_{\scrT_f}$ and $\mathbf{1}_{(\scrF_f)_\Sigma}$ to be invertible, this yields 
\begin{align}
	\mathbf{1}_{\calH_C}=  \mathbf{1}_{\calH_{\mathsf{ex}}} \ast \mathbf{1}_{\scrT_f} \ast \mathbf{1}_{\fSP_C} \ast\mathbf{1}_{\scrT_f}^{-1}\ .
\end{align}
Furthermore, if we assume the existence of an integration map 
$I\colon H_\infty(\scrC_{\leqslant 1})\to \C[\Delta]$, this leads to identity \eqref{eq:framed-pairs-identity-C}. 

As shown in \cite[\S7]{BS-Curve-counting-crepant}, the actual proof makes each step of the above heuristic derivation rigorous in the algebra $H(\scrC_{\leqslant 1})_\Phi$ consisting of Laurent stack functions, as explained in \S\ref{subsec:motHall}. In particular, the Laurent truncation of the stack functions $\mathbf{1}_{\scrT_f}$ and $\mathbf{1}_{\scrF_f}$ are shown to be invertible in \cite[Corollary~77]{BS-Curve-counting-crepant}. The same arguments prove that the Laurent truncation of the stack function $\mathbf{1}_{(\scrF_f)_\Sigma}$ is also invertible. We also note  that the required integration map is provided by Proposition \ref{prop:integrationB}. The details will be omitted since they are completely analogous to \textit{loc.cit.}. 

\section{Openness of orthogonality}\label{sec:openness-orthogonality}

In proving that a torsion pair $(\calT, \calF)$ satisfies openness of flatness, it is often useful to consider a mild variation of the classical semicontinuity theorem.

\begin{warning}
	Contrary to the main body of the paper, we consider the $\infty$-categories of complexes of quasi coherent sheaves $\catQCoh(-)$, of perfect complexes $\catPerf(-)$, of spaces $\calS$, etc. Refer to  \cite[\S1.6]{PS-categorified} for the notation we use. Furthermore, we use the \textit{implicitly derived convention}: given a morphism of derived schemes $f \colon S \to T$, we let $f^\ast \colon \catQCoh(T) \to \catQCoh(S)$ be the \textit{derived} pullback functor, and we let $f_\ast \colon \catQCoh(S) \to \catQCoh(T)$ be the \textit{derived} pushforward. 
	Similarly, all fiber products and tensor products considered in this appendix will be understood in the derived sense.
\end{warning}

\begin{definition}
	Let $f \colon X \to Y$ be a morphism of qcqs derived schemes. We say that $F \in \catPerf(X)$ is \textit{$f$-properly supported} if for every $E \in \catPerf(X)$, $f_\ast(E^\vee \otimes F)$ is a perfect complex on $Y$.
\end{definition}

Let $f \colon X \to Y$ be a morphism of qcqs derived schemes.
For every $S \in \dAff_{/Y}$, let $X_S \coloneqq S \times_Y X$ and let $f_S \colon X_S \to S$ and $p_S \colon X_S \to X$ be the natural projections.
\begin{definition}
	Let $F \in \catPerf(X)$ be a perfect complex and let $a, b \in \Z$ be two integers. The functor $\Phi_{F} \colon \dAff_{/Y} \longrightarrow \calS$
	is given by
	\begin{align}
		\Phi_{F}^{[a,b]}(S) \coloneqq \begin{cases} \ast & \text{if } f_{S, \, \ast} p_S^\ast(F) \textrm{ has tor-amplitude contained in } [a,b]\ , \\[4pt]
			\emptyset & \text{otherwise}\ . \end{cases}
	\end{align}
\end{definition}

Using the natural identification $\dSt_{/Y} \simeq \Sh(\dAff_{/Y}, \tau_{\text{ét}})$, we can see $\Phi_{F}$ as a derived stack equipped with a natural map $\phi \colon \Phi_{F} \to Y$. We have the following.
\begin{proposition}\label{prop:Phi_Fab_open}
	Assume that $F$ is $f$-properly supported.
	Then, the morphism $\phi \colon \Phi_{F}^{[a,b]} \to Y$ is representable by open immersions.
\end{proposition}

\begin{proof}
	Let $g \colon S \to Y$ be a morphism.
	Derived base change yields a canonical identification
	\begin{align}
		f_{S, \, \ast} p_S^\ast(F) \simeq g^\ast f_\ast(F) \ . 
	\end{align}
	Therefore, $\Phi_F^{[a,b]}(S) = \emptyset$ if and only if $g^\ast f_\ast(F)$ has tor-amplitude contained in $[a,b]$.
	Thus, we can replace $f$ by $\id_Y$ and $F$ by $f_\ast(F)$.
	Besides, by definition of tor-amplitude, we see that $\Phi_F^{[a,b]}(S) = \Phi_F^{[a,b]}(\trunc{S})$.
	So we can also replace $Y$ by $\trunc{Y}$ or, equivalently, assume from the very beginning that $Y$ is underived.
	
	Applying \cite[Propositions~6.1.4.4 \& 6.1.4.5]{Lurie_SAG}, we see that $\Phi_F^{[a,b]}(S) = \ast$ if and only if for every geometric point $s \colon \Spec(\kappa(s)) \to S$ one has $\Phi_F^{[a,b]}(\Spec(\kappa(s))) = \ast$.
	Besides, \cite[Corollary~6.1.4.6]{Lurie_SAG} shows as well that the subset $U$ of $Y$ given by those geometric points $s \colon \Spec(\kappa(s)) \to Y$ such that $s^\ast(F)$ has tor-amplitude contained in $[a,b]$ is an open subset.
	Unraveling the definitions, we see that $\Phi_F^{[a,b]}(S) = \ast$ if and only if the morphism $S \to Y$ factors through $U$.
	Therefore, $\Phi_F^{[a,b]} = U$, whence the conclusion.
\end{proof}

\begin{corollary}\label{cor:openness_orthogonality}
	Let $f \colon X \to Y$ be a morphism of qcqs derived schemes.
	Let $F \in \catPerf(X)$ be a perfect complex and let $N \in \Z$ be an integer.
	For every $S \in \dAff_{/Y}$ and every $f_S$-properly supported $G \in \catPerf(X_S)$, the set $U$ of geometric points $t \colon \Spec(\kappa) \to S$ such that
	\begin{align}
		H^i (\Hom_{X_t}(j_t^\ast p_S^\ast(F), j_t^\ast(G)) ) \simeq 0 
	\end{align}
	for every $i \leq N$ is an open subset of $S$. Here, $j_t \colon X_t \longrightarrow X$ is the canonical map from the fiber.
\end{corollary}

\begin{proof}
	Consider the perfect complex $E \coloneqq \calHom_{X_S}( p_S^\ast(F), G)$.
	Inspection reveals that it is $f_S$-properly supported.
	For every $g \colon T \to S$, let $j_g \colon X_T \coloneqq T \times_S X_S \to X_S$ be the induced morphism.
	Then since $p_S^\ast(F)$ is perfect, we have a canonical equivalence
	\begin{align}
		j_g^\ast \calHom_{X_S}(p_S^\ast(F), G) \simeq \calHom_{X_T}(j_g^\ast p_S^\ast(F), j_g^\ast(G)) \ . 
	\end{align}
	Combining this equivalence with \cite[Propositions~6.1.4.4 \& 6.1.4.5]{Lurie_SAG}, we deduce that for $g \colon T \to S$ the following statements are equivalent:
	\begin{enumerate}\itemsep=0.2cm
		\item for every geometric point $t \colon \Spec(\kappa) \to T \to S$, $H^i( \Hom_{X_t}( j_t^\ast p_S^\ast(F), j_t^\ast(G)) ) \simeq 0$ for every $i \leq N$;
		
		\item $f_{T,\, \ast} j_g^\ast(E)$ has tor-amplitude concentrated in cohomological degrees $\geq N$.
	\end{enumerate}
	Observe now that since $E$ is $f_S$-properly supported, $f_{S,\, \ast}(E)$ has tor-amplitude contained in $[n,m]$ for some integers $n,m \in \Z$.
	Set $a \coloneqq \max\{n,N\}$ and $b \coloneqq \max\{m, N\}$.
	Then the above two conditions are further equivalent to $\Phi^{[a,b]}_E(T) = \ast$.
	Applying Proposition~\ref{prop:Phi_Fab_open} to the morphism $f_S \colon X_S \to S$ and to the perfect complex $E$, we deduce that $\Phi^{[a,b]}_E$ is representable by an open subscheme of $S$, which, in virtue of the above discussion, coincides exactly with the subset $U$ of the statement.
	The conclusion follows.
\end{proof}
%

\providecommand{\bysame}{\leavevmode\hbox to3em{\hrulefill}\thinspace}
\providecommand{\MR}{\relax\ifhmode\unskip\space\fi MR }
\providecommand{\MRhref}[2]{%
	\href{http://www.ams.org/mathscinet-getitem?mr=#1}{#2}
}
\providecommand{\href}[2]{#2}

\end{document}